\documentclass[12pt,reqno,twoside]{amsart}

\usepackage{amsmath}
\usepackage{tikz}
\usepackage{tikz-cd}
\usetikzlibrary{intersections}
\usetikzlibrary{calc}
\usepackage{booktabs}
\usetikzlibrary{shapes.geometric}
\usepackage{amssymb}
\usepackage{amscd}
\usepackage{mathabx}
\numberwithin{equation}{section}
\usepackage[original]{imakeidx}
\makeindex 
\usepackage{mathrsfs}
\usepackage{enumitem}

\title[Tremors and horocycles]{Tremors and horocycle dynamics on the moduli space of translation surfaces}
\author{Jon Chaika}
\address{University of Utah, Salt Lake City, United States {\tt chaika@math.utah.edu}} 
\author{John Smillie} 
\address{University of Warwick, United Kingdom {\tt j.smillie@warwick.ac.uk}} 
\author{Barak Weiss}
\address{Tel Aviv University, Tel Aviv, Israel
{\tt barakw@post.tau.ac.il}}
\font\sb = cmbx8 scaled \magstep0 
\font\sn = cmssi8 scaled \magstep0
\font\si = cmti8 scaled \magstep0 

\long\def\comcol#1{{\color{purple}\si #1 }}

\long\def\redtext#1{\ifdraft{{\color{red} #1 }}\else\ignorespaces\fi}

\long\def\combaraknew#1{\ifdraft{{\color{cyan}\sb #1
    }}\else\ignorespaces\fi}
\long\def\combarak#1{\ifdraft{{\color{cyan}\sb #1 }}\else\ignorespaces\fi}

\newcommand\hol{\mathrm{hol}}

\newcommand\trem{\mathrm{trem}}

\newif\ifdraft\drafttrue
\draftfalse
\newcommand\name[1]{\label{#1}{\ifdraft{\sn [#1]}\else\ignorespaces\fi}}

\newcommand\eq[2]{{\ifdraft{\ \tt [#1]}\else\ignorespaces\fi}\begin{equation}\label{#1}{#2}\end{equation}}
\newcommand {\equ}[1]{\eqref{#1}}

\newcommand{\MM}{{\mathcal{M}}}
\newcommand{\tremspace}{{\mathcal{T}}}
\newcommand{\HH}{{\mathcal{H}}}
\newcommand{\opp}{{\mathrm{opp}}}
\newcommand{\SF}{{\mathcal{SF}}}
\newcommand{\HHm}{{\mathcal{H}_{\mathrm{m}}}}
\newcommand{\LL}{{\mathcal{L}}}

\newcommand{\Q}{{\mathbb {Q}}}

\newcommand{\R}{{\mathbb{R}}}

\newcommand{\Dom}{{\mathrm{Dom}}}

\newcommand{\Z}{{\mathbb{Z}}}
\newcommand{\bS}{{\mathbb{S}}}

\newcommand{\Rel}{{\mathrm{Rel}}}

\newcommand{\N}{{\mathbb{N}}}

\newcommand{\diam}{{\operatorname{diam}}}

\newcommand{\dev}{\operatorname{dev}}

\newcommand{\Mod}{\operatorname{Mod}}

\newcommand{\SL}{\operatorname{SL}}

\newcommand {\ignore}[1]  {}

\newcommand{\spa}{{\rm span}}

\newcommand{\EE}{{\mathcal{E}}}

\newcommand{\dist}{{\rm dist}}

\newcommand{\df}{{\, \stackrel{\mathrm{def}}{=}\, }}

\newcommand{\FF}{{\mathcal{F}}}

\newcommand{\s}{{\bf y}}

\newcommand{\til}{\widetilde}

\newcommand{\supp}{{\rm supp}}

\newcommand{\sm}{\smallsetminus}

\newcommand{\vre}{\varepsilon}

\newcommand{\Leb}{{\mathrm{Leb}}}

\newtheorem{thm}{Theorem}[section]

\newtheorem{lem}[thm]{Lemma}

\newtheorem{prop}[thm]{Proposition}

\newtheorem{cor}[thm]{Corollary}

\newtheorem{claim}[thm]{Claim}

\newtheorem{remark}[thm]{Remark}

\newtheorem{example}[thm]{Example}

\newtheorem{dfn}[thm]{Definition}

\usepackage[colorlinks=true,linkcolor=blue]{hyperref}

\begin{document}
\maketitle
\tableofcontents

\begin{abstract}
We
introduce a `tremor' deformation on strata of translation
surfaces. Using it, we give new examples of 
behaviors of horocycle flow orbits $Uq$ in strata of translation surfaces. In the
genus 2 stratum $\HH(1,1)$ we find
orbits $Uq$ which are generic for a measure whose support is strictly
contained in $\overline{Uq}$ and find orbits which are not generic for
any measure. We also describe a horocycle orbit-closure whose Hausdorff
  dimension is not an integer.
\end{abstract}

\section{Introduction}\name{sec: introduction}
A surprisingly fruitful technique for studying certain mathematical objects is to
study dynamics on their moduli spaces. Examples of this
phenomenon occur in the study of integral values of indefinite quadratic
forms (motivating the study of dynamics of Lie group actions on
homogeneous spaces)  and billiard flows on polygonal tables
(motivating the study of the $\SL_2(\R)$-action on the moduli space of
translation surfaces). In both cases, far-reaching results regarding
the actions on the moduli spaces have been used to shed light on a
wide range of problems in number theory, geometry, and ergodic
theory. See \cite{zorich survey, Wright survey, KSS handbook} for surveys
of these developments.  

Let  $B \subset \SL_2(\R)$\index{b@$B$} be the subgroup of upper
triangular matrices, and let  
\eq{eq: hero introduced}{
U \df \{u_s: s
\in \R\}, \ \ \text{ where } \
u_s \df \left( \begin{matrix}1 & s \\ 0 & 1  \end{matrix} \right).
}\index{u@$U$}\index{u@$u_s$}
The $U$-action is an example of a unipotent flow and, in the case of
strata of translation surfaces, is also  
known as the horocycle flow. 
The actions of these groups on moduli spaces are fundamental in both dynamical
settings. For
homogeneous spaces of Lie groups, actions of subgroups such as $\SL_2(\R), B$
and $U$ are strongly constrained and much is known about invariant
measures and orbit-closures. For the action on a stratum $\HH$ 
of translation surfaces, fundamental papers of McMullen, Eskin,
Mirzakhani and Mohammadi \cite{McMullen-SL(2), EM, EMM} have shown that the 
invariant measures and 
orbit closures for the $\SL_2(\R)$-action and $B$-action on $\HH$ are severely
restricted and have remarkable geometric
features; in particular an orbit-closure is the image of a manifold under an
immersion.  This behavior is very much like the behavior observed in the homogeneous setting. 

In this
paper we examine the degree to which such regular behavior might hold for
the $U$-action or  horocycle flow on strata. 
We give examples showing that, with respect to
orbit-closures and the asymptotic behavior of individual orbits, the $U$-action on
$\HH$ has features which are absent in homogeneous dynamics.

In order to set the stage for this comparison we first recall some  results
about the dynamics of  unipotent flows on homogeneous spaces. 
 Special cases of these were proved by several authors and  the
results were proved in complete generality in celebrated work of Ratner (see
\cite{morris book} for a survey, and for the definitions
used in the statement below).

\begin{thm}[Ratner] \name{thm: Ratner}
Let $G$ be a connected Lie group, $\Gamma$ a
  lattice in $G$, $X=G/\Gamma$, and $U = \{u_s: s \in \R\}$ a one-parameter
  Ad-unipotent subgroup of $G$. 
\begin{enumerate}\item
For any $x \in X$, $\overline{Ux}= Hx$ is the orbit of a group $H$
satisfying $U \subset H \subset G$, and $Hx$ is the support of an
$H$-invariant probability measure $\mu_x$.  
\item \label{conc:gen}
Any $x \in X$ is {\em generic for $\mu_x$,} i.e. 
$$
\forall f \in C_c(X), \ \ \lim_{T\to \infty} \frac{1}{T} \int_0^T
f(u_sx)ds = \int_X f d\mu_x.
$$

\end{enumerate}
\end{thm}  
Statement (1) is known as 
the orbit-closure theorem, and statement
(2) is known as 
the genericity theorem.

\subsection{Main results}
We will introduce a method for constructing $U$-orbits with unexpected
properties, and apply it in the genus two stratum
$\HH(1,1)$. 

In the homogeneous setting, orbit-closures of unipotent flows are
manifolds. It was known (see \cite{examples}) that horocycle
orbit-closures could be manifolds with boundary  
in the setting of translation surfaces. We show here that they can be considerably
wilder.

\begin{thm}\name{thm: Hausdorff dim preliminary}
There is $q \in \HH(1,1)$ for which the orbit-closure $\overline{Uq}$
has non-integer Hausdorff dimension.  In fact, by appropriately
varying the initial point, $q$, we can construct an uncountable
nested chain of distinct horocycle orbit-closures of fractional
Hausdorff dimension.  
\end{thm}

We will give additional information about these orbit-closures in
Theorems \ref{thm: spiky fish} and \ref{thm: 
  Hausdorff dim} below.

Let $\EE_4 \subset \HH_1(1,1)$
denote the set of unit-area
surfaces which can be presented as two identical tori glued along a
slit (in the notation and terminology of McMullen \cite{McMullen-SL(2)},
 $\EE_4$ is the subset of area-one surfaces in the eigenform locus of
discriminant $D=4$).

From now on we write $G \df
\SL_2(\R)$\index{G@$G$} and $\EE \df \EE_4$\index{E@$\EE$}. The locus $\EE$ is
5-dimensional, is $G$-invariant, and
is the 
support of a $G$-invariant ergodic probability measure $\mu_\EE$. 
\index{m@$\mu_\EE$}

\begin{thm}\name{thm: 1}
There is $q
\in \HH(1,1)$ which is not contained in $\EE$ but which is  generic
for the measure $\mu_\EE$ supported on $\EE$. 
\end{thm}

Since $\EE = \supp \, \mu_\EE$ is strictly contained in $\overline{Uq}$,
this orbit does not satisfy the analogue of Theorem \ref{thm:
  Ratner}(2). 
The next result shows that the analogue of Ratner's genericity theorem
fails dramatically in $\HH(1,1)$:
\begin{thm}\name{thm: 2}
There is a dense $G_\delta$ subset of $q \in \HH(1,1)$  and $f \in
C_c(\HH(1,1))$ so that \eq{eq: not generic}{
\liminf_{T \to \infty} \frac{1}{T} \int_0^T
f(u_sq)ds 
< 
\limsup_{T \to \infty} \frac{1}{T} \int_0^T
f(u_sq)ds.
}
In particular such points are not generic for any measure on $\HH(1,1)$, 
and there are such points whose forward and backward geodesic trajectories (i.e., in the notation \equ{eq: subgroups notation}, the sets $\{g_tq: t>0\}$ and $\{g_tq: t<0\}$) are both dense.
\end{thm}

One property of unipotent flows on homogeneous spaces
which played a crucial role in Ratner's work is `controlled 
divergence of nearby trajectories'.
The proof of Theorem \ref{thm: 1} 
shows that in strata, divergence of nearby trajectories can be erratic. 
We make this precise in \S \ref{subsec: erratic}, see Theorem
\ref{thm: erratic}.

The proofs of Theorems \ref {thm: Hausdorff dim preliminary}, \ref{thm: 1}, 
 and \ref{thm:
  2} rely on the tremor paths which we now
introduce (the geological nomenclature is inspired by
Thurston's earthquake paths, see \cite{Thurston earthquake}).

\subsection{Tremors}\name{subsec: tremors intro} 
We can describe the action of the horocycle flow on a translation
surface geometrically as giving us a family of surfaces obtained by changing the
flat structure on the original surface by shearing it horizontally.
An interesting modification of this procedure was studied by Alex Wright \cite{Wright
  cylinders}. Let $q \in \HH$, let $M_q$ be the corresponding surface,
and suppose $M_q$ contains a horizontal cylinder $C$. Then
one can deform $M_q$  by horizontally shearing the flat structure on $C$ and leaving
$M_q \sm C$
unchanged. This {\em cylinder shear} \index{cylinder shear} operation
defines a flow on the subset of the stratum corresponding to surfaces
containing a horizontal cylinder.  
This subset of $\HH$ is invariant under the horocycle flow and on it, 
the flow defined by the cylinder shear commutes with the
horocycle flow. 
The tremors we study in this paper are partially defined flows, defined on the set of surfaces whose horizontal foliation is not uniquely ergodic. Tremors  commute with
the horocycle flow on their domains of definition and are a common generalization of both
cylinder shears and the horocycle flow. Wright's analysis of cylinder
shears focused on shears that keep points inside a $G$-invariant
locus. On the other hand, we will study tremors that move 
points in a $G$-invariant locus away from that locus and we will use
these tremors to exhibit new
behaviors of the horocycle flow.

  We can think of both the cylinder shear and the horocycle flow as
  arising from transverse invariant 
  measures  to the horizontal foliation  $\mathcal{F}_q$ on the surface  $M_q$,
   where the amount 
 and location of shearing is  determined by the transverse measure.  
 If the cylinder shear flow takes $q$ to $q'$ then the
  relationship between their period coordinates (see \S \ref{subsec:
    strata} and \S \ref{subsec: atlas of
  charts}, where we will explain the
  notation and make our discussion more precise) is given by 
\eq{eq: horocycle}{
\hol^{(x)}_{q'}(\gamma) =
\hol^{(x)}_{q}(\gamma)+t \cdot \tau(\gamma), \ \ 
\hol^{(y)}_{ q'}(\gamma) = \hol^{(y)}_{q}(\gamma).
}
Here $\hol_q^{(x)}$ and $\hol_q^{(y)}$ denote the
cohomology classes corresponding to the transverse measures $dx$ and
$dy$ on $M_q$ respectively, $\gamma$ is an oriented closed curve or
path joining singularities on $M_q$, 
$t$ is the parameter for the cylinder shear flow, 
and $\tau$ is the cohomology class corresponding to the transverse
measure which is the restriction of $dy$ to the cylinder. 
The
horocycle flow is given in period coordinates as
\eq{eq: horocycle period coordinates}{
\hol^{(x)}_{u_s q}(\gamma) = \hol^{(x)}_{q}(\gamma)+s 
\cdot\hol^{(y)}_{q}(\gamma), \ \ \hol^{(y)}_{u_sq}(\gamma)=
\hol^{(y)}_{q}(\gamma).}
See Figure \ref{fig: horocycle} for an illustration of the geometric meaning of this change in period coordinates.

\begin{figure}
\includegraphics[scale=1.0]{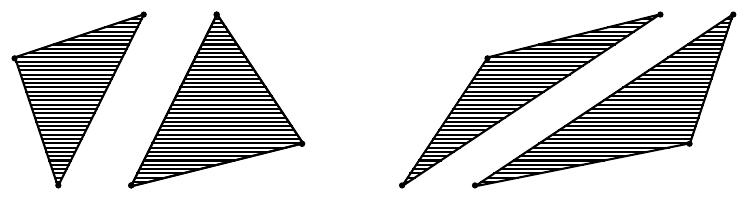}
\caption{The left hand side shows two triangles in $M_q$. The right hand side shows the corresponding triangles in $M_{q'}$ where $q'=u_1(q)$.}\label{fig: horocycle} 
\end{figure}

Recalling that some
surfaces may have additional transverse measures to the horizontal
foliation $\FF_q$, we will define a surface $q'$ via the formula
 \eq{eq: tremor3}{
\hol^{(x)}_{q'}(\gamma) = 
\hol^{(x)}_{q}(\gamma)+t \cdot \beta(\gamma), \ \ 
\hol^{(y)}_{ q'}(\gamma) = \hol^{(y)}_{ q}(\gamma),
}
  where $\beta$ is the cohomology class associated with a transverse
  measure on $\mathcal{F}_q$. In a sense that we will make precise in \S \ref{sec: TCH}, this means that $M_q$ is deformed by shearing nearby horizontal lines relative to each other, where the amount of shearing is specified by $\beta$ and $t$ (see Figure \ref{fig: trem2}). 
We write $\trem_{t,\beta}(q)$ for $q'$ and $\trem_\beta(q)$ for $\trem_{1,\beta}(q)$.
 We refer to a surface of the form $\trem_{t,\beta}(q)$ as a {\em tremor of $q$}.
 As we will show in \S \ref{sec: trem ode} and \S \ref{subsec: for more details},  $q'$ is uniquely determined by $q$, $t$ and $\beta$.  
\index{trem@$\trem_{\beta}(q)$}
\index{trem@$\trem_{t,\beta}(q)$}

 \begin{figure}
\includegraphics[scale=1.0]{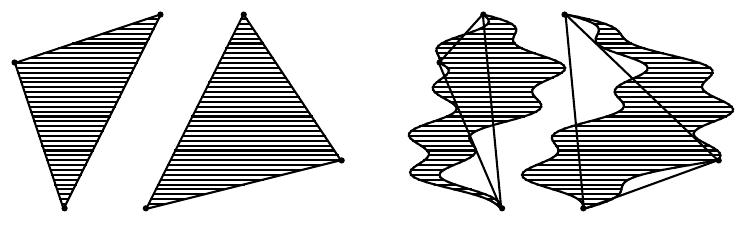}
\caption{ The right hand side shows how the two triangles change with respect to a tremor flow. The periods of the edges change via equation \eqref{eq: tremor3}. 
}\label{fig: trem2} 
\end{figure}
  We now give some additional definitions needed for stating our
  results. 
If the transverse measure corresponding to
$\beta$ is absolutely continuous with respect to $dy$ (see \S
\ref{subsec: ac}) we will say that both $\beta$ and the
tremor $\trem_{\beta}(q)$ are \emph{absolutely continuous.} 
If $q$ has no horizontal saddle connections and  
the transverse measure is not a scalar multiple of $dy$, we will say $\beta$ and $\trem_{\beta}(q)$
are \emph{essential}\index{essential tremor}. We will denote the subspace
of cohomology corresponding 
to signed transverse
measures on $\mathcal{F}_q$ by $\mathcal{T}_q$. This can be related to the tangent space to the
stratum, see \S \ref{subsec: orbifold} and \S \ref{subsec: foliation cycles}. \index{T@$\tremspace_q$}
If the transverse measure is non-atomic, i.e.\ assigns zero  
measure to all horizontal saddle connections or closed leaves, then the tremor path can  
be continued for all time, see Proposition \ref{prop: tremor domain of
  defn}. The case of atomic transverse measures presents   
some technical difficulties which will be discussed in \S \ref{sec:
  atomic tremors}.

\subsection{More detailed results}
The importance of tremor maps for the study of the horocycle flow is
that, where they are defined, they commute  
with the horocycle flow, i.e., $u_s
\trem_{\beta}(q)=\trem_\beta(u_sq)$. 
In particular we will see that for
many tremors, the surfaces 
$u_sq$ and $u_s \trem_{\beta}(q)$ stay  
close to each other, and this leads to the following:

\begin{thm}\name{thm: tremors bounded distance}
Let $\HH$ be any stratum, let $\HH_1$ be its subset of area-one
surfaces, and let $\LL \subset \HH_1$ be a closed $U$-invariant
set which is the support of  a $U$-invariant ergodic measure $\mu$. Let $q
\in \LL$, $\beta\in \tremspace_q$ and $q_1 = \trem_{\beta}(q)$. Then:
\begin{itemize}
\item[(i)]  If   $\beta$ is absolutely continuous then for the sup-norm
 distance $\dist$ on $\HH$ (see \S  
  \ref{subsec: sup norm}), we have
\eq{eq: bounded}{
\sup_{s \in\R} \dist(u_sq, u_sq_1) < \infty. 
}
\item[(ii)] 
If $\beta$ is absolutely continuous
then for any $q'$ in $\overline{Uq_1} \sm \LL$, the surface $M_{q'}$ has a non-uniquely ergodic horizontal
foliation. In particular, if $\LL
\neq \HH_1$ then  $Uq_1$ is not dense in $\HH_1$. 

\item[(iii)] 
 If  $\mu$-a.e. surface in $\LL$ has no horizontal saddle
connection and 
if $q$ is generic for $\mu$, then  $q_1$ is also generic for $\mu$. 
\end{itemize} 
\end{thm}

We will give examples of loci $\LL$ and points $q$ for which the
hypotheses of Theorem \ref{thm: tremors bounded 
  distance} are satisfied, namely we will find $\LL$ and $q$ for which:
\begin{itemize}
\item[(I)]
The locus $\LL$ is $G$-invariant and is the support of a $G$-invariant
ergodic 
measure $\mu$ and the orbit $Uq$ is generic for $\mu$.
\item[(II)]
The surface $M_q$ has no horizontal saddle connections and the
transverse measure corresponding to $dy$
on $M_q$ is not ergodic  (and hence $q$ admits essential absolutely
continuous 
tremors). 
\item[(III)] 
There is an essential absolutely continuous tremor $q_1$ of $q$ which
is not in $\LL$.  
\end{itemize}
There are many examples of strata $\HH$ and loci
$\LL$ for which these properties hold. One particular example which we
will study in detail is $\LL=\EE 
\varsubsetneq \HH_1(1,1)$ (see \S \ref{subsec: locus E} for more
information on $\EE$). 
Namely we will prove the following result which, 
in conjunction with Theorem \ref{thm: tremors bounded distance}, 
immediately implies Theorem \ref{thm: 1}. 
\begin{thm}\name{thm: more detailed}
There are points $q \in \EE$ satisfying (I), (II) and
(III) above. 
Moreover, for any  $q \in \EE$  which admits an essential
tremor $\beta 
\in \tremspace_q$, the points
$$q_r \df 
\trem_{r,\beta}(q)  \in \HH(1,1) \ (\text{where } r >0)$$
satisfy 
\eq{eq: 2nd assertion}{
0< r_1 < r_2 \implies
\overline{Uq_{r_1}} \neq \overline{Uq_{r_2}}.
}
\end{thm}

\begin{remark}\name{remark: other loci} 
Theorem \ref{thm: more detailed} is also true
if $\EE$ is replaced with any of the other eigenform loci $\EE_D
\subset \HH(1,1)$. See \S \ref{subsec: nested} for more details. 
\end{remark}

For certain $q \in \EE$ and  $\beta \in \tremspace_q$, we can 
give a complete description of the 
closure of $Uq_1$ where $q_1 = 
\trem_\beta(q)$. To state this result we will need a
measurement of the size of a tremor and to do this 
we introduce the {\em total variation} $|L|_q(\beta)$ of $\beta \in
\tremspace_q$, see \S \ref{subsec: length} for the definition. 
Also we say
that $q \in \EE$ is {\em aperiodic}\index{aperiodic} if the horizontal foliation of $M_q$ is not
periodic, i.e. it is either minimal or contains a horizontal slit separating the
surface into two tori so that the restriction of the horizontal foliation to each torus is minimal. 

\begin{thm}\name{thm: spiky fish}
For any $a>0$ there is $q_0\in\EE$ and an essential tremor $q_1 = \trem_{\beta_0}(q_0) \in\HH(1,1)$ of
$q_0$ such that 
\eq{eq: spiky fish}{\begin{split}
\overline{Uq_1} & = \overline{
\{\trem_{\beta}(q): q \in \EE \text{ is aperiodic},\  \beta \in
  \tremspace_{q} ,\ |L|_q(\beta)  \leq a\}
} \\
& \subset \{\trem_{\beta}(q): q \in \EE,\  \beta \in
  \tremspace_{q} ,\ |L|_q(\beta)  \leq a\}.
\end{split}
}
Moreover, setting
$q_r \df \trem_{r,\beta_0}(q_0)$, we have that the orbit-closure
$\overline{Uq_r}$ admits the description in \equ{eq:
  spiky fish} with the constant $a$ replaced by $ra$, and the points $q_r$ satisfy the
following strengthening of 
\equ{eq: 2nd assertion}: 
\eq{eq: 3rd assertion}{
  0< r_1 < r_2 \implies
  \overline{Uq_{r_1}} \varsubsetneq \overline{Uq_{r_2}}. }
\end{thm}

The following more explicit result
implies Theorem \ref{thm: Hausdorff dim preliminary}. Its proof relies on \cite{CHM}. 

\begin{thm}\name{thm: Hausdorff dim}
  Let $q_1\in\HH(1,1)$ be the point described in Theorem \ref{thm: spiky fish}. Then
  the Hausdorf dimension of the horocycle orbit closure of $q_1$ satisfies
  $$
5.5 \leq \dim \overline{Uq_1} < 6.
$$
\end{thm}

\subsection{Acknowledgements}
The authors gratefully acknowledge the support of BSF grant 2016256, ISF
grant 2095/15, Wolfson Research Merit Award, NSF grants DMS-135500,
DMS-452762, DMS-1440140 a Warnock chair, and Poincar\'e chair. The
authors thank BIRS-CMO, CIRM, Fields Institute, IHP and
MSRI for their hospitality. The authors
are grateful to Matt Bainbridge and Yair Minsky for
helpful discussions. The authors thank the anonymous referees for their careful reading and excellent comments that greatly improved the accuracy and readability of the paper.

\section{Basics}\name{sec: basics}
In this section we review basic concepts and set up notation. Some readers will find it useful to skip this section on a first reading, and refer back to it as needed.   The main differences between our treatment and other treatments are the attention paid to orbifold loci and the terminology introduced in \S \ref{subsec: transverse}.

\subsection{Strata and period coordinates}\name{subsec: strata}
There are several possible approaches for defining the topology and
geometric structure on strata, see \cite{FM, MT, 
  Wright survey, yoccoz survey, zorich survey}.  For the most part we
follow the approach of \cite{eigenform}, where the reader can find
additional details.

Let $M$  be a compact oriented surface of genus $g$ and 
let $\Sigma \subset M$ be a non-empty finite set with $k$ elements. We make the
convention that the points of $\Sigma$ are labeled as $p_1, \ldots, p_k$.  Let $\mathbf{r}$
be a list of $k$ non-negative integers
satisfying $\sum r_j = 2g-2. $ 
A {\em  translation surface of type $ \mathbf{r}$} is given by an
atlas on $M$ of orientation preserving charts   $\mathcal{A}=\left(
  \psi_\alpha, U_\alpha\right)_{\alpha \in \mathcal{A}}$, where 
the 
$U_\alpha \subset M\sm \Sigma$ are open and cover $M \sm \Sigma$, the
transition maps $\psi_\alpha \circ \psi_\beta^{-1}$ are restrictions of 
translations to the appropriate domains, and such that the planar
structure in a neighborhood of 
each $p_j \in \Sigma$ completes to a cone angle {\em singularity} of
total cone angle $2\pi(r_j +1)$. 
A {\em
  translation equivalence}\index{translation equivalence} between translation surfaces is a
homeomorphism $h$ which preserves the labels and the translation
structure.

These charts determine a metric
on $M$ and a measure which we denote by $\Leb$. 
\index{Leb@$\Leb$} 
These charts also allow us to define natural `coordinate' vector
fields $\partial_x$ and $\partial_y$ and 1-forms $dx$ and $dy$ 
on $M$. The (partially defined) flow corresponding to
$\partial_x$ will be called the {\em horizontal straightline
  flow}\index{horizontal straightline flow}, and we will denote the
trajectory parallel to $\partial_x$ starting at $p \in M_q$ by $ t \mapsto
\Upsilon^{(p)}(t)$\index{U@$\Upsilon^{(p)}_t$}. 
The
corresponding foliation of $M \sm \Sigma$, which we denote by
$\FF$\index{F@$\FF$}, will be called the {\em horizontal
  foliation}. \index{horizontal foliation} 
If we remove from $M$ the horizontal trajectories that hit
singular points,  
then the straightline flow becomes an actual flow defined on a dense
$G_\delta$ subset of
full Lebesgue measure. If this flow is {\em minimal}\index{minimal},
i.e. all infinite horizontal straightline flow trajectories are dense,
we will say that {\em $\FF$ is minimal} or that $M$ is {\em horizontally minimal.}

Fix $\mathbf{r}$ of length $k$, and $g$ satisfying the relation $\sum r_j =
2g-2$. Choose a surface $S$ of genus $g$ and a set $\Sigma
\subset S$ of cardinality $k$, whose elements are labelled by $1, \ldots, k$ (note that we use the same symbol
$\Sigma$ to denote finite subsets of $S$ and of $M$, this should cause no confusion).
We refer to $(S, \Sigma)$ as the {\em model
  surface}.\index{model surface}
A {\em marking map} \index{marking map} of a translation surface $M$
is a homeomorphism $\varphi: (S,\Sigma)\to (M,\Sigma)$ which preserves
labels on $\Sigma$.
We say that two markings maps $\varphi: (S,\Sigma)\to (M,\Sigma)$ and
$\varphi': (S,\Sigma)\to (M',\Sigma)$ are equivalent if there  
is a translation equivalence $h:M\to M'$ so that $h\circ\varphi$ is
isotopic to $\varphi'$ via an isotopy which fixes $\Sigma$. An equivalence class of translation surfaces
with marking maps is a {\em marked translation surface.} There is a
forgetful map which takes a marked translation surface, which is the
equivalence class of $\varphi: S \to M$, to the translation equivalence class of
$M$. We will denote this map by $\pi$ and usually denote an element of
$\pi^{-1}(q)$ by $\til q$. \index{p@$\pi$}
The set of translation self-equivalences of $M$ is a finite group
which we denote by $\Gamma_M$. 
\index{G@$\Gamma_M$} 
In particular we get a left action, by postcomposition,  of
$\Gamma_M$ on the set of marking maps of $M$. 

As we have seen a flat surface structure on $M$ determines two
natural 1-forms $dx$ and $dy$ and these 1-forms determine
cohomology classes in $H^1(M,\Sigma;\R)$ which we denote by
$\hol^{(x)}$ and $\hol^{(y)}$\index{hol@$\hol^{(x)},
  \hol^{(y)}$}. Specifically for an oriented curve $\gamma$ we 
have $\hol^{(x)}(\gamma)=\int_\gamma dx$ and
$\hol^{(y)}(\gamma)=\int_\gamma dy$. We can combine these classes to
obtain an $\R^2$-valued cohomology class $\hol_M=(\hol^{(x)},\hol^{(y)})$ in
$H^1(M,\Sigma;\R^2)$\index{hol@$\hol_M$}. Conversely, any
$\R^2$-valued cohomology class 
gives rise to two $\R$-valued cohomology classes via the
identification $\R^2 = \R \oplus \R$. We denote the corresponding
direct sum decomposition by 
\eq{eq: x and y splitting}{
H^1(M,
\Sigma; \R^2) = H^1(M,
\Sigma; \R_x) \oplus H^1(M,
\Sigma; \R_y).
  }\index{H1@$H^1(M,
    \Sigma; \R_x), \  H^1(M,
\Sigma; \R_y)$}

Now consider a marked translation surface $\til q$ with choice of
marking map $\varphi: (S,\Sigma)\to (M,\Sigma)$, where $M = M_{\til
  q}$ is the underlying translation surface.  In this situation
we have a distinguished element $\hol_{\til q}=\varphi^*(\hol_M) \in
H^1(S, \Sigma ; \R^2)$\index{hol@$\hol_{\til q}$}  given by using the map $\varphi$ to pull back 
the cohomology class $\hol_M$ from $H^1(M, \Sigma ; \R^2)$ to $H^1(S,
\Sigma ; \R^2)$. 
More concretely if $\gamma$ is an oriented curve in $S$ with endpoints in
$\Sigma$ then $\hol_{\til q}(\gamma)=\hol_M(\varphi(\gamma))$. 
The cohomology class $\hol_{\til q}$ is independent of the choice 
 of the particular representative in the equivalence class $\til q$.  
We write $\dev(\til q)$\index{dev@$\dev$} for the cohomology class
$\hol_{\til q}\in H^1(S, \Sigma ; \R^2)$.

\subsection{An atlas of charts on $\HHm$}\name{subsec: atlas of
  charts}
Let $\HHm = \HHm(\mathbf{r})$ 
denote the collection of equivalence classes of marked translation surfaces  of a fixed 
type $\mathbf{r}$. 
Let $\HH = \HH(\mathbf{r})$
denote the collection of
translation equivalence classes of translation surfaces.
We will use the developing map defined above to
equip these sets with a topology, via a local coordinate system which is referred to
as {\em period coordinates}.

We caution the reader that different variants of these definitions can
be found 
in the literature, and they might not be equivalent to our
definitions, specifically as regards the question of whether or not
points of $\Sigma$ are labelled. Our terminology and notation follows \cite{eigenform}, but we introduce some additional notation related to comparison maps, which will be useful in \S \ref{subsec: polygonal tremors} and \S \ref{sec: TCH}. Readers who are familiar with these notions may choose to skip this subsection. 

A  \emph{geodesic triangulation}\index{geodesic triangulation} of a
translation surface is 
a decomposition of the surface
into triangles whose sides are saddle connections, and whose vertices
are singular points, which need not be distinct.
The existence of a geodesic triangulation of any translation surface
is proved in \cite[\S 4]{MS}. 
Let $\varphi:(S,\Sigma)\to (M,\Sigma)$ be a marking map, let $\til q$
be the corresponding marked translation surface, and let $\tau$\index{t@$\tau$} 
denote the pullback of a geodesic triangulation\index{triangulation}
with vertices in 
$\Sigma$, from $(M,\Sigma)$ to $(S,\Sigma)$. The cohomology class 
$\hol_{\til q}$ assigns coordinates in $\R^2$ to edges of the
triangulation and thus can be thought of as
giving a map from the triangles of $\tau$ to triangles in $\R^2$
(well-defined up to translation), 
 and so each triangle in $\tau$ has a Euclidean structure
 coming from $M$.
 Let $U_\tau$ be the collection of all cohomology classes
which map each triangle of $\tau$ into a positively oriented
non-degenerate triangle in $\R^2$. Each $\beta\in U_\tau$ gives a
translation surface $M_{\tau,\beta}$ built by gluing together the
corresponding triangles in $\R^2$ along parallel edges, as well as a
distinguished marking map, which we denote by $\varphi_{\tau,\beta}: (S,\Sigma)\to
(M_{\tau,\beta},\Sigma)$, which is the unique map   taking each triangle of the
triangulation $\tau$ of $S$ to the corresponding triangle of the
triangulation of $M_{\tau,\beta}$
and which is affine on each triangle (with respect to the Euclidean
structure coming from $M$). Let $\til q_{\tau, \beta}$ denote the
marked translation surface corresponding to the marking map
$\varphi_{\tau, \beta} .$ 
Let
$$V_\tau \df \{\til q_{\tau, \beta} :\beta \in U_\tau \} \ \text{ and
   } \ 
\Psi_\tau : U_\tau \to V_\tau, \ \ \Psi_\tau(\beta) = \til q_{\tau, 
  \beta}.$$
By construction, $\beta$ agrees with $\dev(\til q_{\tau, \beta})$ on edges of
$\tau$, and these edges generate $H_1(S, \Sigma)$. 
Thus the map
$$\Phi_\tau : V_\tau \to U_\tau, \ \ \Phi_\tau(\til q) =\dev(\til q)$$
is an inverse to
$\Psi_\tau$ (and in particular $\Psi_\tau$ is injective).
  The collection of maps $\{\Phi_\tau\}$ gives an atlas of
  charts for
  $\HHm$ and the collection of maps $\{\Psi_\tau\}$ gives an inverse atlas
  for $\HHm$.  These charts give $\HHm$ a manifold structure for which
  the map $\dev$ is a local diffeomorphism. In fact this atlas
  determines an affine structure on $\HHm$ so that $\dev$ is an affine
  map. 
 
We denote the tangent space of $\HHm$ at $\til q \in \HHm$ by $T_{\til
  q}(\HHm)$ \index{T@$T_{\til q}(\HHm)$} 
and by $T(\HHm)$ \index{T@$T(\HHm)$}  the tangent bundle of $\HHm$. 
Using the fact
that the developing map 
is a local diffeomorphism we can identify the tangent space at each
point of $\HHm$ with $H^1(S, \Sigma ; \R^2)$ so $T(\HHm)=\HHm\times
H^1(S, \Sigma; \R^2)$. We say that two tangent vectors $v_i \in
T_{\til q_i}(\HHm) \ (i=1,2)$, or two
subspaces $V_i \subset T_{\til q_i}(\HHm)$ are {\em
  parallel}\index{parallel} if they map to the same element or
subspace of $H^1(S,
\Sigma; \R^2)$. We say that a sub-bundle of $T(\HHm)$ is {\em
  flat}\index{flat sub-bundle} if
the fibers over different points are parallel, and that a sub-bundle
of $T(\HH)$ is {\em flat} if each of the connected components of its pullback
to $T(\HHm)$ is flat.

Let 
\eq{eq: x and y splitting model surface}{
 H^1(S,\Sigma;\R^2) = H^1(S,\Sigma;\R_x) \oplus
 H^1(S,\Sigma;\R_y)}
be the analogue of \equ{eq: x and y splitting} for the model
surface $S$. 
This decomposition  determines a foliation of $H^1(S,\Sigma; \R^2)$, whose leaves are pre-images of points under the projection  $H^1(S,\Sigma; \R^2) \to H^1(S,\Sigma; \R_y)$. The pullback of  this foliation to $\HHm$ is the {\em horospherical foliation} (or `strong unstable foliation', see \cite{mahler, SSWY} for more information).  We denote the horospherical leaf of a point $\til q \in \HHm$ by $W^{uu}(\til q)$. \index{horospherical foliation}

Using the explicit marking maps
$\varphi_{\tau,\beta}:(S,\Sigma)\to (M_{\tau,\beta},\Sigma)$, we
get explicit {\em comparison maps}\index{comparison maps} between 
surfaces $M_{\tau,\beta}$ and $M_{\tau,\beta'} \in V_\tau$, taking triangles affinely to triangles, and having the form
$$\varphi_{\tau,\beta,\beta'} \df \varphi_{\tau,\beta}
\circ\varphi_{\tau,\beta'}^{-1}:M_{\tau,\beta'}\to 
M_{\tau,\beta}.$$\index{phi@$\varphi_{\tau,\beta,\beta'}$}The maps
$\varphi_{\tau,\beta,\beta'}$ are continuous and piecewise 
affine but are not in general affine mappings since they may have different derivatives on different triangles. If, in addition, $M_{\tau, \beta}$ and $M_{\tau, \beta'}$ are in the same horospherical leaf, then the comparison map $\varphi_{\tau,\beta,\beta'}$ sends horizontal straightline leaves on $M_{\tau, \beta'}$ to horizontal straightline leaves on $M_{\tau, \beta}$, preserving the vertical distance between plaques (but the length measure on the leaves may be distorted). See 
Figure \ref{fig: comparison}.

\begin{figure}
\includegraphics[scale=1.0]{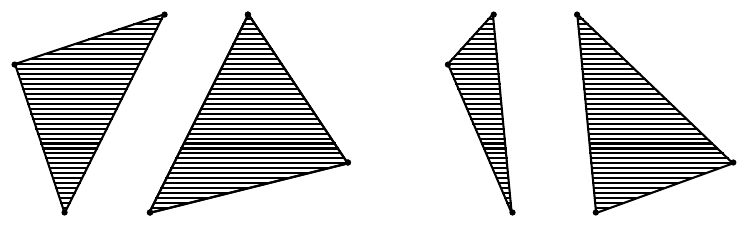}
\caption{ The left hand side shows two triangles in $M_{\tau, \beta}$. The right hand side shows their images under the comparison map. In this case the two surfaces are in the same horospherical leaf.
}
\label{fig: comparison} 
\end{figure}

Let  $\Mod(S, \Sigma)$ be the group of  isotopy classes of homeomorphisms
of $S$ which fix $\Sigma$ pointwise. This is the {\em
  pure mapping class group}\index{mapping class group}. 
It acts on the right on marking maps
by
pre-composition, and this induces a well-defined action on $\HHm$
(note that $\Gamma_M$ acts on the left). It also acts on 
$T(\HHm)=\HHm\times H^1(S, \Sigma; \R^2)$ by
$\gamma: (\varphi,\beta) \mapsto
(\varphi\circ\gamma,\gamma^*(\beta))$.   
The developing
map is $\Mod(S, \Sigma)$-equivariant with respect to these two right
actions and thus 
the action of an element of $\Mod(S, \Sigma)$ on $\HHm$,
when expressed in charts, is linear. This implies that the
$\Mod(S, \Sigma)$-action preserves the affine structure  on $\HHm$. 
This action is properly discontinuous, but not free. Elements with
nontrivial stabilizer groups correspond to surfaces
with nontrivial translation
equivalences.

The group $\Mod(S, \Sigma)$ acts transitively on isotopy classes of
marking maps hence each fiber of the forgetful map $\pi : \HHm \to \HH$ is
a $\Mod(S, \Sigma)$-orbit. We can thus view $\HH$ as the quotient
$\HHm/\Mod(S, \Sigma)$, and equip it with the quotient topology. 
The horospherical foliation on $\HHm$ descends to a 
 well-defined equivalence relation on $\HH$, and we denote the equivalence class of $q \in \HH$ by $W^{uu}(q).$ Loosely speaking, $W^{uu}(q)$ is the set of translation surfaces whose horizontal measured foliation is the same as that of $M_q$. 
 
Viewed as a map between topological spaces the forgetful  map is
typically {\em not} a covering map due to  
 to the presence of translation surfaces in $\HH$ with non-trivial
 translation equivalences.
  To make this map behave more like a covering map we work in the
category of orbifolds.  

\subsection{The affine orbifold structure of a stratum}\name{subsec: orbifold}
An {\em orbifold} structure on a space $X$ is given by an atlas of
inverse charts. This consists of a collection 
of open sets $W_j$ that cover $X$, a collection of maps $\phi_j:
U_j\to W_j$ where $U_j$ are open sets in a vector space $V$, 
and a collection of finite groups $\mathcal{G}_j$ acting linearly on the sets $U_j$
so that  each $\phi_j$ induces a homeomorphism from $U_j/\mathcal{G}_j$ to
$W_j$. Furthermore we require that the transition maps on overlaps respect the group actions. The local groups $\mathcal{G}_j$ give
rise to a local group $\mathcal{G}_x$, depending only on $x \in X$,
and well-defined up to a conjugation. 
More information is contained in \cite[Definitions 2.1 \&  2.2]{orbifolds}. 
If we require that the overlap functions and finite group actions respect the affine structure then we get an {\em affine orbifold}.

 The {\em singular
  set} of an orbifold is the set of points where the local group is
not the identity. The singular set has a stratification into
submanifolds which we will call {\em orbifold
  substrata},\index{orbifold substrata}
defined as the connected components of the subsets of the stratum on
which the local group is constant. We 
will denote the orbifold substratum corresponding to $\mathcal{G}_q$
by $\mathcal{O}_q$.

We now modify our construction of the atlas for $\HHm$ to give an affine
orbifold atlas for $\HH$. 
Let $q\in\HH$, let $M= M_q$ be the underlying translation surface, and
let $\Gamma_q = \Gamma_{M}$ be the 
group of translation equivalences of $M_q$. In order to construct an inverse chart in a neighborhood of $q$ we choose a
marking map $\varphi: (S,\Sigma)\to  (M,\Sigma)$. By pulling back a
triangulation from the quotient of $M$ by
$\Gamma_q$, we can find a geodesic
triangulation $\tau'$ of $M$ which is $\Gamma_M$-invariant, and we let
$\tau = \varphi^{-1}(\tau')$ be the pullback of this triangulation to $S$. As before, let
$U_\tau$ be the set of cohomology classes compatible with 
$\tau$. Let
  ${\mathcal G}_q$\index{Gq@${\mathcal G}_q$} be the (conjugacy class of the) subgroup of
  $\Mod(S,\Sigma)$ corresponding 
  to the isotopy classes of the elements
  $\{\varphi^{-1}\circ h \circ \varphi: h \in \Gamma_q\}$.
Since $\tau'$ is $\Gamma_q$-invariant, the group ${\mathcal G}_q$ acts
on $U_\tau$, and the maps $\pi \circ \Psi_\tau: U_\tau \to \HH$ induce maps from 
$U_\tau/{\mathcal G}_q$ to $\HH$. By possibly replacing $U_\tau$ by a
smaller neighborhood $U'_q \subset U_\tau$ on which this induced map
is injective, we get a 
collection of inverse charts for an
orbifold atlas for $\HH$.

The {\em tangent bundle} of an orbifold is defined  in
\cite[Prop. 4.1]{orbifolds}. It is itself an orbifold, and is equipped
with a projection map $T(X)\to X$, such that the fiber over $x$ 
can be identified with the quotient of a vector space by a linear
action of $\mathcal{G}_x$. 
The projection map $T(X) \to X$ is a bundle map in the category of
orbifolds. Note that its fibers
can vary from point to point. 

We denote the orbifold tangent space of  $\HH$ at $q$ by $T_q(\HH)$, \index{T@$T_q(\HH)$} 
and the tangent bundle of $\HH$  by $T(\HH)$.\index{T@$T(\HH)$}  
We can identify $T(\HH)$ with the quotient of the tangent bundle of $\HHm$
under the action of the pure mapping class group.
%
The bundle $T(\HH)$ has a canonical $\Mod(S, \Sigma)$-invariant
splitting coming from the decomposition
\eqref{eq: x and y splitting model surface}
and we  refer
to the summands  as the {\em horizontal and vertical 
sub-bundles}\index{horizontal sub-bundle}\index{vertical sub-bundle}.
Thus the horizontal sub-bundle is given by the tangent spaces to horospherical leaves in $\HHm.$


Since $\HH$  is the quotient of an affine manifold $\HHm$ by a group
acting affinely and properly discontinuously it inherits  the
structure of an {\em affine} orbifold. A map between affine orbifolds
is {\em affine} if it can be expressed by affine maps in local
charts.



\combaraknew{The statement and proof of the following proposition
  contained inaccuracies stemming from a confusion between $\HH$ and
  $\HHm$, note that $\mathcal{O}_q$ lives in $\HH$ but $\Phi_\tau$
  takes values in $\HHm$. I corrected the statement and omitted the
  proof.}

With the above description of the orbifold tangent bundle of $\HH$, we
obtain a description of the sub-bundle corresponding to the orbifold
substrata.  
\begin{prop}\name{orbifold properties} Let $ q\in\HH$ be a
  surface with a nontrivial local group $\mathcal{G}_q$
  and let $\mathcal{O}_q$ be the corresponding orbifold
  substratum. A choice of 
  $\til q \in \pi^{-1}(q)$ gives a component $\til{\mathcal{O}}_q$ of $ \pi^{-1} (\mathcal{O}_q)$ and a subgroup $\mathcal{G} \subset \Mod(S, \Sigma)$ in the conjugacy class $\mathcal{G}_q$, such that $\til{\mathcal{O}}_q$
  is an affine submanifold of $\HHm$, and  
  its tangent space $T_{\til q}(\til{\mathcal{O}}_q)$ at
  $\til q$ is identified via the developing map with the set of
  vectors in $H^1(S, \Sigma; \R^2)$ fixed by $\mathcal{G}$. 
\end{prop}

The proof is left to the reader.


We will need explicit formulas for the projections
onto the tangent space to an orbifold substratum, and onto a normal sub-bundle. 
Let $M_q$ be a surface with a non-trivial group of translation
equivalences, and choose a chart as above about $M_q$. Choose a marking map of $M_q$ and let  $\mathcal{G}_{q}$ be
the corresponding local group acting on this chart.
%
Define $P^+: H^1(S,\Sigma;\R^2)\to H^1(S,\Sigma;\R^2)$ by 
\eq{eq: formulae projections}{
P^+(\beta) \df
\frac{1}{|\mathcal{G}_q|} \sum_{\gamma\in \mathcal{G}_q}\gamma^*(\beta).
}\index{P@$P^+$}
By Proposition \ref{orbifold properties},  $P^+$ is a projection of
$H^1(S,\Sigma;\R^2)$ onto the tangent space to the substratum. 
The kernel of $P^+$, which we denote by
$\mathscr{N}(\mathcal{O}_q)$,\index{N@$\mathscr{N}(\mathcal{O}_q)$} 
is a natural choice for a normal 
bundle to $\mathcal{O}_q$. We denote by $P^-\df \mathrm{Id}-P^+$
\index{P@$P^-$} the projection onto the normal space  
to the orbifold substratum. Note that $P^\pm$ depend on the orbifold substratum $\mathcal{O}_q$ (via $\mathcal{G}_q$) but 
this will be suppressed in the notation. 
It will also be useful to further decompose the normal bundle into its
intersections with the horizontal and
vertical sub-bundles, and we denote these sub-bundles by
$\mathscr{N}_x(\mathcal{O}_q)$ and
$\mathscr{N}_y(\mathcal{O}_q)$. \index{N@$\mathscr{N}_x(\mathcal{O}_q)$}
\index{N@$\mathscr{N}_y(\mathcal{O}_q)$}

\begin{prop} \name{prop: tangent and normal} Given an orbifold
  sub-locus $\mathcal{O}$, the bundles $T(\mathcal{O})$, 
  $\mathscr{N}(\mathcal{O})$, $\mathscr{N}_x(\mathcal{O})$ and
  $\mathscr{N}_y(\mathcal{O})$ are flat, and each has
  a volume form which is well-defined (independent of a marking). 
\end{prop}

\begin{proof}
To see that the bundles in the statement are flat, note that $\Mod(S, \Sigma)$ acts on $H^1(S, \Sigma; \R)$ and $H^1(S, \Sigma; 
\R^2)$ by linear transformations, and thus the set of vectors fixed by a subgroup $\mathcal{G}$ is a linear subspace. Now flatness follows using Proposition \ref{orbifold properties}.

The map $P^+$ respects the splitting of
cohomology into horizontal and vertical factors, i.e., it commutes
with the two projections onto the summands in 
\equ{eq: x and y splitting model surface}. Moreover, since the
$\Mod(S,\Sigma)$-action on $H^1(S, \Sigma; \R^2)$ preserves $H^1(S,
\Sigma; \Z^2)$,  it 
takes integral classes to rational classes, i.e., is defined over
$\Q$. It thus induces a map
 $$  H^1(S,\Sigma;\R_x) \supset H^1(S,\Sigma;\Z_x) \stackrel{P^+}{\longrightarrow} 
 H^1(S,\Sigma;\Q_x)\subset H^1(S,\Sigma;\R_x)$$
 (with the obvious notations $\Z_x, \Q_x$ for the
corresponding summands), and a corresponding map for
the second summand $\Z_y, \Q_y, \R_y$. The kernels of these 
maps are lattices in $\mathscr{N}_x(\mathcal{O})$ and
$\mathscr{N}_y(\mathcal{O})$ which are parallel. This means that the
Lebesgue measure on $\mathscr{N}_x(\mathcal{O})$, coming from the affine structure
of Proposition \ref{orbifold properties}, has a natural normalization
which does not depend on the choice of a particular lift
$\til{\mathcal{O}} \to \mathcal{O}$.
\end{proof}

Affine structures do not give
a metric geometry but some familiar notions from the theory of
Riemannian manifolds have analogues
for affine manifolds.
Thus an {\em
  affine geodesic} \index{affine geodesic} is a path in an affine
manifold $N$
parametrized by an open interval in the real line which has the
property that in any affine chart the parametrization is linear. We can also describe affine geodesics by saying that
the tangent vector to the curve is invariant under parallel translation.
Affine geodesics are projections of orbits of a partially defined flow
on the tangent bundle which we call the {\em affine geodesic flow}. 
An affine geodesic has a maximal domain of definition which is a
connected open subset of $\R$, which may or may not coincide with
$\R$. We denote by $\Dom(\til q,v)\subset \R$\index{Dom@$\Dom(\til
  q,v)$} the maximal domain of definition of the affine geodesic which
is tangent at time $t =0$ to $v \in T_{\til q}(\HHm)$.

The space of marked translation surfaces with area one is a
submanifold 
$\HHm_{,1}$ of $\HHm$, which is invariant
under $\Mod(S,\Sigma)$. We refer to the quotient orbifold as the {\em
  normalized stratum} \index{normalized stratum} and denote it by
$\HH_1$. 
The normalized stratum is a codimension one sub-orbifold of $\HH$ but
it is not an affine sub-orbifold. 
The developing map $\dev$ maps $\HHm_{,1}$ into a quadric in
$H^1(S, \Sigma; \R^2)$, and the tangent space $T_{\til q}(\HHm_{,1})$  is a
linear subspace of  
$H^1(S,\Sigma; \R^2)$ on which area is constant to first order. This
subspace varies with $\til q$. 
Nevertheless it is often quite useful to use the ambient affine coordinates to 
discuss it.

The intersection of horospherical leaves in $\HH_m$ with $\HH_{\mathrm{m},1}$ give the {\em horospherial foliation} \index{horospherical foliation} of $\HH_{\mathrm{m},1}$. Its leaves are of codimension one in the horospherical leaves of $\HHm$. In general if we consider  a vector
tangent to $\HH_1$ then the affine geodesic determined by this vector need
not lie in $\HH_1$ but in the particular case of vectors tangent to the horospherical foliation (e.g., horocycles and tremors) it
will be the case that these paths lie in $\HH_1$.

\subsection{The action of $G = \SL_2(\R)$ on strata}\name{subsec: G}


We now check that the linear action of $G$  
induces an affine action on  charts.
There is a natural left action of $G$ on $H^1(S,\Sigma;\R^2)$ which is
given by the action of $G$ on the coefficient system, i.e. by
postcomposition of $\R^2$ valued 1-cochains. 
Let $\tau$ be
a triangulation of $S$, and let $U_\tau \subset H^1(S, \Sigma; \R^2)$ be defined as
in \S \ref{subsec: atlas of charts}. For $\beta \in U_\tau$ 
and $g \in G$, we see that 
$g \beta \df g \circ \beta\in U_\tau$. Let $\varphi_{\tau,\beta,g  \beta}:M_\beta\to
M_{g \beta}$ be the comparison map. Notice that it has the same
derivative on each triangle, namely its derivative is everywhere
equal to the linear map $g$.  In particular, the 
comparison map $\varphi_{\tau, \beta, g  \beta}$ does not depend on $\tau$. 
  We will call it the {\em
  affine comparison map corresponding to $g$}\index{affine comparison
  map} and denote it by 
$\psi_g$.\index{p @ $\psi_g$} The action of $g$ on $\HHm$ can now be expressed as
replacing a marking map $\varphi: S \to M$ by $\psi_g \circ \varphi: S
\to gM$. Other affine maps $M_q \to M_{gq}$ with 
derivative $g$ can be 
obtained by composing $\psi_g$ with translation
equivalences.  
Since the $G$-action commutes with the $\Mod(S, \Sigma)$-action, $G$ acts on $\HH$ and
preserves its orbifold stratification. Additionally, the
normal and tangent bundles of Propositions \ref{orbifold properties} and \ref{prop:
  tangent and normal} are
$G$-equivariant. 


We introduce some notation for subgroups of $G$. 
 Recall the group $U
= \{u_s: s\in \R\}$ introduced in \equ{eq: hero introduced}. We will
also use the following 
notation for other subgroups: 
\eq{eq: subgroups notation}{
g_t = \left( \begin{matrix} e^t & 0 \\ 0 & e^{-t} \end{matrix}
\right), \ \ \ \ r_\theta = \left(\begin{matrix}
    \cos \theta & -\sin \theta \\ \sin \theta & \cos
    \theta \end{matrix} \right)
}\index{g@$g_t$} \index{gg@$\til g_t$} \index{r@$r_\theta$}
and
\eq{BU}{ B=\left\{ \left( \begin{matrix} a& b \\ 0 &
        a^{-1} \end{matrix}\right): a>0, b\in\R \right\}.
}
With this notation we note 
that the $U$-action is given in period coordinate charts by 
$$
\hol^{(x)}_{u_s \til q}(\gamma) =
\hol^{(x)}_{\til q}(\gamma)+s\cdot \hol^{(y)}_{\til q} (\gamma), \ \ 
\hol^{(y)}_{u_s \til q}(\gamma) = \hol^{(y)}_{\til q}(\gamma);
$$
this now gives a precise meaning to equation \equ{eq: horocycle period
  coordinates}. We see in particular that horocycle orbits are linearly parametrized 
  affine geodesics.
  
  Our next goal is to give a precise meaning to equation
\equ{eq: tremor3}, by defining transverse
measures and their associated 
cohomology class.

\subsection{Transverse (signed) measures and foliation
  cocycles}\name{subsec: transverse} 
In this section we define transverse measures and cocycles and
cohomology classes associated with a non-atomic transverse
measure. It will be useful to include 
signed transverse measures. 
In some settings
it is  useful to pass to limits of non-atomic transverse measures, and
these limits may be certain atomic transverse measures.   
In \S \ref{sec: atomic tremors} we will discuss the case of these
atomic transverse measures. 

Let $M$ be a translation surface, let $\theta \in \mathbb{S}^1$ be a
direction (i.e., a 
 unit vector $(\cos \theta, \sin \theta) \in
\R^2$), and let $\FF_\theta$ denote the foliation of $M$ obtained by
pulling back the foliation of $\R^2$ by lines parallel to $\theta$. 
A {\em transverse arc} to $\FF_\theta$ is a piecewise
smooth curve $\gamma:(a,b)\to M \sm \Sigma$ of finite length which is everywhere
transverse to leaves of  
$\FF_\theta$. 
A {\em transverse measure}\index{transverse measure} on $\FF_\theta$
is a family $\{\nu_\gamma\}$ where $\gamma$ ranges over the transverse
arcs, the $\nu_\gamma$ are finite regular 
Borel measures defined
on $\gamma$ which are invariant under 
 isotopy along leaves and so that if $\gamma'\subset \gamma$ then
$\nu_{\gamma'}$ is the restriction of $\nu_\gamma $ to $\gamma'$ (in \S\ref{sec: atomic tremors} these two requirements will be referred to respectively as {\em invariance} and {\em restriction}). 
Since transverse measures are defined via measures, the usual notions
of measure theory (absolute 
continuity, Radon-Nikodym theorem, etc.) make sense for transverse
measures (or a pair of transverse measures). In particular it makes
sense to speak of atoms of a  transverse measure, and we will say that
$\nu$ is {\em non-atomic} \index{non-atomic transverse measure} if
none of the $\nu_\gamma$ have atoms. In this paper, if transverse
measures have atoms we require that the atoms be supported
on closed loops, each of which is a closed leaf, or a union of saddle connections  
that meet at angles $\pm \pi$ (see \S \ref{sec: atomic tremors} for  the
complete definition). These are the atomic transverse measures that can
arise as limits of non-atomic transverse measures. We remark that in
the literature, there are several different conventions regarding
atomic transverse measures. 

A (finite) {\em signed measure} on $X$ is a map from Borel subsets of $X$ to
$\R $ satisfying all
the properties satisfied by a measure. 
Recall that every signed measure has a canonical Hahn decomposition, 
 i.e.\ a unique representation $\nu =\nu^+-\nu^-$ as a difference 
of mutually singular finite  measures. A {\em signed transverse measure}\index{signed
  transverse measure} is a system
$\{\nu_\gamma\}$ of signed measures, satisfying the same hypotheses as
a signed measure; or equivalently, the difference of two transverse 
measures $\{\nu^+_\gamma\}, \, \{\nu^-_\gamma\}$. In what 
follows, the words `measure' and `transverse
measure' always refer to non-negative measures (i.e. measures for
which $\nu^-=0$). When we want to allow general signed measures we
will include the word `signed'. We say that
$\nu$ is {\em non-atomic} if $\nu^\pm$ are 
both non-atomic. The sum $\nu^+(X) + \nu^-(X)$ is called the {\em total
  variation} of $\nu$.

If $M$ is a translation surface, $\FF_\theta$ is a directional
foliation on $M$, and $\nu$ is a non-atomic signed transverse measure
on $\FF_\theta$, we  
have a map $\beta_\nu$ from transverse line segments to real
numbers, defined as follows. If $\gamma$ is a transverse oriented line segment and the
(counterclockwise) angle between the direction $\theta$ and the direction of 
$\gamma$ is in $(0, \pi)$, set $\beta_\nu(\gamma)=\nu(\gamma)$.
If the angle is in $(-\pi, 0)$  
 set $\beta_\nu(\gamma)=-\nu(\gamma)$. 
  We extend this to all straight line segments by stipulating that
  $\beta_\nu(\gamma)=0$  for any line segment $\gamma$ that is
  contained in a leaf of the foliation. By linearity we extend $\beta_\nu$ to
  finite concatenations of oriented straight line 
  segments. Similarly we can define $\beta_\nu(\gamma)$ for an oriented
  piecewise smooth curve $\gamma$, where the sign of an intersection
  is measured using the derivative of $\gamma$. \index{b@$\beta_\nu$}

By a
{\em polygon decomposition}\index{polygon decomposition} of a
translation surface $M$, we mean a 
decomposition into 
simply connected polygons for which all the vertices are singular
points. As we saw every $M$
admits a geodesic triangulation which is a special case of a polygon
decomposition. 
  Let $\beta_\nu$ be as in the preceding paragraph. Any element
  $\alpha \in H_1(M,\Sigma)$ has a representative $\til \alpha$ that is a
  concatenation of edges of a polygon decomposition. The invariance
  property of a transverse measure ensures that the value 
  $\beta_\nu(\til \alpha)$ depends only on $\alpha$ and not on the
  representative $\til \alpha$; in particular it does not depend on
  the cell decomposition used, and 
  $\beta_\nu$ is a cochain and defines a cohomology class in
  $H^1(M,\Sigma;\mathbb{R})$. We have defined a mapping $\nu \mapsto
  \beta_\nu$ from non-atomic signed transverse measures to $H^1(M, \Sigma;
  \R^2)$, and in \S \ref{sec: atomic tremors} we will explain how to
  extend this map to atomic transverse measures. We will be primarily 
  interested in transverse measures to the horizontal foliation. An element of
  cohomology which corresponds to a transverse
  measure (resp., a signed transverse measure) to the horizontal foliation will be called a
  \emph{foliation cocycle}\index{foliation cocycle} (respectively,
  {\em signed foliation cocycle}), \index{signed foliation cocycle}
  and $\beta_\nu$ will be called the 
  {\em (signed) foliation cocycle corresponding to $\nu$.}

  Identifying $\R$
  with $\R_x$ and $H^1(M, \Sigma; \R)$ with the first summand in
  \equ{eq: x and y splitting}, we identify the collection of all signed
  foliation cocycles with a subspace $\tremspace_q \subset H^1(M,
  \Sigma; \R_x)$\index{T@$\tremspace_q$}, and the collection of all
  foliation cocycles with a 
  cone $C_q^+ \subset \tremspace_q$. \index{C@$C_q^+$} We refer to
  these respectively as 
  the {\em space of signed foliation cocycles}\index{space of signed
    foliation cocycles} and the {\em cone of
    foliation cocycles}.\index{cone of foliation cocycles} The Hahn
  decomposition of transverse measures implies that every $\beta \in
  \tremspace_q$ can be written uniquely as $\beta = \beta^+ - \beta^-$ for
  $\beta^\pm \in C^+_q$. 
For every $q$, the 1-form $dy$ gives rise to a {\em canonical
  transverse measure}\index{canonical transverse measure} and 
to the corresponding cohomology class $\hol_q^{(y)}$. When we want to
think of this class as a foliation cocycle, we will denote it by $dy$
or $(dy)_q$, and refer to it as the {\em canonical foliation
  cocycle}. \index{canonical foliation cocycle} 
  \index{D@$(dy)_q$}

As discussed above for the horizontal direction, we can define a
 (partially defined) straightline flow in direction $\theta$ by lifting the vector field on $\R^2$ in direction $\theta$ and following lines parallel to $\theta$.  We write $\FF_\theta$ for the foliation by lines in direction $\theta$ and write $\FF$ for $\FF_0$. We say
that a finite Borel measure $\mu$ on $M$ is {\em
  $\FF_\theta$-invariant} if it is invariant under the straightline 
flow in direction $\theta$. We have the following well-known relationship between
transverse measures and 
invariant measures. 

\begin{prop}\name{prop:trans to meas}
For each non-atomic transverse measure $\nu$ on $\FF_{\theta}$ there
exists an $\FF_\theta$-invariant 
measure $\mu_\nu$ with 
\eq{eq: formula above}{\mu_\nu(A) =  \nu(v) \cdot \ell(h)
}
for
every isometrically embedded rectangle $A$ with one side $h$
parallel
to $\theta$, and another  side $v$ orthogonal to $\theta$, where $\ell$ is
the Euclidean length.  The map $\nu \mapsto \mu_\nu$  is a bijection
between non-atomic transverse measures and $\FF_\theta$-invariant
measures that assign zero measure to leaves. It extends to a bijection
between non-atomic signed transverse measures and
$\FF_\theta$-invariant signed measures assigning zero measure to leaves. 

\end{prop}


It is clear from \equ{eq: formula above} that two different transverse 
measures give different measures to some rectangle, and so the
assignment is injective. To see that 
each $\FF_\theta$-invariant measure arises from a transverse
measure, partition $M$ into rectangles and use
disintegration of measures to define a transverse measure on each
rectangle.  
This transverse measure will be non-atomic if the invariant measure
gives zero measure to every horizontal leaf.  

The map $\nu \mapsto \beta_\nu$ is almost injective. More precisely,
we have: 

\begin{prop}[Katok]\name{prop:transverse measures}
If $M_q$ has no horizontal cylinders
and $\nu_1 \neq \nu_2$ are
distinct non-atomic signed transverse measures to the horizontal
foliation, then $\beta_{\nu_1} \neq 
\beta_{\nu_2}$, and moreover the 
restrictions of $\beta_{\nu_i}$ to the absolute period space $H_1(S)$
are different. 
\end{prop}

For a proof see \cite{transverse measures}. Katok considered measures
rather than signed measures, but the passage to
signed measures follows from the uniqueness of the Hahn
decomposition. It is easy to see that the injectivity of the assignment $\nu \mapsto \beta_\nu$ fails if the requirement that $M_q$ has no horizontal saddle connections is omitted. For more on this, see \S \ref{sec: atomic tremors}.

\subsection{The Sup-norm Finsler metric}\name{subsec: sup norm}
We now recall the sup-norm  Finsler metric
on $\HHm$ studied by Avila, Gou\"ezel and Yoccoz in \cite{AGY}. 
 Let $\|
\cdot \|$ denote the Euclidean norm on $\R^2$. 
For a translation surface $q$, denote by $\Lambda_q$ the
collection of saddle connections on $M_q$ and let $\ell_q(\sigma)=
\|\hol_q(\sigma)\|$ be the length of $\sigma \in \Lambda_q$. 
For $\beta\in H^1(M_q, 
\Sigma_q; \R^2)$ we set 
\eq{eq: define a norm downstairs}{
\|\beta\|_q \df \sup_{\sigma \in \Lambda_q}
\frac{\|\beta(\sigma)\|}{\ell_q(\sigma)}.
}\index{B@$\|\beta\|_q$}

We now define a Finsler metric for $\HHm$. Let $\varphi: (S, \Sigma)
\to (M_q, \Sigma)$ be a marking map, which represents $\til q \in
\HHm$. Recall that we can 
identify $T_{\til q}(\HHm)$ with $H^1(S, \Sigma ;
\R^2)$.  
Then $\|\varphi^* \beta\|_{\til q} =\| \beta\|_q$ 
is a norm on $H^1(S, \Sigma ; \R^2)$, or equivalently: 
\eq{eq: define a norm upstairs}{
\|\beta\|_{\til q} \df \sup_{\tau \in \Lambda_{\til q}}
\frac{\|\beta(\varphi(\tau))\|}{\ell_q(\varphi(\tau))}.
}\index{B@$\|\beta\|_q$}
Note that $\Lambda_{\til
  q}$ varies as $\til q$ changes, and that $\|\theta\|_{\til q}$ is
well-defined (i.e. 
depends on $\til q$ and not on the actual marking map $\varphi$). 
%
Recall that using period coordinates, the tangent bundle $T(\HHm)$ is
a product $\HHm \times H^1(S, \Sigma; \R^2)$. As shown in 
\cite[Prop. 2.11]{AGY}, the map 
\begin{equation}\label{eq: for lipschitz}T(\HH_{\mathrm{m}}) \to \R, \ \ (\til q,\beta) \mapsto
\|\beta\|_{\til q}
\end{equation}
 is
continuous.

The Finsler metric  defines a distance function on $\HHm$ which we
call the {\em sup-norm distance} \index{sup-norm distance} and define as follows \index{dist}
\eq{eq: Finsler integrate}{
\dist(\til q_0, \til q_1) \df \inf_{\gamma } \int_0^1
\|\gamma'(\tau)\|_{\gamma(\tau)} d\tau.
}
Here $\gamma$ ranges over smooth paths $\gamma:[0,1] \to
\HH$ with $\gamma(0)=\til q_0$ and $\gamma(1) = \til q_1$. 
This distance is symmetric since $\|\beta\|_{\til q}=\|-\beta\|_{\til q}$.

The following was shown in \cite[\S2.2.2]{AGY}:
\begin{prop}\name{prop: sup norm properties}
The metric $\dist$ is proper,
complete, and induces the topology on $\HHm$ given by period
coordinates. It is invariant under the action of the pure mapping class group. 
\end{prop}

By Proposition \ref{prop: sup norm properties}, 
    in order to compute the length of a path $\rho$, one can  lift the path to $\HHm$ and measure its length there. 
Note that $\dist$ need not be invariant under parallel translation. 

\begin{proof}
The fact that the sup-norm distance is a Finsler metric giving the topology on
period coordinates is \cite[proof of Proposition 2.11]{AGY}.  
The fact that the metric is proper is \cite[Lemma 2.12]{AGY}.
Completeness is \cite[Corollary 2.13]{AGY}.  The metric is invariant under
the action of the mapping class group because its definition depends only on the collection
of saddle connections in $M_q$ which is independent of the marking.
\end{proof}

\

We will now compute the deviation of nearby $G$-orbits with respect to the
sup-norm distance. 
Let $\|g\|_{\mathrm{op}}, \ g^{\mathrm{t}}$ and $\mathrm{tr}(g)$
denote respectively the operator norm, transpose, and trace of $g \in 
G$. The
operator norm can be calculated   
in terms of the singular values of $g$.  Specifically the operator
norm is the square root of the the largest
eigenvalue of $g^{\mathrm{t}}g$. For a 2 by 2 matrix this eigenvalue can be
expressed in terms of the trace and determinant of $g^{\mathrm{t}}g$:
\begin{equation}
  \|g\|_{\mathrm{op}}=\sqrt{\frac{\mathrm{tr}(g^{\mathrm{t}}g)+
      \sqrt{\mathrm{tr}^2(g^{\mathrm{t}}g)-4}}{2}}   
\end{equation}

Recall the affine comparison map $\psi_g: M_q\to M_{gq}$ with
derivative $g$, from \S \ref{subsec: G}. For this map we have
$\hol(\psi(\sigma))=g(\hol(\sigma))$ and hence
$\|\sigma\|_{gq}=\|g(\hol(\sigma))\|_q$. 
From this it is not hard to deduce that 
$$
\|g\beta\|_{g\til q}\le
\|g\|_{\mathrm{op}}\cdot\|g^{-1}\|_{\mathrm{op}}\cdot\|\beta\|_{\til
  q}. 
$$

\begin{cor}[See \cite{AGY}, equation (2.13)] For any $s, t \in \R$ and any $\beta \in H^1(S, \Sigma; \R^2),$ we have 
$$\|u_s(\beta)\|_{u_s\til
    q}\le\left(1+\frac{s^2+|s|\sqrt{s^2+4}}{2}\right) \|\beta\|_{\til
    q}$$ and $$
\|g_t(\beta)\|_{g_t\til
    q}\le e^{2|t|} \|\beta\|_{\til
    q}.
$$
\end{cor}

Integrating these pointwise bounds and using the definition of the sup-norm
distance, we find that nearby horocycle trajectories diverge from each
other at most quadratically and nearby geodesic orbits diverge at most exponentially. Namely:

\begin{cor}\name{eq: sup norm horocycle deviation}
 For  $\til q_0$ and $\til q_1 \in \HHm$ and any $s, t
\in \R$, 

\[
\begin{split}
\left(1+\frac{s^2+|s|\sqrt{s^2+4}}{2}\right)^{-1}\dist (\til q_0, \til
q_1)&\leq \dist(u_s\til q_0, u_s\til q_1)\\ 
&\leq  \left(1+\frac{s^2+|s|\sqrt{s^2+4}}{2}\right) \dist(\til q_0, \til q_1)
\end{split}
\]
and
\eq{eq: geodesic expansion rate}{
e^{-2|t|}\dist (\til q_0, \til q_1) \leq \dist(g_t\til q_0, g_t\til q)
\leq  e^{2|t|}\dist(\til q_0, \til q_1) 
.
}
\end{cor}

In the case of unipotent flows in homogeneous dynamics nearby orbits
diverge at most polynomially with respect to an appropriate
metric. Corollary \ref{eq: sup norm horocycle deviation}  shows that
on strata, nearby horocycles orbits diverge from each other {\em at
  most} quadratically. In \S\ref{subsec: erratic} we will discuss the
more delicate question of  {\em lower} bounds
for the rate of divergence of horocycles, and show that erratic
divergence is possible.

\section{The space of pairs of tori glued along slits}\name{sec: E dynamics}
In this section we collect some information we will need regarding the
structure of $\EE$ and the 
dynamics of the straightline flow on surfaces in $\EE$. We also prove
Proposition \ref{prop: auxiliary}, which plays an important role in \S \ref{sec: spiky 
  fish}. It shows that for surfaces in
$\EE$, the ergodic measures in directions which are not uniquely
ergodic have good approximations by splittings of the surface into two
tori. This may be considered as a converse to a 
construction of Masur and Smillie \cite[\S 3.1]{MT}.

\subsection{The locus $\EE$}\name{subsec: locus E}
McMullen studied
the eigenform loci $\EE_D$, 
which are affine $G$-invariant suborbifolds 
of $\HH(1,1)$ (see
\cite{McMullen-SL(2)} and references therein).
The description of $\EE=\EE_4$ which will be convenient for us is the
following. Recalling that $\HH(0,0)$ is the stratum of tori with two marked points, we have that
$\EE$ is the collection of $q \in \HH(1,1)$ for which 
there is a branched $2$ to $1$ translation cover from $M_q$ onto a torus in
$\HH(0,0)$. To avoid confusion with different conventions used in the
literature, we remind the reader that we take the marked points in
$\HH(0,0)$ and $\HH(1,1)$ to 
be labelled. See \cite[\S 7]{eigenform}
  for additional information.

Given a torus $T \in
\HH(0,0)$ and a saddle connection $\delta$ joining the two marked
points we can build a 
surface  $M\in\HH(1,1)$ by cutting $T$ along $\delta$, viewing the
resulting surface as a surface with boundary. We define $M$ to be the result of
taking two copies of the  surface with boundary  
and gluing along the boundaries. The surface $M$ has a branched
covering map to $T$ and a deck transformation which is an involution
interchanging the two copies of $T$. A {\em slit}\index{slit} on a 
translation surface is a union of homologous saddle connections which
disconnect the surface. Thus in this example, the preimage $\sigma$ of
$\delta$ 
under the map $M \to T$ is a slit. We say that $M$ is built from the
{\em slit construction}\index{slit construction} applied to $\sigma$. Clearly surfaces built from the slit construction belong to $\EE.$

The following proposition shows that, with respect to the terminology
of \S \ref{subsec: orbifold}, $\EE$ consists of points in
$\HH(1,1)$ where the local orbifold group is non-trivial; namely, it
is the group of order two generated by an involution in $\Mod(S,\Sigma)$. 

\begin{prop}[\cite{EMasurSchmoll}]\name{prop: structure of E} The locus $\EE$ is connected.
  It admits a four to one  covering map 
$P: \EE \to \HH(0,0)$ which is characterized by the following property: 
for every $q \in \EE$ there is an order 2 translation equivalence
$\iota = \iota_q: M_q \to
M_q$\index{i@$\iota$}, such that the quotient surface
$M_q /\langle \iota \rangle$ is a translation surface which is
translation equivalent to the torus $T_{P(q)}$.  
\end{prop}

\begin{proof}
Connectedness of $\EE$ is proved in \cite[Theorem 4.4]{EMasurSchmoll}. 
It remains to show that $P$ is four to one. 
By definition, if $q \in \EE$ then $M_q$ has a translation automorphism
$\iota$ such that $M_q/\langle \iota \rangle$ is a torus in
$\HH(0,0)$. 

We begin by determining the fixed points of $\iota$. 
If a translation automorphism fixes a nonsingular point it fixes a neighborhood
of that point. Thus the set of nonsingular fixed points is open and closed.
We conclude that the only possible fixed points are singularities
and singularities are indeed fixed since they are labelled. 
We conclude that $\iota$ induces a branched covering map which has non-trivial
branching at the two singular points.

Let $T$ be a torus with $\Sigma=\{p_1, p_2\}$
corresponding to a point in $\HH(0,0)$. Any $q \in \EE$
for which $P(q) = T$ gives an unbranched cover $M_q \sm P^{-1}(\Sigma)
\to T \sm \Sigma$. Conversely any unbranched cover of $T \sm \Sigma$ can be completed to
a branched cover of $T$. This cover is ramified at $p_j \in \Sigma$
precisely when a small loop $\ell_j$ around $p_j$ in $T$ does not lift as a closed
loop in $M_q$. So the cardinality of $P^{-1}(T)$ is the number of 
topologically distinct degree 2 covers of $T \sm \Sigma$ for which the
loops $\ell_j$ do not lift as closed loops. Equivalently, it is the
number of  conjugacy 
classes of homomorphisms $\pi_1(T \sm \Sigma) \to 
\Z/2\Z$ for which the image of the class of each $\ell_j$ is
nontrivial. Since $\Z/2\Z$ is abelian, the covering
spaces are determined uniquely by elements $\theta \in H^1(T \sm \Sigma;
\Z/2\Z)$ which has dimension 3 and we are counting those $\theta$ for which
both $\theta(\ell_j) \neq 0$. Since the loops $\ell_1$ and $\ell_2$
are homologous, this condition gives a single inhomogeneous linear equation on a  $\Z/2\Z$ vector space of
dimension 3, so we have four solutions. 
\end{proof}

As we saw surfaces built from the slit construction belong to $\EE$. The following is a strong converse to this statement (a similar result holds for all eigenform loci, see \cite[\S 7]{McMullen-SL(2)}).

\begin{prop}\name{prop: slit characterization} Two
  saddle connections $\delta_1$ and $\delta_2$ on the same torus in $\HH(0,0)$, connecting the singularities, give rise to the 
same surface in $\EE$ if and only if the corresponding homology classes
$[\delta_1]$ and $[\delta_2]$ are 
equal as elements of $H_1(T,\Sigma;\Z/2\Z)$.
In particular every surface in $\EE$
  can be built from the slit construction in infinitely many ways (that is, using infinitely many different $\delta$). 
\end{prop}

\begin{proof} As in the proof of Proposition \ref{prop: structure of
    E}, a surface in $\EE$ corresponds to a 
  class $\theta \in H^1(T \sm \Sigma;\Z/2\Z)$ for which the
  $\theta(\ell_j)$ are nonzero, for $j=1,2$. If $\delta$ is any path
  from $p_1$ to $p_2$, it defines a class $[\delta] \in
  H_1(T,\Sigma;\Z/2\Z)$, and we will say $\theta$ is {\em represented by
    $\delta$} if $\theta$ is the class in $H^1(T \sm
  \Sigma;\Z/2\Z)$  
which is Poincar\'e dual to $[\delta]$. Clearly, if $\theta$ is
represented by some $\delta$ then $\theta$ 
satisfies the requirement $\theta(\ell_j) \neq 0$, and by a
dimension count, any such $\theta$ is represented by some path 
$\delta$. It remains to show that each $\theta$ is represented by
infinitely many 
saddle connections $\delta$ from $p_1$
to $p_2$. To
see this, let $\delta_0$ be some path representing $\theta$, let $v_0
\df \hol_T(\delta_0)$, let 
$\Lambda \df \hol_T(H_1(T;\Z))$, and let $\Lambda' \df \Lambda \cup \left(
  v_0 + \Lambda \right).$ Since $\R^2$ is the universal cover of $T$,
$\Lambda$ is a lattice in $\R^2$, $v_0 
\notin \Lambda$, and the required paths $\delta$ are those for which
$\hol_T(\delta) \in v_0 + 2\cdot \Lambda$ and for which the straight segment
in $\R^2$ from the origin to $\hol_T(\delta)$ does not intersect
$\Lambda'$ in its interior. It follows from this description that the
set of such $\delta$ is infinite. 
\end{proof}

For use in the sequel, we record the conclusion of  Proposition \ref{prop: tangent
  and normal} in the special case of the orbifold substratum $\EE$: 

\begin{cor}
  \name{prop: KZ over
  E} We can identify the tangent space  $T(\EE)$ with
the $+1$ eigenspace of the action of $\iota$ on $H^1(S, \Sigma; \R^2)$
and the
normal bundle 
$\mathscr{N}(\EE)$ with the $-1$ eigenspace. The bundle
$\mathscr{N}(\EE)$ has a splitting into flat sub-bundles
$$\mathscr{N}(\EE) = \mathscr{N}_x(\EE) \oplus \mathscr{N}_y(\EE),$$
and each of these sub-bundles has a flat monodromy-invariant volume
form.
\end{cor}

\subsection{Dynamics on $\EE$} \label{subsec: dynamics on E} Here we state some important features
of the straightline flow on surfaces in $\EE$. 
\begin{prop}\name{prop: involution on E and measures}
Let $q \in \EE$, let $M=M_q$ be the underlying surface, let $\iota: M
\to M$  be the involution as described in Proposition 
\ref{prop: structure of E}, let $\FF$ be the horizontal foliation on
$M$, and let $(dy)_q$ be the canonical transverse measure. Suppose
that the foliation $\FF$ is not 
periodic. Then for any transverse
 measure $\nu$ to  $\FF$, $\iota_* \nu$ is also a transverse measure and
there is $c>0$ such that $\nu + \iota_*\nu = 
c\, (dy)_q$. Moreover, if $\FF$ is not uniquely ergodic, then (up to
multiplication by constants) it supports
exactly two ergodic transverse measures which are images of each other
under $\iota_*$, and $\Leb$ is not ergodic for the horizontal
straightline flow.
\end{prop}

This follows from the facts that $\iota$ commutes with the flow and
that, under our aperiodicity assumption, the projection   
of $\FF$ to the torus is uniquely ergodic. We leave the details to the reader.

The following proposition is the main result of this section. 
Recall that
$\FF_\theta$ denotes the foliation 
in direction $\theta$, where $\theta=0$ corresponds to the horizontal direction.

\begin{prop}\name{prop: auxiliary}
Suppose $q \in \EE$ has the property that the horizontal foliation on
$M_q$ is minimal but not ergodic 
and 
let $\mu$ be an invariant ergodic probability measure on $M_q$ for
the horizontal straightline flow.
Then there are directions $\theta_j$, such that the foliations
$\FF_{j}$  in
direction $\theta_j$ contain saddle connections $\delta_j$  satisfying
the following: \begin{itemize} 
\item[(i)] The union $\sigma_j = \delta_j\cup \iota(\delta_j)$ is a
  slit in $\FF_j$ separating $M_q$
into two isometric tori.
\item[(ii)] 
The holonomy $\hol_q(\delta_j) = (x_j, y_j)$ satisfies 
$$|x_j| \to
\infty, \ 0 \neq y_j\to 0 \text{ as } j \to \infty.$$ 
In particular the direction $\theta_j$ is not horizontal but tends to
horizontal,  and the length of 
  $\delta_j$ tends to $\infty$. Moreover there are no saddle
  connections $\delta$ on $M_q$ with holonomy vector satisfying
  $\left|\hol^{(x)}_q(\delta) \right| < |x_j|$ and 
  $\left|\hol_q^{(y)}(\delta) \right|< |y_j|$. 
\item[(iii)] 
For each $j$ we can choose one of the tori $A_j$ in $M_q \sm \sigma_j$,
such that the normalized restriction $\mu_j$ of $\Leb$  to $A_j$ converges to
$\mu$ as $j \to \infty$, w.r.t. the weak-$*$ topology on probability
measures on $M_q$. Thus, letting $\nu$ and $\nu_j$ be the
transverse
measures corresponding to $\mu$ and $\mu_j$ (via Proposition
\ref{prop:trans to meas}),
and letting $\beta_\nu$ and $ \beta_j = \beta_{\nu_j}$ be the corresponding
foliation cocycles in $H^1(M_q, \Sigma_q; \R)$,
  we have
  $\beta_{j} \to \beta_\nu$. 
\end{itemize} 
\end{prop}

We divide the argument below into steps.
\begin{proof}
{\bf Step 1. Finding slits satisfying 
(i) and (ii): divergence in $\EE$ versus convergence in $\HH(0).$}

We consider the projection map $\bar{\pi}:\EE \to \HH(0)$ given by the
composition of the map $P : \EE \to \HH(0,0)$ from Proposition \ref{prop: structure of E}
with the forgetful map forgetting the second marked point. In other
words, $\bar{\pi} : M_q \mapsto
M_q/\langle \iota \rangle$. 
Since $M_q$ has a minimal horizontal foliation, so does $M_{\bar{\pi}(q)}$. 
We normalize the area of $M_q$ to be 2, so that $\bar{\pi}(q)$ has unit area, and thus belongs to the normalized stratum $\HH_1(0)$, which can be identified with the space of unimodular
lattices 
$\SL_2(\mathbb{R})/\SL_2(\mathbb{Z})$. 
 The horizontal foliation is minimal if and only if the corresponding 
 lattice does not contain a nonzero horizontal vector, and this implies that 
  there is a compact set 
 $\mathcal{K} \subset \HH(0)$ for which 
 \begin{equation} \label{eq: return to compact set of tori}\text{ there is } t_j \to \infty \text{ such that } {g}_{-t_j}\bar{\pi}(q) \in \mathcal{K}.
 \end{equation}

Denote by $\mathcal{M}_{g}$ the moduli space of Riemann surfaces of
genus $g$ and let $\overline{\mathcal{M}}_{g}$ be
its Deligne-Mumford compactification (see \cite[\S 5]{Bainbridge thesis} for
a concise introduction). \combaraknew{I added the reference to
  \cite{Bainbridge thesis} because he discusses the genus 2 case in
  detail; note the adjective `concise' which is not usually associated
  with Matt's work.} Passing to a further
 subsequence (which we will continue to denote by $t_j$ to simplify
 notation) we have that $g_{-t_{j}}q$ converges to a
 stable curve in $\overline{\mathcal{M}}_{2}$. This curve projects to some
 torus in $\mathcal{M}_{1}$ (and 
 not in its boundary $\overline{\mathcal{M}}_{1} \sm
 \mathcal{M}_{1}$)  because the projection of $\mathcal{K}$ to 
 $\mathcal{M}_{1}$ is compact. So the limiting stable curve has area
 2. By  \cite[Theorem 1.4]{McM Dio},  the limit of ${g}_{-t_{j}}q$
 is not connected and so, considering the projection to  
 $\mathcal{M}_{1}$ again, it is built from two tori connected at a node. Thus for all large $j$,
 the surfaces 
 \begin{equation}\label{eq: the surfaces mj}
 M^{(j)} \df M_{{g}_{-t_{j}}q}
 \end{equation}
 are built from two copies of a torus $T_j
 \in \mathcal{K}$ glued along slits whose lengths go to
 zero. These slits must be the union of two saddle connections that
 connect the two different singularities of $M^{(j)}$. Indeed,
 the slit cannot project to a short curve on $T_j$ and it must be
 trivial in homology. Write $M^{(0)} = M_q$, let
$
\phi_j: M^{(0)} \to
  M^{(j)}$ 
be 
  the affine comparison map corresponding to $g_{-t_j}$, and 
 let $\delta_j \subset M^{(0)}$ denote the pullback under $\phi_j$ of    one of the saddle
 connections that make up this slit, so that the other is
 $\iota(\delta_j)$. Letting $\theta_j$ be the direction of $\delta_j$, we obtain  (i).  
 
 Because the horizontal flow on $M^{(0)}$ is minimal, the $\delta_j$
 are not horizontal. For any fixed non-horizontal segment $\delta$ on $M^{(0)}$, the length of $\phi_j(\delta)$ on $M^{(j)}$ goes to infinity as $j\to \infty.$ Therefore  we may assume that the $\delta_j$ are all different. By the
 discreteness of holonomies of saddle connections (see \cite{MT}), this implies that $\|(x_j, y_j)\|\to \infty$, 
 where $(x_j, y_j) = \hol_{M^{(0)}}(\delta_j)$. Since
 $$\| \hol_{M^{(j)}}(\phi_j(\delta_j))\|=\|(e^{-t_j}x_j,e^{t_j}y_j)\|\to 0,$$ 
 we have
 that $y_j \to 0$ and so $x_j\to \infty$. Because the torus $T_j$ is
 in the compact set $\mathcal{K}$, the only short saddle
 connections of $M^{(j)}$
 are $\delta_j$ and $\iota(\delta_j)$
 which implies the second assertion in (ii). This establishes (ii).  

\medskip

{\bf Step 2. Sets of uniform convergence for the straightline flow on one side of the slit.}
\medskip

For the proof of (iii), recall that 
$\Upsilon^{(p)}(t)$ denote the horizontal straightline flow (starting at $p$, with time parameter $t$).
By the 
Birkhoff ergodic theorem, there is an increasing sequence $S_k \to
\infty$ and an increasing sequence of 
subsets $E_k \subset 
M^{(0)}$ such that $\lim_{k \to \infty} \mu(M^{(0)}\sm  E_k) =0$, $\Upsilon^{(p)}(t)$ is
defined for all $t\in \R$ and all $p \in  E_k$, and 
for any choice of $p_k \in E_k$, and an interval $I_k \subset \R$ around 0 of length
$|I_k| \geq  S_k$, the
empirical measures $\eta_k $
on $M^{(0)}$ defined by  
$$
\int f d
\eta_k = \frac{1}{|I_k|} \int_{I_k} f\left(\Upsilon^{(p_k)}(t)\right) dt \ \ \
(f \in C_c(M_q))
$$
satisfy 
\eq{eq: empirical}{\eta_k \to_{k \to \infty} \mu, \text{ with respect
    to the weak-$*$ topology.}}


{\bf Step 3. Notation for $\Upsilon(t)$-orbit segments on one side of the slit. }

\medskip
\begin{figure}[h]
\begin{tikzpicture}
 [bdot/.style={circle,draw,fill=black,inner sep=0pt,minimum size=1.0mm},
 wdot/.style={circle,draw,fill=white,inner sep=0pt,minimum size=1.0mm},
 sdot/.style={circle,draw,fill=black,inner sep=0pt,minimum size=0.3mm}]
\def\xa{5};
\def\ya{0};
\def\xb{1};
\def\yb{3};
\def\xc{-4};
\def\yc{1.0};
\def\xp{-2.0};
\def\yp{2.6};
\def\xq{-1.5};
\def\yq{3.2};
\node (A0) at (\xa,\ya) [sdot] {};
\node (A1) at (\xa+\xb,\ya+\yb) [sdot] {};
\node (A2) at (\xa+\xb+\xc,\ya+\yb+\yc) [sdot] {};
\node (A3) at (\xa+\xc,\ya+\yc) [sdot] {};

\node (p) at (\xp+\xa,\yp+\ya) [bdot] {};

\node (q) at (\xq+\xa,\yq+\ya+\yc-\yb+\yc) [wdot] {};

\draw (A0) -- (A1) -- (A2) -- (A3) -- (A0);
\draw (p) -- (q);
\draw [black] (A2) -- node[right] {$\ell$} ++ (0,-3);
\node (C0) at (\xa+\xb+\xc,\ya+\yb+\yc-0.7) [sdot] {};


\path [name path=horizontal line]  (\xa+\xb+\xc,\ya+\yb+\yc-0.7) -- (4+\xa+\xb+\xc,\ya+\yb+\yc-0.7);
\path [name path=second line]  (A1) -- (A2);

\path  [name intersections={of=horizontal line and second line, by=xx}]; 
\node (D0) [sdot] at (xx) {};
\node (D1) [sdot] at ($(xx)-(\xb,\yb)$) {};

\draw [black] (C0) -- (D0);


\path [name path=horizontal line 2]  (D1) -- ++ (2,0);
\path [name path=first line]  (A0) -- (A1);

\path  [name intersections={of=horizontal line 2 and first line, by=yy}]; 
\node (D2) [sdot] at (yy) {};
\draw [black] (D1) -- (D2);

\node (D3) [sdot] at ($(yy)+(\xc,\yc)$) {};
\path [name path=horizontal line 3]  (D3) -- ++ (2,0);
\path [name path=down line]  (A2) -- ++ (0,-3);
\path  [name intersections={of=horizontal line 3 and down line, by=zz}]; 
\node (D4) [sdot] at (zz) {};
\draw [black] (D3) -- (D4);

\begin{scope}[xshift=6cm]:
\node (A0) at (\xa,\ya) [sdot] {};
\node (A1) at (\xa+\xb,\ya+\yb) [sdot] {};
\node (A2) at (\xa+\xb+\xc,\ya+\yb+\yc) [sdot] {};
\node (A3) at (\xa+\xc,\ya+\yc) [sdot] {};
\node (p) at (\xp+\xa,\yp+\ya) [bdot] {};

\node (q) at (\xq+\xa,\yq+\ya+\yc-\yb+\yc) [wdot] {};

\draw (A0) -- (A1) -- (A2) -- (A3) -- (A0);
\draw (p) -- (q);
\end{scope}
\end{tikzpicture}\ \ \ \ \ \ \ \ \
\caption{The picture explains the notation from Step 3. The parallelograms represent tori. The surface  $M^{(j)}=M_{g_{-t_j} q}$ is obtained by gluing the two tori along the slit. The  vertical line on the left connected component of $M^{(j)}\sm \phi_j(\sigma_j)$ is denoted by $\ell$. The horizontal line segment is $H_x$ for some point $x\in\ell$. Because $H_x$ does not intersect $\phi_j(\sigma_j),$ we have that $x$ is in $D_j$, and $H_x$ is in $B_j$ and in $\bar B_j$.}
\label{fig: gluing2} 
\end{figure}

Let $\sigma_j$ be the slit on $M^{(0)}$ as before, and $A_j, A'_j$ be the
two tori comprising $M^{(0)} \sm \sigma_j$  as in (i). We will define certain segments in $M^{(j)}$ and use \eqref{eq: return to compact set of tori} in order to obtain bounds on their length.

Denote by $\iota$ the
involution of Proposition \ref{prop: structure 
  of E}, on both
$M^{(0)}$ and $M^{(j)}$, so that $\phi_j$ commutes with $\iota$.  Then
$\phi_j(\sigma_j)$ is a slit on $M^{(j)}$ and as we saw,  its length 
    $|\phi_j(\sigma_j)|$ satisfies $|\phi_j(\sigma_j)| \to 0.$ Since, by \eqref{eq: return to compact set of tori},
    $\bar{\pi} \left(M^{(j)}\right) \in \mathcal{K}$ 
    for all $j$, the diameter of $M^{(j)}$ is bounded above independently of $j$. 
  Since $\phi_j(A_j)$ is one of the connected components of $M^{(j)} \sm \phi_j(\sigma_j)$, $\phi_j(A_j)$ contains a vertical segment whose length is a fixed number independent of $j$. We denote this segment by $\ell$, and let  $\ell' \df \iota(\ell)$ (the segments $\ell$ and $\ell'$ depend on $j$ but we omit this from the notation). For each $x
    \in \ell \cup \ell'$ let $I(x)$ be the interval starting at 0,
    such that $H_x \df \left\{
      \Upsilon^{(x)}(t): t\in I(x)\right\}$ is the horizontal segment
    on $M^{(j)}$ 
    starting at $x$ and ending at the first return to $\ell \cup
    \ell'$. Then, by considering the projection to $\mathcal{K}$, we see that the length of $I(x)$ is bounded above and below by
    positive constants independent of $j$ and $x$, and by adjusting
    $\ell$ there is a constant $C$ such that
$$\forall j, \, \forall x \in \ell \cup \ell', \ \text{ we have } \ 1 \leq |I(x)| \leq C.$$ 
Let
    $$D_j \df \left\{ x\in \ell \cup \ell' : \phi_j(\sigma_j) \cap H_x 
      = \varnothing
    \right\}$$
    and
    $$
B_j \df \bigcup_{x \in D_j} H_x \ \ \text{ and } \ \bar  B_j \df \bigcup_{x \in D_j \cap \ell} H_x.
    $$
    Thus $B_j$ is the union of trajectories in $M^{(j)}$ starting and ending in $\ell \cup \ell'$ that do not pass through the slit $\phi_j(\sigma_j)$, and $\bar B_j$ is the set of such trajectories that stay in $\phi_j(A_j)$. 
    Then clearly $\iota(B_j)=B_j$ and moreover,  since $|\phi_j(\sigma_j)| \to
    0$, $\Leb(B_j) \to \Leb(M^{(j)}) = 
    2$. Similarly we have $\Leb(\bar B_j) \to
    1$. 

Let $k_j$  be the largest $k$ 
for which $e^{t_j} \geq S_{k}$. Then $k_j \to \infty$. 
By Proposition
\ref{prop: involution on E and measures} we have $\Leb  = \mu +
\iota_* \mu$, and so for large enough 
$j$,  $\phi_j(E_{k_j} \cup \iota(E_{k_j})) \cap B_j \neq \varnothing.$
Since $\iota(B_j)=B_j$ this implies $\phi_j(E_{k_j})\cap B_j \neq
\varnothing.$ Since the two tori $A_j, \, A'_j$ cover $M^{(0)}$, by
replacing $A_j$ with $A'_j$ if necessary, we may assume that for all large enough $j$,
$$\phi_j(A_j \cap  E_{k_j})\cap B_j \neq
\varnothing.$$

\medskip

{\bf Step 4: Comparing $\Upsilon(t)$-orbit segments on one side of the slit,  and Lebesgue measure restricted to that component.}

\medskip

Let $\mu_j$ be the restriction of $\Leb$ to $A_j$, so that $\mu_j$ is a probability measure.
 Our goal is to show that 
    for all $\vre>0$ and $f \in C_c\left (M^{(0)} \right)$, for
    all $j$ large enough we have 
   \eq{eq: our goal}{ \left| 
\int_{M^{(0)}} f d\, \mu_j - \int_{M^{(0)}} f d\mu \right| < \vre.
    }
    We will do this by showing that orbit segments of points in $E_k$, which are almost generic for $\mu$, track orbit segments of other points, which approximate $\Leb$ (\eqref{eq: second third} below).
We  assume with no loss of
    generality that $\|f \|_
    \infty =1$.

Fix $x_1   \in \phi_j( A_j \cap E_{k_j}) \cap B_j$ and let $y_1 \df
\phi_j^{-1}(x_1)$.  There is $x \in \ell \cap D_j$ such that $x_1 \in H_x$, and
we let $y \df \phi_j^{-1}(x)$. Recall that 
  $\phi_j^{-1}$ maps horizontal and vertical straightline segments
  on $M^{(j)}$ to horizontal and vertical straightline segments on
  $M^{(0)}$, multiplying their lengths respectively by $e^{\pm
    t_j}$. In particular $J(y) \df \phi_j^{-1}(H_x)$ is a horizontal line
  segment on $M^{(0)}$ of length at least $e^{t_j}$ and containing
  $y_1$, and since $y_1 \in E_{k_j}$, this implies via \eqref{eq: empirical} that for $j$
  sufficiently large,
\eq{eq: first third}{\left| \frac{1}{e^{t_j}|J(y)|}
\int_0^{e^{t_j}|J(y)|} f\left( \Upsilon^{(y)}(t) \right) dt - \int_{M^{(0)}}
f d \mu \right| < \frac{\vre}{3}. 
}

 We now make the previously mentioned orbit tracking argument: Let $x' \in D_j \cap \ell.$  So there is a vertical subsegment of $\ell$, with length at most $C$,  connecting $x$
and $x'$. Since $\ell \subset \phi_j(A_j)$, this segment lies
completely inside $\phi_j(A_j)$. 
Arcs starting in $\phi_j(A_j)$ can only leave $\phi_j(A_j)$ by passing through the slit $\phi_j(\sigma_j)$. Thus, if $\Upsilon^{(x)}(t)$ is in $\phi_j(A_j)$ and the vertical straightline segment of length $C$ starting at 
$\Upsilon^{(x)}(t)$ misses
$\phi_j(\sigma_j)$, there is also a vertical segment from $\Upsilon^{(x)}(t) $
to $\Upsilon^{(x')}(t )$ of length at most $C$, which lies 
completely inside $\phi_j(A_j)$. 

For any  $x' \in D_j \cap \ell$,  we set $y' \df
  \phi_j^{-1}(x')$.
 Since $|\phi_j(\sigma_j)|\to 0$, the discussion in the preceding paragraph implies that
there is a finite union of subintervals $J_1 = J_1 (y')$ in $J
      =J(y')$, such
    that $|J_1| = O(|\phi_j(\sigma_j)|) \to 0$ and such that for all $t \in
    J 
    \sm J_1$ there is a vertical line segment of 
    length at most  $C$ from $\Upsilon^{(x)}(t)$ to
    $\Upsilon^{(x')}(t)$, and this segment stays completely in $\phi_j(A_j)$. 

  Thus,  for all large enough $j$ we have
  \eq{eq: second third}{\frac{1}{e^{t_j}|J(y')|} 
\int_0^{e^{t_j}|J(y')|} \left|  f\left(\Upsilon^{(y')}(t) \right) -
  f\left(\Upsilon^{(y)}(t)\right)  \right | 
  dt   < \frac{\vre}{3}.  
}
Let $\bar \mu_j$ be the restriction of $\Leb$ to
$\phi_j^{-1}(\bar B_j)$. 
Then using Fubini's theorem to express $\bar \mu_j$ as an integral of
integrals along the lines $\phi_j^{-1}(H_{x'})$, for $x' \in D_j \cap
\ell$, we find from \equ{eq: first third} and \equ{eq: second third} that
\eq{eq: third third}{
\left| \int_{M^{(0)}} f d \bar \mu_j - \int_{M^{(0)}} f d\mu \right| < \frac{2\vre}{3}.
}
Since $\bar B_j \subset \phi_j(A_j)$ and $\phi_j^{-1}$ preserves Lebesgue measure,
we have
$$\phi_j^{-1}(\bar B_j) \subset A_j, \ \ \Leb(\phi_j^{-1}(\bar B_j)) \to 1 = 
\Leb(A_j)$$
and hence for all large $j$,
$$\left| \int_{M^{(0)}} f \, d\bar \mu_j - \int_{M^{(0)}} f \, d \mu_j \right| < \frac{\vre}{3}.$$
Combining this with \equ{eq: third third} gives \equ{eq:
  our goal}. 
\end{proof}

Similar ideas can be used to prove the
following statement.

\begin{thm}\name{thm: Masur Smillie converse}
Suppose $q \in \EE$, and the horizontal measured foliation of the underlying surface $M_q$  is minimal but not ergodic. Then there is a sequence
of decompositions of $M_q$ into pairs of tori $A_j$ and $A'_j$ glued along slits,
 and 
such that the set 
$$
A_\infty = \bigcup_i \bigcap_{j\geq i} A_j
$$
is invariant under the horizontal flow, and has Lebesgue measure $1/2$.

\end{thm} 

The statement will not be used in this paper and its proof is left to
the reader. 

\section{Tremors} \name{sec: tremors}
In this section we give a more detailed treatment of tremors and their
properties. 
\subsection{Definitions and basic properties}\name{subsec: tremor
  definitions} 
\subsubsection{Semi-continuity of foliation
  cocycles}\name{subsec: 
  foliation cycles}
  Let $q\in\HH$ represent a surface $M_q$ with horizontal foliation
  $\FF_q$. Recall from \S \ref{subsec: transverse} that the transverse
  measures (respectively, signed 
  transverse measures) define a cone $C^+_q$ of foliation cocycles
  (resp., a space  $\tremspace_q$ of signed foliation cocycles) and
  these are subsets of $H^1(M_q, \Sigma; \R_x)$.
For a marking map $\varphi: S \to M_q $ representing a marked
translation surface $\til q \in \pi^{-1}(q)$, the pullbacks
$\varphi^*(C_q^+)$ and $\varphi^*(\tremspace_q)$ are subsets of $H^1(S,
\Sigma; \R_x)$ and will be denoted by
$C_{\til q}^+$ and $\tremspace_{\til q}$.  Note that these notions are
well-defined even at orbifold 
points (i.e.\ do not depend on the choice of the marking map) because
translation equivalences map transverse measures to transverse
measures. Recall that 
$\beta \in C^+_q$ is called non-atomic if $\beta = \beta_\nu$ for a
non-atomic  transverse measure $\nu$.
We will mostly work with 
non-atomic transverse measures as described in \S \ref{subsec:
  transverse}, and for completeness explain the atomic  case in \S
\ref{sec: atomic tremors}.

Recall from \S \ref{subsec: atlas of charts} that for any $q$, the
tangent space $T_q(\HH)$ at $q$ 
is identified with $H^1(M_q, \Sigma_q; \R^2)$ (or with  $H^1(M_q,
\Sigma_q; \R^2)/\Gamma_q$ if $q$ is an orbifold point), and that a
marking map identifies the tangent space $T_{\til q_1}(\HHm)$
with $H^1(S, \Sigma; \R^2)$. The
following proposition expresses an 
important semi-continuity property for the cone of foliation cocycles.

\begin{prop}\name{prop: semicontinuity}
The set 
$$
C^+_\HH \df \left \{ (\til q , \beta) \in \HHm \times H^1(S, \Sigma; \R_x) : \beta
\in C^+_{\til q}  \right \}
$$\index{C+H@$C^+_\HH$}is closed. That is,
suppose $\til q_n \to \til q$ is a convergent sequence in 
$\HH_{\mathrm{m}}$, and let $C^+_{\til q_n}, C^+_{\til q} \subset H^1(S,
\Sigma; \R_x)$ be the corresponding cones. Suppose that $\beta_n \in H^1(S,
\Sigma; \R_x)$ is a convergent sequence such that $\beta_n \in C^+_{\til
  q_n}$ for every $n$. Then $\lim_{n\to \infty} \beta_n \in C^+_{\til
  q}$. 
\end{prop}

Proposition \ref{prop: semicontinuity} will be proved in \S
\ref{subsec: polygonal tremors} under an additional assumption  and in
\S \ref{sec: atomic tremors} in general. Note that care is required in
formulating an
analogous property for $\tremspace_q$ because $\dim
\tremspace_q$ can decrease when taking limits.  See Corollary \ref{cor: total
  variation continuous}. Also note the requirement $\til{q}
 \in \HHm$; our definitions of transverse measures are not well-suited to degenerations involving limiting to  $\til q$ in a boundary stratum.

\subsubsection{Signed mass, total variation, and balanced
  tremors}\name{subsec: length} We now define the {\em
  signed mass} and {\em total variation} of a signed foliation
cocycle. Recall from \S\ref{sec: basics} 
that $dx = (dx)_q$  denotes  the
canonical transverse measure for the vertical  
foliation on a translation surface $q$ and $\hol^{(x)}_q$ denotes the
corresponding element of $H^1(M_q, \Sigma_q; \R)$.
Given $q \in \HH$ and $\beta \in H^1(M_q,
\Sigma_q; \R)$, denote by 
$L_{q} (\beta)$ \index{Lq1@$L_q(\beta)$} the evaluation of the cup
product $\hol_{ q}^{(x)} \cup \beta$ on 
the fundamental class of $M_q$. 
In particular, if $\beta = \beta_\nu$ for a non-atomic signed
transverse measure $\nu$ then
\begin{equation}\label{eq: formula for L0}L_q(\beta) =\int_{M_q} dx \wedge \nu;
\end{equation}
or
equivalently, if $\mu = \mu_\nu$ is the horizontally invariant signed measure
associated to $\nu$ by Proposition \ref{prop:trans to meas}, then
$L_q(\beta) = \mu(M_q)$. 
We will refer to $L_q(\beta)$ as the {\em signed mass}\index{signed mass} of $\beta$. 
Our sign conventions imply that $L_q( \beta) >0$ for any nonzero $\beta \in
C^+_q$.

Note
that if 
$h : M_q \to M_q$ is a translation equivalence then $L_q(\beta)
= L_q (h^* (\beta))$. Thus, if $\til q \in \pi^{-1}(q)$ is a marked translation
surface represented by a marking map $\varphi$, and $\beta' \in H^1(S, \Sigma; \R)$ satisfies $\beta = \varphi_* \beta'$, then we can define $L_{\til
  q}(\beta') \df L_q(
\beta)$, and this definition 
does not depend on the choice of the marking map 
$\varphi$ representing $\til q$. In particular the mapping $(q, \beta) \mapsto
L_q(\beta)$ defines a map on  $T(\HH)$, even if $q$ lies in an orbifold
substratum.

 Recall that every signed measure and every signed transverse measure has
a canonical Hahn decomposition $\nu =
\nu^+-\nu^-$ as a difference 
of  measures. Thus  any $\beta \in
\tremspace_q$ can be written as $\beta = \beta^+-\beta^-$ where
$\beta^{\pm} \in C^+_q$. In analogy with the total variation of a measure we now define
\begin{equation}
|L|_q(\beta) = L_q(\beta^+)+L_q(\beta^-),
\end{equation}
\index{Lq1@$\vert L\vert_q(\beta)$}and call this the 
{\em total  variation}\index{total variation} of
$\beta$. Note that the signed mass is defined for every $\beta \in
H^1(M_q, \Sigma; \R)$ but the total variation is only defined for $\beta
\in \tremspace_q$. 
The linearity of the cup product implies that the maps
$$T(\HH) \to \R, \, (q, \beta)
\mapsto L_q(\beta) \  \ \ \mathrm{and } \ \  \ T(\HHm) \to \R, \,  (\til q,
\beta) \mapsto L_{ \til{q}}(\beta)$$ 
are both continuous. In combination with Proposition \ref{prop:
  semicontinuity}, this 
implies: 
\begin{cor}\name{cor: more continuity}
The sets
$$
C^+_{\HHm, 1} \df \{(\til q, \beta):
\beta \in C^+_{\til q}, \, L_{\til q }(\beta) =1 \} 
$$
and 
$$
C^+_{\HH, 1} \df \{(q, \beta):
\beta \in C^+_{q}, \, L_q(\beta)=1 \} 
$$
are closed, and thus define closed subsets of $T(\HHm) $ and $T(\HH) $. 
\end{cor}

The following special case will be important in the proofs of Theorem
\ref{thm: 1}
and Theorem \ref{thm: 2}. 

\begin{cor}\name{cor: semicontinuity} 
 Let $q \in \HH$, and denote its canonical foliation cocycle by $\hol_q^{(y)}$. Suppose the  underlying translation surface $M_q$ has area one and is
horizontally uniquely 
ergodic.  Then for
any sequence $q_n \in \HH$ such that $q_n\to q$, and any $\beta_n \in
C^+_{q_n}$ with $L_{q_n}(\beta_n)=1$,  
we have $ \beta_n \to \hol_q^{(y)}$. 
\end{cor}

The total variation of a foliation cocycle also has a semicontinuity property:

\begin{cor}\name{cor: total variation continuous}
Suppose $\til q_n \to \til q$ in $\HHm$ and $\beta_n \in
\tremspace_{\til q_n} \subset H^1(S, \Sigma; \R)$ is a
sequence of non-atomic signed foliation cocycles for which the limit $\beta =
\lim_{n \to \infty} \beta_n$ exists and $\sup_n |L|_{\til q_n}(\beta_n) <
\infty.$ Then $\beta \in \tremspace_{\til q}$ and
\eq{eq: limit L}{
|L|_{\til q}(\beta) {\leq }
\liminf_{n \to \infty}
|L|_{\til q_n}(\beta_n). 
}
\end{cor}

Corollary \ref{cor: total variation continuous} will also be proved in
\S \ref{subsec: polygonal tremors} under an additional assumption, and the proof in the general case will be given in \S \ref{sec: atomic tremors}.

We say that $\beta \in \tremspace_q$ is {\em balanced}\index{balanced}
if $L(\beta)=0$, 
and we let $\tremspace_q^{(0)}$ denote the set of balanced signed 
foliation cocycles. 
\index{T@$\tremspace_q^{(0)}$}
Combining Corollary \ref{prop: KZ over E} and Proposition \ref{prop: involution on E and measures},  for surfaces in $\EE$ we see that balanced foliation cocycles are those that are `normal' to $\EE$:

  \begin{cor} \name{cor: balanced tremors are in B-} Let
    $\mathcal{O}$ be an orbifold substratum of $\HH$ and $q\in
    \mathcal{O}$. Then $\tremspace_q\cap \mathscr{N}_x(\mathcal{O})
    \subset \tremspace_q^{(0)}$, with equality in the case
    $\mathcal{O} = \EE$; namely, 
if $q \in \EE$ is aperiodic then  $\tremspace_q^{(0)}=\mathscr{N}_x(\EE)$.    
   \end{cor}
\combaraknew{Did fairly extensive editing in the next proof.}

    \begin{proof}
      Let $q \in \mathcal{O}$, let $\Gamma_q$ be the group of
      translation equivalences of $M_q$, let $\mathcal{G} \df
      \mathcal{G}_q$ be the local group as in \S 
      \ref{subsec: orbifold} and let 
      $\gamma \in \mathcal{G}$. Recall that $\Gamma_q$ and
      $\mathcal{G}$ are isomorphic and by fixing a marking map, we
      can think of $\gamma$ simultaneously as acting on $M_q$ by
      translation automorphisms, and on $H^1(S, \Sigma; \R^2)$ by the
      natural map induced by a homeomorphism. 
      Since translation automorphisms of $M_q$ preserve the canonical transverse measure
      $(dx)_q$, we have $\gamma^* \hol^{(x)}_q= \hol^{(x)}_q$, and
      thus for any $\beta$,
      \[
        \begin{split}
 L_q(\gamma^* \beta) = & (\hol_q^{(x)} \cup \gamma^* \beta) (M_q) 
=  (\hol_q^{(x)} \cup \beta)
(\gamma(M_q)) \\ = & (\hol_q^{(x)} \cup \beta)(M_q) = L_q(\beta).
      \end{split}\]
Hence, if $\beta \in \tremspace_q \cap \mathscr{N}_x(\mathcal{O})$
then $P^+(\beta)=0$, where  $P^+$ is the projection onto the tangent space of $\mathcal{O}$ given in \equ{eq: formulae projections}, and we have 
\[
  L_q(\beta) 
    =\frac{1}{|\mathcal{G}|} \sum_{\gamma\in
    \mathcal{G}}L_q(\gamma^*(\beta)) =L_q\left(\frac{1}{|\mathcal{G}|}
    \sum_{\gamma\in \mathcal{G}}\gamma^*(\beta)\right)
  =L_q(P^+(\beta)) 
=0.
\]
Therefore $\beta \in \tremspace_q^{(0)}.$ 
 
  Now if $q\in \EE$ is aperiodic and $\beta \in \tremspace^{(0)}_q$,
  then we can write $\beta = \beta_\nu$ for a signed transverse
  measure $\nu$, and let $\mu = \mu_\nu$ be the 
 associated horizontally invariant signed measure (see
 Proposition \ref{prop:trans to meas}). Since $\beta \in
 \tremspace^{(0)}_q$ we have $\mu(M_q)=0$. Recall from Proposition
 \ref{prop: involution on E and measures} that aperiodic surfaces in
 $\EE$ are either uniquely ergodic, or have two ergodic measures which
 are exchanged by the involution $\iota = \iota_q$. By ergodic
 decomposition (applied to each summand in 
 $\mu = \mu^+-\mu^-$) we can write
 $\mu$ as a linear combination of ergodic measures (where the coefficients
 may be negative). If $M_q$ is uniquely ergodic then this
 gives $\mu = c \cdot \Leb$ and since $\mu(M_q) =0$ we have $\mu=0$. If
 $M_q$ has two ergodic probability measures $\mu_1$ and $\mu_2 = \iota_* \mu_1$
 then $\mu = c_1 \mu_1 + c_2 \iota_* \mu_1$ and
 $$
0 = \mu(M_q) = c_1 \mu_1(M_q) + c_2 \mu_1 (\iota (M_q)) = c_1 +c_2,
$$
so $c_1 = -c_2$. In both cases we obtain $\iota_* \mu = -\mu$, which
implies $\iota_* \beta = -\beta$. Thus, using Corollary \ref{prop:
  KZ over E}, we see that $\beta \in \mathscr{N}_x(\EE)$.
\end{proof}

\subsubsection{Absolutely continuous foliation cocycles}\name{subsec: ac}
Let $\nu_1$ and $\nu_2$ be two signed transverse measures for $\FF_q$.
We say that {\em $\nu_1$ is
absolutely continuous with respect to $\nu_2$}\index{absolute
continuity} if the corresponding signed
measures $\mu_{\nu_1}, \mu_{\nu_2}$ given by Proposition
\ref{prop:trans to meas} satisfy $\mu_{\nu_1} \ll \mu_{\nu_2}$. We say
that {\em $\nu$ is absolutely continuous} if it is absolutely 
continuous with respect to the canonical transverse measure
$(dy)_q$. Since $(dy)_q$ is non-atomic, so is any absolutely 
continuous signed transverse measure.  For $c>0$, we say  $\nu$ is
{\em $c$-absolutely continuous} 
\index{c absolutely continuous@$c$-absolutely continuous}
if 
\eq{eq: equivalently}{
\text{for any
transverse  arc } \gamma \text{ on } M_q, \ \ \ \left| \int_\gamma d\nu \right | \leq c
\left | \int_\gamma dy \right |.
} 
We call a signed foliation cocycle $\beta = \beta_\nu$  {\em absolutely
  continuous} (respectively, {\em $c$-absolutely continuous}) if it
corresponds to a signed transverse measure $\nu$ which 
is absolutely continuous (resp., $c$-absolutely
continuous). Let $\|\nu\|_{RN}$ denote the minimal $c$
such that the above equation holds for all transverse arcs
$\gamma$\index{Cq+@$\|\nu\|_{RN}$} (our notation stems from the fact
that $\|\nu\|_{RN}$ is the 
$L^{\infty}$-norm of the Radon-Nikodym derivative $\frac{d\mu_\nu}{d
  \Leb}$, although we will not be using this in the sequel).
Given $q \in \HH$ and $c>0$, denote by
$C^{+,RN}_q(c)$\index{C@$C^{+,RN}_q(c)$} (respectively, by
$\tremspace^{RN}_q(c)$\index{T@$\tremspace^{RN}_q(c)$})
the set of 
absolutely continuous (signed) foliation cocycles $\beta_\nu$ with $\|\nu\|_{RN} \leq
c$.

\begin{remark}
As the reader will note, we will use both $|L|_q(\beta)$ and $\|\nu\|_{RN}$ to measure the `size' of a foliation cocycle $\beta = \beta_\nu.$ 
For most purposes in this paper, $|L|_{q}$ is easier to work with. Additionally, it is more broadly defined, making sense when the tremor corresponds to a singular measure. However, $\|\cdot\|_{RN}$ is more suitable for estimates involving the distance function $\dist$ (see Proposition \ref{prop: lipschitz tremors}) and plays an essential role in the proof of Proposition \ref{prop: proper on c-ac}.
\end{remark}
It is easy to see that
\eq{eq: bound on length}{C^{+, RN}_q(c) \subset \{\beta \in C^+_q:
  L_q(\beta) \leq c\}}
and
\eq{eq: bound on length signed}{
\tremspace^{RN}_q(c) \subset \{\beta \in \tremspace_q : |L|_q(\beta) \leq c\}.
}

As we will see in Lemma \ref{lem: automatic absolutely continuous},
for some surfaces we will also have a reverse inclusion.

We now observe that for aperiodic surfaces, the assumption of absolute
continuity implies a bound on the Radon-Nikodym derivative:

\begin{lem}\name{lem: absolutely continuous}
Suppose $M_q$ is a horizontally aperiodic surface, $\nu$ is an
absolutely continuous 
transverse measure, and 
$\mu = \mu_\nu$ is the corresponding measure  on $M_q$, so that $\mu
\ll
\Leb$. Then there is $c>0$
such that $\|\nu\|_{RN} \leq c$.
Moreover the constant $c$ depends only on the coefficients  
appearing in the ergodic decompositions of $\mu$ and $\Leb$, and
if $\mu$ is a probability measure and $\Leb = \sum a_i \nu_i,$
where $\{\nu_i\}$ are the horizontally invariant ergodic probability
measures and each $a_i$ is positive, then $\|\nu\|_{RN}  \leq\max_i
\frac{1}{a_i}.$ The same conclusions hold if instead of assuming $M_q$
is aperiodic,\index{aperiodic measure} we assume the measure $\nu$ is
aperiodic, that is  
$\mu$ assigns zero measure to any horizontal cylinder on $M_q$. 
\end{lem}
\begin{proof}
Let $\{\mu_1, \ldots, \mu_d\}$ be the invariant ergodic 
probability measures for the  horizontal straightline flow on
$M_q$. Since $M_q$ is horizontally aperiodic, this is a finite
collection, see e.g. \cite{transverse measures}. Thus
there only finitely many ergodic measures 
which are absolutely continuous with respect to $\mu$, and we denote
them by $\{\mu_1, \ldots, 
\mu_{k}\}$. The measures $\mu_i$ are mutually singular. 
Write $\Leb = \sum_i a_i \mu_i$ and $\mu = \sum_i b_i
\mu_i$, where all $a_i, b_j$ are
non-negative and not all are zero. Since  $\mu \ll 
\Leb$, we have 
$$
b_i > 0 \implies a_i > 0.
$$
Set
\eq{eq: formula for ac const}{
c \df \max \left\{\frac{b_i}{a_i} : b_i \neq 0\right\}.}
For any Borel set $A \subset M_q$ we have 
$$
\mu(A) = \sum_i b_i \mu_i (A) \leq  c \sum_i a_i
\mu_i(A)= c \, \Leb(A). 
$$
This implies that the Radon Nikodym derivative satisfies
$\frac{d\mu}{d \Leb} \leq c$ a.e. The horizontal invariance of $\mu$
and $\Leb$ shows that the Radon-Nikodym derivative $\frac{d\mu}{d \Leb}$
is defined on almost every point of every transverse arc $\gamma$, and
the relation \equ{eq: formula above} shows that it coincides with the 
Radon-Nikodym derivative 
$\frac{d\nu}{(dy)_q}$. Thus we get \equ{eq: equivalently}.

The second
assertion follows from \equ{eq: formula for ac const}, and the last
assertion follows by letting $\mu_i$ denote the horizontally invariant
measures on the complement of the union of the horizontal cylinders in $M_q$, and
repeating the argument given above. 
\end{proof}

\subsubsection{Tremors as affine geodesics, and their domain of
  definition}\label{sec: trem ode} 
Recall  from \S \ref{subsec: atlas of charts} that we
identify $T(\HHm)$ with $\HHm \times H^1(S,\Sigma,\mathbb{R}^2)$. 
 Our
particular interest is in affine geodesics tangent to signed foliation
cocycles. That is, we take $$\beta \in \tremspace_{\til q} \subset
H^1(S,\Sigma;\R_x)$$ (where the last inclusion uses a marking map
$\varphi : S \to M_q$ representing $\til
q$). We write $v = (\beta, 0) \in H^1(S, \Sigma; \R^2)$ and consider
the parameterized line $\theta(t) = \theta_{\til q,v}(t)$ in 
$\HHm$ satisfying 
\eq{eq: linear equation}{\theta(0) = \til q \  \text{ and  } \ \frac{d}{dt} \theta(t)=v}
(where we have again used the marking to identify the tangent space
$T_{\theta(t)}(\HHm)$ with $
H^1(S, \Sigma; \R^2)$). 
By the uniqueness of
solutions of differential equations, these equations uniquely define
the affine geodesic $\theta(t)$ for $t$ in the maximal domain of definition
$\Dom(\til q, v)$.
As in the introduction we now have $\trem_{
  t,\beta}(\til q) = \theta(t)$ and 
  $\trem_{\beta}(\til q)  = \theta(1)$ when $1 \in \Dom (\til q, \beta)$.
Equation \eqref{eq: linear equation} and uniqueness of solutions imply that for $c>0$ we have $\theta_{\til q, cv}(t)=\theta_{\til q,v}(ct)$ and $\Dom(\til q, v) = c \, \Dom(\til q, cv).$ In particular $\trem_{t,c\beta}(\til q) =\trem_{ct,\beta} (\til q)$ and thus $\trem_{t, \beta}(\til q) = \trem_{t\beta}(\til q).$
\index{trem@$\trem_{\beta}(q)$}
\index{trem@$\trem_{t,\beta}(q)$}

Since the developing map is affine, we find 
\eq{eq: tremor}{
\hol^{(x)}_{\trem_\beta(\til q)}(\gamma) =
\hol^{(x)}_{\til q}(\gamma)+\beta (\gamma), \ \ 
\hol^{(y)}_{\trem_\beta(\til q)}(\gamma) = \hol^{(y)}_{\til  
  q}(\gamma).}  
 Comparing equations \equ{eq:
  tremor} and
\equ{eq: tremor3}, we see that we have
given a formal definition of the tremors introduced in \S \ref{subsec: tremors intro}. 

 The pure mapping class group $\Mod(S,\Sigma)$ acts 
on each coordinate of $T(\HHm)=\HHm \times
H^1(S,\Sigma,\mathbb{R}^2)$, and by equivariance we find that
$$\trem_\beta(q) = \pi(\trem_{\beta}(\til q)) \ \ \text{ and } \ \ \Dom(q,
\beta) \df \Dom (\til q, \beta)$$ 
are  well-defined and
independent of the choice of $\til q \in \pi^{-1}(q)$. 

Basic properties of ordinary differential equations now
give us:

\begin{prop}\name{prop: continuity}
The set 
$$\mathcal{D} =\{(\til q, v, s) \in T(\HHm) \times \R :  s \in \Dom(\til q, v)\}$$ 
\index{D@$\mathcal{D}$}
is open in $T(\HHm) \times \R$, and the map 
$$
\mathcal{D} \ni (\til q, v, s) \mapsto \theta_{\til q, v}(s)
$$ is continuous. In particular the tremor map 
$$\{(\til q , \beta) \in T\HHm :
\beta \in \tremspace_q\} 
 \to \HHm, \ \ (\til q, \beta) \mapsto
\trem_{\til q, \beta} $$
is continuous where defined. 
\end{prop}

Comparing equation \equ{eq: tremor} to the definition of the horocycle flow in
period coordinates, we immediately see that for the canonical
foliation cocycle $dy = \hol^{(y)}_{\til q}$, we have
\eq{eq: horocycles are tremors}{
\trem_{sdy}(\til q) = u_s\til q.
}

\subsection{Tremors and polygonal presentations of
  surfaces}\name{subsec: polygonal tremors}
In this section we prove Proposition \ref{prop:
  semicontinuity}, under an additional hypothesis. This special case
is easier to prove and suffices for proving our main results. 
We will prove the
general case of Proposition \ref{prop: semicontinuity} in \S \ref{sec:
  atomic tremors}. At the end of this section we deduce
Corollary \ref{cor: total variation continuous} from Proposition
\ref{prop: semicontinuity}.

\begin{prop}\name{prop: uniform ac closed}
Let $\til q_n \to \til q$ in $\HHm$, $\beta_n \to \beta$ in $H^1(S,
\Sigma; \R_x)$ be as in the statement of
Proposition \ref{prop: semicontinuity}. Write $q_n = \pi(\til q_n), \,
q = \pi(\til q)$ and suppose also that 
\eq{eq: star}{\text{ 
  there is a }c>0 \text{  such that for
  all } n, \  \beta_n \in C^{+,RN}_{q_n}(c). }
Then $
\beta \in C^{+,RN}_{q}(c)$. 
\end{prop}

Clearly Proposition \ref{prop: uniform ac closed} implies
Proposition \ref{prop: semicontinuity} in the case that \equ{eq: star}
holds.

Recall that any translation surface has a polygon decomposition, and that fixing a polygon decomposition on a marked surface makes it possible to consider the same polygon decomposition on nearby marked surfaces.
For the proof of Proposition \ref{prop: uniform ac closed}, we introduce polygon decompositions which are useful for understanding transverse measures to the horizontal foliation.  

In a general polygon decomposition of a surface,  
some edges might be horizontal, and corresponding edges on nearby
surfaces may intersect the horizontal foliation with different
orientations. This will cause complications and in order to avoid
them, we introduce an {\em adapted polygon
  decomposition (APD)} \index{APD} of a surface. An APD is a polygon decomposition in which all
polygons are either 
triangles with no horizontal edges, or quadrilaterals with one
horizontal diagonal. Any surface has an APD, as can be seen by taking
a triangle decomposition and merging adjacent triangles sharing a horizontal
edge into quadrilaterals. We fix an APD of $M_{q}$, with a finite
collection of edges $\{J_i\}$, all of which are transverse to the
horizontal foliation on $M_{ q} $. Since we are considering
marked surfaces, we can use a marking map representing $\til q \in \pi^{-1}(q)$ and the comparison maps of \S\ref{subsec:
  atlas of charts} and think of the arcs $J_i$ as arcs on 
$S$, as well as on $M_{q'}$ for any marked translation surface $\til{q'}$ sufficiently close to $\til q$. Moreover, the edges $\{ J_i\}$
are also  a subset of the  edges of an APD on 
$M_{ q'}$ and they are also transverse to the horizontal foliation on
$M_{ q'}$. Note that on $M_{ q'}$ the APD may contain additional
edges that are not edges on $M_{ q}$, namely some of the horizontal diagonals
on $M_{ q}$ might not be horizontal on $M_{ q'}$ and in this case we add them
to the $\{J_i\}$ to obtain an APD on $M_{ q'}$. 

\begin{figure}[h]
\begin{tabular}{l r}
\begin{tikzpicture}[scale=0.5]
\def\xa{10};
\def\ya{1};
\def\xb{3};
\def\yb{1.2};
\def\xc{0.7};
\def\yc{4.4};
\def\xd{-1.2};
\def\yd{2.3};
\def\xe{-5};
\def\ye{3};

\node (A0) at (\xa,\ya) [circle,draw,fill=black,inner sep=0pt,minimum size=1.0mm] {};
\node (A1) at (\xa+\xb,\ya+\yb) [circle,draw,fill=black,inner sep=0pt,minimum size=1.0mm] {};
\node (A2) at (\xa+\xb+\xc,\ya+\yb+\yc) [circle,draw,fill=black,inner sep=0pt,minimum size=1.0mm] {};
\node (A3) at (\xa+\xb+\xc+\xd,\ya+\yb+\yc+\yd) [circle,draw,fill=black,inner sep=0pt,minimum size=1.0mm] {};
\node (A4) at (\xa+\xb+\xc+\xd+\xe,\ya+\yb+\yc+\yd+\ye) [circle,draw,fill=black,inner sep=0pt,minimum size=1.0mm] {};
\node (A5) at (\xa+\xc+\xd+\xe,\ya+\yc+\yd+\ye) [circle,draw,fill=black,inner sep=0pt,minimum size=1.0mm] {};
\node (A6) at (\xa+\xd+\xe,\ya+\yd+\ye)  [circle,draw,fill=black,inner sep=0pt,minimum size=1.0mm] {};
\node (A7) at (\xa+\xe,\ya+\ye)  [circle,draw,fill=black,inner sep=0pt,minimum size=1.0mm] {};

\draw (A0) -- (A1) -- (A2) -- (A3) -- (A4) -- (A5) -- (A6) -- (A7) -- (A0);
\draw (A7) -- (A1);
\draw (A7) -- (A2);
\draw [dashed] (A6) -- node[anchor=north] {$\sigma$} (A2);
\draw (A6) -- (A3);
\draw (A5) -- (A3);

\end{tikzpicture}\ \ \ \ \ \ \ \ \

\begin{tikzpicture}[scale=0.5]
\def\xa{10};
\def\ya{1};
\def\xb{3};
\def\yb{1.3};
\def\xc{0.7};
\def\yc{5};
\def\xd{-1};
\def\yd{2.5};
\def\xe{-5};
\def\ye{2};

\node (A0) at (\xa,\ya) [circle,draw,fill=black,inner sep=0pt,minimum size=1.0mm] {};
\node (A1) at (\xa+\xb,\ya+\yb) [circle,draw,fill=black,inner sep=0pt,minimum size=1.0mm] {};
\node (A2) at (\xa+\xb+\xc,\ya+\yb+\yc) [circle,draw,fill=black,inner sep=0pt,minimum size=1.0mm] {};
\node (A3) at (\xa+\xb+\xc+\xd,\ya+\yb+\yc+\yd) [circle,draw,fill=black,inner sep=0pt,minimum size=1.0mm] {};
\node (A4) at (\xa+\xb+\xc+\xd+\xe,\ya+\yb+\yc+\yd+\ye) [circle,draw,fill=black,inner sep=0pt,minimum size=1.0mm] {};
\node (A5) at (\xa+\xc+\xd+\xe,\ya+\yc+\yd+\ye) [circle,draw,fill=black,inner sep=0pt,minimum size=1.0mm] {};
\node (A6) at (\xa+\xd+\xe,\ya+\yd+\ye)  [circle,draw,fill=black,inner sep=0pt,minimum size=1.0mm] {};
\node (A7) at (\xa+\xe,\ya+\ye)  [circle,draw,fill=black,inner sep=0pt,minimum size=1.0mm] {};

\draw (A0) -- (A1) -- (A2) -- (A3) -- (A4) -- (A5) -- (A6) -- (A7) -- (A0);
\draw (A7) -- (A1);
\draw (A7) -- (A2);
\draw (A6) -- (A2);
\draw (A6) -- (A3);
\draw (A5) -- (A3);

\end{tikzpicture}\ \ \ \ \ \ \ \ \
\end{tabular}
\caption{Two APD's on nearby surfaces in $\HH(2)$. The dotted horizontal line represents a diagonal of a quadrilateral on the first surface and is an edge of a triangle on the second surface since it is no longer horizontal.}
\label{fig: APD} 
\end{figure}

Since the polygons of a polygon decomposition are simply connected,
a 1-cochain representing an element of $H^1(S, \Sigma; \R)$ is
determined by its values on the edges of the polygons. 
For each $i$, each polygon $P$ of the APD with $J= J_i \subset \partial P$, and each $x
\in J$, there is a horizontal segment in $P$ with endpoints in $\partial P$ one
of which is $x$. The other endpoint of this segment is called the {\em
opposite point (in $P$) to $x$} and is denoted by \index{opposite point map $\mathrm{opp}_K$}
$\mathrm{opp}_P(x)$. The image of $J$ under $\mathrm{opp}_P$ is a
union of one or two sub-arcs contained in the other boundary edges of $P$. 

\begin{figure}[h]
\begin{tikzpicture}[scale=0.7]
\def\xa{10};
\def\ya{1};
\def\xb{3};
\def\yb{1.2};
\def\xc{0.7};
\def\yc{4.4};
\def\xd{-1.2};
\def\yd{2.3};
\def\xe{-5};
\def\ye{3};

\node (A0) at (\xa,\ya) [circle,draw,fill=black,inner sep=0pt,minimum size=1.0mm] {};
\node (A1) at (\xa+\xb,\ya+\yb) [circle,draw,fill=black,inner sep=0pt,minimum size=1.0mm] {};
\node (A2) at (\xa+\xb+\xc,\ya+\yb+\yc) [circle,draw,fill=black,inner sep=0pt,minimum size=1.0mm] {};
\node (A3) at (\xa+\xb+\xc+\xd,\ya+\yb+\yc+\yd) [circle,draw,fill=black,inner sep=0pt,minimum size=1.0mm] {};
\node (A4) at (\xa+\xb+\xc+\xd+\xe,\ya+\yb+\yc+\yd+\ye) [circle,draw,fill=black,inner sep=0pt,minimum size=1.0mm] {};
\node (A5) at (\xa+\xc+\xd+\xe,\ya+\yc+\yd+\ye) [circle,draw,fill=black,inner sep=0pt,minimum size=1.0mm] {};
\node (A6) at (\xa+\xd+\xe,\ya+\yd+\ye)  [circle,draw,fill=black,inner sep=0pt,minimum size=1.0mm] {};
\node (A7) at (\xa+\xe,\ya+\ye)  [circle,draw,fill=black,inner sep=0pt,minimum size=1.0mm] {};

\draw (A0) -- (A1) -- (A2) -- (A3) -- (A4) -- (A5) -- (A6) -- (A7) -- (A0);
\draw (A7) -- (A1);
\draw (A7) -- (A2);
\draw (A6) -- (A3);
\draw (A5) -- (A3);

\path [name path=horizontal line]  (2,8) -- (15,8);
\path [name path=first line]  (A5) -- (A6);
\path [name path=second line]  (A3) -- (A6);
\path [name path=third line]  (A2) -- (A3);
\path [name intersections={of=horizontal line and first line, by=xx}]; 
\path [name intersections={of=horizontal line and second line, by=yy}]; 
\path [name intersections={of=horizontal line and third line, by=zz}]; 

\node (X) at (xx) [circle,draw,fill=black,inner sep=0pt,minimum size=0.7mm] {};
\node [circle,draw,fill=black,inner sep=0pt,minimum size=0.7mm] at (yy) {};
\node (Z)  at (zz) [circle,draw,fill=black,inner sep=0pt,minimum size=0.7mm] {};
\node  [circle,draw,fill=black,inner sep=0pt,label=0:$y_2$] at (zz) {};
\node  [circle,draw,fill=black,inner sep=0pt,label=-180:$y_1$] at (xx) {};
\node  [circle,draw,fill=black,inner sep=0pt,label=-270:$x$] at (yy) {};
\node  [label=-270:$P_1$] at (6,8) {};
\node  [label=-270:$P_2$] at (10,6) {};

\draw [dotted] (X) -- (Z);
\end{tikzpicture}
\caption{The opposite point map, with $y_1=\mathrm{opp}_{P_1}(x)$ and $y_2=\mathrm{opp}_{P_2}(x)$.}
\label{fig: opposite point} 
\end{figure}

A transverse measure $\nu$ for the horizontal foliation on $M_q$ assigns a measure to each $J$. We will
denote this either by $\nu$, or by $\nu|_J$ when confusion may
arise. 
By the invariance property  of a transverse measure,
\eq{eq: invariance property}{
\left(\mathrm{opp}_P\right)_* \nu|_J = \nu|_{\opp_P(J)},
}
and this holds for any $P$ and $J$. We call \equ{eq: invariance
  property} the {\em  
  invariance property}. Note that in this section, all measures under
consideration are
non-atomic, and we will not have to worry about whether intervals are
open or closed (but in \S \ref{sec: atomic tremors} this will be a
concern). 

\begin{prop}\label{prop: they define}
Given an APD for a translation surface $M_q$, and a collection 
of finite non-atomic measures $\nu_J$ on the edges $J$ as above, satisfying the invariance
property, there is a transverse measure $\nu$ on $M_q$ for which $\nu|_J = \nu_J.$
\end{prop}

\begin{proof}
We can reconstruct $\nu$ from the $\nu_J$, by homotoping any transverse arc to subintervals
of edges of the APD along horizontal leaves (this is well-defined in view of the
invariance property). 
\end{proof}

 Proposition \ref{prop: they define} makes it possible to reduce questions about transverse measures on surfaces, to finitely many measures on some arcs. We use this idea in the following:  
\begin{proof}[Proof of Proposition \ref{prop: uniform ac closed}]
 We will
write $\beta_n = \beta_{\nu_n}$ for a sequence of $c$-absolutely
continuous transverse measures $\nu_n$ on $M_{q_n}$ (in particular the
$\nu_n$ are non-atomic). Our goal is to prove that there is a transverse measure
$\nu$ on $M_{ q}$ such that $\beta = \beta_\nu$.
The main idea of the proof is to use APD's to reduce the discussion to measures on finitely many transverse arcs.  It suffices to consider the restriction of the transverse measure
 to a particular finite collection of transverse arcs, which we now
 describe.

Let $\tau$ be the triangulation of $M_{ q}$ obtained by adding
the horizontal diagonals to quadrilaterals in an APD. As discussed in
\S \ref{subsec: atlas of charts}, using $\tau$ and marking maps, 
we obtain maps $\varphi_n: S \to M_{q_n}$, $\varphi:S\to M_q$,  such
that for each $n$, the comparison map $\varphi_n \circ 
\varphi^{-1}: M_q \to M_{q_n}$ is piecewise affine, with derivative
(in planar charts) tending to the identity map as $n \to \infty$.
Let $P$ be one of the polygons of the APD and $K
\subset \partial P$ a subinterval of the form $J$ or $\opp_P(J)$ as above. 
For all large
enough $n$,  none of the sides $\varphi_n \circ \varphi^{-1}(K)$ are
horizontal and all have the 
same orientation as on $M_{ q}$. Let
$\nu_{K}^{(n)}$ be the measure on $\varphi_n \circ \varphi^{-1}(K)$ corresponding to
$\nu_n$. Using the marking $\varphi_n^{-1}$ we will also think of
$\nu_{K}^{(n)}$ as a measure on $\til K = \varphi^{-1}(K)$. 

Passing to
subsequences
and using the compactness of the space of measures of bounded mass on a
bounded interval, we can assume that for each $K$, the sequence
$\left(\nu_{K}^{(n)} \right)_n$ converges 
to a measure $\nu_{K}$ on $\til K$. It follows from \equ{eq: star} that $\nu_K$
is non-atomic, indeed it is $c$-absolutely continuous since all the
$\nu_K^{(n)}$ are. 
Each of the measures $\nu_{K}^{(n)}$ satisfies the invariance property
for the horizontal foliation on $M_{q_n}$, and we claim:
\begin{claim}\name{claim: claim}
The measures
$\nu_K$ satisfy the invariance property for the horizontal foliation
on $M_q$. 
\end{claim}
To see this, suppose $K=J$ in the above notation, the case
$K=\opp_P(J)$ being similar. For
each $n$ let $\mathrm{opp}_P^{(n)}$ be the map 
corresponding to the horizontal foliation on $M_{q_n}$; it maps $J$ to a
subset of an edge or two edges of the APD. Let
$I$ be a compact interval contained in the interior of $J$. Then for all sufficiently large $n$,
$\mathrm{opp}_P^{(n)}(I) \subset \opp_P(J)$, and the maps 
$\mathrm{opp}_P^{(n)}|_{I} $ converge uniformly to $
\mathrm{opp}_P|_{I}$. By our assumption that the measure
is non-atomic, the endpoints of $I$ have
zero $\nu_{J}$-measure. Therefore, since $\nu_{J}^{(n)} \to
\nu_{J}$, by the Portmanteau theorem we have
$
\nu_{J} (I) = \nu_{\opp_P(J)} (\mathrm{opp}_P(I)).
$
Such intervals $I$ generate the Borel $\sigma$-algebra on $J$, and
so we have established the invariance
property. This proves Claim \ref{claim: claim}. \hfill
$\triangle$

\vspace{5mm}

By Proposition \ref{prop: they define}, the $\nu_K$  define a
transverse measure $\nu$, and we let $\beta' = \beta_{\nu}$. Recall
that we have assumed $\beta_n \to \beta$ as 
cohomology classes in $H^1(S, \Sigma; \R)$. For each edge $J$ of the APD,
\begin{equation}\label{eq: each edge} \beta(J) \leftarrow \beta_n(J) = 
m^{(n)}_{J}\to m_{J} =
\beta'(J),
\end{equation}
and so $\beta' = \beta$. 
\end{proof}

We now deduce Corollary \ref{cor: total variation continuous}. 
As in the proof of Proposition \ref{prop: semicontinuity}, 
we will use assumption \eqref{eq: star}.  The general case will be established in \S \ref{sec: atomic
  tremors}.

\begin{proof}[Proof of Corollary \ref{cor: total variation continuous} under 
assumption \eqref{eq: star}]
We first give a formula for $L_q(\beta_\nu)$, where $\nu$ is a transverse measure on the surface $M_q$. Fixing an APD on $M_q$, we can write \begin{equation}\label{eq: formula for L}
L_{ q}(\beta_\nu) = \sum_P \sum_{J \in L(P)} \int_J D(x)  d\nu_J(x),
\end{equation}
where $P$ ranges over the edges of the APD, $L(P)$ is the set of edges on the left-handside of $P$, and for $x \in J, D(x)$ denotes the length of the horizontal segment from $x$ to $\mathrm{opp}_P(x)$. Indeed, this formula is just a more detailed version of \eqref{eq: formula for L0}.

Now, for each $n$ write $\beta_n = \beta_{\nu_n}$ where $\nu_n$ is a
transverse measure on $M_{q_n}$, and let $\nu_n = \nu^+_n- \nu^-_n$ be
the Hahn decomposition.  
By assumption, 
$$
L_{\til q_n}(\beta^\pm_n)  \leq  |L|_{\til
  q_n}(\beta_n)$$
is a bounded sequence. Using the comparison maps $\varphi^{-1} \circ
\varphi_n: M_{q_n} \to M_q$ used in the preceding proof, we can think
of the $(\nu^{\pm}_n)|_J$ as 
measures on $J$ with a uniform bound on their total
mass, and 
we can pass to a subsequence to obtain $(\nu^{\pm}_{n_j})|_J \to
(\nu_\infty^{\pm})|_J$, thus defining (via Proposition \ref{prop: they define} as in the proof of Proposition \ref{prop: semicontinuity}) 
 tranverse measures 
$\nu_\infty^{\pm}$ on $M_q$. 

Let $ \nu \df \nu_\infty^+-\nu_\infty^-$ and let $\beta' \df \beta_\nu$. Recall that a cohomology class is determined by its values on the edges of a triangulation, and non-atomic transverse measures evaluate to zero on horizontal saddle connections. Thus we obtain from \eqref{eq: each edge} that $\beta_{n_j} \to
\beta'$. But since we have assumed $\beta = \lim_n \beta_n$, we have
$\beta = \beta' \in \tremspace_q$.
%
%
The Hahn decomposition $\mu = \mu^+-\mu^-$ of a finite measure is characterized by the following minimizing property: for any pair of measures $\sigma^{\pm}$ with $\mu= \sigma^+-\sigma^-$, and any non-negative integrable function $f,$ we have $\int f d\mu^+ + \int f d\mu^- \leq \int f d\sigma^+ + \int f d\sigma^- $. Thus, even though $\nu|_J = (\nu^+_\infty)|_J- (\nu^-_\infty)|_J$ might not be the Hahn decomposition of $\nu_J$, we have 
from \eqref{eq: formula for L} that 
\begin{equation}\label{eq: from that} \begin{split}|L|_q(\beta_\nu) \leq & \
L_q(\beta_{\nu^+}) + L_q(\beta_{\nu^-}) \\
= & \lim_{j \to \infty} \left( L_q(\beta_{\nu_{n_j}^+}) + L_q(\beta_{\nu_{n_j}^-})\right)
=
\lim_{j \to \infty}
|L|_{\til q_{n_j}}(\beta_{n_j}). 
\end{split}\end{equation}
Since this holds for any choice of the
subsequence, we obtain 
\equ{eq: limit L}.
\end{proof}

\subsection{The domain of definition of a tremor, and foliation cocycles in a fixed horospherical leaf}\label{subsec: for more details}
In this subsection we will
set up a canonical identification of $\tremspace_q$  and $\tremspace_{q'}$, when $q$ and $q'$
 belong to the same horospherical leaf. For this, the notation of an APD, introduced in the \S \ref{subsec: polygonal tremors}, will turn out to be useful. As a consequence, and 
using results of \cite{mahler}, we will show that for a non-atomic tremor, the domain of definition $\Dom(q, \beta)$ is the entire real line, and we will obtain useful `group action' properties of tremors on  a fixed horospherical leaf.

 Recall from \S \ref{subsec: atlas of
  charts} that via the identification of $T(\HHm) $ with the product $ \HHm \times H^1(S, \Sigma; \R^2)$, for any $\til q_1, \til q_2 \in \HHm$, every $v_1 \in T_{\til q_1}(\HHm)$ has a unique parallel vector $v_2 
  \in T_{\til q_2}(\HHm)$. We say that $v_2$ is obtained from $v_1$ by {\em parallel transport.}

\begin{prop} \label{prop: tremspace} (cf. \cite[Theorem 1.2]{mahler}). If $\til q_1$ and $\til q_2$ are elements of $\HHm$ belonging to the same horospherical leaf $W^{uu}$ then parallel transport takes $\tremspace_{\til q_1}$
to $\tremspace_{\til q_2}$. It takes $C^+_{q_1}$ to $C^+_{q_2}$ and takes non-atomic tremors to non-atomic tremors. It takes $(dy)_{q_1} \in\tremspace_{\til q_1}$ to $(dy)_{q_2}\in\tremspace_{\til q_2}$. 

\end{prop}

\begin{proof}
Since $\til q_1, \til q_2$ are both in $W^{uu}$, there is a path $\rho: [a,b] \to W^{uu}$ such that $\rho(a) = \til q_1, \ \rho(b) = \til q_2.$ 
For each $t_0 \in [a,b]$, fix an APD on $\rho(t_0),$ and let $\tau = \tau(t_0)$ be the triangulation obtained from this APD by adding diagonals to quadrilaterals, as in the proof of Proposition \ref{prop: uniform ac closed}. Let 
 $V_\tau$ be the open subset of $\HHm$ associated with $\tau$ as in \S \ref{subsec: atlas of
  charts}. We obtain 
a covering of $[a,b]$ by $\left\{\rho^{-1}\left(V_{\tau(t_0)} \right): t_0 \in [a,b] \right\}$, and by compactness we can pass to a finite covering. Thus in proving the Proposition we may assume that  the image of $\rho$ is contained in one $V_\tau,$ where $\tau = \tau(a)$ is the triangulation obtained from an APD on $M_{q_1}. $ 

Let $\phi : M_{q_1} \to M_{q_2}$
be the comparison map which is affine on triangles of $\tau$, as defined in \S \ref{subsec: atlas of
  charts}. Since $\til q_1, \til q_2$ belong to the same horospherical leaf, a segment is horizontal on $M_{q_1}$ if and only if its image under $\phi$  is horizontal on $M_{q_2}$. In particular the APD on $M_{q_1}$ is sent to an APD on $M_{q_2}$, and the restriction of $\phi$ to edges of the APD commutes with the opposite point maps (this situation is illustrated in Figure \ref{fig: comparison}). This implies via Proposition \ref{prop: they define} that $\phi$ induces a bijection between signed transverse measures on $M_{q_1}$ and $M_{q_2}$, and this bijection maps positive (respectively, atomic) transverse atomic transverse measures to positive (resp. atomic) transverse measures. Also, again using that $\til q_1, \til q_2$ are in the same horospherical leaf, the map $\phi$ sends $(dy)_{q_1}$ to $(dy)_{q_2}$. Thus the map $\phi^* : H^1(M_{q_2}, \Sigma_{q_1}; \R^2) \to H^1(M_{q_2}, \Sigma_{q_2}; \R^2)$ induced by $\phi$ sends $\tremspace_{q_2}$ to $\tremspace_{q_1}$ and sends $(dy)_{q_2}$ to $(dy)_{q_1}.$ Finally, since $\til q_2$ is obtained from $\til q_1$ by pre-composing charts by $\phi$, the definition of parallel vectors given in \S \ref{subsec: atlas of
  charts} shows that the map induced by $\phi^*$ is parallel transport.  
\end{proof}

\begin{prop}\name{prop: tremor domain of defn}
If $\beta \in \tremspace_q$ is non-atomic
then $\Dom(q,\beta)=\R$. 
\end{prop}

The assumption that $\beta$ is non-atomic is important 
here, see \S \ref{sec: atomic tremors}.  

\begin{proof}
 Let $\til q \in \pi^{-1}(q)$, let $\beta \in H^1(S, \Sigma; 
 \R_x)$, let $v = (\beta, 0)$, let $\theta(t)$ be the parameterized line \eqref{eq: linear equation}, and let $\Dom(q, \beta)$ denote its domain of definition. Let $\gamma_s = \hol^{(x)}(\til q) + s \beta$ be the corresponding line in $H^1(S, \Sigma; \R_x).$ We can define $\gamma_s$ for all $s \in \R$, and for $s \in \Dom(q, \beta)$ we have $\gamma_{s} = \dev (\theta(s)).$ Thus $\gamma$ is a line in $H^1(S, \Sigma; \R_x) $, $\theta$ is its lift via dev, and our goal is to show that this lift is well-defined for all $s \in \R.$

We denote by $\FF$ the foliation on $S \sm \Sigma$ obtained by pulling the horizontal foliation on $M_q$ by $\varphi.$  
For all the surfaces $\til{q'}$ in any lift of $\gamma$, $\FF$ is  also the pullback of the horizontal foliation on $M_{q'}$.   
Let $\mathbb{B}(\FF)$ denote the set of cohomology classes $\gamma' \in H^1(S, \Sigma; \R)$ satisfying the following conditions:
 \begin{itemize}
     \item[(i)] For any oriented saddle connection $\delta$ on $M_q$ with $\hol^{(x)}(\delta) >0$, we have $\varphi^* \gamma'(\delta)>0.$ 
     \item[(ii)] 
     For any non-atomic transverse measure $\nu$ to $\FF$, $\gamma'$ has a positive cup product with $\tau \df \beta_\nu$. 
     \end{itemize}
By \cite[Thm. 1.1, see also Thm. 11.2]{mahler} (but swapping the roles of horizontal and vertical foliations), in order to show that the path $\gamma$ lifts, it  suffices to show that $\gamma_s \in \mathbb{B}(\FF)$ for all $s$. Since $\beta$ is non-atomic, it  vanishes on horizontal saddle connections, and this implies that for any horizontal saddle connection $\delta$, the function $s \mapsto \gamma_s(\delta)$ is constant. Therefore $\gamma_s(\delta) = \gamma_0(\delta) = \hol^{(x)}(\delta) >0$, and this implies (i). In order to check (ii), let $\tau$ be the cohomology class corresponding to a non-atomic transverse measure. Then    
$$\int \gamma_s \wedge \tau=\int ((dx)_{\til q} +s\beta)\wedge \tau=\int (dx)_{\til q} \wedge \tau+s\beta\wedge \tau=\int (dx)_{\til q} \wedge\tau >0.$$
 We have used here the fact that two cohomology classes arising from non-atomic measures transverse to the same foliation have cup product zero (see \cite[Prop. 4]{transverse measures}). 
\end{proof}

It follows from Proposition \ref{prop: tremspace} that for any horospherical leaf $W^{uu}$ in $\HHm$, there is a fixed subspace $\tremspace^{(\mathrm{na})}_{W^{uu}} \subset H^1(S, \Sigma; \R_x)$, so that \index{trem@$\tremspace^{(\mathrm{na})}_{W^{uu}} $}
for each $\til q \in W^{uu}$, the collection of non-atomic foliation cocycles in $\tremspace_{\til q}$ is canonically identified with $\tremspace^{(\mathrm{na})}_{W^{uu}}.$ Note that if $\til q$ has no horizontal saddle connections, then the same is true for the same is true for any surface in the horospherical leaf of $\til q$; in this case $\til q$ admits no atomic foliation cocycles and $\tremspace_{\til q} = \tremspace^{(\mathrm{na})}_{W^{uu}} $. We define a map 
\begin{equation}\label{eq: action horospherical}
\tremspace^{(\mathrm{na})}_{W^{uu}} \times W^{uu} \to W^{uu}, \ \ \ (\beta, \til q) \mapsto \trem_{\beta}(\til q). 
\end{equation}
This map is well-defined in light of Proposition \ref{prop: tremor domain of defn}.

\begin{prop}\label{prop: group action law tremors}
The map in \eqref{eq: action horospherical} satisfies the `group-action' law
$$\trem_{\beta_1 + \beta_2}(\til q) = \trem_{\beta_1}(\trem_{\beta_2}(\til q))$$ 
for all $\til q \in W^{uu}$ and $\beta_1, \beta_2 \in \tremspace^{(\mathrm{na})}_{W^{uu}}.$ 
\end{prop}

\begin{proof}
For any $s_1, s_2 \in \R$, the path 
$$\gamma_{s_1, s_2} : \R \to H^1(S, \Sigma; \R_x), \ \ \gamma_{s_1, s_2} (t) \df \hol_{\til q}+ t(s_1 \beta_1 +s_2 \beta_2)$$
can be lifted to a path $\theta_{s_1, s_2}$ by Proposition \ref{prop: tremor domain of defn}. This implies that $\trem_{s_1 \beta_1+s_2\beta_2}(q) $ is well-defined.  Since $\dev$ is a local homeomorphism, it has a unique lifting property. That is, for any path $\gamma : [0,1] \to H^1(S, \Sigma; \R_x)$ and any $\til q_0$ with $\gamma(0) = \dev(q_0)$, there is at most one path $\theta : [0,1] \to \HHm$ with $\theta(0) = \til q_0$ and $\gamma = \dev \circ \, \theta. $ 
The two paths 
$$s \mapsto \trem_{\beta_1}(\trem_{s\beta_2}(\til q)), \ \ \ s \mapsto  \trem_{\beta_1 + s\beta_2}(\til q)$$
are continuous by Proposition \ref{prop: continuity}, and commutativity of addition in $H^1(S, \Sigma; \R_x)$ shows that they are lifts of the same path in $H^1(S, \Sigma ; \R_x)$. Thus they are the same, and setting $s=1$ we get the required result. 

See \cite[Prop. 4.5]{eigenform} for a similar argument.
\end{proof}

\begin{cor}
For any $u \in U$ and $\beta \in \tremspace_q$, we have 
\eq{eq: commutation 0}{
u \, \trem_\beta( q) = \trem_\beta(uq), \ \ \ \ \Dom(uq,
\beta) = \Dom( q, \beta).} 
\end{cor}

\begin{proof}
If $\beta$ is non-atomic, this is immediate from \eqref{eq: horocycles are tremors} and Proposition \ref{prop: group action law tremors}. The proof when $\beta$ is atomic is similar to the proof of Proposition 
\ref{prop: group action law tremors}. In  this paper, we will not be using \eqref{eq: commutation 0} when $\beta$ is atomic, and we leave the details to the reader. 
\end{proof}

\section{The tremor comparison homeomorphism}\name{sec: TCH}
Recall from \S \ref{subsec: for more details} that two points $\til q_0$ and $\til q_1$ in the same horospherical leaf share the same space of foliation cocycles. This was proved in Proposition \ref{prop: tremspace} by analyzing the effect of a composition of finitely many comparison maps $\varphi: M_{q_0} \to M_{q_1}$, each of which is affine on each triangle of a triangulation. The  map $\varphi$ respects horizontal foliations, that is maps the leaves of the  horizontal foliation $\FF$ on $M_{q_0}$ to horizontal leaves on $M_{q_1}$, and preserves the canonical transverse measure $dy$ measuring the `height displacement' between leaves. In this section we will show that if $\til q_1$ is obtained from $\til q_0$ by a non-atomic tremor, then there is a comparison map $M_{q_0} \to M_{q_1} $ 
 that {\em shears along horizontal leaves}; that is, respects the horizontal foliations $\FF$ on $M_{q_0}$ and $M_{q_1}$, preserves the transverse measure $dy$, and in addition, preserves the length parameter along horizontal leaves. In the language of flows, the comparison map from \S \ref{subsec: for more details} commutes with the horizontal straightline flow up to a time change, and in this section we will produce a map commuting with straightline with no time change.  This map need not be affine on triangles. The difference between these maps is illustrated in Figures \ref{fig: trem2} and \ref{fig: comparison}. We note that for the horocycle flow, the affine comparison maps defined in \S \ref{subsec: G} are both affine on triangles, and act by shearing horizontal leaves with respect to each other (see 
Figure \ref{fig: horocycle}).

 As we will see in Proposition \ref{prop: added}, the existence of a comparison homeomorphism that shears along horizontal leaves characterizes the property of lying on the same tremor path.

\begin{prop}\name{prop: tremor comparison}
Let $q_0 \in \HH$ and let $M_0 = M_{q_0}$ be the corresponding surface. Let
$\varphi_0:S\to M_0$  be a marking map and 
let $\til q_0 \in \pi^{-1}(q_0)$ be the corresponding marked translation
surface. 
Let $\nu$ be a non-atomic  signed  
transverse measure on the horizontal foliation of $M_0$ and let
$\beta = \beta_\nu$. Let $q_t = \trem_{t\beta}(q_0)$ and $\til q_t =
\trem_{t\beta}(\til q_0) $, let $M_t = M_{q_t}$ be the underlying
surface, and let $\varphi_t : S \to M_t$ be a marking map representing
$\til q_t$. 
Denote $\hol_{\til q_t} = \left (\hol_t^{(x)}, \hol_t^{(y)} \right)$.
Then there is a unique homeomorphism $\psi_t: M_0\to M_t$ which is isotopic to
$\varphi_t\circ\varphi_0^{-1}$, preserves horizontal foliations and satisfies  
\eq{eq: tremor shear}{
\hol_t^{(x)}(\psi_t(\gamma))=\hol_0^{(x)}(\gamma)+t\int_\gamma \nu \ \
{\rm and} \ \  
 \hol_t^{(y)}(\psi_t(\gamma))=\hol_0^{(y)}(\gamma) 
}
for any piecewise smooth path $\gamma$ in $M_0$ between {\em any} two
points. 
\end{prop}


\begin{dfn} We call $\psi_t: M_0\to M_t$ the {\em tremor comparison
    homeomorphism} (TCH). 
\end{dfn} \index{TCH} \index{tremor comparison homeomorphism}

The uniqueness of a tremor comparison homeomorphism implies 
the following important naturality property:

\begin{cor}
    \name{prop: naturality}
  With the notation of Proposition \ref{prop: tremor comparison}, 
suppose $\varphi_0$ and $\varphi'_0$ are two different marking maps
$S \to M_0$ representing $\til q_0$, so that $\varphi_0' \circ \varphi_0^{-1}$ is isotopic to a translation equivalence
$h$ of $M_0$. Then the TCH's 
$\psi_t$ and $\psi'_t$ satisfy
$\psi_t = \psi'_t \circ h$. 
\end{cor}


 In order to construct $\psi_t$, we start with a comparison map 
 $\varphi$ 
 which is only assumed to satisfy \eqref{eq: tremor shear} in case $\gamma$ is a saddle connection. We then modify $\varphi$ by means of an isotopy which moves points along leaves of the horizontal foliation of the target surface $M_t$. 
 The signed distance along horizontal leaves will be chosen so that \eqref{eq: tremor shear} holds for all piecewise smooth curves  $\gamma$ connecting any two points.  Since the horizontal straightline flow may not be defined for all times, one of the complications we will address is to ensure that we can move points horizontally by the required amount.

\begin{proof}[Proof of Proposition \ref{prop: tremor comparison}] 
We  begin by proving the existence of $\psi_t$.  Let 
  $\tau$ be a triangulation of $S$ obtained as the 
  pullback via $\varphi_0$ of a geodesic triangulation 
  on $M_0$. Since we will be using the opposite point map defined in \S \ref{subsec: polygonal tremors}, we will take $\tau$ to be  given by adding horizontal diagonals to the polygons of an APD, as in the proof of Proposition \ref{prop: tremspace}. Let $U_\tau$ and $V_\tau$ be the open sets in $H^1(S, \Sigma; \R^2)$ and $\HHm$, as in \S
 \ref{subsec: atlas of 
   charts}.
 For a sufficiently small $\vre>0$, in the interval $I = [0, \vre]$ we have
 \eq{eq: contained in Vtau}{
\{\trem_{t\beta}(q) : t \in I \}\subset   V_\tau,
 }
 and we will first prove the
existence of $\psi_t$ for $t \in I$ where $I$ satisfies \equ{eq:
  contained in Vtau}. The existence for all $t$ then follows by composing maps defined on small intervals, as in the first paragraph of the proof of Proposition \ref{prop: tremspace}. With this in mind we can re-parameterize $I$, and replace $\beta$ by its multiple by a positive constant, to assume that $t=1$, $I=[0,1]$ and $\til q_0, \til q_1 \in V_\tau.$


Let $\tau_0, \tau_1 $ denote respectively the pushforward of the triangulation $\tau$ to $M_0, M_1$, and let
 $\varphi:M_0 \to M_1$ be the comparison map which is affine and orientation-preserving on triangles of $\tau$ as in \S \ref{subsec: atlas of
  charts}. Thus $\varphi $ sends $
  \tau_0$ to $ \tau_1$. 
The definition of tremors gives us \eqref{eq: tremor shear} with $\varphi$ in place of $\psi_t$, and 
for any path $\gamma$ on $M_0$  with endpoints in $\Sigma$.
Recall from Proposition \ref{prop: tremspace} that $\varphi$ takes the horizontal foliation of $M$ to the horizontal foliation of $M'$ and  takes $(dy)_{\til q_0}$ to $(dy)_{\til q_1}$. Also $\varphi$ preserves the rightward orientation on horizontal lines. 
We will construct the homeomorphism $\psi$  by composing $\varphi$ with a map which moves a point in $M_1$ along its horizontal leaf. Recall that 
$\Upsilon^{(x)}
(s)$ denotes the image of $x \in M_1$ under horizontal straightline flow, to (signed) distance $s$. With this notation, for a continuous function $\bar s : M_0\to \R$, we 
write 
\eq{eq: defin of psi}{
\psi(p) \df 
\Upsilon^{(\varphi(p))}(\bar s(p)).}

As mentioned above, the straightline flow map $t \mapsto \Upsilon^{(\varphi(p))}(t) $ might not be defined to time $t = \bar s(p)$; we will show in Lemma \ref{lem: upsilon} that it actually is. 
Thus, for $p\in M_0$, $\psi(p)$ is obtained by 
motion along the horizontal leaf of $\varphi(p)$ in $M_1$, by the signed distance $\bar s(p)$; see Figure \ref{fig: trem2}.
Clearly such a map will satisfy the second equation in \eqref{eq: tremor shear}, and 
 $\bar s(p)$  will be chosen so that the first equation in \eqref{eq: tremor shear} holds as well. The construction of $\bar s(p)$ and proof that it has the desired properties will be broken up into several lemmas.

We begin by specifying the values of the function $\bar s$,  on each of the edges of the triangulation $\tau_0$ of $M_0$. On the horizontal edges of the triangulation, we set $\bar s$ equal to zero. Let 
$\sigma:[0,1]\to M_0$ 
denote an affine parameterization of a non-horizontal edge of $\tau_0$. We define
\eq{eq: independent of direction}{\bar s(\sigma(t))=\int_{\sigma(0)}^{\sigma(t)} d\nu-t\int_{\sigma(0)}^{\sigma(1)}d\nu,}
where the integrals are taken along the path $\sigma$ between the indicated limits. 
\begin{lem}\label{lem: orientation} 
The following hold for each edge $\sigma$: 
\begin{itemize}
    \item[(a)] The definition \eqref{eq: independent of direction} does not depend on the choice of orientation for $\sigma$; that is, defines the same function on the edge, if one uses $\bar \sigma(1-t)$ instead of $\sigma(t)$. 
    \item[(b)] 
    The map $t \mapsto \bar s(\sigma(t))$ is continuous. 
    \item[(c)] $\bar{s}(\sigma(0)) = \bar{s}(\sigma(1))=0.$
\end{itemize}
\end{lem}
\noindent \textit{Proof.} Assertion (a) follows from a computation using \eqref{eq: tremor shear} for the curve $\sigma$;  we leave this to the reader. Assertion (b) follows from the fact that $\nu$ is non-atomic. Assertion (c) follows from \eqref{eq: independent of direction}.
$\hfill \triangle$

\medskip

We now check that when using \eqref{eq: independent of direction}, a map defined via \eqref{eq: defin of psi} has the required property of preserving distances along horizontal lines, for two points on opposite sides of a polygon $P$ of the 
 APD. To this end, let $\mathrm{opp}_P$ be the opposite point map as in \S \ref{subsec: polygonal tremors}, and let $\sigma, \sigma'$ denote two affine parameterizations of sides of $P$, so that 
$$x = \sigma(t) \in \partial P, \ \ \  y = \mathrm{opp}_P(x) = \sigma'(t')$$ for appropriate $t,t' \in [0,1].$
Let $d_0, d_1$ denote respectively the horizontal signed distance between $x,y$ and $\varphi(x), \varphi(y)$ in $M_0, M_1$; that is, 
$$y = \Upsilon^{(x)}(d_0), \ \ \ \ \ \varphi(y) = \Upsilon^{(\varphi(x))}(d_1).$$
Here we swap if necessary the roles of $x$ and $ y$ to assume that $d_0>0$, for the definition of $d_i$ we refer to straightline flow on $M_i$, and in case the horizontal trajectory of $x$ is periodic we use the parameterization of paths through the interior of $P$ and $\varphi(P)$. Note that the straightline flow from $x$ to $y$ is well-defined by definition of $\mathrm{opp}_P,$ and straightline flow from $\varphi(x)$ to $\varphi(y)$ is well-defined since $\varphi$ maps horizontal segments to horizontal segments and preserves their orientation.

\begin{lem}\label{lem: on edges}
 We have $d_0  = d_1 -\bar s(x) +\bar s(y).$ 
\end{lem}
Note that Lemma \ref{lem: on edges} does not assume that $\psi $ as in 
\eqref{eq: defin of psi} is well-defined; but if one assumes that $\psi$ is well-defined, one concludes from Lemma 
\ref{lem: on edges} that $d_0$, the signed horizontal distance  between $x$ and $y$, is the same as $d_1-\bar s(x) +\bar s(y)$, the signed horizontal distance between $\psi(x)$ and $\psi(y).$ 

\noindent {\it Proof.}
By decomposing $P$ into triangles, we can assume with no loss of generality that $P$ is a triangle. 
We can further assume, using Lemma \ref{lem: orientation}(a), that $P$ has one vertex at $\xi$, where $\sigma$ and $\sigma'$ are affine parameterizations of opposite edges of $P$ with $\sigma(0) = \sigma'(0) = \xi$ and $\hol^{(y)}(\sigma) < \hol^{(y)}(\sigma')$, as shown in Figure \ref{fig: comparison2}. 
Let $\alpha$ denote a path from $x$ to $y$ along the horizontal segment through $P$. Then $\alpha$ is homotopic to the path from $x$ to $y$ along the edges of $P$, and hence
\eq{eq: an easy computation}{
d_0 = \hol_0(\alpha) = t'\, \hol^{(x)}_{0}(\sigma')-t \, \hol_{0}^{(x)}(\sigma).}
Similarly
\eq{eq: an easy computation2}{d_1 = \hol_1(\varphi(\alpha)) =t'\, \hol^{(x)}_{1}(\varphi(\sigma' ))-t\, \hol_{1}^{(x)}(\varphi(\sigma)).
}

 \begin{figure}[h]
\center

\begin{tikzpicture}[scale=1.0]

\def\xa{-1};
\def\xb{1};
\def\xc{2.6};

\def\ya{-0.5};
\def\yb{1.0};
\def\yc{2.2};
\def\yd{3};

\node (P0) at (\xb,\ya) [circle,draw,fill=black,inner sep=0pt,minimum size=1.2mm,label=-90:$\xi$] {};
\node (P1) at (\xa,\yc) [circle,draw,fill=black,inner sep=0pt,minimum size=1.2mm,label=180:$ $] {};
\node (P2) at (\xc,\yd) [circle,draw,fill=black,inner sep=0pt,minimum size=1.2mm,label=0:$ $] {};

\draw (P0) -- node[left, pos=0.76] {$\sigma$} (P1);
\draw (P0) -- node[right, pos=0.7] {$\sigma'$} (P2);
\draw (P1) -- (P2);

\path [name path=horizontal line]  (-1,\yb) -- (3,\yb);
\path [name path=first line]  (P0) -- (P1);
\path [name path=second line]  (P0) -- (P2);
\path [name intersections={of=horizontal line and first line, by=xx}]; 
\path [name intersections={of=horizontal line and second line, by=yy}]; 

\draw (-1,\yb) -- (3,\yb);
\node [left] (L) at (-1,\yb) {};

\node (X) at (xx) [circle,draw,fill=black,inner sep=0pt,minimum size=0.7mm, label=-100:$x$] {};
\node (Y)  at (yy) [circle,draw,fill=black,inner sep=0pt,minimum size=0.7mm, label=-70:$y$] {};
\begin{scope}[xshift=180]

\def\xa{0.6};
\def\xb{1};
\def\xc{2.0};

\def\ya{-0.5};
\def\yb{1.0};
\def\yc{2.2};
\def\yd{3};

\node (P0) at (\xb,\ya) [circle, draw, fill=black, inner sep=0pt, minimum size=1.2mm, label=-90:$\varphi(\xi)$] {};
\node (P1) at (\xa,\yc) [circle, draw, fill=black, inner sep=0pt, minimum size=1.2mm, label=180:$ $] {};
\node (P2) at (\xc,\yd) [circle, draw, fill=black, inner sep=0pt, minimum size=1.2mm, label=0:$ $] {};

\draw (P0) -- node[left, pos=0.76] {$\varphi(\sigma)$} (P1);
\draw (P0) -- node[right, pos=0.7] {$\varphi(\sigma')$} (P2);
\draw (P1) -- (P2);

\path [name path=horizontal line]  (-1.0,\yb) -- (2.5,\yb);
\path [name path=first line]  (P0) -- (P1);
\path [name path=second line]  (P0) -- (P2);
\path [name intersections={of=horizontal line and first line, by=xx}]; 
\path [name intersections={of=horizontal line and second line, by=yy}]; 

\draw (-1.2,\yb) -- (3.2,\yb);
\node [left] (L) at (-1.2,\yb) {};

\node (X) at (xx) [circle,draw,fill=black,inner sep=0pt,minimum size=0.7mm,  label=-110:$\varphi(x) $] {};
\node (Y)  at (yy) [circle,draw,fill=black,inner sep=0pt,minimum size=0.7mm, label=-70:$\varphi(y)$] {};


\end{scope}

\end{tikzpicture}\ \ \ \ \ \ \ \ \

\caption{Paths used in the proof of Lemma \ref{lem: on edges}.}
\label{fig: comparison2} 
\end{figure}

Applying the opposite point invariance property \eqref{eq: invariance property}, we obtain \eq{eq: using invariance now}{\int_{\sigma(0)}^{\sigma(t)} \nu = \int_{\sigma'(0)}^{\sigma'(t')} \nu. }
 By \eqref{eq: tremor shear} (which holds for the saddle connections $\sigma$ and $\sigma'$), along with \eqref{eq: an easy computation} and \eqref{eq: an easy computation2}, we get 
 \[
 \begin{split}
    d_1-d_0 
    = & t' \, \left( \hol^{(x)}_{1}(\varphi(\sigma')) - \hol_{0}^{(x)}(\sigma')
    \right) - t\, \left( \hol^{(x)}_{1}(\varphi(\sigma)) - \hol_{0}^{(x)}(\sigma)\right) \\
    = & 
    t' \, \int_{\sigma'(0)}^{\sigma'(1)}
 \nu - t \,  \int_{\sigma(0)}^{\sigma(1)} \nu,  
 \end{split}
 \]
 and by \eqref{eq: independent of direction} and \eqref{eq: using invariance now} we also get 
 $$
 \bar s (x) - \bar s (y)  = \bar s (\sigma(t)) -\bar s (\sigma'(t'))  = 
  t' \, \int_{\sigma'(0)}^{\sigma'(1)}
 \nu - t \,  \int_{\sigma(0)}^{\sigma(1)} \nu
  .
 $$
 This gives the required identity.
$\hfill \triangle$

\medskip

We now extend $\bar s$ by affine interpolation to the interiors of triangles. For any point $p \in M_0$, let $x, y$ denote the two intersections of the horizontal leaf of $p$ with $\partial P$, so that $y = \mathrm{opp}_P(x)$, and let $d_0$ be as above. Then there is $t \in (0,1)$ so that $p = \Upsilon^{(x)}(td_0) = \Upsilon^{(y)}((t-1)d_0)$. We define 
\eq{eq: definition s of p}{\bar s (p) \df (1-t) \bar s(x) + t \bar s (y). 
}
Since $\varphi$ and \eqref{eq: definition s of p} are both affine and orientation-preserving, the conclusion of Lemma \ref{lem: on edges} continues to hold; namely, for any two points $x'$ and $y'$ which  are on a horizontal segment passing from side to side of a polygon of the APD, 
we have
\begin{equation}\label{eq: like lemma 5.5}
    d_0  = d_1 -\bar s(x') +\bar s(y'),
\end{equation}
where $d_0, d_1$ denote signed distances defined using $x', y'$. In other words, where defined, $\psi$ maps horizontal segments to horizontal segments isometrically.

With this extended definition we claim:
\begin{lem}\label{lem: upsilon}
For any $p \in M_0,$ the horizontal straightline flow from $\varphi(p)$ to signed distance $\bar s(p)$ on $M_1$ is defined, and thus 
    the map $\psi$ defined  via \eqref{eq: defin of psi} is well-defined.
    \end{lem}

\noindent {\it Proof.}
Suppose by way of contradiction that for some $p \in M_0$, the straightline flow trajectory from $\varphi(p)$ to signed distance $\bar s (p
)$ is not defined. We know from Lemma \ref{lem: orientation}(c) that $p$ is not a singular point. 
Assume with no loss of generality that $\bar s(p)>0$; our assumption means that for some $0<t_{\mathrm{crit}} \leq  \bar s(p)$ we have $\{\Upsilon^{(\varphi(p))}(t) : 0 \leq t<t_{\mathrm{crit}}\}$ is well-defined but the one-sided limit 
$$\xi \df \lim_{t\to t_{\mathrm{crit}}^-}\Upsilon^{(\varphi(p))}(t)$$ 
is a singular point on $M_1$. 
Let $k$ be the number of times the trajectory $\left\{\Upsilon^{(\varphi(p))}(t) : 0 < t<t_{\mathrm{crit}} \right\}$ crosses edges of the APD. We can choose $p$ with the above properties so that $k$ is minimal. We will reach a contradiction in both cases $k=0$ and $k>0.$

If $k=0$ then there is a polygon $\Delta$ of the APD on $M_0$ such that $\xi = \varphi(\xi)$ is a vertex of $\varphi(\Delta)$  and $p \in \Delta$. Let $x$ be the point on $\partial \Delta$ which is opposite to $\xi$, so that $p$ is on the segment from $x$ to $\xi$. 
%
We define $d_0, d_1$ as above, with $p$ and $\xi$ playing respectively the roles of $x$ and $y$.  
Since $t_{\mathrm{crit}}>0$ we must have that $\varphi(p) $ is to the left of $\xi$ in the polygon $\varphi(\Delta),$ and hence $d_1>0$. Since $\varphi$ preserves the orientation of horizontal lines we must have $d_0>0.$ By our contradiction assumption, $t_{\mathrm{crit}}\geq d_1$. 
Finally $\bar s (\xi) =0$ by Lemma 
 \ref{lem: orientation}(c). Putting these together and using equation \eqref{eq: like lemma 5.5} we get the contradiction 
$$ \bar s(p) \geq t_{\mathrm{crit}} \geq  d_1 = d_0 + \bar s(p) - \bar s(\xi) > \bar s(p).$$

Now suppose $k>0.$ Let $\Delta$ be a polygon of the APD containing $p$ and let $y'$ be the endpoint of the rightward oriented segment from $p$ to $\partial \Delta.$ Let $d_0, d_1$ be defined as above, using the points $p$ and $y'$ instead of $x$ and $y$. We compute the numbers $t'_{\mathrm{crit}}, k'$ corresponding to $y'$ instead of $p$. We have $t'_{\mathrm{crit}} = t_{\mathrm{crit}} - d_1$ and $k' = k-1$. Using Lemma \ref{lem: on edges} we have 
$$\bar s(y') = d_0 -d_1 + \bar s(p) \geq d_0 - d_1 + t_{\mathrm{crit}} = d_0 + t'_{\mathrm{crit}} \geq t'_{\mathrm{crit}}.$$ This implies that $y'$ also satisfies that the straightline flow from $\varphi(y)$ to distance $\bar s(y)$ is not defined, and contradicts the minimality in the choice of $p.$
$\hfill \triangle $

\medskip

    \begin{lem}\label{lem: instructive}

    The map $\psi$ is a homeomorphism which is isotopic rel $\Sigma$ to $\varphi$  and satisfies 
    \eq{eq: to see that}{\Upsilon^{(\psi(p))}(t) = \psi ( \Upsilon^{(p)}(t))} 
    for any $p \in M_0 $ and any $t \in \R$ for which one (hence both) of these terms is defined.
    
\end{lem}

\noindent {\it Proof.}
The function $x \mapsto \bar s(x)$
 is continuous by Lemma \ref{lem: orientation} and \eqref{eq: definition s of p}. This implies that $\psi$ is continuous. Since $M_0$ is compact, in order to show that $\psi$ is a homeomorphism, it is enough to show that it is bijective. To this end, we first note that \eqref{eq: to see that} holds. Indeed by equation \eqref{eq: like lemma 5.5}, \eqref{eq: to see that}  holds for any interval $I$ for which the path $\{\Upsilon^{(p)} (t) : t \in I\}$ is contained in a polygon of the APD, and thus, by induction on the number of times a horizontal straightline segment from $p$ to $\Upsilon^{(p)}(t)$ crosses edges of the APD, it holds for all $t.$

It follows from Lemma \ref{lem: orientation}(c) that $\psi|_{\Sigma} = \varphi|_{\Sigma}$ and hence that $\psi$ is a label-preserving bijection on $\Sigma$. 
 It follows from \eqref{eq: to see that} that the restriction of $\psi$ to a horizontal straightline flow trajectory is an isometry (with respect to the metric induced by the 1-form $dx$). 
 Since the restriction of $\psi$ to a horizontal trajectory is an isometry mapping singular points to singular points, the restriction of $\psi$ to any horizontal trajectory is a bijection. Moreover, by \eqref{eq: defin of psi}, the image of a horizontal trajectory under $\psi$ is the same as its image under $\varphi$, and since $\varphi$ is a bijection, we obtain that $\psi$ is also a bijection.

Consider the one-parameter family of maps
 $$
 g^{(r)} (x) \df \Upsilon^{(\varphi(x))}(r\bar s (x)) \ \ \ (r \in [0,1]). 
 $$
 Clearly, this family 
 gives a homotopy between $\varphi$ and $\psi$ fixing $\Sigma$ pointwise. To see that each $g^{(r)}$ is a homeomorphism, arguing as before we see that it suffices to show that it is bijective on each horizontal straightline flow trajectory. For such a trajectory, it is indeed bijective as it is a linear homotopy between order-preserving homeomorphisms. This shows that $\varphi$ and $\psi$ are isotopic rel $\Sigma$. 
 $\hfill \triangle $

\medskip

\begin{lem}\label{lem: map satisfies}
   The map $\psi$ satisfies formula \eqref{eq: tremor shear}.
\end{lem}

\noindent 
{\it{Proof.}} 
We first claim that it is enough to prove the claim for paths $\gamma $ whose image is contained in edges of the APD. Indeed, 
 if \eqref{eq: tremor shear} holds for two paths it holds for their concatenation. Thus, in order to prove the result for an arbitrary path, it suffices to prove the result for a path $\gamma$ contained in one polygon  $\Delta$ of the APD. Suppose $\gamma'$ is  obtained by sliding every point in $\gamma$ to an edge $\sigma$ of $\Delta$; that is, $\gamma'(t) = \Upsilon^{(\gamma(t))}(\rho(t))$, where $\rho(t)$ is the horizontal signed distance from $\gamma(t)$  to $\sigma$. Using formula \eqref{eq: to see that}, we see that $\psi(\gamma')$ is obtained from $\psi(\gamma)$ by sliding horizontally by the same amount $\rho(t)$. From this one easily sees that if \eqref{eq: tremor shear} holds for $\gamma'$, it also holds for $\gamma$. 

 It remains to check that \eqref{eq: tremor shear} holds for paths whose image is contained in an edge $\sigma \subset \partial \Delta.$ This follows easily from the 
 definition \eqref{eq: independent of direction} of $\bar s$ along edges of $\Delta$; we leave the verification to the reader. 
$\hfill \triangle$

\medskip

We now complete the proof of Proposition \ref{prop: tremor comparison}. Lemmas \ref{lem: instructive} -- \ref{lem: map satisfies} establish the existence of $\psi$ with the required properties.
We complete the proof by proving uniqueness. Let $\psi$ and $\psi'$ be isotopic maps from $M_0$ to $M_1$
satisfying \eqref{eq: tremor shear} for arbitrary paths.
This equation implies that $\psi^{-1}\circ\psi'$ preserves the holonomy of paths and is thus a
translation equivalence. Since the maps $\psi$ and $\psi'$ are isotopic the map $\psi^{-1}\circ\psi'$ is isotopic to the identity.
The identity map is
the unique translation equivalence of $M_0$ isotopic to the identity
so we have $\psi^{-1}\circ\psi'=I$ and $\psi=\psi'$. 
\end{proof}

\begin{prop}\label{prop: added}
Suppose that for $i=0,1$, $\til q_i$  are marked translation surfaces represented by the marking maps $\varphi_i: S \to M_i$. Suppose $\til q_0, \til q_1$ belong to the same horospherical leaf, and there is a homeomorphism $\psi: M_0 \to M_1$ isotopic to $\varphi_1 \circ \varphi_0^{-1}$ for which the conclusion of Lemma \ref{lem: instructive} holds. Then $\til q_1 = \trem_\beta(\til q_0)$ for some non-atomic foliation cocycle $\beta \in \tremspace_{q_0}.$
\end{prop}

Since we will not be using this result in this paper, we only outline the argument. 

\begin{proof}[Sketch of proof.]
Since $\til q_0, \til q_1 $ belong to the same horospherical leaf and $\psi$ is isotopic to $\varphi_1 \circ \varphi_0^{-1}$, $\hol^{(y)}_{M_0} = \hol^{(y)}_{M_1}(\psi(\gamma))$ for any path $\gamma$ joining singular points. We will define a non-atomic signed transverse measure $\nu$ satisfying \eq{eq: we will define so that}{\hol^{(x)}_{M_1}(\psi(\gamma)) = \hol^{(x)}_{M_0}(\gamma) + \int_{\gamma}\nu.}
This will show that \eqref{eq: tremor shear} holds (with $t=1, \psi = \psi_1 $), for any path joining singular points, thus showing that $\til q_1 = \trem_{\beta_\nu}(\til q_0)$. 

Let $\vre>0$ be such that horizontal straightline flow is defined on all points of both $\gamma$ and $\psi(\gamma)$, to time $s$, for all $|s|<\vre$. We define the {\em horizontal diameter} of a topological disc in a translation surface  to be the supremum of horizontal holonomies of any curve contained in $\mathcal{U}$. We can cover  the image of $\gamma$ by topological discs $\mathcal{U}$ such that the horizontal diameter of both $\mathcal{U} $ and  $\psi\left(\mathcal{U}\right)$ is smaller than $\vre.$ The subarcs $\gamma'$ of  $\gamma$ contained in  such a topological disc $\mathcal{U}$ generate the Borel $\sigma$-algebra on $\gamma$. For each such  $\gamma'$  we define 
$$\int_{\gamma'} \nu = \hol^{(x)}_{M_1}(\psi(\gamma')) - \hol^{(x)}_{M_0}(\gamma').$$ 
Using the Carath\'eodory extension theorem, one can show that this defines $\nu$ as a signed measure on $\gamma$.
By linearity, $\nu$ satisfies \eqref{eq: we will define so that}, and one can check using \eqref{eq: to see that} that $\nu$ defined in this way is invariant under holonomy along horizontal lines, and thus defines a transverse measure. 
\end{proof}
\begin{remark}\label{remark: from BSW}
It is instructive to compare our discussion of tremors, using
Proposition \ref{prop: tremor comparison}, with the discussion of the
Rel deformations in \cite[\S 6]{eigenform}. Namely in
\cite[Pf. of Thm. 6.1]{eigenform}, a map $\bar{f}_t: M_0 \to
\Rel_t(M_0)$ is constructed but the definition of this map involves
some arbitrary choices. In particular it is not unique and is not
naturally contained in a continuous one-parameter family of maps. 
  \end{remark}

  \section{Properties of tremors}\name{subsec: commutation}
  \combarak{No essential math changes in this section but fairly
    extensive rewriting. This was written imprecisely before we came
    up with TCH's and I made use of TCH's to make it more precise. It
    definitely became more verbose -- I am afraid some readers will be
  annoyed  that we are spending so much time on trivialities. But I am
  not very good at condensing, the previous version took less space
  but was   less precise and maybe more mysterious. As evidence that
  this approach is better, in ``tremor for archeologists'', Cor. 6.2,
  there used to be a mistake, we fixed the statement easily enough,
  but the proof we had in the document was
so convincing that there was no reasonable way to figure out what was
wrong with the proof.}

In this section we will derive further properties of tremors. 
  
  \subsection{Composing tremors and other maps}\name{subsec: composing tremors}
Recall from Proposition \ref{prop: group action law tremors} that we have 
\eq{eq: simplifies to}{
\trem_{\beta_1+\beta_2}(q) = \trem_{ \beta_2}(\trem_{\beta_1}(q)). 
}
Here, and in the rest of this section, we have in mind the identification of $\tremspace_{\til{q_1}}$ with $\tremspace_{\til q_2}$, for all $\til q_1, \til{q_2}$ in the same horospherical leaf; in particular, on the left-hand side of \eqref{eq: simplifies to}, $\beta_2$ belongs to $\tremspace_q$, and on the right-hand side, to $\tremspace_{q_1}$ for $q_1=\trem_{\beta_1}(q),$ and these spaces are identified  by choosing appropriate lifts $\til q, \til{q}_1.$
With this convention recall also from \eqref{eq: commutation 0} that $\trem_\beta(uq)  = u\, \trem_\beta(q),$ for any $u \in U.$

Note that the identification of $\tremspace_{\til q_1}$ with $\tremspace_{\til q_2} $ in Proposition \ref{prop: tremspace} need not send balanced tremors to balanced tremors. However, the horocycle flow commutes with horizontal straightline flows, and therefore
for $u \in U$ and  $\beta \in \tremspace_{ q} \cong
\tremspace_{uq}$, we have $L_{q}(\beta) = L_{u
  q}(\beta)$. From \equ{eq: simplifies to} and \equ{eq: commutation 0} we deduce:

\begin{cor}\name{cor: pass to balanced}
Let $\beta \in \tremspace_q$ and $s \df L_q(\beta)$. Then \begin{itemize} 
\item  $\beta - s (dy)_q \in \tremspace_{u_sq}$  is balanced. 
\item 
If $\beta$ is balanced in $\tremspace_q$ then $\beta$ is balanced in $\tremspace_{uq}$, for any $u \in U$. 
\end{itemize}
\end{cor}

Recall that $B \subset G$ denotes the upper triangular group. We now discuss the interaction between the 
$B$-action and tremors. Note that while an element $\mathbf{b} \in B$  maps horospherical leaves to horospherical leaves, it  does not necessarily preserve individual horospherical leaves, so we cannot use Proposition \ref{prop: tremspace} to identify $\tremspace_{q}$
 with $\tremspace_{\mathbf{b}\til q}$. Instead, we use the derivative of the affine comparison map $\psi_b$ defined in \S \ref{subsec: G} to identify 
 $\tremspace_{\til q}$ with 
$\tremspace_{\mathbf{b}q}$. Note that the subgroup of $B$ preserving horospherical leaves is $U$, and for $u \in U$ the map $\psi_u$ acts on $H^1(S, \Sigma; \R_x)$ trivially, and thus this identification coincides with the identification via parallel transport that is  used in Proposition \ref{prop: tremspace}.



The interaction of tremors with the $B$-action is as
follows.  

\begin{prop}\name{prop: commutation relations}
Let $q \in \HH$ and let 
\eq{eq: def p}{
\mathbf{b} = \left( \begin{matrix} a & z \\ 0 & a^{-1}\end{matrix} \right) \in
B, \text{ with } a = a(\mathbf{b})>0. 
}
Let $M_q$ and $M_{\mathbf{b}q}$ be the underlying
surfaces, and let $\til q \in \pi^{-1}(q)$. The above identification $\tremspace_{\til q} \to \tremspace_{\mathbf{b} \til q} $ multiplies the canonical transverse measure $dy$ by
$a^{-1}$ (where $a = a(\mathbf{b})$ is as in \equ{eq: def p}),  preserves the subsets of atomic and balanced foliation
cocycles, and maps $c$-absolutely continuous foliation cocycles to
$ac$-absolutely continuous foliation cocycles. 
Furthermore, 
\eq{eq: commutation 1}{\mathbf{b}\,
\trem_{ \beta}( q) = \trem_{a \cdot \beta}(\mathbf{b} q), \ \ \ \
\Dom(\mathbf{b}  q, \beta)
= a^{-1} \cdot \Dom( q, \beta).}

\end{prop}

\begin{proof} 
Let $q_1 = \mathbf{b}q$, denote the underlying surfaces by $M =
M_q, \, M_1 = M_{q_1}$ and write $\psi = \psi_{\mathbf{b}} : M\to M_1$ for 
the affine comparison map. 
Since the linear action of $\mathbf{b}$ on $\R^2$ preserves horizontal lines,
$\psi$ 
sends the horizontal foliation on $M$ to the horizontal foliation on
$M_1$. As in Proposition \ref{prop: tremspace}, $\psi$ sends transverse measures to transverse 
measures, non-atomic transverse measures to non-atomic transverse 
measures, and the induced map $\psi^*$ on cohomology sends
$\tremspace_q$ to $\tremspace_{q_1}$ and $C^+_q$ to $C^+_{q_1}$. 
Since $\psi$ is an affine map with derivative $\mathbf{b}$, the canonical transverse
measure $(dy)_q$ on $M_q$ is sent to 
its scalar multiple $a(\mathbf{b})^{-1} \cdot (dy)_{q_1}$ on $M_{q_1}$. 
Hence $c$-absolutely continuous
foliation cocycles are mapped to $a c$-absolutely continuous foliation
cocycles. 
To prove equation \equ{eq: commutation 1}, let $t \mapsto \til q_t$ be the
affine geodesic in $\HHm$ with $\til q_0 = \til q$ and
$\frac{d}{dt}|_{t=0} \til q_t = \beta$, so that $\til q_1 =
\trem_{\beta}(\til q)$. The new path $t \mapsto
\hat{q}_t = \mathbf{b}\til q_t$ is also an affine geodesic and
satisfies $\hat{q}_0 = \mathbf{b} 
\til q$. Now \equ{eq: commutation 1} follows from the fact that 
$\frac{d}{dt}|_{t=0} \hat{q}_t = a(\mathbf{b}) \cdot \beta,$ since
$\tremspace_q$ is embedded in the real space $H^1(S, \Sigma; \R_x)$.

We now show that our affine comparison map sends $\tremspace^{(0)}_q$
to 
$\tremspace^{(0)}_{q_1}$, that is, preserves balanced foliation cocycles. 
Since the horizontal direction is fixed by $\mathbf{b}$ and scaled by a factor of $a= a(\mathbf{b})$, $(dx)_{q_1}$ is obtained from $(dx)_q$ by
multiplication by $a$.
Now suppose $\beta \in \tremspace_{q}^{(0)}$
so that $\hol_q^{(x)} \cup
\beta=0$. By naturality of the cup product we get
$$0 = a^{-1} \hol_q^{(x)}\cup
\beta=\left(\psi^{-1}\right)^*\left( \hol_{q_1}^{(x)} \cup \psi^* \beta \right) =
\hol_{q_1}^{(x)} \cup \psi^*\beta.$$  
\end{proof}

\subsection{Relations between tremors and other maps}\name{subsec:
  new maps} 
We will now prove commutation and normalization relations between
tremors and other maps, which extend those in Proposition \ref{prop:
  commutation  relations}. These results will not be used in the sequel, but we hope they will be useful in the future. We will simultaneously discuss the interaction of tremors
with the action of $B$, all possible tremors for a
fixed surface, real-Rel deformations, and the $\R^*$-action on the
space of tremors. 

We will use
the notation and results of \cite{eigenform} in order to discuss
real-Rel deformations. Let $Z$ be the subspace of $H^1(S, \Sigma; \R_x) $ of cohomology classes 
which evaluate to zero on closed loops.
Thus $Z$ represents the subspace of real
rel deformations of surfaces in $\HH$ (see
\cite[\S 3]{eigenform} for more information).

Let $q \in \HH, \, M_q$ the underlying surface, $\varphi: S \to
M_q$ a marking map and $\til q \in \HHm$ the corresponding element in
$\pi^{-1}(q)$. We define semi-direct products   
$$
S^{(\varphi)}_1 \df B \ltimes \tremspace_{\til q}, \ \ \ \ S^{(\varphi)}_2 \df
B \ltimes
(\tremspace_{\til q} \oplus Z),
$$\index{S@$S^{(\varphi)}_i$}
where the group structure on $ S^{(\varphi)}_2$ is defined by
$$
\left(b_1, v_1, z_1).(b_2, v_2, z_2\right) = \left(b_1b_2, 
a^{-2}(b_2)v_1+v_2, a^{-1}(b_2) z_1+z_2 \right), $$
where 
$$ b_i \in  B, \ v_i \in \tremspace_{\til q}, \ z_i 
\in  Z, 
$$
$a(b)$ is defined in \equ{eq: def  p}. Also define the group structure on
$S^{(\varphi)}_1$ by thinking of it as a subgroup
of $S^{(\varphi)}_2$. 
Define the quotient semidirect products  
$$\bar{S}_1^{(\varphi)} \df S_1^{(\varphi)}/ \sim , \ \
\bar{S}_2^{(\varphi)} \df S_2^{(\varphi)} / \sim,$$
\index{S@$\bar{S}^{(\varphi)}_i$} where
$\sim$ denotes the 
equivalence relation $B \ni u_s  \sim s \cdot \hol_{\til q}^{(y)} \in
\tremspace_{\til q}$.

With this notation we have the following:

\begin{prop}\name{prop: semidirect product action}
  Let $q$, $M_q$, $\varphi$ and $\til q$ be as above, 
  and suppose $M_q$ has no
horizontal saddle connections (so that tremors and 
real-Rel deformations have the maximal domain of definition). Define 
$$
\Theta^{(\varphi)}_1: S^{(\varphi)}_1 \to \HHm, \ \
(b, \beta) \mapsto b \, \trem_{\beta}(\til q) 
$$
\index{U@$\Theta^{(\varphi)}_i $}
and
$$
\Theta^{(\varphi)}_2: S^{( \varphi)}_2 \to \HHm, \ \
(b, \beta, z) \mapsto b \, \mathrm{Rel}_z \, \trem_{\beta}(\til q).
$$
Then the maps $\Theta_i^ {(\varphi)}$ 
obey a `group action' law 
\eq{eq: group action law}{
  \Theta_i^{\left(\psi_{g_2} \circ\varphi
  \right)}(g_1) 
= \Theta_i^{(\varphi)}(g_1g_2) 
\ \ (i=1,2).
}
Moreover these maps are continuous, and descend to
well-defined immersions $\bar \Theta^{(\varphi)}_i:
\bar{S}_i^{(\varphi)} \to \HHm$. \index{U@$\bar \Theta^{(\varphi)}_i $} 
\end{prop}
We will only prove the statement corresponding to $i=1$. The case
$i=2$ will not be needed in the sequel and we will  leave it to 
 the reader. Specifically, in case $i=2$, the comparison map $\psi_{g_2}$ appearing in \eqref{eq: group action law} is defined up to isotopy in \cite{eigenform}, see the map $\bar f_t$ in Remark \ref{remark: from BSW}. 
\begin{proof}
The fact that the map
$\Theta_1^{(\varphi)}$ satisfies the group action law \equ{eq: group
  action law} with respect
to the group structure on $S_1^{(\varphi)}$ is immediate from
Propositions  \ref{prop: group action law tremors} and \ref{prop: commutation relations}. The fact that $\bar
\Theta_1^{(\varphi)}$ is well-defined on $\bar S_1^{(\varphi)}$ follows from \eqref{eq: horocycles are tremors} and \equ{eq: commutation 0}. The maps $
\Theta_1^{(\varphi)}, \, \bar \Theta_1^{(\varphi)}$ are continuous
because they are given as affine geodesics, and because of general
facts on ordinary differential equations. The fact that $\bar
\Theta_1^{(\varphi)}$ is an immersion can be proved by showing that
when $g_1, g_2$ are two elements of $S_1^{(\varphi)}$ that project to
distinct elements of $\bar S_1^{(\varphi)}$, then $ \dev \left(\bar
  \Theta_1^{(\varphi)}(q_i) \right)$ are distinct, i.e. the
operations have a different effect in period coordinates. 
  \end{proof}

There is also a natural action of
the multiplicative group $\R^* = \R \sm \{0\}$ on $\tremspace_{\til
  q}$ given by $(\rho, 
\beta) \mapsto \rho\beta$, where $\rho\in \R^*$ and $\beta \in
\tremspace_{\til q}$. This action preserves the set of balanced tremors
$\tremspace_{\til q}^{(0)}$. By Proposition \ref{prop: tremspace} and
Proposition \ref{prop: commutation 
  relations}, $\tremspace^{(0)}_{\til q}$ is a normal subgroup of $\tremspace_{\til q}$ and $ S_1^{(\varphi)}$. It is not hard to show using Corollary \ref{cor: pass to balanced} that $B \ltimes \tremspace_{\til q}^{(0)}$ is a normal subgroup of $S_1^{(\varphi)}$ isomorphic to the group $\bar S_1^{(\varphi)}.$
We define
a third semidirect product $S^{(\varphi)}_3 \df (\R^* 
\times B) \ltimes \tremspace_{\til q}^{(0)}$, where $\R^*$ acts on
$\tremspace_{\til q}^{(0)}$ by scalar multiplication and $B$ acts on
$\tremspace_{\til q}^{(0)}$ as above.
Arguing as in the proof of Proposition \ref{prop: semidirect
  product action} we obtain:
\begin{prop}\name{prop: more semi direct}
Let $q$, $M_q$, $\varphi$, $\til q$ be as in Proposition \ref{prop:
  semidirect product action}. Then the map  
$$
S_3^{(\varphi)} \to \HHm, \ \ \ \ (\rho, b, \beta) \mapsto b\, \trem_{\rho
  \beta}(\til q),
$$
obeys the group action law and is a continuous immersion.
%
\end{prop}

\begin{remark}
Note that (as reflected by the notation) the objects $S_i^{(\varphi)}$
and $\Theta_i^{(\varphi)}$ 
discussed above depend on the choice of a marking map. This is needed
because the marking map was used to identify $\tremspace_q$ for different
surfaces $q$. On the other hand 
\eqref{eq: commutation 0}  makes sense irrespective of a choice of a marking map. 
\end{remark}

\begin{remark}
In addition to the deformations listed above there is another deformation that could be considered. In the
spirit of \cite[\S 1]{VIET} (see also \cite[\S 2.1]{CMW}),
for each horizontally invariant fully supported probability measure $\nu$ on
$M_q$, there is a 
topological conjugacy sending $\nu$ to Lebesgue measure (on a different
surface $M_{q'}$). This topological conjugacy also induces a
comparison map $M_q \to M_{q'}$ and corresponding maps on foliation 
cocycles and on the resulting tremors, and it is possible to write
down the resulting group-action law which the map obeys when combined with 
those of Propositions
\ref{prop: semidirect product action} and \ref{prop: more semi
  direct}. This is left to the
assiduous reader. 
\end{remark}

\subsection{Tremors and sup-norm distance}
Let $\dist$ denote the sup-norm distance as 
in \S \ref{subsec: sup   norm}.

\begin{prop}\name{prop: lipschitz tremors}
If $q \in \HH$, $\nu$ is a non-atomic absolutely continuous signed
transverse measure on 
the horizontal foliation of $M_q$, and $\beta = \beta_\nu$ 
then 
\eq{eq: bound for distance}{
\dist(q,
\trem_{\beta}(q)) \leq \|\nu\|_{RN}.}
\end{prop}

\begin{proof}
 Let $ q_1 = \trem_\beta(q)$ and 
let $dy$ be the canonical transverse measure on $q$. 
Let  
$$\{\gamma(t) : t \in [0,1]\}, \ \ \text{ where } \gamma_t \df \trem_{t \beta}(q),$$ 
be the affine geodesic from $q$ to $q_1$. The tangent vector of
$\gamma$ is represented by the class $\beta$, 
and by specifying a marking map $\varphi_0 : S \to M_q$ we can lift
the path to $\HHm$, and find $\til q$, $\til q_1$ and $\til 
\gamma(t)$, $t \in [0,1]$ so that
$$\pi(\til q_1) = 
q_1, \ \pi(\til \gamma(t)) = \gamma(t) \ \text{ with } \ \til \gamma(0) = \til
q, \ \  
 \til \gamma(1) = \til q_1,$$ and $\til \gamma (t)$
satisfies 
$$\dev(\til \gamma(t)) = \dev(\til q)+ t
\bar \beta, \ \text{ where } \bar \beta = \left( \varphi_0^{-1}\right)^* \beta
\in H^1(S, \Sigma; \R).$$
We will use this path in 
\equ{eq:
  Finsler integrate} to give an upper bound on the distance from
$q$ and $q_1$. 
For each $t \in [0,1]$, write $q_t = \gamma(t)$ and denote the
underlying surface by $M_t$. Recall that
we denote the 
collection of saddle connections on a surface $q$ by $\Lambda_q$. We let $\Lambda'_q$ denote the saddle connections in $\Lambda_q$ which are not horizontal on $M_q$; for horizontal saddle connections $\sigma$ we have $\bar \beta(\sigma)=0$. For
any $\sigma \in \Lambda_q$, 
we have (with the notation of \S\ref{subsec: strata}) 
\eq{eq: bound on sc}{
\ell_{q_t}(\sigma) = \|\hol_{q_t} (\sigma) \|\geq |\hol^{(y)}_{q_t}(\sigma)|.}

By Proposition \ref{prop: tremspace}, we obtain transverse measures $\nu_t$ and
$(dy)_t$ on each $q_t$.
Using this, for all $t \in [0,1]$ we have 
\[ 
\begin{split}
\|\gamma'(t)\|_{\gamma(t)} = & 
\|\bar \beta \|_{\til q_t} = 
 \sup_{\sigma \in \Lambda_{\til q_t}} \frac{\|\bar \beta(\sigma)
   \|}{\ell_{\til q_t}(\sigma)} =  \sup_{\sigma \in \Lambda'_{\til q_t}} \frac{\|\bar \beta(\sigma)
   \|}{\ell_{\til q_t}(\sigma)}\\ \stackrel{\equ{eq: bound on sc}}{\leq} & \sup_{\sigma \in
  \Lambda_{\til q_t} }
\frac{\|\bar \beta(\sigma) \|}{\left|{\hol^{(y)}_{\til q_t}(\sigma)}\right |} 
 = \sup_{\sigma \in \Lambda_{\til q_t}}
\frac{|\int_{\sigma} d\nu_t |}{|\int_\sigma (dy)_t|}
\stackrel{\equ{eq: equivalently}}{\leq}  \|\nu\|_{RN}. 
\end{split}
\]
Integrating w.r.t. $t \in [0,1]$ in \equ{eq: Finsler integrate} we
obtain the bound \equ{eq: bound for distance}. 
\end{proof}

By moving along a horocycle orbit, small absolutely continuous tremors can be realized by small {\em balanced} tremors. Namely:
\begin{cor}
With the notations and assumptions of Proposition \ref{prop: lipschitz tremors},
    there is $q' \in Uq$ and
$\beta' \in \tremspace^{(0)}_{q'}$ with $|L|_{q'}(\beta') \leq 2\|\nu\|_{RN}$ and
\eq{eq: the identity}{
  \trem_\beta(q) = \trem_{\beta'}(q').}  
\end{cor}

\begin{proof}
This follows from Corollary \ref{cor: pass to balanced}, \eqref{eq: bound for distance}, and
the triangle inequality.  
\end{proof}

%

\section{Proof of Theorem \ref{thm: tremors bounded
    distance}}
  \name{section: horocycles and tremors}

We will now deduce the three assertions of Theorem \ref{thm: tremors bounded distance} from the results of the preceding sections.  
Throughout this section we write $q_1$ for $\trem_{\beta}(q)$ where $\beta \in \tremspace(q)$. The first assertion of the Theorem is that, for $\beta$  absolutely continuous, the distance between $u_sq$ and $u_sq_1$ remains bounded. 
\begin{proof}[Proof of Theorem \ref{thm: tremors bounded distance}(i)]
Let $\beta = \beta_\nu$ be the signed foliation cocycle corresponding
to a signed transverse measure $\nu$. 
We first claim that there is no loss of generality in assuming that
$\nu$ is $c$-absolutely continuous for some $c>0$. To see this, write
$\nu = \nu_1 +\nu_2$ where $\nu_1$ is aperiodic and $\nu_2$ is
supported on horizontal cylinders. By Lemma \ref{lem: absolutely
  continuous}, $\beta_{\nu_1}$ is $c_1$-absolutely continuous for some
$c_1$. Now modify $\nu_2$ so that for any horizontal cylinder
$C$ on $M_q$, the restriction of $\nu_2$ to $C$ is equal to $a_C
\, dy|_C$ for some positive constant $a_C$. Such a modification has no
effect on $\beta_{\nu_2}$, and will thus have no effect on
$\beta = \beta_{\nu_1} + \beta_{\nu_2}$. Thus, 
if $c_2 = \max_C a_C$, then (after the
modification), $\|\nu\|_{RN} \leq c_1+c_2.$
Now  using \eqref{eq: commutation 0} and Proposition \ref{prop: lipschitz tremors}, we see that
the left-hand side of \equ{eq: bounded} is bounded by $c_1+c_2$. 
\end{proof}

The second assertion of the Theorem is that if $\beta$ is absolutely continuous and essential then the horizontal foliation of a surface in the closure of the orbit $Uq_1$ is not uniquely ergodic. For this we will need the following statement, which will also be useful in \S \ref{sec: spiky fish}. 

\begin{prop}\name{prop: proper on c-ac}
Let $F \subset \HH$ be a closed set, and fix $c>0$. Then the sets 
\eq{eq: this set}{F' \df
\bigcup_{q \in F} \bigcup_{\beta \in C^{+, RN}_q(c)} \trem_\beta(q)
}\index{F@ $F'$}
and 
\eq{eq: this set2}{F'' \df
\bigcup_{q \in F} \bigcup_{\beta \in \tremspace^{RN}_q(c)} \trem_\beta(q)
}\index{F@ $F''$}
are also closed. 
\end{prop}
Recall from \S \ref{subsec: ac} that 
$C^{+,RN}_q(c)$ (respectively, 
$\tremspace^{RN}_q(c)$)
denotes 
the set of 
absolutely continuous (signed) foliation cocycles $\beta_\nu \in \tremspace_q$ with $\|\nu\|_{RN} \leq
c$. 
\begin{proof}
We first prove that $F'$ is closed. Let $q'_n \in F'$ be a
convergent sequence with $q' = \lim_n q'_n$. We need 
to show that $q' \in F'$.  Let $q_n \in F$ and $\beta_n \in
C^{+,RN}_{q_n}(c)$ 
such that $q'_n 
= \trem_{\beta_n}(q_n)$. 
We will show that $q'  = \trem_\beta(q)$ where $q$ and $\beta$  are accumulation points of the sequences $(q_n)$ and $(\beta_n)$. According to
Proposition \ref{prop: lipschitz tremors}, the sequence $(q_n)$ is
bounded with respect to the metric $\dist$. Also, a
computation similar to the one appearing in the proof  of Proposition
\ref{prop: lipschitz tremors}, gives $\|\beta_n\|_{q_n} \leq c$,
where $\| \cdot \|_{q_n}$ is the norm given by the Finsler structure defined in \equ{eq: define a
  norm downstairs}. By Proposition \ref{prop:
  sup norm properties} the sup-norm distance is proper, and hence the
sequence $(q_n)$ has a convergent subsequence. Thus passing
to a subsequence and using the fact that $F$ is closed, we can assume $q_n \to
q \in F$. Let $M_n$ be the underlying surfaces of $q_n$. Choose
marking maps $\varphi_n: S \to M_{q_n}$ and $\varphi: S \to M_q$ so that
the corresponding points $\til q_n \in \HHm$ satisfy $\til q_n \to \til q$.
Using these marking maps, identify $\beta_n$ with elements of
$H^1(S, \Sigma; \R^2)$. By the continuity
property of the norms $\| \cdot \|_{q_n}$ (see \S \ref{subsec: sup
  norm}), this sequence of cohomology classes is bounded, and so we can pass
to a further subsequence to assume that $\beta_n$ converges to $\beta \in H^1(S,
\Sigma; \R^2)$. 
%
%
Applying Proposition \ref{prop:
  uniform ac closed} we get that $\beta = \lim_{n\to\infty} \beta_n
\in C^+_{\til q}(c)$ and using Proposition \ref{prop: continuity} we
see that $q' =\trem_\beta(q) \in F'$. The proof that $F''$ is closed is similar. 
\end{proof}

\begin{proof}[Proof of Theorem \ref{thm: tremors bounded distance}(ii)] 
Let $q_1 = \trem_\beta(q)$ where $\beta = \beta_\nu $ and $\nu$ is
absolutely continuous. As in the proof of part (i) of the theorem, we
can assume that $\nu$ is $c$-absolutely continuous for some $c$,
i.e. $\beta \in C^{+, RN}_q(c)$,
and set $F = \overline{Uq}$. By commutation of tremors and horocycles (see \equ{eq: commutation 0}), for any $s \in \R$, we have
$u_s q_1 = \trem_{\beta}(u_sq)$.
By Proposition \ref{prop: tremspace}, $\beta \in C^+_{u_sq}(c)$ for all $s$, and so $u_s q_1 \in
F'$, where $F'$ is defined via \equ{eq: this set}. By Proposition
\ref{prop: proper on c-ac}  we have that any $q_2 \in \overline{Uq_1}
\sm \LL$ also belongs to $F'$, so is a tremor of a surface in
$\LL$. 

So we write $q_2 = \trem_{\beta'}(q_3)$ for $q_3 \in \LL$ and $\beta' \in \tremspace_{q_3}$, and write
$M_2, M_3$ for the underlying surfaces. Our goal is to show that the horizontal foliation on $M_2$ is not uniquely ergodic. Since
$\LL$ is $U$-invariant and $q_2 \notin \LL$, $\beta'$ is not a
multiple of the
canonical foliation cocycle $\hol_{q_3}^{(y)}$, i.e. the horizontal
foliation  on $M_{3}$ is not uniquely 
ergodic. By Proposition \ref{prop: tremspace}, neither is the horizontal foliation on
$M_2$. 
\end{proof}

The third assertion is that when $q$ is generic for some $U$-invariant ergodic measure $\mu$, assigning zero measure to surfaces with horizontal saddle connections, then $q_1$ is also generic for $\mu$ (but note that $q_1$ need not belong to $\supp \, \mu$). A heuristic explanation of this phenomenon is that for most values of $s,$ the surface $u_sq$ is close to surfaces with a uniquely ergodic horizontal foliation, which means that $C^+_{u_sq}$ is a narrow cone centered around the canonical transverse measure tangent to the horocycle flow. By continuity of tremors, in this case $u_sq_1$ is very close to $u_{s+s_0}q$ for some $s_0.$
\begin{proof}[Proof of Theorem \ref{thm: tremors bounded
    distance}(iii)]
We first employ an argument of \cite{LM}, to prove the following: 

\medskip

{\bf Claim 1:} For $\mu$-a.e. surface $q$, the horizontal foliation on the underlying surface $M_q$ is uniquely ergodic.

\medskip 

 Indeed, from  \cite{MW} we find that there is a compact subset $K \subset \HH$ such that any surface $q$ with no horizontal saddle connections
satisfies 
$$\liminf_{T \to \infty} \frac{1}{T} |\{s \in [0, T] : u_sq \in K \}|
> \frac{1}{2}
$$ 
(where $|A|$ denotes the Lebesgue measure of $A \subset \R$. Then by the Birkhoff ergodic theorem, any $U$-invariant ergodic
measure $\nu$ on $\HH$, which gives zero measure to surfaces with
horizontal saddle connections, satisfies $\nu(K) > 1/2.$ 
 If the claim is false, then by ergodicity
$\mu$-a.e. surface has a minimal but non-uniquely ergodic horizontal
foliation. Applying Masur's criterion (see e.g. \cite{MT}) 
to the horizontal foliation,  we find that for $\mu$-a.e.\,$q$, the
ray $\{ g_tq  :t<0\}$ is divergent. Thus for
$\mu$-a.e.\,$q$ there is $t_0 = t_0(q)$ such that for all $t \geq t_0,
\, g_{-t} q \notin K$.  Moreover, we can take $t_1$ large enough
so that $\mu(\{q : t_0(q) < t_1\}) > 1/2$ and hence $\nu = 
(g_{-t_1})_*\mu$ satisfies $\nu(K) < 
1/2$. Since $\nu$ is also $U$-ergodic, and also gives zero measure to
surfaces with horizontal saddle connections, this gives a
contradiction. The claim is proved. 
\hfill $\triangle$

\medskip

Let $\mu$ be the measure on $\LL$, let $q \in \LL$ be generic for
$\mu$, and let $q_1 = \trem_{\beta}(q)$ for some $\beta$. We need to
show that $q_1$ is generic. Let $f$ be a compactly supported
continuous test function and let $\vre>0$. Let $s_0 = L_q(\beta)$ and
let $q_2 = u_{s_0}q$. Since $q_2$ and $q$ are in the same $U$-orbit, 
$q_2$ is also generic.  For this pair $q_1$, $q_2$, we now claim:
\medskip 

{\bf Claim 2:} For 
every $\vre>0$, every $\delta >0$ and for all large enough $T$ there is a subset $A
\subset [0,T]$ with $|A| \geq (1- 
\vre 
)T$ so that for all $s \in A$,
$\dist(u_sq_1, u_sq_2) <\delta$.

\medskip

We 
first use Claim 2 to conclude the proof of the Theorem. 

By  
the uniform continuity of $f$, there is $\delta$ so that whenever
$\dist(x,y) < \delta$ we have $|f(x)-f(y)|<\frac{\vre}{4}$. Apply Claim 2 with $ \frac{\vre}{8 \|f\|_\infty}$ in place of $\vre$. Since $q_2$ is generic, for all 
 large enough  $T$
we have
$$\left|\frac{1}{T}  \int_0^T f(u_sq_2) \, ds - \int f d\mu \right | < \frac{\vre}{2}. $$ 
Using the triangle inequality, we see that for all large enough $T$:
\[
\begin{split}
& \left|\frac{1}{T}  \int_0^T f(u_sq_1) \, ds - \int f d\mu \right | \\ \leq &
\left|\frac{1}{T}  \int_0^T f(u_sq_1) \, ds - \frac{1}{T}  \int_0^T
  f(u_sq_2) \, ds \right | +
\left|\frac{1}{T}  \int_0^T f(u_sq_2) \, ds - \int f d\mu \right | \\ 
\leq & \frac{1}{T} \int_A \left|f(u_sq_1) - f(u_sq_2) \right| ds + \frac{1}{T}
\int_{[0, T] \sm A} 2 \|f\|_\infty \, ds + \frac{\vre}{2} \\ 
\leq & \frac{\vre}{4} + \frac{\vre}{4} + \frac{\vre}{2} = \vre. 
\end{split}
\]
This shows that $q_1$ is generic. 

\medskip

It remains to prove Claim 2. For this we use \cite{MW} again. Let $Q \subset \HH$ be
a compact set such that for all large enough $T$, 
$$
\frac{|A_1|}{T} \geq 1 -\frac{\vre}{2}, \ \text{ where } A_1 = \left \{
  s \in [0,T] : u_s q\in Q \right \}.
$$
Let $\til Q \subset \HHm$ be compact such that $\pi(\til Q) = Q$.
Fix some norm on $H^1(S, \Sigma; \R)$. Since $\til Q$ is compact, and by
the continuity in Proposition \ref{prop: continuity}, there is
$\delta'$ such that for any 
$\til q'\in \til Q,$ and $\beta_1, \beta_2 \in C^+_{\til q'}$ for which $
L_{\til q'}(\beta_1)=L_{\til q'}(\beta_2)=s_0,$ we have 
\eq{eq: first one}{\|\beta_1 -
\beta_2\| < \delta' \implies \dist\left(\trem_{\beta_1}(\til q'),
\trem_{
  \beta_2}(\til q')  \right) < \delta.
}
Let $\LL'$ denote the collection of surfaces in $\LL$ with no horizontal
saddle connections and for which the horizontal foliation is uniquely
ergodic. By Claim 1, $\mu(\LL')= \mu(\LL)=1$, and by Corollary
\ref{cor: semicontinuity}  there is a neighborhood $\mathcal{U}$ of
$\pi^{-1}(\LL')$ such that 
\eq{eq: second one}{
\til q' \in \mathcal{U}, \, \beta \in C^+_{\til q'}, \, L_{\til q'}(\beta) = s_0
\implies \|\beta - s_0 \, (dy)_{\til q'} \| < \delta'. 
}
Clearly
$\pi(\mathcal{U})$ is an open set  of full $\mu$-measure. Since
$q$ is generic, for all sufficiently large
$T$ there is a subset $A_2 \subset [0,T]$ with 
$$
\frac{|A_2|}{T} > 1-
\frac{\vre}{2} \ \text{ and }  
s \in A_2 \implies u_s q \in \pi(\mathcal{U}). 
$$
Now set
$
A = A_1 \cap A_2, $ so that $|A| > (1-\vre)T$. Suppose $s \in A$. Then there is $\til q' \in
\mathcal{U} \cap \til Q$ with $\pi(\til q') = u_sq$. We can view
$\beta$ as an element of $C^+_{u_sq}$ and with respect to the marked
surface $\til q'$ this corresponds to $\beta' \in C^+_{\til q'}$, and
we have 
$$ u_s q_1 =
\trem_{\beta}(u_sq) = \pi(\trem_{\beta'}(\til q'))\ \text{ and } u_s q_2 =
u_{s_0}q' = \pi(\trem_{ s_0 dy}(\til q')). $$
By \equ{eq: first one} and \equ{eq: second one}, we find 
$\dist(u_s q_1, u_s q_2) < \delta$, and the claim is proved. 
\end{proof}

\section{Points outside a locus $\LL$ which are generic for $\mu_\LL$}
In this section, after some
preparations, we prove Theorem \ref{thm: more detailed}. At the end of
the section we also discuss how tremored 
surfaces behave with respect to the divergence
of nearby trajectories under the horocycle flow. 

\subsection{Tremors and rank-one loci}\label{subsec: rank one}
We now recall the notions of Rel deformations and of a rank-one locus. 
Define $W\subset H^1(S, \Sigma; \R^2)$  to be the kernel of the restriction map \index{R@$\mathrm{Res}$}
$\mathrm{Res}: H^1(S, \Sigma) \to H^1(S )$ which takes a cochain to its
restriction to absolute periods. For any $q \in \HH$, and any 
 lift $\til{q} \in \pi^{-1}(q)$,  as in \S \ref{subsec: atlas of charts} we have an
 identification $T_{\til q}(\HHm) \cong H^1(S, 
\Sigma; \R^2)$, and the subspace of $T_q(\HH)$ corresponding to $W$ is
called the Rel subspace and is independent 
of the marking (see \cite[\S3]{eigenform} for more details). 
Let $\mathfrak{g} = \mathfrak{g}_q$ denote the tangent space to the
$G$-orbit of $q$ (we consider this as a subspace of
$T_q(\HH)$ for any $q$).  
A $G$-orbit-closure $\LL$ is said to be a
{\em rank-one locus} if there is a subspace $V\subset W$ such that for
any $q \in \LL$, the 
tangent space 
$T_q(\LL)$ is everywhere equal to $\mathfrak{g}_q \oplus V$. Rank-one loci
were introduced and analyzed by Wright in 
\cite{Wright cylinders}, and the eigenform loci 
$\EE_D$ in $\HH(1,1)$ are examples of rank-one loci. The following result, which  can be seen as a strengthening of an infinitesimal statement given in Corollary \ref{cor: balanced tremors are in B-}, is valid for all rank-one loci. 

\begin{prop}\name{prop: transverse rank one}
Suppose $\LL$ is a rank-one locus. Then for any compact set $K \subset
\LL$ there is an $\vre>0$ such that if $q \in K$ is horizontally aperiodic,
and $\beta \in \tremspace_q$ is 
an essential tremor satisfying $|L|_q(\beta) < \vre$, then
$\trem_{\beta}(q) \notin \LL$. If $q$ is horizontally minimal and 
$\overline{Uq} =\LL$, then no essential tremor of $q$ belongs to $\LL$. 

\end{prop}

\begin{proof}
 We leave it as an exercise to the reader to show that in rank-one loci, by Proposition \ref{prop:transverse measures} and \eqref{eq: tremor}, a small essential tremor of an aperiodic surface $q$ cannot have the same absolute periods as a surface obtained by applying a small element of $G$ to $q$. This establishes the first claim.

For the second assertion, suppose by contradiction that $\trem_{
  \beta} (q)\in \LL$ for some $q \in \LL$ with $\LL = \overline{Uq}$ and
$\beta \in \tremspace_q$ an essential tremor. Let $K$ be a bounded
open subset of $\LL$ and let $\vre>0$ be as in the first
assertion. 
The translated set $g_{t} U q$ is also dense 
in $\LL$, and $g_{t} u\, \trem_{\beta} (q)\in \LL$ for any $u \in
U$. By Proposition
\ref{prop: commutation relations}, $g_{-t} u \, \trem_{\beta} (q)= \trem_{
  e^{-t}\beta}(g_{-t}uq)$. Taking $t$ large enough so that $|L|_q(e^{-t} \beta)
< \vre$, and  choosing $u$ so that $g_{-t}uq \in K$, we get a
contradiction to the choice of 
$\vre$.   
\end{proof}

\begin{cor}\name{cor: how to reach from locus}
Suppose $\LL$ is a rank-one locus, $q_1, q_2 \in \LL$ are horizontally
minimal and have dense $U$-orbits, and for $i=1,2$ there are
$\beta_i \in \tremspace_{q_i}$ such that $\trem_{\beta_1}(q_1) =
\trem_{ \beta_2}(q_2)$. Then there is $u \in U$ such that $uq_1 =
q_2$. Furthermore, if $\beta_1$ and $\beta_2$ are balanced then $q_1 = q_2$ and
$\beta_1$ is obtained from $\beta_2$ by applying a translation equivalence. 
\end{cor}

\begin{proof}
Let $q_3 = \trem_{\beta_i}(q_i)$, let $M_3$ be the underlying surface,
and let $\varphi: S \to M_3$ be a marking map representing $\til q_3 \in
\pi^{-1}(q_3)$. For $i=1,2$, let $$\til \beta_i = \varphi_i^*(\beta_i) \in H^1(S,
\Sigma; \R_x)$$ 
be the cohomology classes for which $$\trem_{\til
  \beta_i}(\til q_i ) = \til q_3 \ \text{ and } \ \til q_i \in \pi^{-1}(q_i).$$ 
  By Proposition
\ref{prop: group action law tremors} we have $\trem_{\til
  \beta_1-\til \beta_2}(\til q_1)=\til q_2$. It
follows from Proposition \ref{prop: transverse rank one} that  
$\til \beta_1-\til \beta_2 = s_0 (dy)_{q_1}$ for some $s_0 \in \R$, i.e. $\trem_{\til
  \beta_1-\til \beta_2}(\til q_1) = 
u_{s_0}\til q_1$ and $ u_{s_0}q_1 = q_2$. If $\beta_1, \beta_2$ are balanced then 
$$
s_0 = \int_{M_{q_1}} dx \wedge s_0\, dy  = \int_{M_{q_1}} dx \wedge
( \beta_1-  \beta_2) = 
L_{q_1}(\beta_1) - L_{q_1}(\beta_2) =0,
$$
and this implies that $q_1=q_2$. Now considering the expression
\equ{eq: tremor} giving $\dev(\trem_{\beta}(\til q))$,
we see that the only possible ambiguity in the choice of $\til \beta_i$ for
which $\trem_{\til \beta_1}(\til q) = \trem_{\til \beta_2}(\til q)$ is
if $\til \beta_1,
\til \beta_2 \in H^1(S, \Sigma; \R_x)$ are exchanged by the action of
$\varphi^{-1} \circ h \circ \varphi$, where $h$ is a translation
equivalence of the underlying surface $M_{q}$. This gives the last
assertion. 
\end{proof}

We can use Proposition \ref{prop: transverse rank one} to construct
examples fulfilling property (III) in the discussion preceding the
formulation of Theorem 
\ref{thm: more detailed}; namely we will show there is $q\in \LL =
\EE$ and $q_1\notin \LL$, where $q_1$ is an essential tremor of $q$. We remark 
that in the introduction  we explicitly required that $q$ admit a  
tremor which is both essential and absolutely continuous. In fact this
assumption is redundant, that is for surfaces in $\EE$, foliation cocycles are
absolutely continuous. More precisely we have:

\begin{lem}\name{lem: automatic absolutely continuous}
For each aperiodic $q \in \EE$, and any $\beta \in \tremspace_q$,  
\eq{eq: automatic}{
|L|_q(\beta) \leq 1 \ \implies \beta \text{ is 2-absolutely
  continuous.}
}
\end{lem}
Recall that \eqref{eq: bound on length signed} gives that if $\beta\in \mathcal{T}^{RN}_q(2)$ then $|L|_q(\beta)\leq 2$.
\begin{proof}
First suppose $\beta = \beta_\nu \in C^+_q$ with
$L_q(\beta) = 1$. By Proposition \ref{prop: involution on E and
  measures} there is $c_1$ 
such that $\nu + \iota_* \nu = c_1 (dy)_q$. Since 
$$\int_{M_q} dx \wedge
dy = 1 =L_q(\beta)=
\int_{M_q}  dx \wedge
\nu  = \int_{M_q} dx \wedge d \iota_* \nu,$$ 
we must have  $c_1 = 2$, i.e. 
$$
(dy)_q = \frac{1}{2} d\nu + \frac{1}{2} d \iota_* \nu.
$$
This implies that $\beta \in C^{+, RN}_q(c)$ for $c = 2$. 
For a
general $\beta \in \tremspace_q$, with $|L|_q(\beta) \leq 1$, write
$\beta = \beta_{\nu^+} - \beta_{\nu^-}$, with $\beta_{\nu^\pm} \in C^+_q$ and repeat the
argument. For any transverse positive arc $\gamma$ we have $\int_\gamma
d\nu^\pm  \in \left[ 0, 2
\int_\gamma dy \right ],$ which implies \equ{eq: equivalently} with
$c=2$.

\ignore{
For the second assertion, assume first that $\nu$ is a positive
transverse measure which is 2-absolutely continuous and ergodic,
and let $\nu' = \iota_*
\nu$ and let $c $ so that $(dy)_q = c(d\nu + d\nu')$. The measures
$\nu, \nu'$ are mutually singular so for any $\vre>0 $ we can find a
short arc $\gamma$ such that $\int_\gamma d\nu' \leq \vre \int_\gamma
d\nu$ and $\int_\gamma d\nu >0$. Since $\nu$ is 2-absolutely continuous, this gives
$$
\int_\gamma d\nu \leq 2 \int_\gamma dy = 2c \int_\gamma (d\nu + d\nu')
\leq 2c(1 + \vre) \int_\gamma d\nu.
$$
Taking $\vre \to 0$ we see that $c \geq \frac{1}{2}.$
Since $dx = \iota_* dx$ we have $\int_{M_q} dx \wedge d\nu =
\int_{M_q} dx  \wedge d\nu'$, and thus 
$$1 = \int_{M_q} dx \wedge dy = c \int_{M_q} dx \wedge (d\nu + d\nu')
 = 2c \int_{M_q} dx \wedge d\nu = 2c L_q(\beta_\nu),$$
and $L_q(\beta_\nu) = \frac{1}{2c} \leq 1$.

If $\nu$ is a positive transverse measure which is not necessarily
ergodic, and is 2-absolutely continuous, then we can write $\nu =
\nu_1 + \nu_2$ where each $\nu_i$ is ergodic and 2-absolutely
continuous, and by the previous paragraph, $L_q(\beta) = L_q(\beta_{\nu_1})+
L_q(\beta_{\nu_2}) \leq 2$. If $\nu$ is a signed transverse measure,
then the conclusion follows by considering its Hahn decomposition. }
\end{proof}

\subsection{Nested orbit closures}\name{subsec: nested}
Theorems \ref{thm: more detailed} and \ref{thm: spiky fish} both
exhibit one-parameter families of distinct orbit-closures for the
$U$-action (see \equ{eq: 2nd assertion} and \equ{eq: 3rd
  assertion}). This  property is proved using the following 
general statement.
\begin{prop}\name{prop: properly nested} 
Let $F = \EE$, let $c>0$, and let $F''$ be the set defined by \equ{eq: this
  set2}. Let 
   $q_0$ be a surface in $\EE$ whose $U$-orbit is dense in $\EE$, and let 
  $\mathfrak{F}_1$ be a subset of $F''$ containing an essential tremor of $q_0$. 
For
each $\rho>0$ define 
\eq{eq: what we need for nested}{
\mathfrak{F}_\rho \df \left\{\trem_{\rho \beta}(q): q \in \EE, \ \beta \in
\tremspace_q^{(0)}, \ \trem_\beta(q) \in \mathfrak{F}_1 \right\}.
}
Then for $0< \rho_1 < \rho_2$ we have $\mathfrak{F}_{\rho_1} \neq
\mathfrak{F}_{\rho_2}.$
\end{prop}

\begin{proof}
By Corollary \ref{cor: pass to balanced}, replacing $q_0$ with an
element in its $U$-orbit, there is no loss of 
generality in assuming that $\mathfrak{F}_1$ contains an essential balanced
tremor of $q_0$. Thus if we define
$$
\tremspace_{q_0}^{(0)}(\rho) \df \left\{\beta \in
\tremspace_{q_0}^{(0)}:  \trem_{\beta}(q_0) \in \mathfrak{F}_\rho \right\},
$$
then $\tremspace_{q_0}^{(0)}(1) $ contains a nonzero vector. Clearly
for all $\rho>0$ we have $\tremspace_{q_0}^{(0)}(\rho) = \rho
\tremspace_{q_0}^{(0)}(1),$ so each of the sets
$\tremspace_{q_0}^{(0)}(\rho)$ contain nonzero vectors as well. By
\equ{eq: this set2} and Corollary 
\ref{cor: how to reach from locus}, the sets
$\tremspace_{q_0}^{(0)}(\rho)$ are bounded for each $\rho$. Now
suppose by contradiction that for $\rho_1 < \rho_2$ we have 
$\mathfrak{F}_{\rho_1} = \mathfrak{F}_{\rho_2}$. Then 
$$
\tremspace_{q_0}^{(0)}(\rho_1) = \tremspace_{q_0}^{(0)}(\rho_2) =
\frac{\rho_2}{\rho_1} \tremspace_{q_0}^{(0)}(\rho_1) .$$
But $\frac{\rho_2}{\rho_1}>1$ and a bounded subset of
$\tremspace^{(0)}_{q_0}$ cannot be invariant under a nontrivial
dilation if it contains nonzero points. This is a contradiction. 
\end{proof}

\begin{proof}[Proof of Theorem \ref{thm: more detailed}] We will find
  a surface satisfying conditions (I), (II) and (III) of the theorem.
It was  shown by Katok and Stepin \cite{KS}  that there is a surface
$q \in \EE$ with a 
horizontal foliation which is not uniquely ergodic and has no
horizontal saddle connection (Veech \cite{Veech
  Kronecker} proved an equivalent result on $\mathbb{Z}_2$-skew
products of rotations, see \cite{MT}). 
 Thus the underlying surface $M_q$ satisfies 
condition (II). To see that $q$ satisfies condition (III) we apply
Proposition \ref{prop: transverse rank one} to the rank-one locus
$\EE$.

To see that $q$ satisfies condition (I), we use
\cite[Thm. 10.1]{eigenform}, which states that the $U$-orbit of
every point in $\EE$ is generic for some
measure; furthermore, the result identifies the measure. In the terminology of
\cite{eigenform},
the $G$-invariant `flat' measure on $\EE$ is the measure of type 7. The last bullet
point of the theorem states that a 
 surface is equidistributed with respect to flat measure if it has no
 horizontal saddle connection and is 
 not the result of applying a real-Rel flow to a lattice surface.
However lattice surfaces without horizontal saddle connections have a
uniquely ergodic 
horizontal foliation (\cite{Veech - alternative}) and the horizontal foliation is
preserved under real-Rel deformations. This implies that $q$ cannot
be a real-Rel deformation of a lattice surface.

For the proof of the second assertion, equation \eqref{eq: 2nd
  assertion},  we combine Propositions 
\ref{prop: more semi direct} 
and \ref{prop: properly nested}. Namely, we let $q_r = \trem_{
  r,\beta}(q)$ be as in the statement of the Theorem and define 
$$\hat{\mathfrak{F}}_\rho \df \overline{Uq_{\rho}} \ \ \text {and } \ 
\mathfrak{F}_\rho \df \left\{\trem_{\rho \beta}(q): q \in \EE, \ \beta \in
\tremspace_q^{(0)}, \ \trem_\beta(q) \in \hat{\mathfrak{F}}_1 \right\}.
$$
Recall the $\R^*$-action multiplying elements of $\tremspace_q$ by positive scalars (see \S \ref{subsec:
  new maps}). 
Since $q_r$ is obtained from $q_1$ using the
$\R^*$-action with parameter $r$, by naturality of the $\R^*$-action (see Proposition
\ref{prop: more semi direct}) we obtain that  $\hat{\mathfrak{F}}_\rho=\mathfrak{F}_\rho$.  
So $\hat{\mathfrak{F}}_{r_1} \varsubsetneq \hat{\mathfrak{F}}_{r_2}$ for $r_1<r_2,$ and  \equ{eq: 2nd assertion} follows by Proposition \ref{prop: more semi direct}. 
\end{proof}

\begin{remark}
As we remarked in the introduction (see Remark \ref{remark: other
  loci}), Theorem \ref{thm: more detailed} remains valid for
other eigenform loci $\EE_D$ in place of $\EE = \EE_4$. Indeed, the
results of \cite{eigenform} used above are valid for all eigenform loci, and
to prove the existence of surfaces in $\EE_D$ whose horizontal foliations
are minimal but not ergodic, one can use \cite{CM} in place of
\cite{KS}. Thus the proof given above goes through with obvious modifications. Finally we note
that Lemma \ref{lem: automatic absolutely continuous} is also true for 
other eigenform loci, provided the constant 2 on the right hand side
of \equ{eq: automatic} is replaced with an appropriate constant
depending on the discriminant $D$. We leave the details to the
reader.
\end{remark}

\subsection{Erratic divergence of nearby horocycle orbits}
\name{subsec: erratic}
A crucial ingredient in Ratner's measure classification
  theorem is the polynomial divergence of 
nearby trajectories for unipotent flows. As we have seen in Corollary
\ref{eq: sup norm horocycle deviation} there is a quadratic 
upper bound on the distance between two nearby horocycle trajectories
in a stratum $\HH$, with respect to the sup-norm
distance. Such upper bounds can also be obtained in the homogeneous
space setting, but in that setting they are accompanied by
complementary lower bounds. Namely, Ratner used the fact 
that if 
$\{u_s\}$ is a unipotent flow on a homogeneous space $X$, for some
metric $d$ on $X$ we have (see e.g. \cite[Cor. 1.5.18]{morris book}):

\medskip

{\em $(*)$ for any $\vre>0$ and every $K \subset X$ compact, there is
  $\delta>0$ such that if $x_1, x_2 \in   X$ and for some $T>0$ we
  have 
$$ \left| \left\{ s \in [0,T]: d(u_sx_1, u_sx_2) < \delta, u_s x_1 \in
    K \right\} \right| \geq \frac{T}{2},
$$
then for all $s \in [0,T]$ for which $u_sx_1 \in K$ we have 
$
d(u_s x_1, u_sx_2)<\vre.
$
}

\medskip 

Our proof of Theorem
\ref{thm: more detailed} shows that $(*)$ fails for strata and in fact
we have:

\begin{thm}\name{thm: erratic}
There is a stratum $\HH$, a compact set $K \subset \HH$, $\vre>0$, and
$q_1, q_2 \in \HH$, so that for any $\delta>0$, 
\eq{eq: weird behavior 1}{
\liminf_{T \to \infty} \frac{1}{T} \left| \left\{s \in [0,T] : \dist(u_s
    q_1, u_s q_2) < \delta, \, u_s q_1 \in K \right\} \right| > \frac{1}{2},
}
but the set
\eq{eq: weird behavior 2}{ \{s\geq  0: 
u_{s}q_1 \in K \text{ and } \dist(u_{s}q_1, u_{s}q_2) \geq \vre \}
 } is nonempty. 
 \end{thm}
 
\begin{proof}
Take $q_1 \in \LL$ for some $\LL$ as in the proof of Theorem \ref{thm:
  more detailed}, where $q_1$ admits an essential tremor, and is
generic for the $G$-invariant measure on $\LL$, and let
$q_2$ be a balanced essential tremor of $q_1$. Let $0 < \vre <
\dist(q_1, q_2)$, so that \equ{eq: weird behavior 2} holds. Claim 2 in the proof of Theorem \ref{thm: tremors bounded
    distance}(iii) implies \equ{eq: weird behavior 1}.  
\end{proof}

\begin{remark} 
The construction in \S \ref{sec: spiky fish} exhibits a
stronger contrast to assertion $(*)$: it gives 
examples in which equation \equ{eq: weird behavior 1} holds while the set in equation \equ{eq:
  weird behavior 2} is unbounded. 
\end{remark}

\section{Existence of non-generic surfaces} \name{sec: nongeneric}
In this section we will prove Theorem \ref{thm: 2}. Let $B$ be the
upper-triangular group. We will need the following useful consequence
of the interaction of tremors 
with the $B$-action.

\begin{thm}\name{thm: dense tremors} Let $\HH$ be a stratum of
  translation surfaces and let $\LL \varsubsetneq \HH$ be a
  $G$-invariant locus such that there is $q \in \LL$ with
  $\overline{Gq} = \LL$ and such that $q$ admits an
  essential absolutely continuous tremor which does not belong
  to $\LL$. Then the closure of the set 
\eq{eq: obtained by tremoring}{
\bigcup_{q' \in Bq} \{\trem_{\beta}(q'): \beta \in C^+_{q'} \text{  
  is an essential absolutely continuous tremor}\}
}
is $G$-invariant and contains a $G$-invariant locus $\LL'$ with $\dim
\LL' > \dim \LL$. 

In particular, if $\LL = \EE \subset \HH(1,1)$, then the set in
\equ{eq: obtained by tremoring}  
is dense in $\HH(1,1)$. 
\end{thm}

\begin{proof}
Let $\Omega$ be the set in equation \equ{eq: obtained by tremoring}, and let $F$ be the
closure of $\Omega$. By assumption there is $q
\in \LL$ and an
absolutely continuous $\beta \in C^+_q \sm T_q(\LL),$ and hence for
$\vre>0$ sufficiently small, the curve 
$$t \mapsto q(t)\df  \trem_{t\beta}(q), \ \  t \in (-\vre, \vre)$$ 
satisfies $q(t) \in \Omega \sm
\LL$ for $t\neq 0$ and $q=\lim_{t\to 0} q(t)$; i.e., $q \in
\overline{\Omega \sm \LL}$. By Proposition \ref{prop: commutation
  relations}, $\Omega$ is $B$-invariant, and hence so is $F$. According to
\cite[Thm. 2.1]{EMM}, any $B$-invariant closed set is $G$-invariant, and is a
finite disjoint union of $G$-invariant loci. This implies that $\LL = \overline{Bq} 
\subset F$, and also that we can write $F = F_1 \sqcup
\cdots \sqcup F_k$ where each $F_i$ is a closed $G$-invariant
locus supporting an ergodic $G$-invariant measure, and for $i \neq j$
we have $F_i \not \subset F_j$. There is an $i$ 
so that $\LL \subset F_i$, and we claim
$\LL \varsubsetneq F_i$. Suppose $\LL = F_i$ and let $q(t)$ as
above. Then for sufficiently small $t>0$ we have $q(t) \notin F_i$. So
there is some $j$ such that $F_j$ contains a sequence $q(t_n)$ with
$t_n>0$ and $t_n \to 0$. Since $F_j$ is closed we find that $q \in F_j$. But since $F_i
= \overline{Gq}$ and $F_j$ is $G$-invariant and closed, we obtain that
$F_i \subset F_j$, a contradiction proving the claim. 

Thus if we set $\LL' \df F_i$ we have $\LL \varsubsetneq \LL',$ and
since both $\LL$ and $\LL'$ are manifolds and 
each is the support of a smooth ergodic measure, we must have
$\dim \LL < \dim 
\LL'$, as claimed.  To prove the second assertion, that $\mathcal{L}'=\mathcal{H}(1,1)$ 
we note that by McMullen's classification \cite{McMullen-SL(2)}, there
are no $G$-invariant loci $\LL'$ 
satisfying $\EE \varsubsetneq \LL' \varsubsetneq
\HH(1,1)$. 
\end{proof}

\begin{proof}[Proof of Theorem \ref{thm: 2}] First we
  claim that a dense set of surfaces in $\HH(1,1)$ are generic for   $\mu_1 =
  \mu_\EE$, the natural measure on $\EE$. By Theorem \ref{thm: tremors
    bounded distance}(iii) it suffices to show that tremors of
  surfaces in $\mathcal{E}$ with no horizontal saddle connections are
  dense in  
$\mathcal{H}(1,1)$. By Theorem \ref{thm: dense tremors} it
suffices to show that there exists a surface in $\mathcal{E}$ with no
horizontal saddle connections that admits an essential
tremor. Theorem \ref{thm: more detailed} establishes this, and the
claim is proved.

We now use a
  Baire category argument.
Let $\mu_2$ be the
natural flat measure on the entire stratum $\HH(1,1)$. Let $f$
  be a compactly supported continuous function with $\int f d\mu_1
  \neq  \int f d\mu_{2}$, and let $\vre>0$ be small enough so that 
$$2\vre < \left|\int f d\mu_{1} - \int f d\mu_{2} \right|.$$
For $j =1, 2$ and $T>0$ let 
$$\mathcal{C}_{j,T} \df \left \{ q \in \HH(1,1) : \left| \frac{1}{T} \int_0^T
  f(u_sq) ds - \int f d\mu_{j} \right| < \vre\right \}
$$ 
(which is an open subset of $\HH(1,1)$), and let 
$$
\mathcal{C}_j \df \bigcap_{n \in \N} \bigcup_{T \geq n} \mathcal{C}_{j,T}.
$$
If $q$ is generic for $\mu_j$ then $q \in \mathcal{C}_{j,T}$ for all
$T$ sufficiently large. Since generic surfaces for $\mu_j$
are dense in $\HH(1,1)$, each $\mathcal{C}_j$ is a dense
$G_\delta$-subset of $\HH(1,1)$. By definition, for $q \in
\mathcal{C}_j$ we have a subsequence $T_n \to \infty$ such that
$\frac{1}{T_n} \int_0^{T_n} f(u_sq)ds$ converges to a number $L$ with
$|L-\int f d\mu_j| \leq \vre$.  In particular, any 
$q \in \mathcal{C}_1 \cap \mathcal{C}_2$ satisfies \equ{eq: not
  generic}. For the last assertion, note that the set of surfaces with a dense orbit under the diagonal group $\{g_t\},$ in either forward or backward time,  is also a dense $G_\delta$ subset, and
so intersects $\mathcal{C}_1 \cap \mathcal{C}_2$ nontrivially.  
\end{proof}


\section{A new horocycle orbit closure}\name{sec: spiky fish}
In this section we will prove Theorem \ref{thm: spiky
  fish}. 
We first show the inclusion between the two subsets of $\HH(1,1)$ described in equation \equ{eq: spiky
  fish}, namely we show that
\eq{eq: 1.8 revisited}{
\begin{split}
  & \overline{\{\trem_\beta(q): q \in \EE \text{ is aperiodic},\  \beta \in
  \tremspace_{q} ,\ |L|_q(\beta)  \leq a\}} \\ \subset & \{
  \trem_\beta(q) : q \in \EE,\  \beta \in
  \tremspace_{q} ,\,  |L|_q(\beta)  \leq a\}.
  \end{split}
  }
 To see this note that Proposition \ref{prop: proper 
  on c-ac} and Lemma \ref{lem: automatic
  absolutely continuous}  imply that the first set is contained in the closed set $$\left\{\trem_{\beta}(q):q \in \EE, \ \beta \in \tremspace^{RN}(2a) \right\}.$$  Corollary \ref{cor: total variation continuous} implies that any limit point must satisfy $|L|_q(\beta)\leq a$. 

For the last assertion of the Theorem, note that the inclusion in equation
\equ{eq: 3rd assertion} is obvious from the first line of equation \equ{eq: spiky fish}, and the
naturality of the $\R^*$-action (Proposition 
\ref{prop: more semi direct}). The inclusion is proper by
Theorem \ref{thm: more detailed}. 

It remains to show the existence of a surface $q_1$ for which we have 
equality in equation \equ{eq: spiky fish}, namely for which
\eq{eq: namely for which}{
\overline{Uq_1}  = \overline{
\{\trem_\beta(q): q \in \EE \text{ is aperiodic},\  \beta \in
  \tremspace_{q} ,\ |L|_q(\beta)  \leq a\} }
.}
Before doing this, we 
set up some notation to be used throughout this section and describe
our strategy. We 
partition $\EE$ into 
the following subsets:
\[
\begin{split}
\EE^{(\mathrm{per})} & = \{q \in \EE: M_q \text{ is horizontally periodic}
\}\index{EA@$\EE^{(\mathrm{per})}$}, \\
\EE^{(\mathrm{tor})} & = \{q \in \EE: M_q
\text{ is two tori glued along a horizontal slit} \} \sm \EE^{(\mathrm{per})} 
\index{EB@$\EE^{(\mathrm{tor})}$}, \\
\EE^{(\mathrm{min})}\index{EC@$\EE^{(\mathrm{min})}$} & = \EE \sm \left(\EE^{(\mathrm{per})} \cup
\EE^{(\mathrm{tor})} \right)  \\
& = \{q \in \EE : \text{all infinite horizontal
  trajectories are dense}\}.
\end{split}
\]

Note that the set of aperiodic surfaces in $\EE$ is precisely
$\EE^{(\mathrm{tor})} \cup \EE^{(\mathrm{min})}$. It is easy to check that the sets $\EE^{(\mathrm{per})}$  and
$\EE^{(\mathrm{min})} $ are both dense in $\EE$; this follows easily from \cite[Thms 4.1 \& 1.8]{MT}. The set $\EE^{(\mathrm{tor})}$ is also dense --- this can be derived from \cite{EMM}, or in a more elementary fashion from Proposition \ref{prop: auxiliary}(2), see the proof of Proposition \ref{prop: tor dense in min}. 
We further partition $\EE^{(\mathrm{tor})}$ according to the length of the
slit. 

Let $\EE^{(\mathrm{tor}, H)}\index{EE@$\EE^{(\mathrm{tor}, H)}$}$ be the set of 
$q\in\EE^{(\mathrm{tor})}$ for which $M_q$ is made of two tori glued along a horizontal slit of length 
exactly $H$.



Although the individual sets $\EE^{(\mathrm{tor}, H)}$ are not dense in $\EE$, for each
$H_0>0$ the 
union $\bigcup_{H>H_0} \EE^{(\mathrm{tor}, H)}$ is dense in $\EE$. 

Now for positive parameters $a$ and $H$ we define subsets of $\HH(1,1)$:
\[
\begin{split}
\SF_{(\leq a)}^{(\mathrm{min})}\index{SF@$\SF_{(\leq a)}^{(\mathrm{min})} $} & = \left\{
\trem_\beta(q) : q \in \EE^{(\mathrm{min})},\,  \beta \in
\tremspace_q,\, |L|_q(\beta) \leq a 
\right \}\\
\SF_{(\leq a)}^{(\mathrm{tor})}\index{SF@$\SF_{(\leq a)}^{(\mathrm{tor})}$} & = \left\{
\trem_\beta(q) : q \in  
\EE^{(\mathrm{tor})},\,  \beta \in \tremspace_q,\, |L|_q(\beta) \leq a
\right \} \\
\SF_{(\leq a)} \index{SF@$\SF_{(\leq a)}$}& = \SF_{(\leq
  a)}^{(\mathrm{min})} \, \cup \, \SF_{(\leq a)}^{(\mathrm{tor})} 
\\
\SF_{(\leq a)}^{(\mathrm{tor}, H)}\index{SF@$\SF_{(\leq
    a)}^{(\mathrm{tor}, H)}$} & = \left \{\trem_\beta(q) \in
  \SF^{(\mathrm{tor})}_{(\leq a)} : q \in  
\EE^{(\mathrm{tor}, H)} \right\}.
\end{split}
\]
To lighten the notation, in the remainder of this section we will denote the closure  $\overline{\SF_{(\leq
  a)}}$ by $\overline{\SF}$\index{SF@$\overline{\SF}$}.
The letters $\SF$ stand for `spiky fish', and one can think of
$\overline{\SF }\sm \EE$ as the spikes of the spiky
fish. 
For $q \in \EE^{(\mathrm{tor})} \cup \EE^{(\mathrm{min})}$, denote by
$C^{+, \mathrm{erg}}_q$ \index{C@$C^{+, \mathrm{erg}}_q$} the extreme 
rays in the cone of foliation cocycles.
If the horizontal
direction is not uniquely ergodic on $M_q$ then Proposition \ref{prop:
involution on E and measures} shows that $C^{+, \mathrm{erg}}_q$ consists of two rays interchanged by the involution
$\iota$.
Further denote
\[
\begin{split}
\SF^{(\mathrm{min})}_{(=a)} & = \left\{ \trem_\beta(q) : q \in 
\EE^{(\mathrm{min})}, \beta \in C^{+, \mathrm{erg}}_q, L_q(\beta)=a \right\} \\ 
\SF^{(\mathrm{tor})}_{(=a)} & = \left\{ \trem_\beta(q) : q \in 
\EE^{(\mathrm{tor})}, \beta \in C^{+, \mathrm{erg}}_q, L_q(\beta)=a \right \} \\ 
\SF_{(=a)} & = \SF^{(\mathrm{min})}_{(=a)}  \, \cup \,
\SF^{(\mathrm{tor})}_{(=a)} \\
\SF^{(\mathrm{tor}, H)}_{(=a)} & = \left \{ \trem_\beta(q) : q \in 
\EE^{(\mathrm{tor}, H)}, \beta \in C^{+, \mathrm{erg}}_q, L_q(\beta)=a \right \} 
.
\end{split}
\]
Note that for $\beta \in C^{+, \mathrm{erg}}_q$, $L_q(\beta)=|L|_q(\beta)$.

With this terminology it is clear that equation \equ{eq: namely for which} (and
hence Theorem \ref{thm: spiky fish}) follows from:
\begin{thm}\name{thm: spiky fish 1}
For any $a>0$ there is $q_1 \in \SF^{(\mathrm{min})}_{(= a)}$,
such that $\overline{Uq_1} = \overline{\SF} =
\overline{\SF^{(\mathrm{tor})}_{(\leq a)}}$. 
\end{thm}

The proof of Theorem \ref{thm: spiky fish 1} will make use of the following intermediate
statements. Throughout this section, dist refers to the sup-norm
distance discussed 
in \S \ref{subsec: sup norm}. We will restrict dist to
$\overline{\SF}$, in particular the balls which will appear in the
proof are subsets of $\overline{\SF}$.

\begin{prop}\name{prop: tor dense in min}
For any $q \in \SF_{(\leq a)}^{(\mathrm{min})}$ and any $\vre>0$ there
is $q' \in \SF_{(\leq a)}^{(\mathrm{tor})}$ such that $\dist(q, q')<\vre.$
\end{prop}

 \begin{prop}\name{prop: min H dense in min}
For any $a>0$, any $q \in \SF_{(\leq a)}^{(\mathrm{tor})}$ and any $\vre>0$ there
is an $H_0$ such that for each $H>H_0$ there is a $q'\in \SF_{(=
  a)}^{(\mathrm{tor}, H)}$ such that $\dist(q,q')<\vre$. 
\end{prop}

Note that the approximation described in Proposition \ref{prop: min H dense in min} needs to accomplish two goals: approximate a tremor with total mass at most $a$ by tremors of total mass exactly $a$; and do so with a prescribed slit length $H.$

\begin{prop}\name{prop: min dense}
For positive constants $a$ and $H$ and any $q \in  \SF_{(=
  a)}^{(\mathrm{tor}, H)}$ the set $\overline{Uq}$ contains all of $\SF_{(=
  a)}^{(\mathrm{tor},H)}.$  
\end{prop}

\begin{proof}[Proof of Theorem \ref{thm: spiky fish 1} assuming
  Propositions \ref{prop: tor dense in min}, \ref{prop: min H dense in
    min}  
  and \ref{prop: min dense}]
  The equality $ \overline{\SF}
  =
\overline{\SF}^{(\mathrm{tor})}_{(\leq a)}$ is clear from Proposition
\ref{prop: tor dense in min}.  
 We will
prove:
\begin{itemize} \item[(i)]
  There is $q_1 \in \overline{\SF}
  $ for which the
  orbit $Uq_1$ is dense in $\overline{\SF}
  $. 
\item[(ii)] Any $q_1$ as in (i) satisfies $q_1 = \trem_\beta(q)$ for some $q \in
\EE^{(\mathrm{min})}$ and $\beta \in C^{+, \mathrm{erg}}_q$ with $L_q(\beta)
=a$.  
\end{itemize}

To prove (i), we will use the Baire category
theorem. { In this argument we will consider $\overline{\SF}$ as a  metric space in its own right, with respect to the restriction of the metric $\dist.$ Since $\overline{\SF}$ is closed and $U$-invariant, this is a complete metric space on which the $U$-action is continuous.} Given 
$\vre >0$ and a compact set $K \subset \overline{\SF}$, let
$\mathcal{V}_{K, \vre}$ denote the set of points in $\overline{\SF}$
whose $U$-orbit is $\vre$-dense in $K$. By continuity of the horocycle
flow and compactness of $K$, one sees that $\mathcal{V}_{K, \vre}$ is relatively open. We will
show that $\mathcal{V}_{K, \vre}$ is not empty. To see
this, note that by Proposition \ref{prop: tor dense in min}, given a
compact $K 
\subset \overline{\SF}$ and $\vre>0$ there is a 
finite set $F\subset \SF^{(\mathrm{tor})}_{(\leq
  a)}$ which is $\vre/2$-dense
in $K$. For $p \in F$, let $H_0=H_0(p)$ be the constant given in Proposition
\ref{prop: min H dense in min}, where we substitute $p$ for $q$ and replace $\vre$ with $\vre/2$.  Let $H> \max_{p\in
  F} H_0(p)$. Then for each $p$ 
there is $q'_p \in \SF^{(\mathrm{tor},
  H)}_{(=a)}$ such that $\dist(p, q_p')<\vre/2.$ Finally by
Proposition \ref{prop: min dense}, for any $q \in \SF^{(\mathrm{tor},
  H)}_{(=a)}$, the closure of $Uq$ contains all of the $q'_p$. Thus the orbit
$Uq$ comes within distance $\vre/2$ of each  $p \in F$ and in
particular is $\vre$-dense in $K$. {We have now shown that for any $\varepsilon>0$ and $K \subset \overline{\SF}$, $\mathcal{V}_{K,\varepsilon}\neq \emptyset$. 

We additionally claim that $\mathcal{V}_{K,\varepsilon}$ is dense for all $K$ compact and $\varepsilon>0$. To see this, first observe that  $$K\subset K' \text{ and } 0<\varepsilon'<\varepsilon \ \ \implies \ \ \mathcal{V}_{K,\varepsilon}\supset \mathcal{V}_{K',\varepsilon'}.$$ 
Given $x \in \overline{\SF}$ and $\varepsilon'>0,$ assume with no loss of generality that $\varepsilon'<\epsilon $ and apply the preceding statement, to $\varepsilon'$ instead of $\varepsilon$ and $K' \df K \cup \{x\}$ instead of $K. $ The $U$-orbit of any point in $\mathcal{V}_{K', \varepsilon'}$ intersects $B(x, \varepsilon'),$ and since $\mathcal{V}_{K', \varepsilon'}$ is contained in $\mathcal{V}_{K, \varepsilon}$ and is $U$-invariant, we have found that $\mathcal{V}_{K, \varepsilon}$ intersects $B(x, \varepsilon').$ Since $\varepsilon'$ was arbitrary, this shows that $\mathcal{V}_{K, \varepsilon}$ is dense.

}

Now let $K_1 \subset K_2 \subset \cdots$ be an exhaustion of
$\overline{\SF}$ by compact sets { and $\varepsilon_1>\varepsilon_2>...>0$ with $\lim \varepsilon_j=0$. By the Baire category theorem, and since all sets of the form $\mathcal{V}_{K, \vre}$ are open and dense in $\overline{\SF},$
$$\bigcap_{n=1}^{\infty} \mathcal{V}_{K_n,\varepsilon_n} \neq \emptyset.$$  
Clearly, any point in this intersection has a $U$-orbit which is dense in $\overline{\SF}.$ We have proved (i).}

To prove assertion (ii), recall from \equ{eq: 1.8 revisited} that $\overline{\SF}_{(\leq
  a)}$  is contained in
the set
$$\mathcal{S}_{(\leq a)} \df 
\{\trem_\beta(q): q
\in \EE, \, \beta\in \tremspace_q, \, |L|_q(\beta) \leq a\}.
\index{S@$\mathcal{S}_{(\leq a)}$}$$
Thus $q_1$ is of the form $\trem_{\beta}(q)$ for some $q \in \EE$ and $\beta \in \tremspace_q$ with
$|L|_q(\beta) \leq a$. We cannot have $q \in \EE^{(\mathrm{tor})} \cup
\EE^{(\mathrm{per})}$ since in both of these cases $M_q$ would have a
horizontal saddle connection of some length $H$, hence so would $q_1$,
and hence any surface in $\overline{Uq_1}$ would have a horizontal saddle
connection of length at most $H$. This would  contradict the fact that $Uq_1$
is dense in $\SF_{(\leq a)}$. So we must have $q \in
\EE^{(\mathrm{min})}$, and moreover $q$ has no horizontal saddle
connection.  Similarly, $\beta$ is not a multiple of the canonical foliation cocycle $(dy)_q$, because this would imply via \eqref{eq: horocycles are tremors} that $Uq_1\subset \EE$. In particular $M_q$ is not horizontally  uniquely ergodic.

Let $\nu_1$ and $\nu_2 = \iota_* \nu_1$ be the ergodic
transverse measures for the horizontal straightline flow on  $M_q$, normalized so that
$L_q(\beta_{i})=1, $ where $\beta_i \df \beta_{\nu_i}$ for $i=1,2$ and
write $\beta = a_1 \beta_1 + 
a_2\beta_2$ where $|a_1|+|a_2| \leq a$. We can assume with no loss of
generality that $a_2 \geq a_1$. Since $\beta$ is not a multiple of
$(dy)_q =\frac{1}{2} \nu_1 + \frac{1}{2} \nu_2$, we have $a_2 >
a_1$. Defining $s = 2a_1$ and using \equ{eq: commutation 0} we get
\begin{equation}
    \begin{split}
        \trem_\beta(q) &= \trem_{a_1 \beta_1 +a_2\beta_2}(q)\\
        &= \trem_{ a_1\left(2\hol^{(y)}_q- \beta_2 \right) +a_2\beta_2}(q)\\
        &= \trem_{(a_2 -a_1)\beta_2}(u_sq)
    \end{split}
\end{equation}  
and this shows that we may replace $q$ with $u_sq$ and  $\beta$ with
$(a_2-a_1)\beta_2$, which is an element of $C_{u_sq}^{+, 
  \mathrm{erg}}.$  So we assume that $\beta \in C_q^{+, \mathrm{erg}}$
and $L_q(\beta) \leq a$. Suppose $L_q(\beta) = a'< a$, then
writing $\rho = \frac{a}{a'}>1$ and letting 
$$q_1 = \trem_\beta(q) \text{ and }q_2
= \trem_{\rho \beta}(q) \in \SF_{(\leq a)} = \overline{Uq_1},
$$

Proposition \ref{prop: more semi direct} implies that
$$
\SF_{(\leq \rho a)} = \overline{Uq_2} \subset \overline{Uq_1} =
\SF_{(\leq a)} \subset \SF_{(\leq \rho a)},
$$ 
and thus $\SF_{(\leq \rho a)} = \SF_{(\leq a)}.$ This contradicts
Proposition \ref{prop: properly nested}, and hence $L_q(\beta)=a$. We have shown that there is $q_1\in \overline{\SF}$ with $\overline{Uq_1}=\overline{\SF}$, and moreover $q_1$ must be in $\SF^{(\min)}_{(=a)}$, proving the theorem. 
\end{proof}

We proceed with the proofs of Propositions \ref{prop: tor dense in
  min}, \ref{prop: min dense} and \ref{prop: min H dense in min}. As
we will see now, the main ingredient for proving  
Proposition \ref{prop: tor dense in min} is Proposition
\ref{prop: auxiliary}. 

\begin{proof}[Proof of Proposition \ref{prop: tor dense in min}]
By Proposition \ref{prop: continuity}, it is enough to show that for any $q$ in
$\EE^{(\mathrm{min})}$, any $\beta \in \tremspace_q$, and any 
$\vre'>0$, there is $q_1 \in \EE^{(\mathrm{tor})}$ and $\beta_1 \in
C^+_{q_1}$, such that $\dist(q, q_1) < \vre'$ and $\|\beta -
\beta_1\|<\vre'$.  Here $\| \cdot \|$ is some norm on $H^1(S, \Sigma; 
\R_x)$,  and we identify the cones $C^+_q$ and $C^+_{q_1}$ with subsets of
this vector space by choosing a marking and using period coordinates.
We would like to
use Proposition \ref{prop: auxiliary} (iii) 
and take $q_1 = r_{-\theta_j} q$, where $ r_{-\theta_j}$ is the
rotation of $M_q$ which makes the 
slit  $\sigma_j$ horizontal, and for $\beta_1$ take the cohomology class
corresponding to restriction of Lebesgue
measure to a torus on $M_{q_1}$ which is a connected component of the
complement of the horizontal slit; i.e. the rotation of $A_j$. It is clear that
for large $j$ this choice would fulfill all our requirements,
except perhaps the requirement that $q_1 \in
\EE^{(\mathrm{tor})}$. Namely it could be the case that the two translation
equivalent slit tori which
appear in Proposition \ref{prop: auxiliary} are periodic in direction
$\theta_j$. If this were to happen, we recall that $M_{q_1}$ is presented as two
tori glued along a horizontal slit, but the tori are horizontally
periodic, so a small perturbation of these tori (in the space of tori
$\HH(0)$) will make them horizontally aperiodic. Pulling back to
$\EE$, i.e. regluing the aperiodic tori along the same slit, we get a
new surface $q'_1$ which is not horizontally 
periodic and can be made arbitrarily close to $q_1$. The cohomology
class $\beta'_1$ corresponding to the restriction
of Lebesgue measure to one of the two 
aperiodic tori can be made arbitrarily close to $\beta_1$, completing
the proof. 
\end{proof}

Proposition \ref{prop: min dense} follows from a classical result of
Hedlund \cite{Hedlund} asserting that any horizontally aperiodic surface has a dense
$U$-orbit in the space of tori $\HH(0) \cong \SL_2(\R)/\SL_2(\Z)$. 

\begin{figure}[h]

\begin{tikzpicture}[scale=1.0]
\def\xa{5};
\def\ya{0};
\def\xb{1};
\def\yb{3};
\def\xc{-4};
\def\yc{1.0};
\def\xp{-2.0};
\def\yp{2.6};
\def\xq{-1};
\def\yq{1.5};

\node (A0) at (\xa,\ya) [circle,draw,fill=black,inner sep=0pt,minimum size=0.3mm] {};
\node (A1) at (\xa+\xb,\ya+\yb) [circle,draw,fill=black,inner sep=0pt,minimum size=0.3mm] {};
\node (A2) at (\xa+\xb+\xc,\ya+\yb+\yc) [circle,draw,fill=black,inner sep=0pt,minimum size=0.3mm] {};
\node (A3) at (\xa+\xc,\ya+\yc) [circle,draw,fill=black,inner sep=0pt,minimum size=0.3mm] {};

\node (p) at (\xp+\xa,\yp+\ya) [circle,draw,fill=black,inner sep=0pt,minimum size=1.2mm] {};

\path [name path=horizontal line 1]  (0,\yp+\ya) -- (8,\yp+\ya);
\path [name path=first line]  (A0) -- (A1);
\path  [name intersections={of=horizontal line 1 and first line, by=xx}]; 

\node (B0) [circle,draw,fill=black,inner sep=0pt,minimum size=0.3mm] at (xx) {};
\node (B1) [circle,draw,fill=black,inner sep=0pt,minimum size=0.3mm] at ($(xx)+(\xc,\yc)$) {};

\draw [black] (p) -- (B0);

\path [name path=horizontal line 2]  (0,\yp+\ya+\yc) -- (8,\yp+\ya+\yc);

\path [name path=second line]  (A1) -- (A2);
\path [name intersections={of=horizontal line 2 and second line, by=yy}]; 

\node (B2) [circle,draw,fill=black,inner sep=0pt,minimum size=0.3mm] at (yy) {};
\draw [black] (B1) -- (B2);

\node (B3) [circle,draw,fill=black,inner sep=0pt,minimum size=0.3mm] at ($(B2)-(\xb,\yb)$) {};

\path [name path=horizontal line 3]  (0,\yp+\ya+\yc-\yb) -- (8,\yp+\ya+\yc-\yb);

\path [name path=third line]  (A0) -- (A1);
\path [name intersections={of=horizontal line 3 and third line, by=zz}]; 

\node (B4) [circle,draw,fill=black,inner sep=0pt,minimum size=0.3mm] at (zz) {};
\draw [black] (B3) -- (B4);

\node (B5) [circle,draw,fill=black,inner sep=0pt,minimum size=0.3mm] at ($(B4)+(\xc,\yc)$) {};
\node (q) at (\xq+\xa,\yp+\ya+\yc-\yb+\yc) [circle,draw,fill=white,inner sep=0pt,minimum size=1.2mm] {};

\draw [black] (B5) -- (q);
\draw (A0) -- (A1) -- (A2) -- (A3) -- (A0);

\begin{scope}[xshift=220]

\def\xa{3};
\def\ya{0};
\def\xb{1};
\def\yb{3};
\def\xc{-4};
\def\yc{1.0};
\def\xp{-2.0};
\def\yp{2.6};
\def\xq{-1};
\def\yq{1.5};

\node (A0) at (\xa,\ya) [circle,draw,fill=black,inner sep=0pt,minimum size=0.3mm] {};
\node (A1) at (\xa+\xb,\ya+\yb) [circle,draw,fill=black,inner sep=0pt,minimum size=0.3mm] {};
\node (A2) at (\xa+\xb+\xc,\ya+\yb+\yc) [circle,draw,fill=black,inner sep=0pt,minimum size=0.3mm] {};
\node (A3) at (\xa+\xc,\ya+\yc) [circle,draw,fill=black,inner sep=0pt,minimum size=0.3mm] {};
\node (p) at (\xp+\xa,\yp+\ya) [circle,draw,fill=black,inner sep=0pt,minimum size=1.2mm] {};

\path [name path=horizontal line 1]  (0,\yp+\ya) -- (8,\yp+\ya);
\path [name path=first line]  (A0) -- (A1);

\path  [name intersections={of=horizontal line 1 and first line, by=xx}]; 
\node (B0) [circle,draw,fill=black,inner sep=0pt,minimum size=0.3mm] at (xx) {};
\node (B1) [circle,draw,fill=black,inner sep=0pt,minimum size=0.3mm] at ($(xx)+(\xc,\yc)$) {};

\draw [black] (p) -- (B0);

\path [name path=horizontal line 2]  (0,\yp+\ya+\yc) -- (8,\yp+\ya+\yc);

\path [name path=second line]  (A1) -- (A2);
\path [name intersections={of=horizontal line 2 and second line, by=yy}]; 
\node (B2) [circle,draw,fill=black,inner sep=0pt,minimum size=0.3mm] at (yy) {};
\draw [black] (B1) -- (B2);

\node (B3) [circle,draw,fill=black,inner sep=0pt,minimum size=0.3mm] at ($(B2)-(\xb,\yb)$) {};

\path [name path=horizontal line 3]  (0,\yp+\ya+\yc-\yb) -- (8,\yp+\ya+\yc-\yb);

\path [name path=third line]  (A0) -- (A1);
\path [name intersections={of=horizontal line 3 and third line, by=zz}]; 
\node (B4) [circle,draw,fill=black,inner sep=0pt,minimum size=0.3mm] at (zz) {};
\draw [black] (B3) -- (B4);

\node (B5) [circle,draw,fill=black,inner sep=0pt,minimum size=0.3mm] at ($(B4)+(\xc,\yc)$) {};
\node (q) at (\xq+\xa,\yp+\ya+\yc-\yb+\yc) [circle,draw,fill=white,inner sep=0pt,minimum size=1.2mm] {};

\draw [black] (B5) -- (q);

\draw (A0) -- (A1) -- (A2) -- (A3) -- (A0);
\end{scope}
\end{tikzpicture}\ \ \ \ \ \ \ \ \

\caption{A surface  $M_q\in\EE$ obtained by gluing two identical horizontally aperiodic tori  along a horizontal slit.}
\label{fig: gluing1} 
\end{figure}

\begin{figure}[h]

\begin{tikzpicture}[scale=1.0]
\def\xa{5};
\def\ya{0};
\def\xb{1};
\def\yb{3};
\def\xc{-4};
\def\yc{1.0};
\def\xp{-2.0};
\def\yp{2.6};
\def\xq{-1};
\def\yq{1.5};

\node (A0) at (\xa,\ya) [circle,draw,fill=black,inner sep=0pt,minimum size=0.3mm] {};
\node (A1) at (\xa+\xb,\ya+\yb) [circle,draw,fill=black,inner sep=0pt,minimum size=0.3mm] {};
\node (A2) at (\xa+\xb+\xc,\ya+\yb+\yc) [circle,draw,fill=black,inner sep=0pt,minimum size=0.3mm] {};
\node (A3) at (\xa+\xc,\ya+\yc) [circle,draw,fill=black,inner sep=0pt,minimum size=0.3mm] {};

\node (p) at (\xp+\xa,\yp+\ya) [circle,draw,fill=black,inner sep=0pt,minimum size=1.2mm] {};

\path [name path=horizontal line 1]  (0,\yp+\ya) -- (8,\yp+\ya);
\path [name path=first line]  (A0) -- (A1);

\path  [name intersections={of=horizontal line 1 and first line, by=xx}]; 
\node (B0) [circle,draw,fill=black,inner sep=0pt,minimum size=0.3mm] at (xx) {};
\node (B1) [circle,draw,fill=black,inner sep=0pt,minimum size=0.3mm] at ($(xx)+(\xc,\yc)$) {};

\draw [black] (p) -- (B0);

\path [name path=horizontal line 2]  (0,\yp+\ya+\yc) -- (8,\yp+\ya+\yc);

\path [name path=second line]  (A1) -- (A2);
\path [name intersections={of=horizontal line 2 and second line, by=yy}]; 
\node (B2) [circle,draw,fill=black,inner sep=0pt,minimum size=0.3mm] at (yy) {};
\draw [black] (B1) -- (B2);

\node (B3) [circle,draw,fill=black,inner sep=0pt,minimum size=0.3mm] at ($(B2)-(\xb,\yb)$) {};

\path [name path=horizontal line 3]  (0,\yp+\ya+\yc-\yb) -- (8,\yp+\ya+\yc-\yb);

\path [name path=third line]  (A0) -- (A1);
\path [name intersections={of=horizontal line 3 and third line, by=zz}]; 
\node (B4) [circle,draw,fill=black,inner sep=0pt,minimum size=0.3mm] at (zz) {};
\draw [black] (B3) -- (B4);

\node (B5) [circle,draw,fill=black,inner sep=0pt,minimum size=0.3mm] at ($(B4)+(\xc,\yc)$) {};

\node (q) at (\xq+\xa,\yp+\ya+\yc-\yb+\yc) [circle,draw,fill=white,inner sep=0pt,minimum size=1.2mm] {};

\draw [black] (B5) -- (q);

\draw (A0) -- (A1) -- (A2) -- (A3) -- (A0);

\begin{scope}[xshift=220]

\def\s{0.2};
\def\xa{2};
\def\ya{0};
\def\xbb{1};
\def\ybb{3};
\def\xcc{-4};
\def\ycc{1.0};
\def\xb{\xbb+0.5*\ybb};
\def\yb{3};
\def\xc{\xcc+0.5*\ycc};
\def\yc{1.0};
\def\xp{-1.0};
\def\yp{2.6};
\def\xq{-0.0};
\def\yq{1.5};

\node (A0) at (\xa,\ya) [circle,draw,fill=black,inner sep=0pt,minimum size=0.3mm] {};
\node (A1) at (\xa+\xb,\ya+\yb) [circle,draw,fill=black,inner sep=0pt,minimum size=0.3mm] {};
\node (A2) at (\xa+\xb+\xc,\ya+\yb+\yc) [circle,draw,fill=black,inner sep=0pt,minimum size=0.3mm] {};
\node (A3) at (\xa+\xc,\ya+\yc) [circle,draw,fill=black,inner sep=0pt,minimum size=0.3mm] {};

\node (p) at (\xp+\xa,\yp+\ya) [circle,draw,fill=black,inner sep=0pt,minimum size=1.2mm] {};

\path [name path=horizontal line 1]  (0,\yp+\ya) -- (8,\yp+\ya);
\path [name path=first line]  (A0) -- (A1);

\path  [name intersections={of=horizontal line 1 and first line, by=xx}]; 
\node (B0) [circle,draw,fill=black,inner sep=0pt,minimum size=0.3mm] at (xx) {};
\node (B1) [circle,draw,fill=black,inner sep=0pt,minimum size=0.3mm] at ($(xx)+(\xc,\yc)$) {};

\draw [black] (p) -- (B0);

\path [name path=horizontal line 2]  (0,\yp+\ya+\yc) -- (8,\yp+\ya+\yc);

\path [name path=second line]  (A1) -- (A2);
\path [name intersections={of=horizontal line 2 and second line, by=yy}]; 
\node (B2) [circle,draw,fill=black,inner sep=0pt,minimum size=0.3mm] at (yy) {};
\draw [black] (B1) -- (B2);

\node (B3) [circle,draw,fill=black,inner sep=0pt,minimum size=0.3mm] at ($(B2)-(\xb,\yb)$) {};

\path [name path=horizontal line 3]  (0,\yp+\ya+\yc-\yb) -- (8,\yp+\ya+\yc-\yb);

\path [name path=third line]  (A0) -- (A1);
\path [name intersections={of=horizontal line 3 and third line, by=zz}]; 
\node (B4) [circle,draw,fill=black,inner sep=0pt,minimum size=0.3mm] at (zz) {};
\draw [black] (B3) -- (B4);

\node (B5) [circle,draw,fill=black,inner sep=0pt,minimum size=0.3mm] at ($(B4)+(\xc,\yc)$) {};

\node (q) at (\xq+\xa,\yp+\ya+\yc-\yb+\yc) [circle,draw,fill=white,inner sep=0pt,minimum size=1.2mm] {};

\draw [black] (B5) -- (q);

\draw (A0) -- (A1) -- (A2) -- (A3) -- (A0);
\end{scope}
\end{tikzpicture}\ \ \ \ \ \ \ \ \

\caption{Applying a tremor in $C^{+, \mathrm{erg}}_q$ to  $M_q$ amounts to  applying a horocycle shear to one of the two tori. The resulting surface is not in $\EE$. Note that the length of the slit is unchanged.}
\label{fig: gluing2} 
\end{figure}

\begin{proof}[Proof of Proposition \ref{prop: min dense}] 
Note that each surface $q$ in $\EE^{(\mathrm{tor}, H)}$ has a
splitting into two translation equivalent tori $A_1$ and $A_2$ glued along a horizontal slit of
length $H$, and interchanged by the map $\iota$ of Proposition \ref{prop:
  structure of E}. The two rays in $C^{+, \mathrm{erg}}_q$ correspond, up
to multiplication by scalars, to
the restriction of the transverse measure  $(dy)_q$ to each of the two
tori. Thus if we set $s = 2a$, then each $q' \in
\SF_{(=a)}^{(\mathrm{tor}, H)}$ is obtained by 
a `subsurface shear' of a surface in $\EE^{(\mathrm{tor},H)}$, namely by 
applying $u_s$ to one of the tori $A_i$ and not changing the other
torus --- see Figures \ref{fig: gluing1} and \ref{fig: gluing2}. The reason for taking $s=2a$ is that the area of each of the
$A_i$ is exactly 1/2. 
This description
implies in particular that $\SF_{(=a)}^{(\mathrm{tor}, H)}$ is the
image of $\EE^{(\mathrm{tor}, H)} $ under  a continuous map commuting
with the $U$-action. So it suffices to show that the $U$-orbit of any  $q
\in \EE^{(\mathrm{tor}, H)} $ is dense in $ \EE^{(\mathrm{tor}, H)}
$. 

We do this by defining a $U$-equivariant inclusion of $\HH(0)^{(\mathrm{tor})}$, the set of tori that are horizontally aperiodic, into $\EE^{(\mathrm{tor},H)}$, and using the previously mentioned theorem of Hedlund.  Note  that any surface in $\EE^{(\mathrm{tor}, H)}
$ is obtained from a surface $M_{q'}$ for $q' \in \HH(0)^{(\mathrm{tor})}$ by
forming two copies of $M_{q'}$ and gluing them along a slit of length $H$
starting at the marked point (the fact that the surface is aperiodic
ensures that the slit exists).
This defines a 
$U$-equivariant map $\HH(0)^{(\mathrm{tor})} \to
\EE^{(\mathrm{tor}, H)}$, which is continuous when
$\HH(0)^{(\mathrm{tor})}$ is equipped with its topology as a subset of
$\HH(0)$. Thus to complete the proof it suffices to show that any
surface in $\HH(0)^{(\mathrm{tor})}$ has a $U$-orbit which is dense in
$\HH(0)$ --- which is Hedlund's theorem. 
\end{proof}

\subsection{Controlling tremors using checkerboards}
\name{subsec: checks}
\combarak{Rewrote these 3 lead-in paragraphs incorporating Jon's
  suggestions} In order to prove   
Proposition \ref{prop: min H dense in min} we will (among other
things) have to deal with the following situation. Given
$q \in \EE$ and $\beta \in C^{+, \mathrm{erg}}_q$,
with $L_q(\beta)<a$, we would like to find 
 a surface $M_{q'}$ and $\beta' \in C^{+,
  \mathrm{erg}}_{q'}$, such that $L_{q'}(\beta')=a$ and $\trem_{\beta}(q)$ is close to $\trem_{\beta'}(q')$. We find $q'$ close to the horocycle orbit of $q$. More specifically, we will choose $s$ so that $q_0 =
u_{-s} q$ and $\beta_0 = \beta + s$ $\hol_q^{(y)}$ satisfy $\trem_{\beta}(q) = \trem_{\beta_0}(q_0)$ and
$L_{q_0}(\beta_0) =a$, and take $q'$ close to $q_0$. This transforms
our problem into finding $\beta' \in C^
{+,\mathrm{erg}}_{q'}$ which closely approximates $\beta_0 \in
C^+_{q_0}$, where $\beta_0$ is not ergodic but rather is a nontrivial
convex combination of $\hol_{q_0}^{(y)}$ and an ergodic foliation
cocycle.

Controlling such convex combinations is achieved using what we will refer to informally as a
`checkerboard pattern'. A checkerboard on a 
torus $T$ is a pair of non-parallel
line segments $\sigma_1$ and $\sigma_2$ on $T$ which form the boundary of a
finite collection of polygons, which can be colored in two colors so that no two
adjacent polygons have the same color (see Figures \ref{fig:
  checkerboard1} and \ref{fig: checkerboard3}). If we
equip two identical tori $T_1, T_2$ with checkerboard patterns defined by
the same lines $\sigma_1, \sigma_2$, and in which the colors in the
coloring are swapped, we can 
form a surface $M$  in $\EE$ by gluing $T_1$ to $T_2$ in two different
ways, namely along each of the $\sigma_i$. Both of these
gluings give the same surface $M$, but it is decomposed as a union of
two tori glued along a slit in
two different ways (see Proposition \ref{prop: slit characterization}). One decomposition is into the original tori $T_1$
and $T_2$, and the other is into the unions $T'_1, T'_2$ of parallelograms of a fixed
color. Our interest will be in the `area imbalance' of the checkerboard, which is
the difference between the areas of $T_1 \cap T'_1$ and $T_2 \cap
T'_1$. Informally, the area imbalance tells us how close these two decompositions 
are to each other.

In our application the lines $\sigma_1$ and $\sigma_2$ will both be nearly
horizontal. Taking the normalized restriction $\Leb|_{T'_1}$ to one of the
tori in the decomposition $M = T'_1 \cup T'_2$ gives an ergodic
foliation cocycle for the flow in the direction of $\sigma_2$, and the checkerboard picture shows that it closely
approximates a nontrivial 
convex combination of the two ergodic components of the other
foliation cocycle, in the direction of $\sigma_1$, namely the one coming from the normalized
restrictions $\Leb|_{T_1},
\Leb|_{T_2}$. Controlling the coefficients in this convex combination
amounts to controlling the area imbalance parameter, and this will be achieved below in Lemma \ref{lem: sublemma}, item (IV).

Checkerboards were originally introduced by Masur and Smillie in order
to provide  
a geometric way to understand Veech's examples of surfaces with a
minimal and non-ergodic horizontal foliation, see \cite[p. 1039 \&
Fig. 7]{MT}. We now proceed to a more precise discussion.

\begin{figure}
\includegraphics[scale=0.90, trim = 0mm 0mm 0mm 0mm]{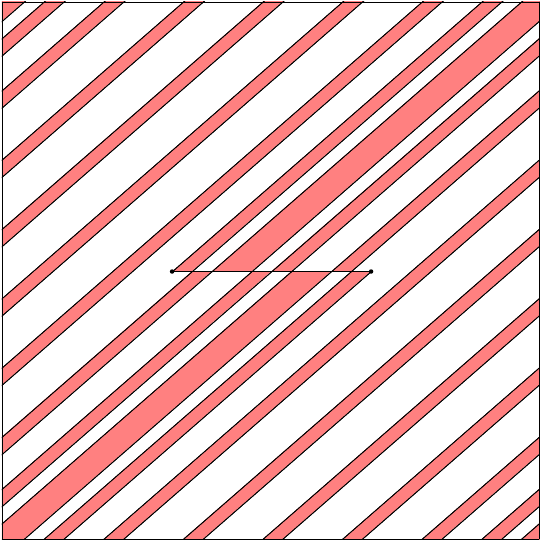}
\caption{A checkerboard: when the 
  $\sigma_i$ (drawn in black) are long and orthogonal, the torus will be
  partitioned into small rectangles of alternating colors. The
  difference between the areas occupied by the colors is the {\em
    area imbalance}. }\name{fig: checkerboard1} 
\end{figure}

\begin{figure}
\includegraphics[scale=0.90, trim = 0mm 0mm 0mm 0mm]{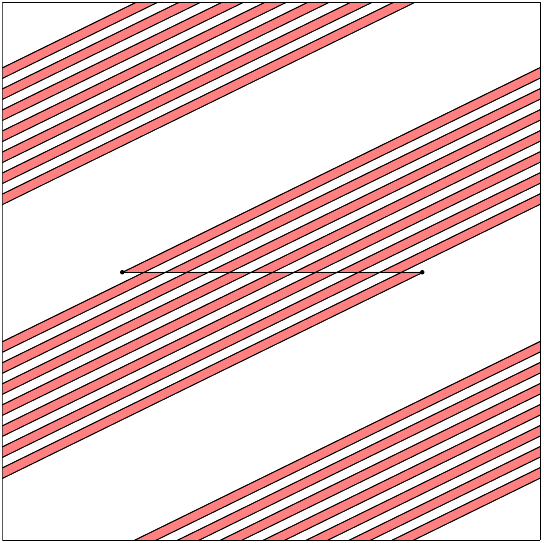}
\caption{A key feature of this checkerboard is that the non-horizontal
  black segment crosses the horizontal segment immediately adjacent to
  its previous crossing, leading to strips of equal width and length.}
\name{fig: checkerboard3} 

\end{figure}

Let $p \in
\HH(0,0)$ be a torus with two marked points $\xi_1$ and 
$\xi_2$. Let $T = T_p$ be the underlying surface. 
  Let $\sigma_1, 
\sigma_2$ be two non-parallel saddle connections on $p$ from $\xi_1$
to $\xi_2$. 
Let $\bar \sigma_2$ be the segment obtained by reversing the orientation on
$\sigma_2$, and let $\sigma$ be the concatenation of $\sigma_1$ and
$\bar \sigma_2$ so that $\sigma$ is a closed loop on $T$.
We have: 

\begin{lem}\name{lem: checkerboard}
The following are equivalent:
\begin{itemize}

\item[(i)]
 The loop $\sigma$ is homologous to zero in
  $H_1(T; \Z/2\Z)$.
\item[(ii)]
It is possible to color the connected components of $T \sm \sigma$
with two colors so that components which  
are adjacent along a segment forming part of $\sigma$ have different
colors. 
\item[(iii)]
For $i=1,2$ let 
 $M_i$  be the surface obtained from the slit construction applied to
 $\sigma_i$ (as in \S \ref{subsec: locus E}). Then $M_1$ and $M_2$ are
translation equivalent. 
\end{itemize}
\end{lem}

\begin{proof}
The equivalence of (i) and (iii) follows from Proposition \ref{prop: slit
  characterization}. 
We now show that (ii) is equivalent to the triviality of the class
represented by $\sigma$. 
Consider the $\Z/2\Z$ valued 1-cochain Poincar\'e dual  to
$\sigma$. This cochain represents a trivial cocycle if and only if it
is the coboundary of a 
$\Z/2\Z$-valued function. Associating colors to the values of such a
function as in Figure \ref{fig: checkerboard1} we have the
checkerboard picture. Specifically being a coboundary with $\Z/2\Z$
coefficients means that two regions have the same color iff a generic
path crosses $\sigma$  an even number of times to get from one to the
other.  
\end{proof}

Assume that $\sigma_1$ and $\sigma_2$ cross each other an odd number of times and
satisfy the conditions of Lemma \ref{lem: 
  checkerboard}, 
let $A$ be the area of $T$ and let $A_1, A_2$ be the areas of the
regions colored by the two colors in the coloring in (ii) above, so that $A_1+A_2 = A$. We
will refer to the quantity 
$\frac{|A_1-A_2|}{A}$
 as the {\em area imbalance}\index{area imbalance} of
the subdivision given by $\sigma_1, \sigma_2$ (note that when $T_p$ has
area one this is the same as $|A_1-A_2|$).

We will need the following two lemmas on tori.

\begin{lem}\label{lem: sublemma}
Suppose $T$ is a torus for which the horizontal direction is aperiodic. Given $c \in [0,1)$, a horizontal segment $\sigma_1$ on $T$, and $\eta>0$,  there is $H_0$ such that for any 
$H>H_0$, there is a second segment $\sigma_2$ on $T$ joining the two
endpoints of $\sigma_1$ for which the following hold:
\begin{itemize}
\item[(I)]
The segments  $\sigma_1, \sigma_2$ on $T$  intersect an odd
number of times and  satisfy the conditions of Lemma \ref{lem:
  checkerboard};
\item[(II)] 
Let  $\theta \in (-\pi,\pi)$ be the direction of $\sigma_2$. Then
$|\theta| < \eta$ and the flow in direction $\theta$  
is aperiodic on $T$;
\item[(III)]
the length of $\sigma_2$ is in the interval $\left (H,
(1+\eta)H \right)$;
\item[(IV)]
 the area imbalance of $\sigma_1, \sigma_2$ is in the interval
 $\left(c- \eta, c+\eta
 \right).$ 
\end{itemize}
\end{lem}

\begin{lem}\label{lem: torus equidistribution}
   Let $T$ be a horizontally minimal torus, and let $\sigma_1$ be a horizontal segment on $T$. Let $\sigma_2^{(k)}$ be a sequence of straight segments in $T$ in direction $\theta_k\neq 0$, connecting the endpoints of $\sigma_1$, so that the loop $\sigma$ above satisfies the conditions in Lemma \ref{lem: checkerboard}, and satisfying $\lim_{k \to \infty}\theta_k=0.$ Let $T^{(k)}$ be any one of the two monochromatic regions, in the checkerboard coloring described in Lemma \ref{lem: checkerboard}(ii). 
 Then 
    for any piecewise smooth bounded curve $\gamma \subset T$, which is transverse to the horizontal foliation,  we have 
\eq{eq: integrals appearing in}{
\lim_{k \to \infty} 
\frac{1}{\Leb(T^{(k)})}\int_\gamma dy|_{T^{(k)}} 
=\frac{1}{\Leb( T)}\int_\gamma dy. 
}

\end{lem}

We will give the proof of Lemmas \ref{lem: sublemma} and \ref{lem: torus equidistribution} at the end of this section. First we conclude the proof of 
Proposition \ref{prop: min H
  dense in min} assuming their validity.

  \begin{proof}[Proof of Proposition \ref{prop: min H dense in min}]
Let $q$ be as in the statement of Proposition \ref{prop: min H dense
  in min}, that is $q$ is obtained from $p 
\in \HH(0)$ with minimal horizontal foliation, and from parameters
$H_1>0$ and $s_1, s_2 \in \R$ satisfying \eq{eq: s1 plus s2}{|s_1|+|s_2| \leq 2a,} as
follows. First put a horizontal segment $\sigma_1$ of length $H_1$ on
the underlying torus 
$T = T_p$ giving rise to a surface in $\HH(0,0)$. Then apply the slit
construction described in \S \ref{subsec: locus E} to obtain a surface $M_{q_0}$ for $q_0 \in
\EE^{(\mathrm{tor})}$ which is a union of two tori $T_1$ and $T_2$ with
minimal horizontal foliations, glued along a horizontal slit of length
$H_1$. Rescale so that this surface has area one, i.e. each $T_i$ has
area $1/2$. Then for $i=1,2$, apply the horocycle shear map $u_{s_i}$
to $T_i$, and glue the resulting aperiodic tori to each other to
obtain $M_q$. In light of the factor 2 appearing in \eqref{eq: s1 plus s2}, $q = \trem_\beta(q_0)$ for $\beta \in \tremspace_{q_0}$ satisfying $|L|_{q_0}(\beta) \leq a,$ so $q \in \SF_{(\leq a)}^{(\mathrm{tor})},$ and all surfaces in $\SF_{(\leq a)}^{(\mathrm{tor})}$ can be described in this way.

By swapping the roles of $T_1$ and $T_2$, replacing $p$ with $u_{-s}p$, where $s=2a-(s_1+s_2)$, and replacing 
$s_i$ with $s_i+s$ for some $s \in \R$, 
we can assume that 
\eq{eq: formula 2a}{
 0 \leq s_1 \leq s_2 \text{ and } s_1 +s_2=2a.} 
Let 
\begin{equation}
  \label{eq: definition of c}  
c
\df \frac{s_2-s_1}{2a}.
\end{equation}

  

Let $M_{q_0} \in \EE^{(\mathrm{tor})}$ be the surface constructed as in the above
  discussion (starting with $T$ and $\sigma_1$ as in the paragraph above equation \eqref{eq: formula 2a}). This means that $q = \trem_{ s_1\beta_1 +
    s_2\beta_2}(q_0)$, where $\beta_i = \beta_{\nu_i}$ is the
  cohomology class corresponding to the transverse measure $\nu_i$
  obtained by
  restricting the canonical transverse measure $(dy)_{q_0}$ to the torus $T_i$,  and the tori are glued along a horizontal
  slit of length $H_1$. 
  Let $\mathcal{U}$ denote the (\dist) $\vre$-ball
  around $q$. Our goal is to show that $\mathcal{U}$ contains 
  some $q'$ which is also a tremor of a  
  surface $q'_0 \in \EE^{(\mathrm{tor})}$, but for which the
  parameters $s_1$ and $s_2$ and the slit length
  $H$ are prescribed. More precisely $M_{q'_0}$ is
  built from two minimal tori $T'$ and $T''$ glued along a horizontal slit of length $H$, $M_{q'}$
  is obtained by applying the horocycle
  flow $u_{2a}$ to $T'$ and leaving $T''$ fixed (since
  $T'$ has area $\frac{1}{2}$ this will give a tremor of
  total variation exactly $a$), and we need to carry the construction out for all
  $H>H_0$ where $H_0$ is allowed to depend on  $\mathcal{U}$. 

We obtain $q'_0$ as follows. 
Using Lemma \ref{lem: sublemma}, we find $\sigma_2$ satisfying conditions (I--IV), for $\eta$
sufficiently small (to be determined 
below). 
Define $q'_0 \df  gq_0$ where $g \in \SL_2(\R)$ is the (unique) composition of a
small rotation and small diagonal matrix, 
satisfying 
$$
g \, \hol_{T}( \sigma_2) = (H,0).
$$
By swapping $T'$ and $T''$ if needed, we will assume 
  \eq{eq: using the}{A_2 = \mathrm{Leb}(\psi_{g^{-1}} (T') \cap T_2) \geq A_1 = \mathrm{Leb}(\psi_{g^{-1}} (T') \cap T_1),}
  where $\psi_{g^{-1}}: M_{q'_0}\to M_{q_0}$ is the comparison map. 
Note that in light of
(II) and (III), $g$ is close to the identity in the sense that we can
bound  the norm $\|g - \mathrm{Id}\|$ with a bound which goes to zero as $\eta
\to 0$, so that by choosing $\eta$ small we can make $\dist(q_0, q'_0)$ as
small as we wish. 

Recall 
$q$ is obtained from $q_0$ by shearing the two tori $T_i$ (for
$i=1,2$) by $u_{s_i}$. Define $q'$ to be the surface obtained from $q'_0$ by shearing the
torus $T'$ by $u_{2a}$.
We now show using (II) and  (IV) that by
making $\eta$ small and $H$ large we can ensure that $q' \in
\mathcal{U}$. To see this, we will work in period coordinates, which by Proposition \ref{prop: sup norm properties} gives the same topology as $\dist$. 
We will choose a marking map $\varphi: S \to M_{q_0}$ and use it to
define an
explicit basis for $H_1(S, \Sigma)$, by pulling back a basis
of $H_1(M_{q_0}, \Sigma_{q_0})$. Then we will show that for all $\eta$
small enough and $H$ large enough, when
evaluating $\hol_q$ and $\hol_{q'}$ on the elements $\alpha$ of this basis, 
the differences $\|\hol_q(\alpha) - \hol_{q'}(\alpha)\|$ can be made
as small as we wish. The basis is described 
as follows. For $i=1,2$, let 
$\alpha^{(i)}_1, \alpha^{(i)}_2$ be
straight segments in $T_i$
generating the homology, so that $\left\{\alpha^{(i)}_j: i,j=1,2
\right \} \cup   \{\bar{\sigma}_1\}$ form  a basis
  for $H_1(M_{q_0}, \Sigma_{q_0}; \Z)$. 
We now compute the holonomy vectors of these elements, corresponding
to $q$ and $q'$.

By the description of $q$ from the preceding paragraph, and since $\alpha^{(i)}_j \subset T_i,$ we have 
\eq{eq: first for q}{
\hol_{q}\left(\alpha_j^{(i)} \right) = u_{s_i} \hol_{q_0}
\left(\alpha_j^{(i)} \right) = \hol_{q_0}\left(\alpha_j^{(i)} \right)
+ s_i \left( \begin{matrix} \hol^{(y)}_{q_0} \left(\alpha_j^{(i)}
   \right )\\ 0\end{matrix} \right)} 
and
\eq{eq: second for q}{\hol_{q}\left(\bar{\sigma}_1 \right) =
  \hol_{q_0}\left(\bar{\sigma}_1 \right).} 
Now let $\nu'$ be the transverse measure given by restricting the
canonical transverse measure $(dy)_{q'_0}$ to $T'$. 
Then by the
description of $q'$ from the preceding paragraph we also have
that 
\eq{eq: first for q'}{
\hol_{q'}\left(\alpha_j^{(i)} \right) = \hol_{q'_0}
\left(\alpha_j^{(i)} \right) + 2a \left(\begin{matrix} 
    \nu'\left(\alpha_j^{(i)} \right) \\ 0 \end{matrix} \right)
}
and 
\eq{eq: second for q'}{
\hol_{q'}\left(\bar{\sigma}_1\right) = \hol_{q'_0}
\left(\bar{\sigma}_1 \right) + 2a \left(\begin{matrix} 
    \nu'\left(\bar{\sigma}_1\right) \\ 0 \end{matrix} \right).
}
We want to show that by making $\eta$ small, we can make the difference between 
\equ{eq: first for q} and \equ{eq: first for q'},
as well as the difference
 between \equ{eq: second for q} and \equ{eq: second for q'}, as small as we like. 

Let $\mu'$ be the restriction of Lebesgue measure to $T'$ so that, in
the notation of Proposition \ref{prop:trans to meas}, we have $\mu'
= \mu_{\nu'}$, a (positive) measure with total variation $\frac 12 $. 
Using the definition of the area imbalance and \eqref{eq: using the}, 
we see that the area imbalance is $4A_2-1=1-4A_1.$
This implies 
$$\mu'(T_i) = A_i 
= \frac{1}{4} \left(1 + (-1)^i \cdot \text{area imbalance} \right) \ \ \ (i=1,2).$$  
Therefore, using equation \equ{eq: formula 2a}, the choice of $c$  in \eqref{eq: definition of c}, along with (IV), we have 
$$4a\mu'(T_i) \asymp a \left(1 +(-1)^i c \right) = a \left(\frac{s_1+s_2}{2a}+ (-1)^i \frac{s_2-s_1}{2a} \right) = s_i,$$
where by $A \asymp B$ we mean that $A$ can be made arbitrarily close to $B$ by choosing $\eta$ small enough. 
 By (II), choosing $\eta$ small forces $\theta$ to be close to 0, which is a uniquely ergodic direction on $T_i.$ 
We apply Lemma \ref{lem: torus equidistribution}, with $T = T_i, \gamma = \alpha_j^{(i)}$, and with $T^{(k)}$ any sequence of $T'$ as above corresponding to $\eta \to 0.$ We obtain that 
the second summands on the right hand sides of equations \equ{eq: first for q} and \equ{eq: first for q'} can be made arbitrarily close to each other by taking $\eta$ sufficiently small.

 Furthermore, since $\dist(q_0, q'_0) \asymp 0$, we have  
 $$\left\|\hol_{q'_0}\left(\alpha_j^{(i)}
\right) - \hol_{q_0} \left(\alpha_j^{(i)} \right ) \right\| \asymp \left\|\hol_{q'_0}\left(\bar{\sigma}_1
\right) - \hol_{q_0}\left(\bar{\sigma}_1\right) \right\|.$$
Thus for $\eta $ small enough we can make the difference between the quantities
\equ{eq: first for q} and \equ{eq: first for q'} as small as we like. 
We also have 
$$
\nu'(\bar{\sigma}_1 )\leq \int_{\bar{\sigma}_1} (dy)_{q'_0} = |\sin
(\theta)| \ell(\bar{\sigma}_1), 
$$
where $\ell (\bar{\sigma}_1)$ denotes the length of
$\bar{\sigma}_1$. Thus by (II) 
and \equ{eq: second for q'},
$$\|\hol_{q'_0}(\bar{\sigma}_1) - 
 \hol_{q'}(\bar{\sigma}_1) \| \asymp 0.$$ 
 Putting these estimates together we see that the difference
 between \equ{eq: second for q} and \equ{eq: second for q'} can also
 be made as small as we like. 
\end{proof}

\begin{proof}[Proof of Lemma \ref{lem: sublemma}] 
 Let $T_0$ be the standard torus $\R^2/\Z^2$, and let 
 $$\psi: T \to T_0$$ 
 be an affine homeomorphism.  Since the horizontal direction is aperiodic on $T$,
$\psi$ maps $\sigma_1$ to a segment on $\hat  \sigma_1 \df \psi(\sigma)$ on $T_0$  with holonomy
$(x, \alpha x)$ for some $\alpha \notin \Q$ and $x>0$. Let $\xi_1,
\xi_2$ be the endpoints of $\hat  \sigma_1$ in $T_0$. We
will choose $k$ an even positive integer, and a simple  closed curve $\ell$ from
$\xi_1$ to $\xi_1$, and let $\hat  \sigma_2$ be the shortest curve homotopic
to the concatenation of $k$ copies of $\ell$, followed by one copy of $\hat  \sigma_1$. Also we will denote $\sigma_2 = \psi^{-1}(\hat  \sigma_2).$ Since $k$
is even, the curve $\sigma$ of Lemma \ref{lem: checkerboard} is
homologous to an even multiple of $\psi^{-1}(\ell)$ and thus (I) holds. The
choice of the curve $\ell$ corresponds to the choice of $(m,n) \in \Z^2$ with
$\gcd(m,n)=1$. Since $\alpha$
is irrational, the linear form $(m,n) \mapsto m\alpha -n$ assumes a
dense set of values on pairs $(m,n) \in \Z^2$ with $\gcd(m,n)=1$
(see \cite{CE} for a stronger statement). We choose $m,n$ so that 
\eq{eq: the choice of m n}{\left|x(m\alpha -n) - (1-c) \right| <\eta.
}
We can make this choice with $m,n$ large enough, so that the direction
of $\ell$ approaches the direction of slope $\alpha$. Note that for
all $k$, the direction of $\hat \sigma_2$ is closer to the direction of
$\hat \sigma_1$ than the direction of $\ell$, and this means that the direction $\theta$ of $\sigma_2$ is nearly horizontal. Hence for such
$(m,n)$ and all large $k$,  $|\theta|$ is small. 
 Because $\alpha \notin \mathbb{Q}$ the slope of $\sigma_2$
is irrational and 
so we have (II). As we incrementally increase $k  \in 2\mathbb{N}$, the length of
$\hat \sigma_2$ increases by approximately twice the length of $\ell$. So for all large enough $H$, we can find  $k$ 
so that (III) holds. 

We now verify (IV), which requires describing the region and coloring
given by $\sigma_1$ and $\sigma_2$ as in Lemma \ref{lem:
  checkerboard}. (It may be helpful to consult Figure \ref{fig: checkerboard3}, which has 15 intersections between the curves, counting the initial and terminal points, 7 dark strips and 6 white strips.)
We will work in $T_0$ instead of $T$. The holonomy of $\hat \sigma_2$ is $k(m,n)+(x,x\alpha)$. The curves
$\hat \sigma_1$ and $\hat \sigma_2$ intersect in $k+1$ points (including
$\xi_1,\xi_2$) and these intersection points divide each $\hat \sigma_i$
into $k$ equal length pieces. Consecutive pieces of the division
of $\hat \sigma_2$ bound strips of the coloring given by Lemma \ref{lem:
  checkerboard}. 
 So we obtain a region $R$ composed of $k-1$ strips of alternating
 color where each strip is a flat parallelogram with sides $\frac 1 k
 (k(m,n)+(x,x\alpha))$ and 
 $\frac 1 k(x,x\alpha)$. As $k-1$ is odd, the areas of  all but one of these strips cancel out. 
 This gives that the contribution of $R$ to the area imbalance of
 $R$ is equal to the area $A$ of one strip. We have 
 $$A = \left|\det \left( \begin{matrix} m +\frac x k & \frac x k \\
       n+\frac {x\alpha}k & \frac{x\alpha}k \end{matrix} 
  \right) \right| = \frac{\left| m \alpha - n\right|}{k}.$$ The complement of $R$ has one color and area
$1-(k-1)A$. This implies that the total area imbalance is 
  $$1-(k-1) A  - A 
 = 1-kA = 1-x|m \alpha - n|
.$$
So (IV) follows from \equ{eq: the choice of m n}, and the proof is
complete. 
\end{proof}

\begin{proof}[Proof of Lemma \ref{lem: torus equidistribution}]
    Let $\mu_0 = \frac{1}{\Leb(T)} \Leb$ be normalized Lebesgue measure on $T$. Since we have assumed that $T$ is horizontally minimal, and minimal straightline flows on tori are uniquely ergodic, $\mu_0$ is the unique Borel probability measure on $T$ invariant under horizontal straightline flow. 
    
    For each $k$ define a measure $\mu_k = \frac{1}{\Leb(T^{(k)})}$ (the normalized restriction of Lebesgue measure to $T^{(k)}$). We claim that $\mu_k$ converges weak-$*$ to $\mu_0$, as $k \to \infty.$ Indeed, let $\Upsilon^{(x)}_k(t)$ denote the image of $x \in T$ under straightline flow in direction $\theta_k$ to time $t$. We can write $\mu_k$ as a convex combination of normalized length measures along segments $$\left\{\Upsilon^{(x)}_k(t) : t \in [0, S] \right\},$$
    for $x \in \sigma_1$ and with $S$ the first return time of $x$ to $\sigma_1$ along its orbit in direction $\theta_k$ (that is, segments passing parallel to the long sides in the   parallelograms of the checkerboard pattern). The length of these segments goes to infinity and their direction becomes more and more horizontal as $k \to \infty. $ By unique ergodicity, for any continuous test function $f$ on $T$, any $\vre>0$, and any sufficiently large $S$ (independent of $x$), 
    $$
    \left| \frac{1}{S} \int_0^S f \left(\Upsilon^{(x)}_0(t) \right) \, dt - \int f d\mu_0 \right| < \frac{\vre}{2},
    $$
    and by uniform continuity of $f$, for any fixed $S$ and all large enough $k,$
    $$
     \left| \frac{1}{S} \int_0^S f \left(\Upsilon^{(x)}_0(t) \right ) \, dt - \frac{1}{S} \int_0^S f \left(\Upsilon^{(x)}_k(t) \right) \, dt \right| < \frac{\vre}{2}.
    $$
    Putting these together we get  $\mu_k \to \mu_0$. 

    We can now recover the integrals appearing in equation \eqref{eq: integrals appearing in} from $\mu_0$ and $\mu_k$, as follows. Let $\bar \gamma \subset T$ denote the image of $\gamma$, and let $r>0$ be small enough so that for all large enough $k$, the maps
    $$
    \bar \gamma \times [0, r] \to T, \ \ \ (x, t) \mapsto \Upsilon_k^{(x)}(t)
    $$
    are injective, and their image does not intersect $\sigma_1.$
 For $k \geq 0$, let 
    $$
  A_{k} \df \bigcup_{x \in \bar \gamma} \left\{\Upsilon_k^{(x)}(t): t \in [0, r] \right \}.$$
Then by Fubini's formula for Lebesgue measure, we have 
    \eq{eq: then by Fubini}{ \frac{1}{\Leb(T)}
\int_\gamma dy = \frac{1}{r} \, \mu_0(A_{0}) \ \ \ \text{ and } \ \ \frac{1}{\Leb(T^{(k)})} \int_{\gamma } dy|_{T^{(k)}} = \frac{1}{r} \, \mu_{k}(A_k).     }
The Lebesgue measure of $\partial A_0$ is zero, and hence by weak-$*$ convergence, 
$$\lim_{k\to \infty} \mu_k(A_0) =\mu_0(A_0).$$ 
Also, the symmetric difference $A_0 \triangle A_k$ satisfies $\mu_k(A_0 \triangle A_k) \to_{k \to \infty} 0,$ as can be shown by an elementary argument which we leave to the reader. This shows that 
$$\lim_{k \to \infty} \mu_k(A_k) = \mu_0(A_0).$$
    Together with equation \eqref{eq: then by Fubini}, this implies equation \eqref{eq: integrals appearing in}. 
\end{proof}

\section{Non-integer Hausdorff dimension}\name{sec: Hausdorff dim}
The purpose of this section and the following one is to prove Theorem \ref{thm: Hausdorff
  dim}. Throughout this section we use the notation of \S \ref{sec:
  spiky fish}. We briefly explain the basic idea of the proof. We can think of a neighborhood of $\EE$ as being modelled on a
neighborhood of the zero section in the total space of the normal
bundle $\mathscr{N}(\EE)$ (see Corollary \ref{prop: KZ over
  E}). Thus we can think of $\SF_{(\leq a)}$ as a subset of 
the total space of $\mathscr{N}(\EE)$.  For all $q \in \EE$, the intersection of
$\left(\mathscr{N}(\EE)\right)_q$ with $\SF_{(\leq a)}$  
 is either a point or a line segment, contained in the two-dimensional space $\left(\mathscr{N}_x(\EE)\right)_q$.
By \cite{CHM} the set of $q\in \EE$, for which this set is not a point  has
 Hausdorff dimension 4.5. 
 
 Obtaining the lower bound is easier, and we use Proposition \ref{prop: H-dim facts} to say that the Hausdorff dimension is at least 4.5 +1. Obtaining the upper bound is more involved, occupying \S \ref{sec: upper} and \S\ref{subsec: sec 8 aux}. We denote the Hausdorff dimension of a subset $A$ of a
metric space $X$ by $\dim A$. We will use the following  well-known
facts about Hausdorff dimension (see e.g. \cite{Falconer, Mattila}):

\begin{prop}\name{prop: H-dim facts}
  Let $X$ and $X'$ be metric spaces. 
  \begin{enumerate}
  \item
    If $f: X \to X'$ is a Lipschitz map then $\dim X \geq  \dim
    f(X)$. In particular, Hausdorff dimension is invariant under
    bi-Lipschitz homeomorphisms. 
  \item
 For a countable collection $X_1, X_2, \ldots$ of subsets of $X$ we
 have $\dim \bigcup X_i = \sup_i \dim X_i.$
\item
  Let $A$ and $B$ be  subsets of Euclidean space and let $X \subset
  A\times B$ be such that for all $a \in A$, $\dim \{b \in B
  : (a,b) \in X\} \geq d$. Then \eq{eq: fiber lower bound}{\dim X \geq \dim A + d.}
In
  particular $\dim 
  (A \times B) \geq \dim A + \dim B.$

    \end{enumerate}
\end{prop}
Note that when stating Theorem \ref{thm: Hausdorff dim} we did not
specify a metric on $\HH(1,1)$. For concreteness one can take the
metric to be the metric 
$\dist$ defined in \S \ref{subsec: sup norm}, but note that in view of items (1) and (2) of Proposition
\ref{prop: H-dim facts}, the Hausdorff dimension of a set with respect to two different
metrics on $\HH(1,1)$ is equal, as long as they are mutually bi-Lipschitz on
compact sets. We will use this fact repeatedly.

 The next definition fixes an identification of an open set in the stratum with cohomology via period coordinates. This is helpful for working with the metric $\dist$. Formally, let $\mathcal{U} \subset \HH$ be an open set  and $\pi \colon \HHm \to
\HH$ be the forgetful map of \S \ref{subsec: strata}. In this section,
we say that 
$\mathcal{U}$ is  an {\em
 adapted neighborhood }\index{adapted neighborhood} if it is precompact, and there is a
triangulation of $S$ such that 
  a connected component of $\pi^{-1}(\mathcal{U})$ is contained in  $V_\tau$, 
 where $V_\tau$ is described in \S \ref{subsec: atlas of
   charts}. Additionally we
 will say that a relatively open $\mathcal{U} \subset \EE$ is an {\em adapted
 neighborhood (in $\EE$)} if it is the intersection of an 
  adapted neighborhood   in $\HH(1,1)$, with the locus
 $\EE$. 

\subsection{Proof of lower bound}
We use the notation introduced in \S \ref{sec: spiky fish}, and begin
with the proof of the easier 
half of the theorem.
\begin{proof}[Proof of lower bound in Theorem \ref{thm: Hausdorff
      dim}]
  For each $\delta>0$, we will define a 
  subset $X_0 
  \subset \SF_{(\leq
    a)}$, subsets $X_1 \subset \EE, \, X_2 \subset \R,$ and a surjective Lipschitz map $f: X_0 \to
  X_1
  \times X_2,$ 
  where $\dim X_1 \geq 4.5 - \delta$ and $\dim X_2=1$. The statement will then
  follow via Proposition \ref{prop: H-dim facts}.

Let $\mathcal{U} \subset \HH(1,1)$ be an adapted neighborhood,
so that 
we can identify $\mathcal{U}$ with 
 an open subset of $H^1(S, \Sigma; \R^2). $ 
Fix a  norm $\| \cdot \|$
on $H^1(S, \Sigma; \R_x)$ which is invariant under translation
equivalence arising from the orbifold group of $\EE$, as in Proposition \ref{orbifold properties}. According to Corollaries \ref{cor: pass to balanced} and 
  \ref{cor: how to reach from locus}, for any $q' \in
  \SF^{(\mathrm{min})}_{(\leq a)}$ there is a unique $q = q(q') \in
  \EE^{(\mathrm{min})}$ and a unique $\beta = \beta(q') 
  \in \tremspace_q^{(0)} $ (up to translation equivalence) such that $q' =
  \trem_{q,\beta}$. Define \begin{equation}\label{eq: def bar f}
  \bar f: \SF^{(\mathrm{min})}_{(\leq a)}
  \to \EE^{(\mathrm{min})}  \times \R_{\geq 0} \text { by }
  \bar{f}(q') \df  \big(q(q'),
  \|\beta(q')\|\big).
  \end{equation}
  Note that because translation equivalences preserve $\|\cdot \|$
  this is well-defined. By Corollaries \ref{cor:
    balanced tremors are in B-}  and \ref{prop: KZ over E}
  we have that $\beta(q') \in \mathscr{N}_x(\EE)$ for all $q'$, where
$\mathscr{N}_x(\EE)$ 
 is a flat subbundle. 
 
  We claim that by making $\mathcal{U}$ 
  small enough, $\bar
  f$ restricted to $\mathcal{U}$ is a Lipschitz 
  map (where we use the metric arising from dist on the domain and first summand of the range of $\bar f$). 
  Indeed, by the continuity of the map in \eqref{eq: for lipschitz}, and the fact that $\mathcal{U}$ is precompact, the metric dist is bi-Lipschitz to the metric $\mathrm{dist}'(q_1, q_2) \df \|\hol(\til q_1 )- \hol(\til q_2) \|$ 
  arising from period coordinates and the chosen norm $\| \cdot \|$ (when $\til q_i \in \pi^{-1}(q_i)$ belong to some fixed connected lift of $\mathcal{U}$).  Furthermore, if $\mathcal{U}$ is small enough, then for the projections introduced in \S \ref{subsec: orbifold}, we have from Corollary \ref{prop: KZ over E} that 
  $$\hol(q(q')) = P^+(\hol(q')) \ \ \text{ and } \ \ \hol(\beta(q')) = P^-(\hol(q'));$$ 
  that is, in period coordinates on $\mathcal{U}$, $\bar f$ is obtained by writing a vector in $H^1(S, \Sigma; \R^2)$ in terms of its coordinates with respect  to the two factors in a direct sum  decomposition, composed with taking the norm on the second coordinate. This is clearly a Lipschitz map with respect to $\dist'$. 
  
Fix $\eta>0$ and set 
\[
\begin{split} 
 X^{(\eta)}_1
  \df & \left\{q \in \EE^{(\mathrm{min})}: \text{ there is } \beta \in
    \tremspace_q^{(0)} 
  \text{ with } |L|_q(\beta)\leq a \text{ and } \|\beta\| = \eta \right\}, \\
X_0 \df & \left\{q' \in \SF^{(\mathrm{min})}_{(\leq a)} : q(q') \in
X^{(\eta)}_1, \ 
  \|\beta(q')\|\leq \eta \right\}, \\
 X_2 \df & [0,\eta],
\end{split}
\]
and define 
$$f:X_0 \to X_1^{(\eta)} \times X_2, \ \  f\df \bar f|_{X_0}.$$
Then $f$ is
 Lipschitz on the intersection of $X_0$ with any compact set, and the
 definitions ensure 
 that $f$ is surjective. So it remains to show that for $\eta>0$ small
 enough we have 
 \eq{eq: what we need hdim}{
   \dim X_1^{(\eta)} \geq
   4.5 -\delta.}
Let 
$$X_1 = \left \{q \in
\EE^{(\mathrm{min})} : \text{horizontal flow on } M_q
\text{ is not uniquely ergodic} 
\right \}.$$ 
Since $X_1  = \bigcup_{\eta>0} X_1^{(\eta)}$, by Proposition
\ref{prop: H-dim facts} (2)
it suffices to show that $\dim X_1 \geq 4.5.$
 This is deduced from work of Cheung, Hubert and Masur as
 follows. By the general theory of local cross-sections (see
 e.g. \cite{MSY}), the action of the group $\{ r_\theta : \theta \in 
 \mathbb{S}^1\}$ on $\EE$ admits a cross-section, that is, we can
 parameterize  a small neighborhood  
  in $\EE$ by $(q, \theta) \mapsto r_\theta q$, where $q$ ranges over
  a 4-dimensional smooth manifold $\mathcal{V}$,  $\theta$ ranges over
  an open set in $\mathbb{S}^1$, and the parameterizing map is
  Bi-Lipschitz. Thus these coordinates identify a neighborhood in 
  $\EE$ with a Cartesian product $\mathcal{V}\times I$ where $I$ is
  an interval in $
  \mathbb{S}^1$. It is shown in \cite{CHM} that $\mathcal{V}$ contains a Borel
  subset $A$ of full measure, such that for each $q \in
  A$ there is a subset $\Theta_q \subset \mathbb{S}^1$ so
  that for $q \in A$,  $\theta \in \Theta_q$ we have
  $r_\theta q \in X_1$, and $\dim \Theta_q = 0.5.$ Proposition
  \ref{prop: H-dim facts}, item (1) and formula \eqref{eq: fiber lower bound}
   now imply \equ{eq: what
    we need hdim}. 
\end{proof}
\begin{remark} 
We remark without proof that the map $\bar{f}$ introduced in \eqref{eq: def bar f} would not be Lipschitz if we defined the second coordinate to be $|L|_q(\beta)$. Indeed, if we were to define $\bar f$ in this way and extend it to tremors of surfaces in $\mathcal{SF}^{(\mathrm{tor})}_{(\leq a)}$, then Proposition \ref{prop: min H dense in min} would show that $\bar f$ is not even continuous. 
Also, it is likely that $\bar{f}$ is not bi-Lipschitz, and this is part of the challenge in proving the upper bound. 
  \end{remark}

\subsection{Proof of upper bound}\name{sec: upper}
We now begin the proof of the upper bound, starting with a brief guide to its proof. 
In order to cover 
  $\SF_{(\leq a)}$ efficiently, we will view a subset of this set as lying in a product space, namely a local trivialization of the bundle $\mathscr{N}(\EE)$ as in the proof of the lower bound.  To efficiently cover 
 $\SF_{(\leq a)}$ in this product space we find convex sets $J_i\subset \EE$ so that the fixed-size tremors of points in $J_i$ vary in a controlled way. 
 
    Proposition \ref{prop: upper bounds Hdim} gives an upper bound for the Hausdorff dimension that fits this strategy. The remainder of this section is devoted to proving the upper bound assuming this result. 
    We prove the Proposition in  \S \ref{subsec: sec 8 aux}.

\subsubsection{Preparations for the upper bound: general result for efficient covers}
We begin with our general result for exploiting efficient covers 
of convex sets. Let $Y \subset \R^d$ and let $|Y|$
    denote the Lebesgue measure  of 
     $Y$. 
     Let $\mathcal{N}^{(\vre)}(Y)$\index{n@$\mathcal{N}^{(\vre)}(Y)$} denote the
    $\vre$-neighborhood of $Y$, that is $\mathcal{N}^{(\vre)}(Y)
    = \bigcup_{y \in Y} B(y, \vre)$.
    \index{N@$\mathcal{N}^{(\vre)}(Y)$} The {\em inradius}\index{inradius} of 
 $Y \subset \R^d$ is defined to be the supremum of $r\ge0$ such that $Y$
 contains a ball of radius $r.$  
\begin{prop}\name{prop: upper bounds Hdim}
Let $P_1 \subset \R^d$, $P_2 \subset \R^2$ be balls. Let 
$Z \subset P_1 \times P_2$, and $\{Z(t) : t \in \N\}$ be a collection of subsets
of $P_1\times P_2$, such that for any $T>0$, $Z \subset
\bigcup_{t=T}^\infty Z(t).$ Assume furthermore that there are positive constants
$c_1$, $c_2$, and $\delta<1$  and that for each $t \in \N$, 
$Z(t)$ is a
finite disjoint union of sets $X_i(t) \times
Y_i(t),$ with $X_i(t) \subset P_1, \, Y_i(t) \subset P_2$, for which
the following hold: 
\begin{itemize}
\item[(i)] 
Each $X_i(t)$ is contained in a convex set $J_i(t) $ such that the
$J_i(t)$ are pairwise disjoint, and each has inradius at least
$c_1e^{-2t}$. 
\item[(ii)]
Each $Y_i(t)$ is a rectangle whose shorter side has length at most 
$c_2e^{-2t}$.  
\item[(iii)]
$ \left|\bigcup_i \mathcal{N}^{(e^{-2t})} (X_i(t)) \right| \leq c_2 e^{-\delta t}$.
\end{itemize}
Then 
\eq{eq: hausdorff upper bound}{\dim Z \leq d+1 -\frac{\delta}{5}.}

\end{prop}
To obtain an upper bound on the Hausdorff
dimension of $\SF_{(\leq a)}$, 
we will verify the assumptions of Proposition \ref{prop: upper bounds 
  Hdim}, with $d=5$. In our setup, a small  adapted neighborhood 
$\mathcal{U} \subset \EE$ 
will play the role of a neighborhood 
in $\R^5$, and the 2-dimensional 
subspace $\mathscr{N}_x(\EE)$ will play the role of $\R^2$.

\subsubsection{Preparations for upper bound: transverse systems}
In order to verify hypotheses (i) and (ii) of Proposition \ref{prop:
  upper bounds Hdim} we need to choose 
convex sets in $\EE$ so that the $\mathscr{N}_x(\EE)$ fibers intersected with
$\SF_{(\leq a)}$ vary in a controlled way. To do this,  
we now get good approximations for the cone of foliation cocycles
which will be constant on our convex subsets of $\EE$. Our strategy will be to define convex regions, on which the horizontal flow is combinatorially similar up to some fixed time. Arguments like this are standard when using Rauzy-Veech induction. In our setup it will be more convenient to use  transverse systems, which we now introduce. The advantage of transverse systems is that they have a more transparent interaction with the geodesic flow $\{g_t\}$. See  \cite[\S2]{mahler} for a related construction. 

\combarak{Changes in the definition of transverse systems and the
  related structures.}
Let $\til q \in \HHm$ and let $M_q$ be the underlying translation
surface. A {\em transverse system}\index{transverse system} on $M_q$ is 
   a finite collection of disjoint arcs of finite length
which are transverse to the horizontal foliation on $M_q$, do not
contain points of $\Sigma$, and intersect every horizontal leaf (see \cite[Fig. 2.1]{mahler}).  The
arcs may contain points of $\Sigma$ in their closure. For
example, if the horizontal foliation on $M_q$ is minimal then
$\sigma$ could be any short vertical arc not passing through
singularities, and if $M_q$ is aperiodic and $\vre$ is an arbitrary
positive number, $\sigma$ could be the union
of vertical arcs of length $\vre$ intersecting the horizontal saddle connections, along with downward pointing vertical prongs of length 
$\vre$ starting at all singular points (and where the singular points
at their extremities are not considered a part of the prong).

We now define some structures 
associated with a transverse system. We mark one point on each
connected component of $\sigma$. A {\em $\sigma$-almost horizontal
  segment}\index{s@$\sigma$-almost horizontal segment} is a continuous
oriented path $\ell$ from $\sigma$ to $\sigma$, which 
starts and ends at marked points,  is a concatenation of an edge along
$\sigma$, a piece of a horizontal leaf in $M_q \sm \Sigma_q$ which does not intersect
$\sigma$ in its interior, and another edge along $\sigma$. The
orientation of a $\sigma$-almost horizontal segment is the one given
by rightward motion along horizontal leaves. Two $\sigma$-almost
horizontal segments are said to be {\em isotopy equivalent} if they
are homotopic with fixed endpoints, and where the homotopy is through
$\sigma$-almost horizontal segments. Up to isotopy equivalence there are only
finitely many $\sigma$-almost horizontal segments. A {\em $\sigma$-almost
  horizontal loop}\index{s@$\sigma$-almost horizontal segment} is a
continuous oriented loop which is a concatenation of 
$\sigma$-almost horizontal segments, where the orientation of the loop is
consistent with the orientation of each of the segments. We say
that a $\sigma$-almost horizontal loop is {\em 
  reduced}\index{reduced $\sigma$-almost horizontal loop} if it
intersects each connected component of 
$\sigma$ at most once. With each $\sigma$-almost horizontal loop
$\gamma$ we associate a
  cohomology class $\beta_\gamma \in H^1(M_q, \Sigma_q; \R)$ via
  Poincar\'e duality.

  We will need the following: 
  \begin{lem}\name{lem: basis up to finite index}
Let $M_q$ be a surface with no horizontal saddle connections. Then for any transverse system $\sigma,$ the cohomology classes corresponding to all
$\sigma$-almost horizontal loops generate
$H^1(M_q, \Sigma;\Z)$. 
    \end{lem}

    \begin{proof}
      \combarak{It took me a long time to figure this out but in
        hindsight it is trivial. Probably my proof is too long or
        there is a reference.}
The union of
$\sigma$-almost horizontal segments in one isotopy equivalence class
is the union of sub-arcs of 
$\sigma$ and a topological disc foliated by parallel horizontal
segments. The union of these topological discs gives a presentation of
$M_q \sm \Sigma$ as a 
cell complex. We call it the {\em cell complex associated with
  $\sigma$}\index{cell complex associated with a transverse
  system} (see \cite[\S 2.4]{mahler}). This generalizes the well-known Veech zippered rectangles 
construction \cite{Veech rectangles}; namely the zippered rectangle  construction arises when 
 $\sigma$ has one
connected component which intersects all the horizontal saddle connections of $M_q$, and the two endpoints of $\sigma$ are mapped by the horizontal straightline flow to singular points in forward time. In the zippered rectangle case, a proof of the Lemma is given in \cite[\S 4.5]{yoccoz survey}.

Since we have assumed that $M_q$ is horizontally minimal, any open subinterval $\sigma' \subset \sigma$ can serve as a transverse system. We choose $\sigma' \subset \sigma$
so that it satisfies the conditions mentioned above, namely, the cell complex associated with $\sigma'$ is a zippered rectangle construction. Since the $\sigma'$-almost horizontal loops are a subset of the $\sigma$-almost horizontal loops, the statement for $\sigma$ follows from the statement for $\sigma'$. 
      \end{proof}
  
  Given a marking map $S \to M_q$ we can
  think of each $\beta_\gamma$ as an element of $H^1(S, \Sigma ; \R)$. We
  denote by $C^+_q(\sigma)$ the convex cone over all of the $\beta_\gamma$,
  that is
  $$
C^+_q(\sigma) = \mathrm{conv} \left( \left\{t \beta_\gamma : \gamma
    \text{ is a }
    \sigma\text{-almost horizontal loop on } M_q \text{ and } t>0 \right\} \right).
$$


Note that $C^+_q(\sigma)$ is a finitely generated cone. Indeed, if we
let $\mathscr{L} = \mathscr{L}_{q, \sigma}$ denote 
the reduced $\sigma$-almost horizontal loops, then $C_q^+(\sigma)$ is the
convex cone generated by $\beta_{\gamma}, \ \gamma \in
\mathscr{L}$. Since $\beta_\gamma$ only depends on the homotopy class
of $\gamma$, and there are only finitely many isotopy classes of
$\sigma$-almost horizontal segments, this shows the finite generation
of $C^+_q(\sigma)$.

 Let $C_q^+$ be the cone of foliation cocycles as in \S \ref{subsec: transverse}. Clearly, if $\sigma \subset \sigma'$ are transverse
systems then $C_q^+(\sigma) 
\subset C_q^+(\sigma')$.  We have the following
standard fact. \combarak{omitted all mention of Schwartzman asymptotic
  cycles, as this only caused confusion.}

\begin{prop}\name{prop: Schwartzman cycles}
Suppose $M_q$ has no horizontal saddle connections and let $\sigma_1
\supset \sigma_2 \supset \cdots $ be a nested sequence of 
transverse systems for the horizontal foliation on $M_q$, with total length
going to zero. Then 
\eq{eq: Schwartzman intersection}{
C^+_q \subset \bigcap_{n=1}^{\infty} C^+_q(\sigma_n).}
\end{prop}

\begin{remark}
    In fact we have equality in \eqref{eq: Schwartzman intersection}. In this paper we only need the inclusion stated above. The reverse inclusion can be proved along the lines of \cite[Proof of Thm. 1.1]{mahler}; for similar results
    in the context of interval exchange transformations  and measured foliations see 
    \cite[Lemma 1.5]{VIET} and \cite[Theorem 5.1.1]{Mosher} respectively. 
\end{remark}
\begin{proof}[Proof of Proposition \ref{prop: Schwartzman cycles}]
  \combarak{Changes in this proof.}
We need to show that $C^+_q \subset  C^+_q(\sigma_n)$ for every  $n$. We use the
Birkhoff ergodic theorem. Take an ergodic invariant probability
measure $\mu$ for the straightline flow on $M_q$, let $\nu$ be a transverse measure corresponding to $\mu$ as in Proposition \ref{prop:trans to meas}, and let $\beta_\nu$ be the corresponding foliation cocycle. Since $M_q$ has no horizontal saddle connections, $\nu$ is non-atomic, the horizontal straighline flow on $M_q$ is minimal, and $C_q^+$ is the convex cone generated by the foliation cocycles $\beta_\nu$ arising in this way. Take a horizontal 
leaf $\ell$ which lies on a generic horizontal straightline trajectory
for $\mu$. This implies that $\ell $ intersects any transverse system infinitely many times. 
Genericity means that for a transverse arc $\gamma$, $\nu(\gamma) = \lim_{S \to \infty} \frac{1}{S} \# (\gamma \cap \ell_S),$ where $\ell_S$ is a piece of the leaf starting at some fixed point on $\ell$ and of length $S$ (and the limit exists). Let $\sigma'_n$ be a connected component of $\sigma_n$
which intersects $\ell$ infinitely many times.  Then we can find a
sequence of intersections of $\ell$ and 
$\sigma'_n$ such that the horizontal lengths of subsegments of $\ell$
between consecutive intersections grow longer and longer. Closing
up these segments along $\sigma'_n$ gives longer and longer
$\sigma_n$-almost horizontal loops, and taking the Poincar\'e dual of
a renormalized sum of a large number of them gives a sequence 
approaching $\nu$ (as can be seen by evaluating these sums on closed loops $\gamma$).  This implies $\beta_\nu \in  C^+_q(\sigma_n)$. 
\end{proof}

 \combarak{Extensive edits from  here until the end of the proof of 
   the two Lemmas. }

 \subsubsection{Transverse systems in $\EE$ and $\HH(1,1)$} \label{subsubsec: transverse systems}
We now specialize to $\HH(1,1)$ and specify the collection
of transverse systems 
$\{\sigma_n\}$ explicitly. Recall our convention that singularities for a
surface in $\HH(1,1)$ are labeled. Each $q \in \HH(1,1)$ has two vertical prongs
issuing from the first singular point in a downward direction, and we
denote by $\bar \sigma_t$ the union of the corresponding vertical
segments of length $e^{-t}$. On any compact subset of $\HH(1,1)$ there
is a lower bound on the length of a shortest saddle connection, and so
for $t$ large enough the vertical prongs do not hit singular points
and so $\bar \sigma_t$ is well-defined. If $M_q$ is horizontally minimal then each
horizontal leaf intersects $\bar \sigma_t$ and in particular each
horizontal separatrix starting at a singularity has a first
intersection (as seen along the separatrix) with $\bar \sigma_t$. Denote by $\vre = \vre(q,t)$ the
maximal length, along $\bar \sigma_t$, of a segment starting at a
singularity and ending at the first intersection of some horizontal
separatrix $\xi$ with $\bar \sigma_t$. Let $\hat \sigma_t \subset \bar
\sigma_t$ be the union of 
the two vertical prongs taken of length $\vre$.
Note that $\hat
\sigma_t$ is a transverse system on $M_q$ if $M_q$ is horizontally
minimal, but some non-minimal surfaces 
have  horizontal leaves that miss $\hat \sigma_t$.

Fix an adapted neighborhood $\mathcal{U}$, and recall that by choosing
a connected component of $\pi^{-1}(\mathcal{U})$, we can equip all
$q \in \mathcal{U}$ with a marking map (up to equivalence), and this
identifies each $C^+_q$ with a cone in $H^1(S, \Sigma; \R_x)$. For
those $q \in \mathcal{U}$ for which $M_q$ has no horizontal saddle
connections, the
marking map also determines the cone 
$C^+_q(\hat \sigma_t)$ as a cone in $H^1(S, \Sigma; \R_x)$. We denote it by $\til
C_q^+(t)$\index{C@$\til C^+_q(t)$} in order to lighten the notation.  Since
$\hat \sigma_t$ is invariant under the map $\iota$, this
identification does not depend on the choice of the marking map
(within its equivalence class). As in Corollary \ref{prop: KZ over E} 
 let $H^1(S, \Sigma; \R^2) = T(\EE)\oplus \mathscr{N}(\EE)$ be the 
decomposition into $\iota$ invariant and anti-invariant
classes.
By Corollary \ref{cor: balanced tremors are in B-}, a balanced signed
foliation cocycle belongs to $\mathscr{N}_x(\EE)$.
As in the proof of Proposition \ref{prop: auxiliary}, let
$\bar{\pi}: \EE \to \HH(0)$ be the projection which maps a surface
$q \in \EE$ to the torus $M_q /\langle \iota \rangle$, and forgets the
 marked point (one of the two endpoints of the slit) corresponding to
 the second singular point of $M_q$. 

 The area-one condition in the definition of $\EE$ means that $\EE$ is
 not a linear space. For our proof we will need to cover $\EE$ by
 convex subsets, and in order to make the notion of convexity
 meaningful we work locally, as follows. 
 Recall that $\mathcal{U} \subset \EE$ is an adapted neighborhood (in
$\EE$) if it is
the intersection of $\EE$ with an adapted neighborhood in the
stratum. 
 In this case there is a triangulation $\tau$ of $S$ such that any connected component of 
 $\pi^{-1}(\mathcal{U})$ is contained in the intersection of the set $V_\tau$ (as in \S
 \ref{subsec: atlas of charts}) with the 
fixed point set of the involution described in Proposition \ref{prop:
  structure of E}, and with the locus of area-one surfaces.
Let $q \in \mathcal{U}$ and fix a marking map of
$\varphi: S \to q$ representing a surface $\til q \in V_\tau$. Let $\Phi = \Phi_q$ be the
map which sends $x \in
T_{q}(\EE)$ to the surface $q'$ satisfying 
$\hol(\til q') = c(\hol(\til q) +
x)$, where $\til q'$ is given by the marking map determined by
$\varphi$ and $\tau$ (see \S \ref{subsec: atlas of charts}) and the rescaling
factor $c$ is chosen so that the surface $q'$ 
has area one. A {\em convex 
   adapted neighborhood}\index{convex adapted neighborhood} of $q$ is
$\Phi(\mathcal{W})$ where $\mathcal{W}$ is an open convex subset of 
$T_{q}(\EE)$ so that $\Phi|_{\mathcal{W}}$ is 
a homeomorphism onto its image, which is contained in $\mathcal{U}$. When 
discussing diameters, convex sets, etc., we will do this with respect
to the linear structure on $\mathcal{W}$. We say that a
collection $\mathcal{J}$ of convex subsets of a convex 
adapted neighborhood is a {\em weak convex partition}\index{weak convex partition} if the interiors $\{ J^\circ: J \in \mathcal{J}\}$ are disjoint, and the union of closures $\bigcup_{J \in \mathcal{J}} \bar J$ covers all
horizontally minimal surfaces in $\mathcal{U}$.


It is clear from definitions that for $t \in \R,$
\eq{eq: clear from definitions that}{
\til C^+_{g_{-t}q}(0) = g_{-t}\left(\til C^+_q(t) \right).}

Let 
$\til \EE_{\mathrm{m}} = \pi^{-1}(\EE)$, and let $\til \EE_{\mathrm{m},t}$ denote the surfaces in $\til \EE_{\mathrm{m}}$ which have no vertical saddle connections of length at most $e^{-t}$, and for which every horizontal straightline leaf intersects $\hat \sigma_t$. Note that for these surfaces, the cone $\til C^+_q(t)$ is well-defined, that the set of horizontally minimal marked surfaces in $\til \EE_{\mathrm{m}}$ is contained in $\bigcup_{t>0} \til{ \EE}_{\mathrm{m},t}$, and that a collection of horizontally minimal marked surfaces belonging to a compact subset of $\til{\mathcal{E}}_{\mathrm{m}}$ is contained in $\til{\EE}_{\mathrm{m}, t}$ for all $t$ small enough. For each $t$ we define a partition $\mathcal{J}_t$ of  $\til \EE_{\mathrm{m},t}$ 
into  $t$-equivalence classes, with the property that $t$-equivalent surfaces $\til q_1, \til q_2$ have $\hat \sigma_t$-almost horizontal segments  which are homotopic and have the same intersection pattern with $\hat \sigma_t$.

Let $\xi_1(q), 
\ldots, \xi_k(q)$ be the paths made by concatenating a horizontal and vertical segment on $M_q$ as follows. The $\xi_i(q)$ begin from $\Sigma$ and move along  horizontal separatrices until the first intersection with $\hat \sigma_t$,    and are  continued vertically along $\hat \sigma_t$ so that they end at points of $\Sigma$. In the situation at hand, of surfaces in $\HH(1,1),$ we have $k=8$ since there are four horizontal prongs issuing from each of the two singularities. By choice of the orientations, we have 
\eq{eq: by choice of orientation}{\hol^{(y)}_{ q} (\xi_j( q))>0 \ \ \text{ for each } j.} 
Since $\hat \sigma_t$ and the collection of $\xi_j(\til q)$ is invariant under the involution $\iota$, there are two indices $j$ realizing the maximum in the definition of $\vre(q,t)$, and we permute indices so that $\xi_2 = \iota(\xi_1)$ and 
\eq{eq: achieves maximum}{\hol^{(y)}_{ q} (\xi_1(q)) = \hol^{(y)}_{ q} (\xi_2(q)) = \max_j \hol^{(y)}_{q} (\xi_j(q)) \leq e^{-t}.}
We add two more segments $\xi_9, \xi_{10}$ which are horizontal continuations of  $\xi_1, \xi_2$, starting from the endpoints of $\xi_1, \xi_2$ on  $\hat \sigma_t$ and end at the next intersection point with $\hat \sigma_t$, and we switch the orientation of $\xi_9, \xi_{10}$ so that \eqref{eq: by choice of orientation} continues to hold.

We choose an equivalence class of marking maps $\til q \in \pi^{-1}(q)$. By $\iota$-invariance we can think of the $\xi_j(\til q)$
as representing paths on the topological marking surface $(S, \Sigma)$. 
We say that $\til q_1$ and $\til q_2$ are {\em $t$-equivalent} \index{t@$t$-equivalent} if, possibly after permuting the indices $j$, for $i=1,2$ the paths $\xi_j(\til q_i)$ represent the same homotopy classes when pulled back to $S$, \eqref{eq: by choice of orientation} and \eqref{eq: achieves maximum} continue to hold, and the order of intersections of the $\xi_j$ with each connected component of $\hat \sigma_t$ is the same. 

Recall the cell complex associated with $\hat \sigma_t$, discussed in the proof of Lemma \ref{lem: basis up to finite index}. This complex gives a polygon decomposition of $M_q$ into rectangles, with vertical and horizontal sides being subsegments of $\hat \sigma_t$ and concatenations of some of the $\xi_j$. From this it is easy to see that any 
$\hat \sigma_t$-almost horizontal segment on $M_{q}$ is homotopic to a concatenation of some of the $\xi_j$. This implies that if $\til q_1, \til q_2$ are $t$-equivalent then there is a bijection between the homotopy classes represented by their $\hat \sigma_t$-almost horizontal segments, which preserves the order in which they intersect the transverse system.

Note that the definition of $t$-equivalence only involved the intersection pattern of certain horizontal and vertical lines on the surface. From this, and the rescaling properties of the geodesic flow, we obtain the equivariance property
 \eq{eq: naturality partition}{
\til q \in J \in \mathcal{J}_t \ \iff \ g_{-t} \til q \in g_{-t}( J) \in \mathcal{J}_0. 
  }

\begin{lem}\name{lem: Schwartzman joined1}
  Let $\mathcal{U} \subset \EE$ be a convex adapted
  neighborhood, and let $\mathcal{V} \subset \EE_{\mathrm{m}}$ be a connected component of $\pi^{-1}(\mathcal{U})$. Then for all large enough $t$, 
the partition 
\begin{equation}\label{eq: the partition of U}
\left\{\pi \left(\mathcal{V}\cap \bar J \right) : J \in \mathcal{J}_t \right\}
\end{equation}
is a weak convex partition of $\mathcal{U}$. For $J \in \mathcal{J}_t$, surfaces in  the boundary $\bar J \sm (\bar J)^\circ$ have horizontal saddle connections, and are either horizontally non-minimal, or horizontally uniquely ergodic.

\end{lem}

Lemma \ref{lem: Schwartzman joined1} gives some geometrical control over  
the elements of the partition $\mathcal{J}_t$; and in light of Proposition \ref{prop: Schwartzman cycles}, the same partition can also be used in order  to control the direction of foliation cocycles.

\begin{proof}
Since $\mathcal{U}$ is precompact, there is a lower bound on the length of a vertical saddle connection of surfaces in $\mathcal{U}$, so for all large enough $t$, $\mathcal{U} \cap \pi(\til \EE_{\mathrm{m},t})$ contains the set of horizontally minimal surfaces in $\mathcal{U}$. Since the sets $J \in \mathcal{J}_t$ give a partition of $\EE_{\mathrm{m},t}$, in order to show that the sets in \eqref{eq: the partition of U} form a weak convex partition of $\mathcal{U},$ we only need to show that 
 each of the sets in \eqref{eq: the partition of U} is convex, and that the interiors of these sets are disjoint.
 
 By construction of $\mathcal{V}$ and $\mathcal{U}$, the map $\pi|_{\mathcal{V}} : \mathcal{V} \to \mathcal{U}$ is injective, modulo the local group, and for each $q \in \mathcal{U}$ we denote by $\til q$ its pre-image in $\mathcal{V}.$ 
Then $\til q'$ belongs to the $t$-equivalence class $J$ of $\til q$ if the following hold: 
\begin{itemize} 
\item all horizontal leaves on the underlying surface $M_{q'}$ intersect the transverse system $\hat \sigma_t$;
\item formulas \eqref{eq: by choice of orientation}, \eqref{eq: achieves maximum} hold for $q'$ (possibly up to permutation);  
\item
%
 for all
$i,j$,
\begin{equation}\label{eq: possibly c}
\hol_{\til q}^{(y)} (\xi_i) > \hol_{\til q}^{(y)} (\xi_j) \ \ \iff \ \
\hol_{\til q'}^{(y)} (\xi_i) > \hol_{\til q'}^{(y)} (\xi_j).
\end{equation}
\end{itemize}
The first of these conditions holds if the horizontal foliation on $M_{q'}$ is minimal, which holds for a dense set of surfaces.  
Conditions \eqref{eq: by choice of orientation} and \eqref{eq: possibly c} involve inequalities between holonomies and thus give convex conditions in period coordinates. Therefore the set 
$(\bar J)^\circ$ is precisely the set of
surfaces satisfying the inequalities in \eqref{eq: by choice of orientation} and \eqref{eq: possibly c}. 
This implies that the sets 
$\{\bar J: J \in \mathcal{J}_t \}$ are convex, and their interiors 
$\{(\bar J)^\circ: J \in \mathcal{J}_t \}$
are disjoint. 

For the last assertion, let $q \in \bar J \sm (\bar J)^\circ$. Then on $M_q$ there are two $\xi_i, \xi_j$ with the same vertical holonomy; their concatenation gives a horizontal saddle connection. Applying the translation automorphism $\iota$ we get at least two horizontal saddle connections on $M_q,$ and now results about surfaces in eigenform loci, summarized in \cite[Thm. 7.13]{eigenform}, show that there are three possibilities for the horizontal foliation: $M_q$ could have a horizontal cylinder decomposition, could be made of two horizontally minimal tori glued along a slit, or could be horizontally uniquely ergodic. 
%
  \end{proof}

We note that the first assertion in Lemma \ref{lem: Schwartzman joined1} remains true, with a very similar proof, if $\EE$ is replaced by any 
$G$-invariant locus, and $\hat\sigma_t$ is replaced with any transverse
system satisfying the equivariance property \eqref{eq: naturality partition}. We now use the additional 
structure of $\EE$ in order to state and prove  bounds on the objects associated
with a transverse system. 

\begin{lem}\name{lem: Schwartzman joined2}
  Let $\mathcal{U} \subset \EE$ be a convex adapted neighborhood, let
  $\mathcal{J}_t$ be the partitions as in Lemma
  \ref{lem: Schwartzman joined1}, let $K_1
\subset \HH(0)$ be compact, and let $a>0$. 
If 
$q \in \mathcal{U} \cap \EE^{(\mathrm{min})}$ is horizontally minimal
then there are positive constants $c_1$ 
and $c_2$ (depending on $q$) such that if 
$t>0$ satisfies $g_{-t} \bar{\pi}(q) 
\in K_1$ (where $\bar \pi : \EE \to \HH(0)$ is the projection defined at the beginning of \S \ref{subsubsec: transverse systems}), then the following hold: 
\begin{enumerate}
\item[(a)] The length of each $\hat \sigma_t$-almost horizontal loop is at least
$c_1 e^{t}$, and the inradius of $J$ is at least $c_1 e^{-2t}$, where
$J \in \mathcal{J}_t$ is the partition element containing $q$.
\end{enumerate}
Suppose furthermore that $q$ is not horizontally uniquely ergodic, let
$P^-$ be the projection onto the orthocomplement of involution invariant classes as in \S \ref{subsec: orbifold}, and let $\til C^+_q (t)  = C_q^+(\hat \sigma_t)$ as above. Then
\begin{enumerate}
  \item[(b)]
\eq{eq: essentially 1-dim}{
P^-\left( \left\{\beta \in \til C^+_q (t): L_q(\beta) \leq a \right\}\right)
}
is contained in a rectangle with diameter in the interval $[c_1,
c_2]$. 
\item[(c)] The rectangle in (b) can be chosen so that one of its sides 
has length bounded above by $c_2 e^{-2t}$. 
\end{enumerate}

\end{lem}

\begin{proof}
In order to obtain the bounds in (a), note that the existence of a short $\sigma$-almost horizontal segment implies the existence of a short saddle connection. Note also that 
  the transverse system $\hat \sigma_t$ is the 
  preimage under $\bar{\pi}$ of a transverse system $\sigma_0$ on the torus
  $\bar{\pi}(M_q)$. Using the affine comparison map $\psi_{g_{-t}}$ corresponding to
  $g_{-t}$ as in \S\ref{subsec: G}, we
  can consider the image of this transverse 
  system on ${g}_{-t}\bar{\pi}(q)$. If ${g}_{-t}\bar{\pi}(q) \in K_1$ there
  exists $c'_1$ depending only on $K_1$ so that any almost-horizontal
  loop, with respect to a transverse system of bounded length, has
  length at least $c'_1$. Considering the effect of 
  the map $\psi_{g_{-t}}^{-1}$, we obtain the required lower bound
  on the length of a $\hat \sigma_t$-almost horizontal segment on
  $M_{q}$. Now take some lower bound $c''_1$ for the inradius of an
  element $J$ in the partition $\mathcal{J}_0$, satisfying $\bar \pi \circ \pi(J) \cap K_1 \neq \varnothing$.  Such a lower bound exists because $K_1$ is
  compact and the collection $\mathcal{J}_0$ is locally finite. By
  \equ{eq: naturality partition}, we can pull
  back to $\mathcal{J}_t$ using $g_t$ 
  and use \equ{eq:
    geodesic expansion rate} 
    to obtain the lower bound of $c''_1e^{-2t}$
  on the inradius of elements of $\mathcal{J}_t$. Taking $c_1 =
  \min(c'_1, c''_1)$ we obtain (a). 
  
We now prove assertion (b). Note that here $c_1, c_2$ are allowed to depend on $q$. 
The continuity of $L_q$, the fact that $\til C^+_q(t)$ is a finitely generated convex cone, and the fact that $L_q(\beta)>0$ when $\beta$ is a $\sigma$-almost horizontal loop, imply that the set $\left\{\beta \in \til C^+_q(t):L_q(\beta)\leq a\right \}$ is compact. Hence so is the set appearing in \eqref{eq: essentially 1-dim}. Boundedness follows from the properness of the metric dist. 
  Since $q$ admits an
essential tremor, there is $\beta_0 \in 
\til C^+_q$ for which $P^-(\beta_0) \neq 0$ and this implies the lower
bound in (b).

The proof of (c) combines the upper bound in (b), the effect of renormalization by the flow $g_t$, and the fact that the action of $g_t$ preserves the Lebesgue measure on $\mathscr{N}_x(\EE)$, the real part of the normal bundle. 
In the proof of (c) we will write $A \ll B$ if
    $A$ and $B$ are two quantities depending on several parameters
    and $A \leq
    CB$ for some constant $C$ (the implicit constant) independent of
    these parameters. 
    In this proof the implicit
    constant is allowed
    to depend on $q$ but not on $t$.

It follows from Proposition \ref{prop: tangent and normal} and Corollaries 
\ref{prop: KZ over E} and \ref{cor: balanced tremors are in
  B-}, that the projection
$P^-$ is defined over $\Q$. This implies that $P^1$ maps the lattice of $\Z$-points
$H^1(S; \Z_x)$ to a sublattice $\Lambda$ in
$\mathscr{N}_x(\EE,\Z) \df \mathscr{N}_x(\EE) 
\cap H^1(S,\Sigma; \Z)$.

 Let $M_t $ be the underlying surface
  of $g_{-t} q$ and
  denote by $\psi_t : M_q \to M_t$ the affine comparison map  
  defined in \S \ref{subsec: G}. 
Let $\mathscr{L}(q)$ and $\mathscr{L}(g_{-t}q)$ denote respectively the set of
reduced 
$\hat \sigma_t$- (resp., $\psi_t(\hat \sigma_t)$-) almost horizontal
loops on $q$ (resp., on $g_{-t}q$). 
By Lemma \ref{lem: basis up to finite index}, for $\mathscr{L}$ equal
to either of $\mathscr{L}(q)$ and $\mathscr{L}(g_{-t}q)$, we have
that $\{\beta_\gamma : \gamma
\in \mathscr{L}\}$ contains a basis of $H^1(S ; \Z)$, and hence
the projection $P^-\left(\left\{\beta_\gamma: \gamma \in \mathscr{L}
  \right \}
  \right)$ generates $\Lambda$. Let $\Psi_t$ be the map $q \mapsto g_{-t}q.$
  By choosing a marking map $\varphi: S\to M_q$ and using $\psi_t
  \circ \varphi$ as a marking map for $M_t$, this induces a map
$\bar\Psi_t: H^1(S, \Sigma; \R^2) \to  H^1(S, \Sigma; \R^2)$. Since
the map $\iota$ of Proposition \ref{prop: structure of E} commutes with the map
$\psi_t$, the map $P^-$ commutes 
with $\Psi_t$, and hence we have the following
diagram:
\[
  \begin{tikzcd}
    H^1(S, \Sigma; \R_x)  \cong T_q \mathcal{U}  \arrow{r}{\bar{\Psi}_t}
    \arrow{d}{P^-} & H^1(S, \Sigma; \R_x) \cong T_{g_{-t}q} \HH  \arrow{d}{P^-} \\
   \mathscr{N}_x(\EE)
     \arrow{r}{\bar{\Psi}_t|_{\mathscr{N}_x(\EE)}} & \mathscr{N}_x(\EE)
  \end{tikzcd}
\]
The preceding discussion shows that 
$\bar{\Psi}_t(\Lambda) = \Lambda$, and therefore
\eq{eq: volume unchanged}{
  \left| \det \left(\bar{\Psi}_t|_{\mathscr{N}_x(\EE)} \right)\right | =1.}



Similarly to \eqref{eq: clear from definitions that}, we have an equivariance relation 
$$\mathscr{L}(q) = \bar{\Psi}_t^{-1}(\mathscr{L}(g_{-t}q )).$$
 Also, as in
Proposition \ref{prop: commutation relations},  we
have that for $\beta \in \tremspace_q$, if we set $\beta' \df
\bar{\Psi}_t(\beta)$ then $L_{g_{-t}q}(\beta') = e^{-t}L_q(\beta)$. 
This gives
 \[ \begin{split}
& P^- \left( \left\{ \beta\in \til C_q^+(t) : L_q(\beta) \leq a\right\}
\right) \\=&
\bar{\Psi}_t^{-1} \circ P^- \circ \bar{\Psi}_t\left( \left\{ \beta\in
    \til C_q^+(t) : L_q(\beta) \leq a\right\} \right)  
\\ =&\bar{\Psi}_t^{-1} \circ P^- \left( \left\{ \beta'\in
    \til C_{g_{-t}q}^+(0) : L_{g_{-t}q}(\beta') \leq e^{-t}
    a\right\} \right)   \\ \subset &
\bar{\Psi}_t^{-1}\left(\left\{\beta'' \in \mathscr{N}_x(\EE) :
    \|\beta'' \| \ll e^{-t}a \right\} \right),
\end{split}
\]
where the bound in the last inclusion follows from Proposition 
\ref{prop: lipschitz tremors} and Lemma \ref{lem: automatic absolutely continuous}, and the fact that $\bar \pi (g_{-t}q )\in K_1$ and on compact sets, the metric dist is bi-Lipschitz to any norm in period coordinates.
Thus, using \equ{eq: volume unchanged}, the set in the left hand side
of \equ{eq: essentially 1-dim} is a convex subset of
$\mathscr{N}_x(\EE)$ of 
area $\ll e^{-2t}$. On the other hand, by (b), it contains a vector of
length $\gg 1$. This means that it is contained in a rectangle whose
small sidelength is $\ll e^{-2t}$, as claimed. 
  \end{proof}

\subsubsection{Preparations for proving the upper bound: Nondivergence estimates}
Masur's criterion states that if the vertical foliation on a surface $M_q$ is not uniquely ergodic 
then $g_t q \to \infty$ as $t \to \infty$. In this paper we are dealing with horizontal foliations so
we have that if the horizontal foliation on $M_q$
    is not uniquely ergodic then $g_{-t} q \to
    \infty$ as $t \to \infty$; i.e., the backward
    trajectory eventually leaves every compact  
    set. The following 
    result gives (for a fixed surface) an upper bound for the measure
    of directions in which 
    the orbit has escaped a large compact set by a fixed time.

\begin{prop}[Athreya]
\name{prop: jayadev thesis}
For any stratum $\HH$ there is 
$\delta >0 $, and a compact subset $K \subset \HH$ such that for any
compact set $Q \subset \HH$ and any  $T_0>0$ 
 there is
$C>0$ so that  
for all $q \in Q$ and all $T>0$, we have 
$$
\left|\left\{\theta \in \mathbb{S}^1 : \forall t \in [T_0, T_0+T], \,
    g_{-t}
    r_\theta q \notin K \right\} \right | \leq C e^{-\delta T}.
$$
\end{prop}
The formulation given above is stronger than the statement of
\cite[Thm. 2.2]{Jayadev thesis}. Namely, in 
\cite{Jayadev thesis}, the constant $C$ is allowed to depend on $q$,
while we claim that $C$ can be chosen uniformly over the compact 
set $Q$. One can check that the stronger 
Proposition \ref{prop: jayadev thesis} follows from the proof given in
\cite{Jayadev thesis}. Alternatively, one can  derive it from 
\cite[Prop. 3.7]{AAEKMU}.
Indeed, in the notation of \cite{AAEKMU}, set
$\delta=\frac 2 3 $,  $a< 2^{-\frac{5}{2}}\, C_1^{-\frac{3}{2}}$ and
$C=a^{-2T_0}C(x)$, and note that  $C(x)$ is uniform when $x$ ranges over a compact set, and
for $N>\frac{2\, T_0}{t}$
we have    
$$Z\left(X_{\leq M},N,1,\frac 2 3 \right)\supset \{q: \alpha(g_tq) \leq M
\text{ for all }T_0\leq t\leq N\}.$$

\begin{remark}\label{rem: divergent better} Proposition \ref{prop:
    jayadev thesis}
  is convenient for our covering arguments  
because if we take a compact set $K'$ whose interior contains $K$, and slightly larger, and if $g_{-t} q \notin K'$ for all $t\in [T_0,T+T_0]$ ,then for
$q'$ in a small neighborhood of $q$ we have 
 $g_{-t} q' \notin K$ for all $t\in [T_0,T+T_0]$.  
 Applying Proposition \ref{prop: jayadev thesis} to $K'$ 
 we have exponential decay (in $T$) of the measure of a 
 neighborhood of the set we are covering. 
\end{remark} 

\subsubsection{Proof of upper bound}
We now prove the upper bound, assuming 
Proposition \ref{prop: upper bounds Hdim}, which will be proved in the next section.

  \begin{proof}[Proof of the upper bound in Theorem \ref{thm: Hausdorff
      dim}]
      We divide the argument into steps. 
      
      {\textbf{Step 1: Reduction to $\SF^{(\min)}_{(\leq a)}\cap \mathrm{NUE}$.}}

For each $H_0>0$, the set $\overline{\bigcup_{H \leq H_0}\EE^{(\mathrm{tor}, H)}}$ is a proper submanifold of $\EE$ with boundary (in the closure we pick up surfaces made of identical periodic tori glued along an embedded slit). On 
$\bigcup_{H \leq H_0}\EE^{(\mathrm{tor}, H)}$,
if $\beta \in \tremspace_q^{(0)}$
satisfies $|L|_q(\beta)=s$, the map $(q, s)
\mapsto \trem_\beta(q)$ performs symmetric horocycle  shears in opposite directions corresponding to $u_{s}$ and $ u_{-s}$ on the two tori which are connected components of the complement of the slit. Therefore this map  is locally Lipschitz for the metric coming from any norm in period coordinates, and as in the proof of the lower bound (see the discussion of the map $\bar f$), this means it is locally Lipschitz for dist. Thus by Proposition 
\ref{prop: H-dim facts}, taking the union over all $H_0 \in \N,$ the subset of $\SF_{(\leq a)}$ consisting of tremors of surfaces in
$\EE^{(\mathrm{tor})} \cup \EE^{(\mathrm{per})}$ has Hausdorff dimension at most 5.  So we need only bound
the Hausdorff dimension of the set of surfaces $\trem_\beta(q)$ where $q$
is horizontally minimal and non-uniquely ergodic, i.e., bound the
dimension of the essential tremors in $\SF_{(\leq a)}$. Note that by Lemma \ref{lem: Schwartzman joined1}, the collections of such surfaces is covered by the sets $\{ (\bar J)^\circ: J \in \mathcal{J}_t\}$ for all sufficiently large $t.$

{\textbf{Step 2: A countable cover.}}
In light of Proposition \ref{prop: H-dim facts}(2), it is enough to cover
$\SF_{(\leq a)}$ by countably many subsets, and give a uniform upper bound on the Hausdorff dimension of each. The countable collection we will use, which is denoted below by $Z$, exhausts the set of essential tremors $\SF^{(\min)}_{(\leq a)}\cap \mathrm{NUE}$, and depends on several parameters: the adapted subset in $\EE$ containing the surface $q$ for which $\trem_\beta(q) \in SF^{(\min)}_{(\leq a)}$,  the return time under $g_{-t}$ to a certain compact set $K'$, and constants coming from Lemma \ref{lem: Schwartzman joined2}. 

To make this precise, define 
$$\EE' \df \{q \in \EE: M_q \text{ admits an essential tremor}
\},$$
and write $\HH$ for $\HH(1,1)$. Let $\delta>0$ and $K \subset \HH$
be a compact set as in Proposition \ref{prop: jayadev
  thesis}. We assume with no loss of generality that $\delta
<1$.
Let $\dist$ be the metric of \S \ref{subsec: sup norm} and let  
$$
K' \df \{ q \in \HH(1,1): \dist(q, K) \leq 1\}.
$$
By Proposition \ref{prop: sup norm properties}, $K'$ is
compact. 

We can cover
$\EE'$ with countably many convex adapted neighborhoods with compact
closures. 
Given such a convex adapted neighborhood $\mathcal{U} \subset 
\EE$, and given a parameter $T_0>0$, let $C =
C(\mathcal{U}, T_0)$ be as in Proposition \ref{prop: jayadev thesis}
with $Q \df \overline{\mathcal{U}}$. 
If $q \in \mathcal{U} \cap \EE'$ and $\beta \in 
\tremspace_q^{(0)}$, there are $c_1 = c_1(q), \, c_2 = c_2(q)$ so 
the conclusions of Lemma \ref{lem:
  Schwartzman joined2} are satisfied. 
 Masur's criterion \cite{MT} applied to the horizontal foliation of $M_q$ implies that 
the trajectory $\{g_{-t}q : t>0\} $ is divergent, and in 
particular, 
there is $
T_1(q)$  such that for all $t \geq T_1(q)$, 
$g_{-t}q\notin K'$. For each $\mathcal{U}$ in the above countable
collection, each $T_0 \in \N$, and each $c \in \N$ with $c \geq
C(\mathcal{U}, T_0)\,
e^{\delta  T_0},$ let $Z = Z(\mathcal{U}, T_0, c)$ denote the set of
tremors $\trem_\beta(q)$ where $q \in  \mathcal{U}\cap \EE'$ and $\beta 
\in \tremspace^{(0)}_q$ satisfy  the bounds 
$$
|L|_q(\beta) \leq a, \ \ \ T_1(q) \leq T_0, \ \ \ 
c_2(q) \leq c, \ \ \ c_1(q) \geq \frac{1}{c}.
$$
Then in light of Proposition \ref{prop: H-dim facts}(2) it 
suffices to show that 
\begin{equation}\label{eq:Z bound} \dim Z
\leq 6 -\frac{\delta}{5}.
\end{equation}

{\textbf{Step 3: Applying Proposition \ref{prop: upper bounds Hdim}.}}

Let $K_1 \subset \HH(0)$ be a compact set so that for each 
$q \in \HH(0)$ for which the horizontal foliation is aperiodic, the
set of return times $\{t \in \N: g_{-t} q \in K_1\}$ is
unbounded. The choice of $K_1$ ensures that for any $T>0$, 
$$
Z\subset \bigcup_{t \in \N, \, t \geq T_0} Z(t),$$
where 
$$Z(t) \df \left\{\trem_{q,\beta} \in Z : q \in \mathcal{U} \cap \EE',
  \, \beta \in
  \tremspace^{(0)}_q, \, g_{-t} \bar{\pi}(q) \in K_1 \right\}
$$
and $\bar \pi: \mathcal{E} \to \HH(0)$ is as in \S \ref{subsec: dynamics on E}. 
Let $$X(t)\df  \{q
\in  Z \cap \EE'  : g_{-t}\bar{\pi}( q) \in K_1\}.$$

We now check that all the conditions of Proposition \ref{prop: upper bounds Hdim} are satisfied. 
We first check (iii). By \equ{eq: geodesic expansion rate} and the definition of $K'$ we see 
that for any $q_0 \in \mathcal{N}^{(e^{-2t})}(q)$ and $t \geq T_1(q_0)$  we must have $g_{-t} q_0 \notin
K$. Thus if $\mu_\EE$ denotes the flat measure on $\EE$, 
Proposition \ref{prop: jayadev thesis} and a Fubini argument show that for
each $t \in \N$, 
\begin{equation}\label{eq: athreya bound}
\mu_\EE \left( \mathcal{N}^{(e^{-2t})} \left(\left\{ q \in \mathcal{U}
      \cap \EE':
      T_1(q) \leq T_0 \right \} \right) \right )
\leq C e^{\delta T_0} \, e^{-\delta t}, 
\end{equation}
where $ C = C(\mathcal{U}, T_0)$. 

We now check conditions (i) and (ii).
Using Lemmas \ref{lem: Schwartzman joined1} and \ref{lem: Schwartzman
  joined2}, for each $t$ define  
finitely many
convex sets $J_i(t)$ of inradius at least $c_1 e^{-2t}$ which cover
$X(t)$ and for which the map $q \mapsto \til C^+_q(t)$ is constant on
$J_i(t)$, and set 
$$X_i(t) \df X(t) \cap J_i(t)$$
and 
$$Y_i(t) \df  \bigcup_{q \in X_i(t)} P^-\left( \left\{\beta \in \til
C^+_q(t): L_q(\beta)\leq a \right\} \right).$$
With these definitions, it follows from Lemma \ref{lem:
  Schwartzman joined2} (with $c_2 = c = 1/c_1$)
that all conditions of Proposition \ref{prop: upper 
  bounds Hdim} are
satisfied and the result follows. 
  \end{proof}

\section{Effective covers of convex sets}\name{subsec: sec 8 aux}  
  In this section we prove Proposition \ref{prop: upper
     bounds Hdim}.  First, we briefly outline the idea of the proof. The main difficulty is to find efficient covers of  $\bigcup_i X_i(t)$ by small balls of a fixed radius. 
If the intersection of a  ball with one of the sets $J_i(t)$ appearing in (i) has significant measure, it will contribute significantly to our cover, and it follows from (iii) that the number of such balls is not too large  (see \eqref{eq: occupies definite amount}). The subset of $\bigcup_i X_i(t) $ not covered by such balls requires more work, and in particular, the key technical result Corollary \ref{cor: for covering convex}.

In this section the notation $|A|$ may mean one of several
different things: if $A\subset \R^d$ 
then  $|A|$ denotes the Lebesgue measure of 
$A$. Let
$\bS^{d-1}$ denote the $d-1$ dimensional unit sphere in $\R^d$, then for
$A \subset \bS^{d-1}$, $|A|$ denotes the measure of $A$ with respect
to the unique rotation invariant probability measure on
$\bS^{d-1}$. If $A \subset \R^{d} \times \bS^{d-1}$, then $|A|$ denotes the
measure of $A$ with respect to the product of these measures.

 The next Proposition contains the main geometric idea, and implies  Corollary \ref{cor: for covering convex} via standard covering arguments for Euclidean spaces. The Proposition provides power law savings for the measure of the subset of a convex set $K$ for which the ball centered at such a  point intersects $K$ in small measure.

\begin{prop}\name{prop: convex main}
  For any $d \geq 2$ there are positive constants $c, C$, depending only on $d$, such that for any
  compact convex set $K\subset \mathbb{R}^d$ with inradius
  $R>0$, and any $\vre\in (0,1)$, the set
 $$
K^{(\vre)} \df \left\{x \in K: |B(x,\vre R) \cap K| \leq c (\vre R)^d
\right\} 
$$
satisfies
$$
\left |K^{(\vre)} \right| \leq C \vre^2 |K|. 
$$
\end{prop}
We briefly discuss the proposition and its proof. Observe that the condition of being in $K^{(\varepsilon)}$ is more restrictive than being near the boundary of $ K$. For example, if $K$ is a line segment then $K^{(\varepsilon)}$ is empty for small enough $\varepsilon$. It turns out to be useful to think of convex sets in two dimensions, and  the main idea of the proof is to reduce the problem to a two-dimensional statement via polar coordinates. The two-dimensional case is proved   by comparing the measure of a `bulk' (which is denoted by $K'$ in the proof) to a quantity that bounds $K^{(\varepsilon)}$.

Since the statement of Proposition \ref{prop: convex main}
is invariant under homotheties, we can and will assume that
$R=1$. For $\psi \in \bS^{d-1}$, and $x \in 
\R^d$, let $\tau_\psi(x) \df \{x + s \psi : s \in \R\}$ be the line
through $x$ in direction $\psi$, and let
$$
K^{(\vre)}(\psi) \df \left\{ x \in K^{(\vre)} : \left|\tau_\psi(x) \cap
    K \right| < \vre \right\}.
  $$

  \begin{lem}\name{lem: convex 1}
    For any $d  \geq 2$ there is a positive constant $c$ so
    that for any $\vre \in (0,1),$ 
    there is $\psi \in \bS^{d-1}$ such
    that
    \eq{eq: for lemma}{
\left| K^{(\vre)}(\psi)\right| \geq \frac{\left| K^{(\vre)}\right|}{2}. 
    }
  \end{lem}

  \begin{proof}
Let $c =\frac{1}{2^{d+2}d}$, and suppose $x \in
K^{(\vre)}$, so that $\left|B(x, \vre) \cap K \right| \leq 
c \vre^d.$ For each $\theta \in \bS^{d-1}$, we write $$T_\theta(x) =
\left| \tau_\theta(x) \cap K\right| \ \ \text{ and } \  
\rho(\theta) = \sup\{s > 0 : x + s\theta \in K\}.$$
Then
$\max (\rho(\theta) , \rho(-\theta)) \geq
\frac{T_\theta (x) }{2}.
$
Computing
the volume of $B(x,\vre) \cap K$ in polar coordinates, we have
\[
  \begin{split}
    c\vre^d \geq & \left|B(x, \vre) \cap K \right|
    = \int_{\bS^{d-1}}\int_0^{\rho(\theta)} r^{d-1}dr \, d\theta
    \\ \geq & \frac{1}{2} \int_{\bS^{d-1}} \int_0^{\frac{T_\theta(x)}{2}}
    r^{d-1} dr d\theta 
    \geq  \frac{1}{2^{d+1}d} \int_{\bS^{d-1}} T_\theta(x)^d d\theta.
  \end{split}
\]
So by Markov's inequality and the choice of $c$, 
\eq{eq: from Markov}{
  |\{\theta \in \bS^{d-1} :
  T_\theta(x) < \vre\}| \geq \frac{1}{2}.}
Now consider the set
$$A \df \left\{(x, \theta) \in
  K^{(\vre)}\times \bS^{d-1} : T_\theta(x) < \vre \right\}.$$
From \equ{eq: from Markov}
and Fubini we have
$$
\frac{\left|K^{(\vre)} \right|}{2} \leq |A| = \int_{\bS^{d-1}}
\left|K^{(\vre)}(\theta) \right| \, d\theta. 
$$
Thus for some $\psi \in \bS^{d-1}$ we have \equ{eq: for lemma}.
    \end{proof}

\begin{proof}[Proof of Proposition \ref{prop: convex main}]   Let $\mathbf{e}_1,
  \ldots, \mathbf{e}_d$ denote the standard basis of 
    $\R^d$ and let $p_0$ be a point
for which $B(p_0,1) \subset K$. Applying a rotation and a translation, we may assume
that $p_0=0$ and $\psi = \mathbf{e}_d$, where $\psi$ is as in Lemma
\ref{lem: convex 1}. We will make computations in cylindrical
coordinates, i.e. we will consider the sphere
$\bS^{d-2}$ as embedded in $\spa \left(\mathbf{e}_1, \ldots,
  \mathbf{e}_{d-1}\right)$ and write vectors in $\R^d$ as $ r \theta + z
  \mathbf{e}_d$. In these coordinates, $d$-dimensional Lebesgue
  measure is given by $\alpha r^{d-2} \, dr \, d\theta \, dz,$ where $d\theta$ is the
  rotation invariant probability measure on $\bS^{d-2}$ and $\alpha = \alpha_{d-1}$
  is a
  constant. For each $\theta \in \bS^{d-2}$, define
  $$\rho_\theta = \sup \{r \in \R : r\theta \in K\} \text{ and } 
    f_\theta(r)= \left|\tau_{\mathbf{e}_d}(r\theta) \cap K\right|,
$$
i.e., $f_\theta(r)$ is the length of the intersection with $K$ of
the vertical line through $r\theta$. 
Let
$$
K' = K \cap \left \{r\theta + z \mathbf{e}_d  : r \in
  \left[\frac{\rho_\theta }{3}, \frac{2\rho_\theta}{3} \right] \right
\} .
  $$
Since $K$ is convex, the function $f_\theta$ is concave, and since
$B(0,1) \subset K$, $f_\theta(0) \geq 1$. This implies that whenever
$r \theta + z \mathbf{e}_d \in K^{(\vre)}(\mathbf{e}_d)$, $r \geq
(1-\vre)\rho_\theta$. Furthermore, whenever $r\theta + z \mathbf{e}_d
\in K'$ we have $f_\theta(r) \geq \frac{1}{3}$. Clearly $f_\theta(r) \leq \vre$ whenever there
is $z$ for which $r \theta + z \in K^{(\vre)}$, and hence
\[
  \begin{split}
& \left|K^{(\vre)} (\mathbf{e}_d)\right| \leq  \alpha \int_{\bS^{d-2}}
\int_{(1-\vre)\rho_\theta}^{\rho_\theta} \vre r^{d-2} dr \, d\theta \\
\leq & \alpha \vre \int_{\bS^{d-2}}
\int_{(1-\vre)\rho_\theta}^{\rho_\theta} \rho_\theta^{d-2} dr \, d\theta  
=  C'  \alpha  \vre^2 
\int_{\bS^{d-2}}
\int_{\frac{\rho_\theta}{3}}^{\frac{2\rho_\theta}{3}}
r^{d-2}  dr \, d\theta \\
\leq & C' \alpha  \vre^2 
3
\int_{\bS^{d-2}}\int_{\frac{\rho_\theta}{3}}^{\frac{2\rho_\theta}{3}}
f_\theta(r) r^{d-2} dr \, d\theta 
=  3 C' \vre^2 |K'|,
\end{split}
\]
where
$$
C' =  \frac{3^{d-1}(d-1)}{2^{d-1}-1}.
$$
Since $K' \subset K$, we have shown
\eq{eq: have shown}{
 \left|K^{(\vre)}(\mathbf{e}_d)\right | \leq 3 C' \vre^2 |K|.
  }
Now taking $C = 6C'$, recalling that $\psi = \mathbf{e}_d$, and
combining Lemma \ref{lem: convex 1} with \equ{eq: have 
  shown} we obtain the desired result. 
\end{proof}

Let $N(A, R)$\index{N(A,R)@$N(A, R)$} denote the minimal
  number of balls of radius $R$ needed 
to cover $A \subset \R^d$.

\begin{cor}\name{cor: for covering convex}
For any $d \geq 2$ there exist positive constants $\bar{c},\bar{C}$ so that if
$K\subset \mathbb{R}^d$ is a convex set with inradius $R$ then the set
\begin{equation}\label{eq:K def}
K^{(\vre, \bar{c})} \df \left\{x\in K:|B(x,\vre R) \cap K|< \bar{c} \left|B(x,\vre R)
  \right| \right\}
  \end{equation}
satisfies
$$
N\left(K^{(\vre, \bar{c})}, \vre R \right) \leq \bar{C} \, |K| \, \vre^{2-d} R^{-d}.$$
\end{cor}
\begin{proof} Let $K^{(\vre)}, c,C$ be as in Proposition \ref{prop: convex main},
  and let $\bar{c}$ be small enough so that
  $$
\bar{c} \, \left| B(x, \vre R)\right| < c \left(\frac{\vre}{2} \, R  \right)^d.
  $$
This choice ensures that  if $x \in K^{(\vre, \bar{c})}$ and $y \in
  B\left(x,\frac \vre 2 R\right)$ then
$y \in K^{(\vre/2)}$; i.e., $B\left(x, \frac{\vre }{2} R \right) \subset
 K^{(\vre/2)}$. Let
$B_1, \ldots, B_N$ be a minimal collection of balls of radius
$\vre R$ which cover $K^{(\vre, \bar{c})}$ and have centers $x_1, \ldots,
x_N $ in $K^{(\vre, \bar{c})}$. Then for each $i$, $\left|B_i \cap
  K^{(\vre/2)} \right| \geq
\left|B\left(x_i, \frac \vre 2 R\right )\right| = \kappa \vre^d R^d$ for
a constant $\kappa$ depending on $d$. By the Besicovitch covering theorem (see e.g.
\cite[Chap. 2]{Mattila}), each point in $K^{(\vre, \bar{c})}$ is covered at most
$N_d$ times, where $N_d$ is a number depending only on $d$. Therefore,
\[\begin{split}
& N \kappa \vre^d R^d  = \sum_{i=1}^N \left| B \left (x_i, \frac{\vre}{2} R\right) \right| 
\leq \sum_{i=1}^N \left | B_i \cap K^{(\vre/2)}\right| \\ \leq &  
N_d \left| K^{(\vre/2)}\right| \leq N_d C \frac{\vre^2}{4} |K|,
\end{split}\]
where we used Proposition \ref{prop: convex main} for the last inequality.
Setting $\bar{C} = \frac{ N_d C}{4 \kappa}$ we obtain the required estimate.
\end{proof}

%
%

We are now ready for the
\begin{proof}[Proof of Proposition \ref{prop: upper bounds Hdim}]
For each $t \in \N$ we will find an efficient cover of $Z(t)$ by balls of
radius $e^{-\left(2 + \frac{\delta}{2}\right)t}$, from which we will derive the Hausdorff dimension bound. 
We will lighten the notation by writing 
$\hat N(P, t)$ \index{N(P, t)@$\hat N(P, t)$} for $N\left(P, e^{-\left(2+\frac{\delta}{2}\right)t}
\right).$ 
We will continue with the notation $A \ll B$ used in the proof of
Lemma \ref{lem: Schwartzman joined2}, and write $A \asymp B$ if $A \ll B $ and $B \ll A$. 
In this proof the implicit
    constant is allowed
    to depend on $d, c_1, c_2, \delta, P_1, P_2$. 

We claim that \eq{eq: main claim cover}{
\hat N(Z(t), t) \ll
e^{\left(\left(2+\frac{\delta}{2}\right)(d+1)-\frac{\delta}{2}\right)t}.}
To prove \equ{eq: main claim cover}, we will find an efficient cover
for each set $X_i(t)$ and each $Y_i(t)$, and combine them. By
assumption (ii),  
$\hat N(Y_i(t), t)  \ll e^{\left(2 + \frac{\delta}{2}\right)t}\,
  e^{\frac{\delta}{2}t} = e^{\left(2 + \delta\right) t}$ for each
$i$. Indeed, the first term in this product comes from covering the
long side, of length $\ll 1$, and the second term is needed for
covering the short side of length $\ll e^{-2t}$. 
So it suffices
to show 
\eq{eq: suffices for Xi}{\sum_i \hat N(X_i(t), t) 
\ll e^{\left(\left(2+\frac{\delta}{2} \right)d-\delta \right)t}.}

With the notation of  \eqref{eq:K def}  
define
$$
J'_i(t) \df J_i^{\left(e^{-\frac{\delta}{2}t}, \bar{c} \right)}. 
$$
We will consider the sets
$
\bar{X}_i(t) = X_i(t) \sm J'_i(t)$
and $X_i(t) \cap J'_i(t) $
 separately, finding efficient covers
for each. 
If $x \in \bar{X}_i(t)$ then 
\eq{eq: occupies definite amount}{
 \begin{split} & \left|B\left(x,e^{-\left(2+\frac \delta 2 \right)t} \right) \cap J_i(t)  \cap
  \mathcal{N}^{\left(e^{-2t}\right)}(X_i(t)) \right| \\ \asymp &
\left|B\left(x,e^{-\left(2+\frac \delta 2 \right)t} 
  \right)\right| \asymp e^{-d\left(2 + \frac{\delta}{2}\right)t}.
  \end{split}
}
Let $\left \{B^{(i)}_j  \right \}_j$ be a
minimal collection of balls of radius 
$e^{-\left(2+\frac{\delta}{2}\right) t}$ 
centered at points in $\bar{X}_i(t)$ needed to cover
$\bar{X}_i(t)$. By the Besicovitch covering theorem, the collection $\left\{B^{(i)}_j\right\}$
has bounded multiplicity, i.e. for each
$x$ and $i$, $\# \left\{j : x \in B^{(i)}_j \right\} \ll 1.$
Since the $J_i(t)$ are disjoint, the collection
$\mathcal{B}_t = \left\{B^{(i)}_j \cap J_i(t) \right\}_{i,j}$ is also
of bounded multiplicity.
Taking into account \equ{eq: occupies definite amount}, we have 
\eq{eq: bound non prickly part}{
\sum_i \hat N\left(\bar{X}_i(t), t \right) \ll \# \mathcal{B}_t \ll
e^{d\left(2+\frac{\delta}{2}\right)t} \, \left|\bigcup_i  \mathcal{N}^{(e^{-2t})}(X_i(t))  
\right| \stackrel{(\text{iii})}{\ll}
e^{\left(d \left(2+\frac{\delta}{2}\right) -\delta \right)t}.
}
We also have from Corollary \ref{cor: for covering convex} (with $R =
e^{-2t}$ and $\vre = e^{-\frac{\delta}{2} \, t}$) that 
\eq{eq: bound prickly part}{
  \begin{split}
& \sum_i \hat N\left(J'_i(t),t \right) \ll   \sum_i
e^{\frac{\delta}{2}(d-2)t} \, e^{2dt}  |J_i(t)| \\
\ll & e^{\left( \left(2+\frac{\delta}{2} \right) d - \delta
\right)t}
\left|\bigcup_i J_i(t)
\right| \ll e^{\left( \left(2+\frac{\delta}{2} \right) d -
    {\delta} \right)t}.
\end{split}
}
Combining the estimates \equ{eq: bound non
  prickly part} and \equ{eq: bound prickly part}, we obtain \equ{eq:
  suffices for Xi}, and thus \equ{eq: main claim cover}.

We now prove \equ{eq: hausdorff
  upper bound}. 
Let
$$s > d+1 - \frac{\delta}{5}$$
and set  
\eq{eq: def s'}{
  \begin{split} s' & \df \frac{\delta}{2} -\left( 2 + \frac{\delta}{2}
  \right) \cdot \frac{\delta}{5}
   >0
\end{split}
}
  (where we have used $\delta <1$).
We need to show that for any
$\eta>0$, we can cover $Z$ by a collection of  balls $\mathcal{B}$ of radius at
most $\eta$, so that  $\sum_{B \in \mathcal{B}} \diam(B)^s \ll 1$.
To this end, choose $T$ so that $e^{-\left(2+\frac{\delta}{2}\right) T} <
\eta$. For each $t \geq T$ let $\mathcal{B}_t$ be a collection
of $\hat N(Z(t), t)
$ balls of
radius $e^{-\left(2+\frac{\delta}{2}\right) t}$ covering $Z(t)$ and
let $\mathcal{B} = \bigcup_t 
\mathcal{B}_t.$ Then by \equ{eq: main claim cover} 
 we have 
\[\begin{split}
& \sum_{B \in \mathcal{B}} \diam(B)^s \ll  \sum_{t \geq T} \hat N(Z(t), t)
e^{-\left(2 + \frac{\delta}{2} \right) st} \\ \ll & 
\sum_{t \geq T}
e^{\left( \left(2 +\frac{\delta}{2} \right) (d+1) -\frac{\delta}{2} -
    \left(2 +\frac{\delta}{2} \right)\left (d+1 - \frac{\delta}{5}
  \right  )\right)t}
=
\sum_{t \geq T} e^{-s't} 
\to_{T \to \infty} 0.
\end{split}\]
So for large enough $T$ we have our required cover.
\end{proof}

\section{Atomic transverse measures}\name{sec: atomic tremors}
In this section we  
complete the proofs of Proposition \ref{prop:
  semicontinuity} and Corollary \ref{cor: total variation
  continuous}. We recall that in \S \ref{subsec: polygonal tremors},
these results were already 
proved in a special case (namely assuming \equ{eq: star}, that the transverse measure is absolutely continuous),
  and that this special case is sufficient for the proofs of Theorems
  \ref{thm: tremors bounded distance}, \ref{thm: spiky fish} and
  \ref{thm: Hausdorff dim}. In this section we give a more robust treatment that does not assume absolute continuity.

  We note that in the literature there are several different
  conventions regarding atomic transverse measures. Recall from 
  the second paragraph of \S \ref{subsec: transverse} that in this paper, atomic transverse measures can only  be supported on loops of a certain kind. As we will now see, these loops arise on boundaries of cylinders, but also arise as  `ghosts of departed  cylinders', that is loops comprised of finitely many horizontal saddle connections, which are not boundaries of cylinders, but might represent core curves of cylinders on nearby surfaces. We first define these loops precisely, and then give our definition of atomic transverse measures, and the associated cohomology classes. For defining the latter, we will need to introduce `decorations' of atomic transverse measures. It will become evident in the course of the proof of Proposition \ref{prop: semicontinuity} that our definition is a useful and natural one. 
  
    We say 
  that a finite, cyclically ordered collection of horizontal saddle  
connections $\delta_1, \ldots, \delta_t$  {\em
forms a loop } if the right endpoint of $\delta_i$
is the left endpoint of $\delta_{i+1}$ (addition mod $t$).
Any singular point $\xi \in \Sigma$ of degree $a$, is
contained in a
neighborhood $\mathcal{U}_\xi$ naturally parameterized by
polar coordinates $(r\cos 
\theta, r\sin \theta)$, for $0 \leq r < r_0$ and $\theta \in \R/(
2\pi(a+1)\Z)$, where $r=0$ corresponds to $\xi$ (see \cite[\S
2.5]{eigenform}). If $\xi \in
\Sigma$ is a right endpoint of 
$\delta_i$ and  a left endpoint of $\delta_{i+1}$, we can parameterize
the intersections of $\delta_i, \delta_{i+1}$ with $\mathcal{U}_\xi$ using
polar coordinates, and the {\em $i$-th turning
  angle} is the difference in angle  between $\delta_i$ and
$\delta_{i+1}$. The turning angle is well-defined modulo $2\pi(a+1)\Z$
and is an odd multiple of $\pi$. We say that the 
loop is {\em continuously extendable} 
\index{continuously extendable loop}
if for each $i$ the $i$-th turning
angle is $\pm \pi$, and we say a continuously extendable loop is {\em primitive} if whenever we
have a repetition 
$\delta_i = \delta_j, \, i\neq j,$ we must have that 
the turning angle at both of the endpoints of $\delta_i$ differs in sign from
that of $\delta_j$. Thus on each surface there are only finitely many primitive
continuously extendable loops,  and their number is bounded by a number depending only on the stratum containing the surface.

One source of continuously extendable loops, are cylinders on nearby surfaces; see Figures \ref{fig: ghost14} and \ref{fig: ghost}.
\begin{figure}[h]
\begin{tabular}{l r}
\begin{tikzpicture}[scale=0.8]

\def\u{2};
\def\v{3};
\def\w{4};

\def\ya{1};
\def\yb{3};
\def\yc{3.5};
\def\yd{5};

\node (A0) at (0,\yb) [circle,draw,fill=black,inner sep=0pt,minimum size=1.2mm] {};
\node (A1) at (\u,\yc) [circle,draw,fill=white,inner sep=0pt,minimum size=1.2mm] {};
\node (A2) at (\u,\yd) [circle,draw,fill=black,inner sep=0pt,minimum size=1.2mm] {};
\node (A3) at (\v,\ya) [circle,draw,fill=white,inner sep=0pt,minimum size=1.2mm] {};
\node (A4) at (\v,\yb) [circle,draw,fill=black,inner sep=0pt,minimum size=1.2mm] {};
\node (A5) at (\u+\v,\yc) [circle,draw,fill=white,inner sep=0pt,minimum size=1.2mm] {};
\node (A6) at (\u+\v,\yd)  [circle,draw,fill=black,inner sep=0pt,minimum size=1.2mm] {};
\node (A7) at (\v+\w,\ya)  [circle,draw,fill=white,inner sep=0pt,minimum size=1.2mm] {};
\node (A8) at (\v+\w,\yb)  [circle,draw,fill=black,inner sep=0pt,minimum size=1.2mm] {};
\node (A9) at (\u+\v+\w,\yc)  [circle,draw,fill=white,inner sep=0pt,minimum size=1.2mm] {};

\draw (A0) -- (A4) ;
\draw (A0) -- (A1);
\draw (A1) -- (A2);
\draw (A1) -- (A5);
\draw (A2) -- (A6);
\draw (A3) -- (A4);
\draw (A3) -- (A7) ;
\draw (A4) -- (A8);
\draw (A5) -- (A6);
\draw (A5) -- (A9);
\draw (A7) -- (A8);
\draw (A8) -- (A9);

\draw [dashed] ($(A0)!0.5!(A1)$) -- ($(A8)!0.5!(A9)$);
\end{tikzpicture}\ \ \ \ \ \ \ \ \

\begin{tikzpicture}
[scale=0.8,
 bdot/.style={circle,draw,fill=black,inner sep=0pt,minimum size=1.2mm},
 wdot/.style={circle,draw,fill=white,inner sep=0pt,minimum size=1.2mm},
 sdot/.style={circle,draw,fill=black,inner sep=0pt,minimum size=0.3mm}]

\def\u{2};
\def\v{3};
\def\w{4};

\def\ya{1};
\def\yb{3};
\def\yc{3.5};
\def\yd{5};

\node (A1) at (\u,\yc) [wdot] {};
\node (A2) at (\u,\yd) [bdot] {};
\node (A3) at (\v,\ya) [wdot] {};
\node (A4) at (\v,\yb) [bdot] {};
\node (A5) at (\u+\v,\yc) [wdot] {};
\node (A6) at (\u+\v,\yd)  [bdot] {};
\node (A7) at (\v+\w,\ya)  [wdot] {};
\node (A8) at (\v+\w,\yb)  [bdot] {};
\node (A10) at ($ (A2)+(A1)-(A0) $)  [wdot] {};
\node (A11) at ($ (A3)+(A8)-(A5) $)  [bdot] {};

\draw (A1) -- (A2);
\draw (A1) -- (A5);
\draw (A2) -- (A6);
\draw (A3) -- (A4);
\draw (A3) -- (A7) ;
\draw (A4) -- (A8);
\draw (A5) -- (A6);
\draw (A7) -- (A8);
\draw (A1) -- (A4);
\draw (A5) -- (A8);

\draw (A2) -- (A10);
\draw (A6) -- (A10);

\draw (A3) -- (A11);
\draw (A7) -- (A11);

\draw [dashed] ($(A1)!0.5!(A4)$) -- ($(A5)!0.5!(A8)$);
\draw [dashed] ($(A2)!0.5!(A10)$) -- ($(A6)!0.5!(A10)$);
\draw [dashed] ($(A3)!0.5!(A11)$) -- ($(A7)!0.5!(A11)$);

\end{tikzpicture}\ \ \ \ \ \ \ \ \

\end{tabular}
\caption{ In the 3-cylinder surface on the left, the dotted line represents a thin cylinder. Collapsing it gives rise to a continuously extendable loop.  The presentation of the same surface on the right helps show how the continuously extendable loop arises as a limit.  }
\label{fig: ghost14} 
\end{figure}

\begin{figure}[h]
\begin{tabular}{l r}
\begin{tikzpicture}[scale=0.8]

\def\u{2};
\def\v{3};
\def\w{4};

\def\ya{1};
\def\yb{3.5};
\def\yc{3.5};
\def\yd{5.5};

\def\sp{1.5mm};

\node (A1) at (\u,\yc) [circle,draw,fill=white,inner sep=0pt,minimum size=1.2mm] {};
\node (A2) at (\u,\yd) [circle,draw,fill=black,inner sep=0pt,minimum size=1.2mm] {};
\node (A3) at (\v,\ya) [circle,draw,fill=white,inner sep=0pt,minimum size=1.2mm] {};
\node (A4) at (\v,\yb) [circle,draw,fill=black,inner sep=0pt,minimum size=1.2mm] {};
\node (A5) at (\u+\v,\yc) [circle,draw,fill=white,inner sep=0pt,minimum size=1.2mm] {};
\node (A6) at (\u+\v,\yd)  [circle,draw,fill=black,inner sep=0pt,minimum size=1.2mm] {};
\node (A7) at (\v+\w,\ya)  [circle,draw,fill=white,inner sep=0pt,minimum size=1.2mm] {};
\node (A8) at (\v+\w,\yb)  [circle,draw,fill=black,inner sep=0pt,minimum size=1.2mm] {};
\node (A10) at (\u+\u,\yd)  [circle,draw,fill=white,inner sep=0pt,minimum size=1.2mm] {};
\node (A11) at (\v+\w-\u,\ya)  [circle,draw,fill=black,inner sep=0pt,minimum size=1.2mm] {};

\draw (A1) -- (A2);
\draw (A1) -- (A4);
\draw (A4) -- (A5);
\draw (A5) -- (A8);
\draw (A2) -- (A10);
\draw (A10) -- (A6);
\draw (A3) -- (A4);
\draw (A3) -- (A11);
\draw (A11) -- (A7);
\draw (A5) -- (A6);
\draw (A7) -- (A8);

\end{tikzpicture}\ \ \ \ \ \ \ \ \

\begin{tikzpicture}[scale=0.8]

\def\u{2};
\def\v{3};
\def\w{4};

\def\ya{1};
\def\yb{3.5};
\def\yc{3.5};
\def\yd{5.5};

\def\sp{1.5mm};

\node (A1) at (\u,\yc) [circle,draw,fill=white,inner sep=0pt,minimum size=1.2mm] {};
\node (A2) at (\u,\yd) [circle,draw,fill=black,inner sep=0pt,minimum size=1.2mm] {};
\node (A3) at (\v,\ya) [circle,draw,fill=white,inner sep=0pt,minimum size=1.2mm] {};
\node (A4) at (\v,\yb) [circle,draw,fill=black,inner sep=0pt,minimum size=1.2mm] {};
\node (A5) at (\u+\v,\yc) [circle,draw,fill=white,inner sep=0pt,minimum size=1.2mm] {};
\node (A6) at (\u+\v,\yd)  [circle,draw,fill=black,inner sep=0pt,minimum size=1.2mm] {};
\node (A7) at (\v+\w,\ya)  [circle,draw,fill=white,inner sep=0pt,minimum size=1.2mm] {};
\node (A8) at (\v+\w,\yb)  [circle,draw,fill=black,inner sep=0pt,minimum size=1.2mm] {};
\node (A10) at (\u+\u,\yd)  [circle,draw,fill=white,inner sep=0pt,minimum size=1.2mm] {};
\node (A11) at (\v+\w-\u,\ya)  [circle,draw,fill=black,inner sep=0pt,minimum size=1.2mm] {};

\draw (A1) -- (A2);
\draw (A1) -- (A4);
\draw (A4) -- (A5);
\draw (A5) -- (A8);
\draw (A2) -- (A10);
\draw (A10) -- (A6);
\draw (A3) -- (A4);
\draw (A3) -- (A11);
\draw (A11) -- (A7);
\draw (A5) -- (A6);
\draw (A7) -- (A8);
\draw [line width=1.2pt] ([shift=(0:2.5mm)] $(A4)$) arc (0:180:2.5mm);
\draw [line width=1.2pt] ([shift=(0:-2.5mm)] $(A5)$) arc (180:360:2.5mm);
\draw [line width=1.2pt] ([shift=(0:-2.5mm)] $(A10)$) arc (180:360:2.5mm);
\draw [line width=1.2pt] ([shift=(0:2.5mm)] $(A11)$) arc (0:180:2.5mm);
\draw [line width=1.2pt] ([shift=(0:2.5mm)] $(A4)$) -- ([shift=(0:-2.5mm)] $(A5)$);
\draw [line width=1.2pt] ([shift=(0:2.5mm)] $(A5)$) -- ([shift=(0:-2.5mm)] $(A8)$);
\draw [line width=1.2pt] ([shift=(0:2.5mm)] $(A1)$) -- ([shift=(0:-2.5mm)] $(A4)$);
\draw [line width=1.2pt] ([shift=(0:2.5mm)] $(A2)$) -- ([shift=(0:-2.5mm)] $(A10)$);
\draw [line width=1.2pt] ([shift=(0:2.5mm)] $(A10)$) -- ([shift=(0:-2.5mm)] $(A6)$);
\draw [line width=1.2pt] ([shift=(0:2.5mm)] $(A3)$) -- ([shift=(0:-2.5mm)] $(A11)$);
\draw [line width=1.2pt] ([shift=(0:2.5mm)] $(A11)$) -- ([shift=(0:-2.5mm)] $(A7)$);

\end{tikzpicture}\ \ \ \ \ \ \ \ \
\end{tabular}
\caption{This is the surface  obtained by collapsing the middle cylinder in Figure \ref{fig: ghost14}. The union of all horizontal saddle connections on the resulting surface is a continuously extendable loop.
The half-circles extending this curve to the punctured surface are shown. This extended curve is the `ghost of the departed cylinder' from Figure \ref{fig: ghost14}.}
\label{fig: ghost} 
\end{figure}

Recall from \S \ref{subsec: transverse} that a non-atomic transverse measure is a collection of non-atomic finite measures $\{\nu_\gamma\}$ indexed by finite-length transverse arcs $\gamma \subset M \sm \Sigma, $ satisfying the invariance and restriction condition.

By a {\em closed horizontal leaf} on $M$ we mean a loop contained in one leaf of the horizontal foliation on $M \sm \Sigma$. Given a closed horizontal leaf $\lambda$, and a finite length transverse arc $\gamma$, the number of intersection points $\#(\lambda \cap \gamma)$ is finite, and we define measures $\theta^{(\lambda)}_\gamma$ by  
$$\theta^{(\lambda)}_\gamma(A) \df \#  (\lambda \cap A).$$ 
It is clear that the collection of measures 
$\left\{\theta^{(\lambda)}_\gamma \right\}$ satisfies the invariance and restriction conditions. Now given a primitive continuously extendable loop $\ell$, obtained as a concatenation of horizontal saddle connections $\delta_1, \ldots, \delta_t$ (possibly with repetition), and a finite length transverse arc $\gamma$, the number of intersection points $\#(\delta_i \cap \gamma)$ is again finite for each $i$, and we define a collection of measures $\left\{\theta^{(\ell)}_{\gamma} \right\}$ by 
\begin{equation}\label{eq: theta extendable}
\theta^{(\ell)}_\gamma(A) = \sum_{i=1}^t \#  (\delta_i \cap A).
\end{equation}
For each $\gamma$ let $\nu_\gamma^{(\mathrm{at})}$ denote the restriction of $\nu_\gamma $ to its atoms,
and let 
$$\nu^{(\mathrm{at})} \df \left\{\nu_{\gamma}^{(\mathrm{at})} \right\}, \ \ 
 \nu_\gamma^{(\mathrm{na})}  \df \nu_\gamma - \nu_\gamma^{(\mathrm{at})} \text{ and } \nu^{(\mathrm{na})} \df \left\{\nu_{\gamma}^{(\mathrm{na})} \right\}.$$
 \index{nu@$\nu^{(\mathrm{at})}$}  
Here is the definition of transverse measures which we will use in this paper. 
\begin{dfn}\label{def: atomic transverse measure}
A {\em transverse measure (to the horizontal foliation on $M$)} is a family of measures $\{\nu_\gamma\}$, indexed by finite length transverse arcs in $M \sm \Sigma$, such that 
\begin{itemize}
    \item The non-atomic part 
$\nu^{(\mathrm{na})}$ satisfies the invariance and restriction conditions given in \S \ref{subsec: transverse};  
\item there are at most finitely many primitive continuously extendable
loops $\ell_r$, at most finitely many closed horizontal leaves $\lambda_s$, and positive weights $a_r ,b_s$ such that the atomic part $\nu^{(\mathrm{at})}$ satisfies that for each transverse arc $\gamma$,  
\begin{equation}\label{eq: does not uniquely determine}
\nu_\gamma^{(\mathrm{at})} = \sum_r a_r \theta^{(\ell_r)}_{\gamma} + \sum_s b_s \theta_\gamma^{(\lambda_s)}.
\end{equation}
\end{itemize}
\end{dfn}
Our next goal is to define the cohomology class $\beta_\nu \in H^1(M_q, \Sigma_q; \R)$ associated with a transverse measure $\nu$, extending the assignment given in Proposition \ref{prop:trans to meas} to atomic transverse measures with atoms. To this end, given a continuously extendable loop $\ell  = (\delta_1, \ldots, \delta_t)$, a {\em continuous extension} $\check \ell$ of $\ell$ is a continuous
closed curve
homotopic to $\ell$ with all its points in $S\sm \bigcup_{\xi \in \Sigma} \mathcal{U}_\xi$, which is
the same as $\ell$ outside the neighborhoods $\mathcal{U}_\xi$, and
such that for each $i$, the intersection of $\delta_i, \delta_{i+1}$
with $\mathcal{U}_{\xi} $ is replaced with a curve on $\partial
\mathcal{U}_\xi$ corresponding to $r = r_0$ and $\theta$ in an
interval of length $\pi$. See Figure \ref{fig: ghost}. If $\lambda$ is a closed horizontal leaf on $M_q$, we let $\check \lambda$ denote the corresponding curve oriented in the direction of increasing $x$ coordinate. Both $\check \ell$ and $\check \lambda$ are closed oriented loops avoiding
$\Sigma_q$, and thus represent elements of
$H_1(M \sm \Sigma)$. By Poincar\'e-Lefschetz duality, these loops represent elements of
$H^1(M, \Sigma; \R)$, which we will denote by $\left[\check \ell \right], \left[\check \lambda \right]$.

We would like to use these cohomology classes in order to define the cohomology class associated with an atomic transverse measure. 
A complication  is that the objects defined above are not uniquely determined by the measure.  

\begin{example}\label{example: below}
We list some examples which are related to lack of uniqueness in our discussion above. 
    \begin{enumerate}
        \item 
        If $\xi \in \Sigma$ is a removable singularity (singularity of order 0) and $\delta_i$, $\delta_{i+1}$ are horizontal saddle connections which meet at $\xi$ and are consecutive along an extendable loop $\ell$,  there are two possibilities for a continuous extension $\check \ell$, corresponding to taking angles $+\pi$ or $-\pi$ for the $i$-th turning angle. 

        \item 
        If $M_q$ has a horizontal cylinder $C$, $\lambda_1$ and $\lambda_2$ are two parallel closed horizontal leaves in the interior of $C$, we have $\left [ \check \lambda_1 \right]  = \left [ \check \lambda_2 \right] $ (as elements of  $H_1(M_q \sm \Sigma_q)$). Moreover, if $h>0$ is the height of $C$ and $\nu_C$ is the restriction of Lebesgue measure to $C$ and $\beta_C$ is the cohomology class corresponding to $\nu_C$ as in 
        Proposition \ref{prop:trans to meas}, then  $\left[ \lambda_i \right] = \frac{1}{h}\beta_C.$ Finally, the class $[\lambda_i]$ can also be obtained from the two continuously extendable loops forming the top and bottom  boundary of $C$.
        \item 
        Two different surfaces $M_{q_1}, M_{q_2}$, each with a horizontal cylinder $C_1 \subset M_{q_1}, \, C_2 \subset M_{q_2}$, can be deformed into a surface $M_q$ on which the height of the cylinders $C_i$ has been taken to zero. This results in the same continuously extendable loop $\ell$ on $M_q$, for which the $C_i$ on $M_{q_i}$ correspond to two different continuous extensions $\check \ell_1, \check \ell_2.$ Such examples can be found using imaginary Rel deformations of horizontally periodic surfaces, see \cite{McMullen snow}. In Figure \ref{fig: ghost2} we show an example giving the same $\ell$ as in Figure \ref{fig: ghost}, via a different cut and paste operation involving cylinders. 
       \end{enumerate}
\end{example}      

  \begin{figure}
   \begin{tabular}{l r}

\begin{tikzpicture}[scale=0.8]

\def\u{2};
\def\v{3};
\def\w{4};

\def\ya{1};
\def\yb{3.5};
\def\yc{3.3};
\def\yd{5.5};

\def\sp{1.5mm};
\node (A0) at (0,\yb)  {};
\node (A1) at (\u,\yc) [circle,draw,fill=white,inner sep=0pt,minimum size=1.2mm] {};
\node (A2) at (\u,\yd) [circle,draw,fill=black,inner sep=0pt,minimum size=1.2mm] {};
\node (A3) at (\v,\ya) [circle,draw,fill=white,inner sep=0pt,minimum size=1.2mm] {};
\node (A4) at (\v,\yb) [circle,draw,fill=black,inner sep=0pt,minimum size=1.2mm] {};
\node (A5) at (\u+\v,\yc) [circle,draw,fill=white,inner sep=0pt,minimum size=1.2mm] {};
\node (A6) at (\u+\v,\yd)  [circle,draw,fill=black,inner sep=0pt,minimum size=1.2mm] {};
\node (A7) at (\v+\w,\ya)  [circle,draw,fill=white,inner sep=0pt,minimum size=1.2mm] {};
\node (A8) at (\v+\w,\yb)  [circle,draw,fill=black,inner sep=0pt,minimum size=1.2mm] {};
\node (A10) at ($(A2)+(A1)-(A0)$)  [circle,draw,fill=white,inner sep=0pt,minimum size=1.2mm] {};

\node (A11) at ($(A3)+(A8)-(A5)$)  [circle,draw,fill=black,inner sep=0pt,minimum size=1.2mm] {};

\draw (A1) -- (A2);
\draw (A1) -- (A4);
\draw (A4) -- (A5);
\draw (A5) -- (A8);
\draw (A2) -- (A10);
\draw (A10) -- (A6);
\draw (A3) -- (A4);
\draw (A3) -- (A11);
\draw (A11) -- (A7);
\draw (A5) -- (A6);
\draw (A7) -- (A8);

\end{tikzpicture}\ \ \ \ \ \ \ \ \

\begin{tikzpicture}[scale=0.8,
 bdot/.style={circle,draw,fill=black,inner sep=0pt,minimum size=1.2mm},
 wdot/.style={circle,draw,fill=white,inner sep=0pt,minimum size=1.2mm},
 sdot/.style={circle,draw,fill=black,inner sep=0pt,minimum size=0.3mm}]

\def\u{2};
\def\v{3};
\def\w{4};

\def\ya{1};
\def\yb{3.5};
\def\yc{3.5};
\def\yd{5.5};

\def\sp{1.5mm};

\node (A1) at (\u,\yc) [wdot] {};
\node (A2) at (\u,\yd) [bdot] {};
\node (A3) at (\v,\ya) [wdot] {};
\node (A4) at (\v,\yb) [bdot] {};
\node (A5) at (\u+\v,\yc) [wdot] {};
\node (A6) at (\u+\v,\yd)  [bdot] {};
\node (A7) at (\v+\w,\ya)  [wdot] {};
\node (A8) at (\v+\w,\yb)  [bdot] {};
\node (A10) at (\u+\u,\yd)  [wdot] {};
\node (A11) at (\v+\w-\u,\ya)  [bdot] {};

\draw (A1) -- (A2);
\draw (A1) -- (A4);
\draw (A4) -- (A5);
\draw (A5) -- (A8);
\draw (A2) -- (A10);
\draw (A10) -- (A6);
\draw (A3) -- (A4);
\draw (A3) -- (A11);
\draw (A11) -- (A7);
\draw (A5) -- (A6);
\draw (A7) -- (A8);
\draw [line width=1.2pt] ([shift=(0:2.5mm)] $(A1)$) arc (0:90:2.5mm);
\draw [line width=1.2pt] ([shift=(0:2.5mm)] $(A2)$) arc (0:-90:2.5mm);
\draw [line width=1.2pt] ([shift=(0:2.5mm)] $(A3)$) arc (0:90:2.5mm);
\draw [line width=1.2pt] ([shift=(0:2.5mm)] $(A4)$) arc (0:-90:2.5mm);
\draw [line width=1.2pt] ([shift=(0:-2.5mm)] $(A5)$) arc (180:90:2.5mm);
\draw [line width=1.2pt] ([shift=(0:-2.5mm)] $(A6)$) arc (180:270:2.5mm);
\draw [line width=1.2pt] ([shift=(0:-2.5mm)] $(A7)$) arc (180:90:2.5mm);
\draw [line width=1.2pt] ([shift=(0:-2.5mm)] $(A8)$) arc (180:270:2.5mm);
\draw [line width=1.2pt] ([shift=(0:-2.5mm)] $(A10)$);
\draw [line width=1.2pt] ([shift=(0:2.5mm)] $(A11)$); 
\draw [line width=1.2pt] ([shift=(0:2.5mm)] $(A4)$) -- ([shift=(0:-2.5mm)] $(A5)$);
\draw [line width=1.2pt] ([shift=(0:2.5mm)] $(A5)$) -- ([shift=(0:-2.5mm)] $(A8)$);
\draw [line width=1.2pt] ([shift=(0:2.5mm)] $(A1)$) -- ([shift=(0:-2.5mm)] $(A4)$);
\draw [line width=1.2pt] ([shift=(0:2.5mm)] $(A2)$) -- ([shift=(0:-2.5mm)] $(A10)$);
\draw [line width=1.2pt] ([shift=(0:2.5mm)] $(A10)$) -- ([shift=(0:-2.5mm)] $(A6)$);
\draw [line width=1.2pt] ([shift=(0:2.5mm)] $(A3)$) -- ([shift=(0:-2.5mm)] $(A11)$);
\draw [line width=1.2pt] ([shift=(0:2.5mm)] $(A11)$) -- ([shift=(0:-2.5mm)] $(A7)$);
\end{tikzpicture}
\end{tabular}
\caption{The figure on the left is another 3-cylinder surface obtained by deforming the surface in Figure \ref{fig: ghost}. This deformation  creates another cylinder. The solid line in the figure on the right represents the corresponding continuously extendable loop. }
\label{fig: ghost2} 
\end{figure}

In order to deal with the lack of uniqueness exhibited in these examples, we will make the following definition. 

\begin{dfn}\label{def: decoration}
    Let $\nu = \nu^{(\mathrm{na})}+\nu^{(\mathrm{at})}$ be a transverse measure, let $(\ell_r), (\lambda_s), (a_r), (b_s)$ be as in \eqref{eq: does not uniquely determine}, and for each $\ell_r$, let $\check \ell_r$ be a continuous extension of $\ell_r$. We refer to the quintuple 
     $\left[\nu, (\check \ell_r), (\lambda_s), (a_r), (b_s) \right]$ as a {\em decorated transverse measure}.
   The quadruple   $\left[(\check \ell_r), (\lambda_s), (a_r), (b_s) \right]$
    will be referred to as a {\em decoration of $\nu$.} \index{decoration (of a transverse measure)}
\end{dfn}

We now define the cohomology class $\beta_{\bar \nu} \in H^1(M_q, \Sigma_q; \R)$ associated with a decorated transverse measure $\bar \nu = \left[\nu, (\check \ell_r), (\lambda_s), (a_r), (b_s) \right].$ 
 The reader should note that whereas the definition of $\nu^{(\mathrm{na})}$ involves constructing an explicit cochain, we will only give $ \beta_{\bar \nu^{(\mathrm{at})}}$  as a
cohomology class.

Let $\nu = \nu^{(\mathrm{na})} + \nu^{(\mathrm{at})}$ be the decomposition of $\nu$ into its non-atomic and atomic parts. 
As
explained in \S\ref{subsec: transverse}, $\nu^{(\mathrm{na})}$ determines a 1-cochain $\beta_{\nu^{(\mathrm{na})}} \in H^1(M_q, \Sigma_q)$. 
We define 
$$\beta_{\bar \nu^{(\mathrm{at})}} \df \sum a_r \left[ \check \ell_r \right]+\sum b_s \left[ \check \lambda_s \right] \ \ \text{ and } \ \ \beta_{\bar \nu} \df \beta_{\nu^{(\mathrm{na})}} +\beta_{\bar \nu^{(\mathrm{at})}}. $$ 
\index{b@$\beta_{\bar \nu}$}

As reflected by the notation, the reader will notice that this  depends not only on $\nu$ but also on the choice of its decoration $\bar \nu$. Nevertheless we will sometimes abuse notation by writing $\beta_\nu$ instead of $\beta_{\bar \nu}$. Clearly we have equality when $\nu$ is non-atomic or when the atomic part  $\nu^{(\mathrm{at})}$ is supported only on closed horizontal leaves, and not on continuously extendable loops.

Recalling from \S \ref{subsec: length} that  $L_{q} (\beta_{\nu})$ is our notation for the evaluation of the cup
product $\hol_{ q}^{(x)} \cup \beta_{\bar \nu}$ on 
the fundamental class of $M_q$, the reader can check that 
\begin{equation}\label{eq: intersection length}
L_q (\beta_{\bar \nu^{(\mathrm{at})}}) = \sum_r  a_r |\ell_r| + \sum_s  b_s |\lambda_s|,
\end{equation}
where $|\ell_r|, \, |\lambda_s|$ denote respectively the horizontal length of $\ell_r$ and $\lambda_s.$
In particular, this number does not depend on the decoration $\bar \nu$ of $\nu.$ In addition, the positivity property $L_q(\beta_{\nu})>0$ (see
the first paragraph of \S \ref{subsec: length}) extends to foliation cocycles arising from 
atomic transverse measures, and we have a continuity property 
\begin{equation}\label{eq: a continuity property}
\nu_n \stackrel{\mathrm{weak-*}}{\longrightarrow} \nu_\infty \ \implies \ L_q(\beta_{\nu_n}) \stackrel{n \to \infty}{\longrightarrow} L_q(\beta_{\nu_\infty}),
\end{equation}
where by weak-$*$ convergence we mean weak-$*$ convergence of the corresponding measures on each closed transverse finite length arc. 

Let  $\alpha \mapsto \beta_\nu({\alpha})$ be the evaluation map. 
It is clear from the definition in \S \ref{subsec: transverse}, that if $\nu$ is
non-atomic and $\alpha$ is represented by a concatenation of
horizontal saddle connections, then $\beta_\nu(\alpha)=0.$
With our definition of $\beta_{\nu}$, we also have $ \beta_{\nu}(\alpha) 
=0$ if $\nu$ is atomic and $\alpha$ is a cycle
represented by a closed horizontal leaf, 
 because the leaf may be homotoped away from horizontal saddle connections. However it is possible to have  continuously extendable loops $\alpha$ and $\ell$ such that for the decorated atomic foliation cocycle $\beta_\nu$ associated with $\check \ell$ we have $\beta_\nu(\alpha) \neq 0.$



In case 
$\alpha$ is represented by an oriented horizontal saddle connection on $M_q$, $\check \ell$ is the continuous extension of a continuously extendable loop $\ell$ on $M_q$ which has a nontrivial intersection with $\alpha $, 
and $\bar \nu$ is the decorated transverse measure corresponding to $\check \ell,$ 
then the 
tremor 
$q_s \df \trem_{s\beta_{\bar \nu}}(q)$ will not be defined for all $s$. Indeed,  using \equ{eq: tremor}, for $s_0 = -\frac{L}{\left[ \check \ell \right](\alpha) }$ , where $L$ is the (oriented) length of $\alpha$,
we would have
$\hol_{q_{s_0}}(\alpha) = (0,0)$,
which is impossible. 
For instance, in Figure \ref{fig: ghost}, this situation will arise if $\alpha$ is the class represented by the horizontal saddle connections in the middle of the diagram. This shows why the
requirement in Proposition \ref{prop: tremor domain of defn}
that the tremor is non-atomic, is essential.  

\begin{remark}
    We do not define a version of a TCH for atomic transverse measures (note the assumption of non-atomicity in Proposition \ref{prop: tremor comparison}). There is a natural surgery, associated with a {\em decorated} atomic transverse measure, in which the surface is cut along a continuously extendable loop $\ell$ and reglued after a twist. One can show that for limits as in Figures \ref{fig: ghost} and \ref{fig: ghost2} this discontinuous map is the pointwise limit of cylinder twists, which are the corresponding TCH's on nearby surfaces. Furthermore, in some cases, including those shown in the figures, for small enough values of the twist parameter, this map coincides with a real Rel deformation.  
\end{remark}

\subsection{Refining an APD}
Our discussion of atomic transverse measures will rely on
the construction in \S \ref{subsec: polygonal tremors}. 
Recall from \S \ref{subsec: polygonal tremors}
that an APD for $q$ is a polygon decomposition of the underlying
surface $M_q$, into 
triangles and quadrilaterals, without horizontal edges, and such that the quadrilaterals contain a
horizontal diagonal.   We
consider all edges of an APD as open, i.e., they do not contain their
endpoints. In order to pay attention to atomic measures, we
further subdivide each edge of an APD into finitely many subintervals
by removing the points that lie on horizontal saddle connections. We
will denote by $J_i$ these open intervals lying on edges of an
APD. We will refer to an APD whose edges have been additionally
subdivided as above, as a {\em refined APD}. 
For each $i$, each polygon $P$ with $J_i \subset \partial P$, and each $x
\in J_i$, we define the opposite point $\mathrm{opp}_P(x)$ as in \S \ref{subsec: polygonal tremors}.  

Let $J=J_{i_0}$ for some $i_0$, $J \subset \partial
P$, and let $J' = \mathrm{opp}_P(J).$ Then $J'$ is
a union of either one or two of the intervals $J_i,$ for $
i \neq i_0$, depending on whether a point of $J$ has an opposite point
in $\Sigma$. In the former case we set $J_0=J$ and in the latter case
we set $J_0$ to be one of the two components of $ J \sm \mathrm{opp}_P^{-1}(\Sigma)$ and we replace $J'$ with $ \mathrm{opp}_P(J_0)$. With these
definitions 
$\mathrm{opp}_P|_{J_0}: J_0\to J'$ is a bijection. Note that each endpoint of
$J$ lies on a horizontal saddle connection or in $\Sigma$, and each endpoint
of $J_0$ is either an endpoint of some $J_i$ or lies on an
infinite critical leaf. 

We say that 
a transverse measure $\nu$ on $M_q$ {\em does not charge extendable loops} if all of the atoms of $\nu$ lie on 
closed horizontal leaves. That is, in \eqref{eq: does not uniquely determine}, the collection $(\ell_s)$ is empty. We now extend Proposition \ref{prop: they define} to such measures: 
\begin{prop}
\label{prop: they define2} Let $M_q$ be a translation surface equipped with a refined APD. The map which sends a transverse measure on $M_q$ to its restriction to the edges of the refined APD, is a bijection between a system of finite measures $\nu_J$ on the edges of the refined APD, satisfying the invariance property \eqref{eq: invariance
  property}, and transverse measures which do not charge extendable loops. 
\end{prop}
\begin{proof} If $\nu$ is a transverse measure which does not charge extendable loops, then it assigns a measure to each of the
intervals $J, J', J_0$,
and by our condition that any atoms lie on closed horizontal leaves, the restriction to
$J$ has the same mass as the restriction to $J_0$. The measures will
be denoted by $\nu_J, \nu_{J'}, \nu_{J_0}$. Their non-atomic part satisfies the
invariance property \eqref{eq: invariance
  property} by Proposition \ref{prop: they define}, and their atomic part is a finite combination of measures $\theta^{(\lambda)}_\gamma$, and these measures are easily seen to also satisfy \eqref{eq: invariance
  property}.

Conversely, suppose we are given a collection 
of finite measures $\nu_J$ on the edges $J$ as above, satisfying the invariance
property. Since an
infinite leaf has an accumulation point in one of the 
$J$, by the invariance property, any atoms of the measures $\nu_J$ lie
on closed horizontal leaves. The points of $M_q$ lying on horizontal saddle connections are
not in any of the $J$'s, and thus we can reconstruct from the $\nu_J$ a
transverse measure which does not charge extendable loops. 
\end{proof}

\subsection{Beginning the proof of  Proposition \ref{prop: semicontinuity}}

\color{black}
We will use the same proof strategy as in  \S \ref{subsec: polygonal tremors}. Namely, we will use refined APD's to describe transverse measures as measures on the edges of the APD, and discuss what happens to measures when taking limits. In this section, we will have to be more careful in treating limits of atomic measures. We will give the proof in three stages, each dealing with a more general case.
\begin{proof}[Proof of modified  Proposition \ref{prop: semicontinuity}]
We will prove the Proposition under two extra hypotheses given below as equations \eqref{eq: simplifying 1} and \eqref{eq: simplifying 2}. 
Let $ \til q_n \to \til q$ and $\beta_n \to \beta$ be as in
the statement of the Proposition, let $q_n  = \pi(\til q_n),  q = \pi(q)$
be the projections to $\HH$, and let $M_{q_n}, M_q$ be the underlying
surfaces. 
As in \S \ref{subsec: polygonal tremors}, 
we can assume that $\til q_n$ and $\til q$ are represented by marking
maps $\varphi_n \to M_{q_n}, \, \varphi: S \to M_q$ such that $\varphi_n
\circ \varphi^{-1}$ is piecewise affine with derivative tending to
$\mathrm{Id}$ as $n \to \infty$. Choose a refined APD for $q$, and let $K \subset M_q$ denote any one of the intervals
$J, J', J_0$.  We will sometimes use the same notation $K$ to refer to the corresponding arc on $M_{q_n}$ given by $\varphi_n \circ \varphi^{-1}(K)$. By our choice of marking maps, $K$ is a straight segment which is a subset of an edge of the same triangulation $\varphi_n(\tau)$, on each of the surfaces $M_{q_n}$; thus this inaccuracy should cause no confusion. Clearly we can pass to convergent subsequences in the course of the proof, and we will do so several times below.

Our first simplifying assumption is 
\begin{equation} \label{eq: simplifying 1} \begin{split} & \text{ each
 } \beta_n \text{  is equal to } \beta_{\nu_n} \text{  for some transverse measure
} \nu_n  \\ & \text{ which does not charge extendable loops}.
\end{split}
\end{equation}

Let $\nu_{K}^{(n)}$ denote the measure on $K$ given by the pushforward
of $\nu_n|K$ under $\varphi \circ \varphi_n^{-1}$, and 
denote the total variation of $\nu_{K}^{(n)}$ by $m_{K}^{(n)}$. This number can be
expressed as the evaluation of $\beta_n$ on a path $\sigma = \sigma_K$
from singular 
points to singular points that is a concatenation of $K$ with segments contained in horizontal leaves.  Since $\beta_n \to \beta$, we have
$m_{K}^{(n)} \to_{n  \to \infty} m_{K} = \beta(\sigma).$ Let $\til K =
\varphi^{-1}(K) \subset S$. Since $K$
is open and not horizontal, $\til K$ has a natural compactification
$\bar{K}$ in
which we add bottom and top endpoints $x_K^{\mathrm{b}}, x_K^{\mathrm{t}}$
to $\til K$. Note that we consider $\bar{K}$ abstractly, and not as a
subset of $S$. Because the $\nu^{(n)}_K$ do not charge extendable loops, each measure $\nu^{(n)}_{K}$ can be viewed as a measure
on the compact interval $\bar{K}$, assigning mass zero to
endpoints. Passing to further
subsequences, we can assume each sequence $\left(\nu_{K}^{(n)} \right)_n$ converges
to a measure $\nu_{\bar{K}}$ on $\bar{K}$ such that $\nu_K = \nu_{\bar{K}}|_K$. 
We have \eq{eq: gets used}{
m_{K} = \nu_{\bar{K}}(\bar{K})=\nu_{K}(\til K) +
\mathrm{e}_K^{\mathrm{b}} + \mathrm{e}_K^{\mathrm{t}}, }
where we call the numbers 
$$\mathrm{e}_K^{\mathrm{b}} \df \nu_{\bar{K}}(\{x_K^{\mathrm{b}}\}), \ \ 
\mathrm{e}_K^{\mathrm{t}} \df 
\nu_{\bar{K}}(\{x_K^{\mathrm{t}}\})
$$
{\em  the escape of mass parameters} to endpoints. \index{escape of mass parameters} We can concretely express the
$\mathrm{e}_{K}^{\mathrm{b,t}}$ by subdividing $K$ into two half-intervals
$K^{\mathrm{b}}, 
K^{\mathrm{t}}$ whose common endpoint is an interior point of $K$
which has zero measure under $\nu_K$. In these terms 
\eq{eq: loss of mass explicit}{
\mathrm{e}^{\mathrm{b}}_K
= \lim_{n \to \infty} \nu_K^{(n)}(K^{\mathrm{b}}) -
\nu_K(K^{\mathrm{b}})
}
(and this limit does not depend on the decomposition $K =
K^{\mathrm{b}} \cup K^{\mathrm{t}}$).

Since the collection of measures $\{\nu_K\}$ satisfies the invariance
property, where all atoms that appear lie on closed horizontal leaves, it defines  a 
transverse measure which does not charge extendable loops, and we let $\beta'$ be the corresponding cohomology
class. 

Our second  simplifying assumption is 
 that there is no escape of mass, i.e.
\begin{equation}\label{eq: simplifying 2}
   \text{ all the numbers } \mathrm{e}_K^{\mathrm{b, t}} \text{ are equal to 0}. 
\end{equation}

Using the fact that $\beta_n$ does not charge extendable loops, for
each edge $E$ of the APD we have: 
$$\beta(E) \leftarrow \beta_n(E) = \sum_{K }
m^{(n)}_{K}\to \sum_{K} m_{K} = \sum_{K} \nu_{K}(\til K) =
\beta'(E),$$ 
where the sum ranges over open intervals $K \subset E$ covering all but
finitely many points of $E$, and the first equality follows from formulas \equ{eq: gets used} and \eqref{eq: simplifying 2}. In this case we have shown that $\beta =
\beta'$ corresponds to a transverse measure, and we are done. 
This establishes the statement under our simplifying assumptions \eqref{eq: simplifying 1} and \eqref{eq: simplifying 2}. 
\end{proof}

\subsection{Using boundary-marked surfaces}\label{subsec: using marked prongs}
We continue under assumption \eqref{eq: simplifying 1} but without assuming \eqref{eq: simplifying 2}. That is, the measures $\nu_n$ do not charge extendable loops, but some of the 
$\mathrm{e}_K^{\mathrm{b, t}} $ are positive, and the limit measures have atoms on the horizontal saddle connections at the endpoints of $K$. In order to treat this case, we will need to
record additional information about the 
invariance property satisfied by the measures $\nu_{\bar{K}}$. Informally, if we have escape of mass to a point $\xi$ which is either a singularity, or  the intersection of a horizontal saddle connection  with an edge of the refined APD,  we will want to record the angular sector of length $\pi$ at $\xi$, bounded by horizontal sides, to which the mass escaped. 
 Recording this additional information will give rise to continuous extensions of extendable loops. More precisely, after passing to subsequences, the information encoded in the numbers $\mathrm{e}_K^{\mathrm{b,t}}$ will yield a limiting  transverse measure as in Definition \ref{def: atomic transverse measure}  and a decoration as in Definition \ref{def: decoration}.

In order to formalize this, 
it will be useful to use boundary-marked surfaces \index{boundary marked surfaces} (see
\cite[\S 2.5]{eigenform}). 
Let $\check S \to \check q$ be a blown-up \index{blown-up surfaces} marked version of the marked
surface $S\to q$. Let $\xi \in
\Sigma$ and recall that $\check q$ replaces $\xi$ with a circle
parameterized by an angular variable $\theta$ taking values in $\R /
(2(a+1)\pi\Z) $, where $a$ is 
the order of $\xi$. Each $\theta$ will be called a {\em prong at $\xi$} which
can be thought of as the tangent direction of an infinitesimal line
segment of angle $\theta \mod 2\pi\Z$ ending at $\xi$. The
infinitesimal line is horizontal if and only if $\theta \in\pi\Z$. In a similar
way we can blow up nonsingular points of $S$, replacing them with a
circle parameterized by $\R/2\pi\Z$, and thus talk about the prongs at
a regular point (this corresponds to a singularity of order $a=0$). For each $k \in
\Z/(2(a+1)\pi \Z)$, and each $\xi$, 
two prongs at $\xi$ are
called {\em 
  bottom-adjacent} (resp. {\em top-adjacent}) if their angular
parameter belongs to the same interval $[k\pi, (k+1)\pi]$ with $k$
even (resp. odd), and {\em adjacent} \index{adjacent prongs} if they are either bottom- or
top-adjacent. For example, two horizontal prongs corresponding to two saddle connections meeting at a singular point $\xi$ on a bottom component of the boundary of a horizontal cylinder are bottom-adjacent to each other, and are also bottom-adjacent  to any prong moving upward from $\xi$ into the interior of the cylinder.

By definition of an APD, at each $\xi$
and each $k$, there is at least one edge $E$ with an endpoint in $(k
\pi, (k+1)\pi)$. 
We have compactified the line segments $K$ corresponding to $J, J_0,
J'$ as above by abstract points $x_K^{\mathrm{b}},x_K^{\mathrm{t}}$,
and these points map to points in $\check S$ by continuously extending the
embedding $\til K \to \check S$.  We will denote these points in
$\check S$ by their angular parameters $\theta_K^{\mathrm{b,t}} $ and call them {\em
  prongs of the APD}. Any point which is a regular point on the surface $\til q$, and which is on the interior of an edge $J$ (in the above notation, points of $J \sm J_0$),  will only be the endpoint of one top prong and one bottom prong, and the adjacency classes of these prongs will be singletons. In order to keep the notation consistent we will still refer to these endpoints as prongs, although we do not need to mark these points or blow up $\til q$ at these points.   
  
Since the APD contains no horizontal segments,
$\theta_K^{\mathrm{b,t}} \notin \pi \Z.$
Note that for $k$ even (resp. odd), 
all prongs of the APD with angular parameter in $(k\pi, (k+1)\pi)$ are
of form $\theta_K^{\mathrm{b}}$ (resp. $\theta_K^{\mathrm{t}}$).
In the preceding discussion (see formula \eqref{eq: loss of mass explicit}), we have associated to
each of these prongs an 
`escape of mass' quantity $\mathrm{e}_K^{\mathrm{b,t}}$. 

\medskip

\noindent

\begin{claim}\name{claim: claim 1}
\begin{enumerate}
\item
The weights of prongs of the APD only depend on their adjacency
class. More precisely, if $K, K'$ are edges of the APD with bottom-
(resp. top-) adjacent prongs
$\theta_{K}^{\mathrm{b}}, \theta_{K'}^{\mathrm{b}}$
(resp. $\theta_{K}^{\mathrm{t}}, \theta_{K'}^{\mathrm{t}}$) 
then $\mathrm{e}_K^{\mathrm{b}} = \mathrm{e}_{K'}^{\mathrm{b}}$
(resp. $\mathrm{e}_K^{\mathrm{t}} = \mathrm{e}_{K'}^{\mathrm{t}}$). 

\item
For any horizontal saddle connection $\sigma$, let
$\xi_1,
\xi_2$ in $S$ be consecutive points of $\sigma$ lying on edges of the
APD (the $\xi_i$ could either be singular points or interior points of edges of
the APD which are endpoints of subintervals $K$). For $i=1,2$, let 
$\theta_i^{(\sigma)}$ represent the two prongs of
$\sigma$ at $\xi_i$, and let $K_i$ (resp. $ L_i$) be intervals with
prongs at $\xi_i$ which are part of the APD, such that $\theta_{K_i}$
(resp. $\theta_{L_i}$) is bottom-
(resp. top-) adjacent to $\theta^{(\sigma)}_i$. See Figure \ref{fig: adjacency1}. Then 
\eq{eq: invariance endpoints}{
\mathrm{e}_{K_1}^{\mathrm{b}}+
\mathrm{e}_{L_1}^{\mathrm{t}} = 
\mathrm{e}_{K_2}^{\mathrm{b}}+ 
\mathrm{e}_{L_2}^{\mathrm{t}}.
}
\item
If a horizontal prong adjacent to $\theta_K^{\mathrm{b,t}}$ is on an
infinite critical leaf then $\mathrm{e}_K^{\mathrm{b,t}}=0.$
\end{enumerate}

\end{claim}

 \begin{figure}[h]
\begin{tikzpicture}[scale=1.0]

\def\xa{-1};
\def\xb{1};
\def\xc{2.6};

\def\ya{-0.5};
\def\yb{1.0};
\def\yc{2.2};
\def\yd{3};

\node (P0) at (\xb,\ya) [circle,draw,fill=black,inner sep=0pt,minimum size=1.2mm] {};
\node (P1) at (\xa,\yc) [circle,draw,fill=black,inner sep=0pt,minimum size=1.2mm,label=180:$ $] {};
\node (P2) at (\xc,\yd) [circle,draw,fill=black,inner sep=0pt,minimum size=1.2mm,label=0:$ $] {};

\draw (P0) -- node[left, pos=0.76] {$K_1$} (P1);
\draw (P0) -- node[left, pos=0.24] {$K_2$} (P1);
\draw (P0) -- node[right, pos=0.7] {$L_1$} (P2);
\draw (P0) -- node[right, pos=0.15] {$L_2$} (P2);
\draw (P1) -- (P2);

\path [name path=horizontal line]  (-1,\yb) -- (3,\yb);
\path [name path=first line]  (P0) -- (P1);
\path [name path=second line]  (P0) -- (P2);
\path [name intersections={of=horizontal line and first line, by=xx}]; 
\path [name intersections={of=horizontal line and second line, by=yy}]; 

\draw (-1,\yb) -- node[below]{$\sigma$} (3,\yb);
\node [left] (L) at (-1,\yb) {};

\node (X) at (xx) [circle,draw,fill=black,inner sep=0pt,minimum size=0.7mm, label=-100:$\xi_1$] {};
\node (Y)  at (yy) [circle,draw,fill=black,inner sep=0pt,minimum size=0.7mm, label=-70:$\xi_2$] {};

\end{tikzpicture}\ \ \ \ \ \ \ \ \

\caption{If $\sigma$ is a horizontal saddle connection passing through a polygon $P$ in a refined APD, then the total mass lost to the intersection points of $\sigma$ with edges of $P$ is the same. }
\label{fig: adjacency1} 
\end{figure}

\noindent {\em Proof of Claim \ref{claim: claim 1}.}
 Because adjacent prongs are in the same $(k\pi,(k+1)\pi)$ interval of
 direction, they are exchanged by $\opp_P$ and so statement (1)
 follows from \equ{eq: loss of mass explicit} and the invariance 
property \eqref{eq: invariance
  property} which the measures $\nu_K^{(n)}$ satisfy.
To see (2), note that the assumption
that $\xi_i$ are consecutive along $\sigma$ means that
$K_1, L_1$ are both
subintervals of an edge of the APD, and similarly for $K_2, L_2,$ where the two edges are edges of one polygon $P$, and with
$$\mathrm{opp}_P(K_1) = K_2 \ \ \text{ and } \ \mathrm{opp}_P(L_1) = L_2.$$ 
\combarak{I drew a figure in my notebook to explain the notation here
  and put a photo of it in
  the dropbox folder.}
By \equ{eq: loss of mass explicit} we have 
$$\mathrm{e}^{\mathrm{b}}_{K_i} + \mathrm{e}^{\mathrm{t}}_{L_i} = \lim_{n\to \infty}
\left(\nu^{(n)}_{K_i}(K_i^{\mathrm{b}})+\nu^{(n)}_{L_i}(L_i^{\mathrm{t}})
  \right)
- \left(\nu_{K_i}(K^{\mathrm{b}}_i)+\nu_{L_i}(L^{\mathrm{t}}_i) \right)$$ 
for each $i$, and \equ{eq:
  invariance endpoints} follows from the invariance property of each
of the $\nu^{(n)}_K$ on $K^{\mathrm{b}}_1, L^{\mathrm{t}}_1,
K^{\mathrm{b}}_1 \cup L^{\mathrm{t}}_1$. 

For (3), any critical leaf $\ell$ 
intersects some  interval $J$ of the APD
in its interior infinitely many times. If
$\mathrm{e}_K^{\mathrm{b,t}}\neq 0$ for a prong 
$\theta_K^{\mathrm{b, t}}$ adjacent to a prong defined by an endpoint
of $\ell$, we obtain 
infinitely many atoms in the
interior of $J$, and by the invariance property,
they all have the same $\nu_J$-mass. This contradicts the finiteness of the
measure $\nu_J$.  \hfill $\triangle$

\vspace{3mm}

We can now 
interpret extendable loops for boundary-marked
surfaces using our notion of adjacency:
an extendable loop $\ell$ is 
a loop formed as a concatenation of saddle connections which are either bottom-adjacent or top-adjacent at each of their endpoints. Moreover, if one of these saddle connections passes through a singular point of order zero,  a continuous extension $\check \ell$ of $\ell$ specifies a particular adjacency class. Thus each meeting point 
$\xi$ of consecutive horizontal saddle connections $\delta_i, \delta_{i+1}$ along a continuous extension $\check \ell $ of extendable loop $\ell$ represents an adjacency class at $\xi$  and we say
that $\check \ell$ {\em represents} this class. A primitive extendable loop can represent a given adjacency class at most once. By item (1) of Claim \ref{claim: claim 1}, the escape of mass parameters $\mathrm{e}_K^{\mathrm{b,t}}$ 
assign numbers $\mathrm{e}_{\mathcal{A}}$ to each bottom/top adjacency class
$\mathcal{A}$. 
The following claim shows that these numbers can be expressed in terms of extendable loops (and in fact, explicit continuous extensions of these loops).

\begin{claim}\name{claim: claim 2}
There is a finite collection $(\check \ell_s)$ of continuous extensions of primitive extendable loops and
finitely many positive real numbers $b_s$ such that for
each adjacency class $\mathcal{A}$, 
\begin{equation}\label{eq: sum}
\mathrm{e}_{\mathcal{A}} = 
\sum_{\check \ell_s \text{ represents } \mathcal{A}} b_s. 
\end{equation}

\end{claim}
\noindent {\em Proof of Claim \ref{claim: claim 2}.} 
The proof is by induction on the number of adjacency classes $\mathcal{A}$ for which $\mathrm{e}_{\mathcal{A}} \neq 0$. When this number is zero, we can take $s=0$ and the claim holds vacuously.
Choose the adjacency class $\mathcal{A}_1$ for which 
$$\mathrm{e}_{\mathcal{A}_1}  = \min\{\mathrm{e}_{\mathcal{A}}:
\mathrm{e}_{\mathcal{A}}>0\}. $$
For the induction step,
 we will  show
that $M_q$ contains a primitive extendable loop, such that all the adjacency
classes $\mathcal{A}$ represented by this loop satisfy
$\mathrm{e}_{\mathcal{A}} \geq \mathrm{e}_{\mathcal{A}_1}$. To see
this, let $\delta_1$ be an outgoing (i.e., right-pointing) prong in 
$\mathcal{A}_1$. According to item (3) of Claim \ref{claim: claim 1}, $ \delta_1$ is the initial point of a horizontal
saddle connection $\delta_1$. Let $\mathcal{A}_2^{\mathrm{b,t}}$ be the two adjacency classes of the
terminal point of $\delta_1$. Then according to \equ{eq: invariance
  endpoints}, at least one of
$\mathrm{e}_{\mathcal{A}^{\mathrm{b,t}}_2}^{\mathrm{b,t}}$ is
positive, and hence is
bounded below by $\mathrm{e}_{\mathcal{A}_1}.$ We label this adjacency class $\mathcal{A}_2$,  choose $\delta_2$ to lie on an
outgoing prong representing $\mathcal{A}_2$. Continuing, we find consecutive saddle connections 
$\delta_1, \delta_2, \ldots$, with turning angles $\pm \pi,$ whose endpoints represent adjacency classes $\mathcal{A}_i$ for which $\mathrm{e}_{\mathcal{A}_i} \geq \mathrm{e}_{\mathcal{A}_1}$. Eventually an adjacency class must be represented twice along this sequence, which means that some subset $\delta_{i_0}, \ldots, \delta_{j_0}$ of consecutive loops in  $\delta_1, \delta_2, \dots$  forms an extendable loop $\ell$, with $\mathrm{e}_{\mathcal{A}_i}\geq
\mathrm{e}_{\mathcal{A}_1}$ for each $i$. The adjacency classes $\mathrm{e}_{\mathcal{A}_i}$ equip $\ell$ with a continuous extension $\check \ell.$
We define 
$$b \df \min\{\mathrm{e}_{\mathcal{A}_i} : i_0 \leq i \leq j_0\}.$$ 
Replacing $\mathrm{e}_{\mathcal{A}_i}$ with $\mathrm{e}_{\mathcal{A}_i}-b$ for each $i \in \{i_0, \ldots, j_0\}$ we have a new collection with a smaller number of adjacency classes for which $\mathrm{e}_{\mathcal{A}} \neq 0.$ We can apply the induction hypothesis to this new collection, and obtain our statement by induction.  
\hfill $\triangle$

\medskip

\begin{proof}[Proof of Proposition \ref{prop: semicontinuity} under one simplifying assumption]
We continue with the notation used above, and we assume \eqref{eq: simplifying 1} but not \eqref{eq: simplifying 2}. We have that $\beta = \lim_{n\to \infty} \beta_n$ is a limit of cohomology classes corresponding to transverse measures which do not charge extendable loops, and $\beta'$ is the cohomology class corresponding to the  limiting transverse measure on the interior of edges of the refined APD (see the paragraph before equation \eqref{eq: simplifying 2}). 

We now show
\eq{eq: finally}{
\beta - \beta' = \sum_s b_s \left[\check \ell_s \right],
} 
where the $\check \ell_s$ and $b_s$ are provided by Claim \ref{claim: claim 2}. 
Indeed, it is enough to check this identity by evaluating 
on the paths $\alpha = \sigma_K$ introduced in the paragraph above \eqref{eq: gets used}, since such paths represent cycles which generate $H_1(M_q,
\Sigma_q)$. For such paths, \equ{eq: finally} is immediate from \equ{eq:
  gets used} and \equ{eq:
  sum}. 
Equation \equ{eq: finally} completes the proof of Proposition
\ref{prop: semicontinuity}, under the assumption that the $\beta_n$ do not charge extendable loops. 
\end{proof}

\subsection{Limits of extendable loops}
Our next goal is to remove assumption \eqref{eq: simplifying 1}. To this end we prove the following Proposition, which is another special case of Proposition \ref{prop: semicontinuity}: 
\begin{prop}\label{prop: one more simplification}
Suppose $\til q_n \to \til q$ is a convergent sequence of marked translation surfaces, corresponding to marking maps $\varphi_n, \varphi$. Suppose  $\nu_n$ is an atomic transverse measure of the form $a_n \theta^{(\ell_n)}$ on $M_{q_n},$ where $a_n>0, \, \ell_n$ is a primitive continuously extendable loop, and $\theta^{(\ell_n)}$ is a collection of measures on transverse arcs as in \eqref{eq: theta extendable}. For each $n$ let $\check \ell_n$ be a continuous extension of $\ell_n$ and let $\beta_n \in H^1(S, \Sigma; \R^2)$ be the  cohomology classes corresponding to $a_n \check \ell_n$. Assume that $\beta_n \to \beta.$ Then there is a decorated transverse measure $\bar \nu$ on $M_q$ such that $\beta = \beta_{\bar \nu}$.
\end{prop}

\begin{proof}
Let $\til q_n \to \til q$ be marked translation surfaces as in the preceding discussion, corresponding to marking maps $\varphi_n, \varphi$, which are chosen so that the transitions $\varphi_n \circ \varphi^{-1}$ are piecewise affine and map the edges of an APD on $M_{q}$ to edges of a triangulation on each $M_{q_n}$. Let $\check M_{q_n}, \check M_q$ be their boundary-marked versions; as before, these blow-ups allow us to speak of each of the adjacency classes represented by $\check \ell_n$, on all of the surfaces $\check M_{q_n}, \check M_q. $ Passing to a subsequence we can assume that the cyclically ordered list of adjacency classes represented by  $\check \ell_n$ is the same for all $n.$ We can also assume that $a_n \to a$ for some $a\geq 0.$

{\bf Case 1.}
Along a subsequence, the total horizontal length of $\ell_n$ is bounded on $M_{q_n}$. 
In this case, we will show that after passing to a subsequence: 
\setlist[enumerate,1]{label=(\roman*)}
\begin{enumerate}
    \item \label{item: i boundary}
for all  $n$, the continuous extensions $\varphi_n^{-1}(\check \ell_n)$ are homotopic to each other rel $\Sigma$; 
\item \label{item: iii boundary} if $a=0$ then the measures $\nu_n$ converge to $0. $
\item \label{item: ii boundary} there is a decorated  atomic transverse measure $\bar \nu$ on $M_q$, supported on a continuously extendable loop whose extension is also homotopic to $\check \ell_n$, such that $\beta = \beta_{\bar \nu}$.
\end{enumerate}

On $M_q$ there are only finitely many saddle connections of a bounded length. Using the blown-up translation surface structure, each of them is uniquely determined up to orientation with its initial prong. 
Each of the saddle connections $\delta^{(n)}_i$ comprising $\ell_n$ is a horizontal saddle connection of bounded length on each $M_{q_n}$, and the corresponding prongs converge to those on the boundary-marked surface $\check M_q$. This implies that up to taking subsequences, we can assume that for all large enough $n_1, n_2, $ the number of saddle connections comprising $\ell_{n_1}$ is the same as that for $\ell_{n_2},$ and for every $i$, the segments $ \varphi_{n_1}^{-1}(\delta^{(n_1)}_i), \ \varphi_{n_2}^{-1}(\delta^{(n_2)}_i)$ are homotopic to each other on $S$ rel $\Sigma.$ This means that for all large $n,$ the $\varphi_n^{-1}(\ell_n)$ are homotopic to each other, and after passing to a subsequence, \ref{item: i boundary} holds. We denote the homology classes represented by the eventual value of $\varphi_n^{-1}(\check \ell_n)$ by $\check \ell_\infty$. 

Given our refined APD on $M_q,$ we see that in Case 1, the number of intersection points of $\varphi \circ \varphi_n^{-1}(\ell_n)$ with each edge of each polygon is bounded above by a number independent of $n$. 
It follows from \eqref{eq: theta extendable} that the total mass of the pushforward $(\varphi \circ \varphi_n^{-1})_* \theta^{(\ell_n)}$ to each edge of the APD is bounded above, uniformly in $n$. From this 
\ref{item: iii boundary} follows. 

In particular, in proving \ref{item: ii boundary}, we can assume $a>0,$ and we can replace each $\nu_n$ by $\frac{1}{a_n}\, \nu_n $ to assume that $\nu_n = \theta^{(\ell_n)}$. We see from \ref{item: i boundary} that $\beta_n = [\check \ell_\infty]$ is a  constant sequence. 
We now show that on $M_q$, the loop $\varphi (\check \ell_\infty)$ is homotopic to a continuous extension of an extendable loop. Note that for each $i$, the path $\delta'_i \df \varphi \circ \varphi_n^{-1}(\delta_i)$ is a limit of horizontal saddle connections of bounded length on nearby surfaces, so is either homotopic to a horizontal saddle connection, or to a concatenation of several horizontal saddle connections. 
Passing to subsequences we can assume that on $M_q$,  $\delta'_i$ is homotopic rel $\Sigma$ to a concatenation of horizontal saddle  connections $\delta'_{i,1}, \ldots, \delta'_{i,j}$ for some $j = j(i)\geq 1$, and we need to show that the turning angle at the terminal endpoints of each of the saddle connections $\delta_{i,r}$ is $\pm \pi$. This is clear if $r=j(i)$, 
because 
 the terminal prong at $\delta_{i,j(i)}$ is the terminal prong of $\delta'_i$ and is represented by the  extendable loop $\varphi_n^{-1}(\ell_n)$. 
 If $1 \leq r < j(i)$ then on the surface $M_{q_n}$, the terminal endpoint of $\varphi_n \circ \varphi^{-1}(\delta_{i,r})$ is nearly on the interior of $\delta^{(n)}_i$, is either slightly below it or slightly above it, and is not very close to other singular points. Passing to subsequences we can assume that for all $i,j$, the direction from which $\varphi_n \circ \varphi^{-1}(\delta_{i,r})$ approaches the interior of $\delta_{i}^{(n)}$ is the same for all $n.$ This shows that $\varphi(\check \ell_\infty)$ is homotopic to the continuous extension of an extendable loop on $M_q$, which we denote by $\check \ell$.


 Define $\nu = \theta^{(\ell)},$ and let $\bar \nu$ be its decoration by $\check \ell.$ We find that 
$$\nu = \lim_{n \to \infty} (\varphi \circ 
\varphi_n^{-1})_* \nu_n \ \ \text{ and} \ \ \beta_{\bar \nu} = [\check \varphi^{-1}(\ell)] = [\check  \ell_\infty] = \lim_{n\to \infty} \beta_n. $$ 
This completes the proof in Case 1.  

\medskip

  {\bf Case 2.} The total length of $\ell_n$ is a sequence tending to infinity as $n\to \infty.$

For each edge $K$ of the refined APD on $M_q$ fixed above, we continue to denote its image under $\varphi_n \circ \varphi^{-1}$ by $K$, repeating our plea to the reader to overlook this inaccuracy. With this notation the measures $(\nu_n)|_K$ can all be considered as measures on the same interval $K.$ Let 
$$N_n(K) \df \#(\ell_n \cap K), \ \ \text{ and } \ N_n \df \max_K N_n(K).$$ 
Then in Case 2 we have $N_n \to \infty.$ Since the cohomology classes $\beta_n$ converge, the sequence of numbers $L_{q_n}(\beta_n)$ is bounded, and using \eqref{eq: intersection length} we find that \begin{equation}\label{eq: b to zero}
a= \lim_{n \to \infty} a_n =0.
\end{equation}
Fix an edge $K$ of the refined APD and simplify notation by writing $\eta_n \df (\nu_n)|_K$.
If $N_n(K)$ is a bounded sequence, then the measure $\theta_K^{(\ell_n)}$ is bounded, and hence by \eqref{eq: b to zero}, the sequence of measures $\eta_n$ tends to 0.

Now suppose 
\begin{equation} \label{eq: now suppose}
N_n(K) \to \infty.
\end{equation}
Passing to a subsequence (same subsequence for all $K$), we have that the measures $(\eta_n)$ converge to a limit measure $\eta_{\infty} = \eta_\infty(K)$ (perhaps with a smaller mass than the liminf of the masses of $\eta_n$), and the measures $(\eta_n)$ also determine the escape of mass parameters $\mathrm{e}_K^{\mathrm{b,t}}$ via formula \eqref{eq: loss of mass explicit}. In order to complete the proof, following the strategy used in Case 1, it suffices to prove the following:
\setlist[enumerate,1]{label=(\alph*)}
\begin{enumerate}
    \item \label{item: do not charge} The measures $\eta_\infty$ do not charge extendable loops, and the corresponding system of measures $(\eta_\infty(K))_K$ satisfies the invariance property.
    \item \label{item: adjacency only} 
    The numbers $\mathrm{e}_K^{\mathrm{b,t}}$ satisfy the conclusions of Claim \ref{claim: claim 1}. In particular, they depend only on the adjacency class represented by the bottom and top prongs of $K$ respectively, and thus all adjacency classes $\mathcal{A}$ of the refined APD are assigned numbers $\mathrm{e}_{\mathcal{A}}$. 
    \item \label{item: as in Claim 2} the collection $(\mathrm{e}_{\mathcal{A}})_{\mathcal{A}}$ satisfies the conclusion of Claim \ref{claim: claim 2}. 
\end{enumerate}

Here $\eta_\infty$ and $\mathrm{e}_J^{\mathrm{b,t}}$ correspond respectively to the non-atomic and atomic part of the limiting transverse  measure. 

To see that the measures $\eta_\infty$ do not charge extendable loops, recall that the interval $K$ is open and does not intersect horizontal saddle connections. This means that any measure supported on $K$ does not charge extendable loops.  

For the invariance property we argue similarly to the proof of Claim \ref{claim: claim}. Namely, for a compact subinterval $J$ of $K$ we let $K' = \mathrm{opp}_P(K)$ be the opposite interval on the refined APD on $M_q$, and let $\mathrm{opp}_P^{(n)}: J \to K'$ denote the map 
$$\mathrm{opp}_P^{(n)}(x) \df  \varphi \circ \varphi_n^{-1}(\mathrm{opp}_{P,n}(x)),$$
where $\mathrm{opp}_{P,n}(x)$ is the intersection (on $M_{q_n}$) of the horizontal line through $x$ with the edge opposite to $K.$ Note that this map might not be defined for given $n$, for some $x$ near the endpoints of $K$, but is defined for all $x \in J$ and all large enough $n,$ depending on $J$. 
With this notation we need to show that if $f$ is a continuous compactly supported function on $K'$, then 
\begin{equation}\label{eq: need Portmanteau}
\int f \circ \opp_P \, d\eta_\infty = \lim_{n \to \infty} \int f \circ \opp_P^{(n)} d\eta_n, 
\end{equation}
 where the map on the right-hand side is well-defined for all large enough $n$ depending on $\mathrm{opp}_P^{-1}(\mathrm{supp}(f)).$ The left hand side of \eqref{eq: need Portmanteau} is equal to $\lim_{n \to \infty} \int f \circ \mathrm{opp}_P d\eta_n$ by definition of $\eta_\infty$, and, using the fact that the maps $f \circ \mathrm{opp}_P^{(n)}$ converge to $f \circ \mathrm{opp}_P$ uniformly on $K$, this is equal to the right-hand side of \eqref{eq: need Portmanteau}. We have proved \ref{item: do not charge}. 

For the proof of Claim \ref{claim: claim 1}, we used the invariance property of the measures $\nu_K^{(n)}.$ The measures $\eta_n$ do not satisfy the invariance property but they almost do so. Namely, let $K$ be an edge of the refined APD, and let $K' = \mathrm{opp}_P(K)$ be the opposite edge for some polygon $P$. Any connected component of $\ell_n \cap P$ gives rise to two intersection points with $K$ and $K'$, which are images of each other under $\mathrm{opp}_{P}^{(n)}$, unless the connected component ends at a singular point at one of the endpoints of $K$ or $K'$. Thus, up to possible removing a bounded number of points from $K \cup K'$, corresponding to endpoints and their image under $\mathrm{opp}^{(n)}_P$,  the map $\mathrm{opp}_P^{(n)}$ induces a matching between points of $\ell_n \cap K$ and points of $\ell_n \cap K'. $ Removing the contributions of these points from the formula \eqref{eq: theta extendable} we modify $\eta_n$ slightly to obtain a new sequence of measures $\eta'_n$. In view of \eqref{eq: b to zero} and \eqref{eq: now suppose}, this new sequence $(\eta'_n)$ has the same limit $\eta_\infty$ and defines the same numbers $\mathrm{e}_K^{\mathrm{b,t}}.$ Thus, we can replace $\eta_n$ by $\eta'_n$ and the proof of Claim \ref{claim: claim 1} goes through to prove \ref{item: adjacency only}. Finally the proof of Claim \ref{claim: claim 2} only uses the conclusions of Claim \ref{claim: claim 1}, so we get \ref{item: as in Claim 2}.
 \end{proof}

\begin{proof}[Completing the proof of Proposition \ref{prop: semicontinuity}]
For each $n$ let $\bar \nu_n$ be a decorated version of $\nu_n$ so that $\beta_n = \beta_{\bar \nu_n}$. We write each $\nu_n$ as a sum $\nu'_n + \nu''_n,$ where $\nu'_n$ does not charge extendable loops, and $\nu''_n$ is a finite linear combination, with positive coefficients, of measures $\theta^{(\ell_{n,k})}$ supported on primitive continuously extendable paths $\ell_{n,k}$.
For each $\ell_{n,k}, $ the decoration $\bar \nu_n$ induces  a decoration of $\theta^{(\ell_{n,k})}$. This  amounts to choosing a continuous extension $\check \ell_{n,k}$ of each $\ell_{n,k}$. Decompose $\beta_n = \beta'_n + \beta''_n$ where $\beta'_n$ and $\beta''_n$ are the cohomology classes corresponding to $\nu'_n$ and $\nu''_n$. Further decompose 
$$\beta''_n =\sum_{k=1}^{m_n} \beta''_{n,k}, \ \ \ \text{ where } \ \ \beta''_{n,k} \df a_{n,k} \left[\check \ell_{n,k} \right],$$ 
for positive coefficients $a_{n,k}$ and where $m_n$ is the number of summands. The sequence  $m_n$ is bounded since the number of primitive countinuously extendable loops is bounded, and we will show below that  the cohomology classes $\beta'_n$ and $\beta''_{n,k}$ are all bounded. Assuming this, by passing to further subsequences we can assume that $m_n =m$ is constant, $\beta'_n \to \beta'$ and $ \beta''_{n,k} \to \beta_{k}''$ for $k=1, \ldots, m,$ where $\beta'+\sum_k \beta_k'' = \beta.$  The measures $\nu'_n$ satisfy \eqref{eq: simplifying 1}, and by the special case of Proposition \ref{prop: semicontinuity} established in \S \ref{subsec: using marked prongs} we have $\beta' \in C^+_{\til q}.$  The measures $\nu''_n$ are finite linear combinations of measures, each of which satisfies the conditions of Proposition \ref{prop: one more simplification}. By linearity, we obtain Proposition \ref{prop: semicontinuity}.

It remains to show that the sequences $(\beta'_n), \, (\beta''_{n,k})$ are bounded in $H^1(S, \Sigma; \R^2)$. For this it suffices to find a basis $v_1, \ldots, v_N$ of $H_1(S, \Sigma)$ such that the  sequences of evaluations
\begin{equation}\label{eq: sequence of evaluations}
    \left(\beta'_n(v_i)\right)_{n \in \N}, \ \ \ \left(\beta''_{n,k} (v_i)\right)_{n \in \N}
\end{equation}
are bounded, for each $i$ and each $k$. The basis we will use consists of the edges of a triangulation obtained from an APD, by adding horizontal diagonals to quadrilaterals. From continuity of $q \mapsto L_q$,  and from the convergence $\beta_n \to \beta$, we have that the terms appearing in \eqref{eq: intersection length} are bounded. In particular, if $v_i$ is an edge of the APD then the sequence $\left(\beta'_n(v_i) \right)_{n \in \N}$
is bounded. 
By definition, $\nu'_n$ assigns mass zero to the horizontal diagonals of the APD, and thus the sequence  $\left(\beta'_n(v_i)\right)_{n \in \N}$ is bounded for every edge $v_i$ of our triangulation. For the sequence 
$\left(\beta''_{n,k}(v_i) \right)_{n \in \N},$ as in the discussion in Case 2 of the proof of Proposition \ref{prop: one more simplification}, we have that the number of intersections of $\check \ell_{n,k}$ with an edge of the triangulation is bounded above by $C \ell_{n,k}$ for some $C>0$. 
Thus the boundedness of \eqref{eq: intersection length} implies that 
$\left(\beta'_n(v_i)\right)_{n \in \N}$ is bounded. 
\end{proof}

\begin{proof}[Completing the proof of Corollary \ref{cor: total variation continuous}]
The proof is almost identical to the one we gave in \S \ref{subsec: polygonal tremors}, with the following modifications. 
In \S \ref{subsec: polygonal tremors},  assumption \eqref{eq: star} was used in order to be able to apply Proposition \ref{prop: semicontinuity}; we now have Proposition \ref{prop: semicontinuity} without this assumption. Additionally, we invoked \eqref{eq: each edge} and non-atomicity  in order to say that  the limiting transverse measures $\nu^\pm_\infty$ 
satisfy $ \lim_{n \to \infty} \beta_{\nu^\pm_n} = \beta_{\nu_\infty^\pm},$ implying
\begin{equation}\label{eq: this continuity}
\lim_{n \to \infty} L_q(\beta_{\nu_n^\pm}) = L_q(\beta_{\nu_\infty^\pm}),
 \end{equation}
which we needed in \eqref{eq: from that}.
 In our case the limiting measures $\nu^\pm_\infty$ might have atoms, but \eqref{eq: this continuity} still holds by \eqref{eq: a continuity property}. 
 Finally, the minimization property of the Hahn decomposition was used in connection with formula \eqref{eq: formula for L}. 
Here we use the same minimization property, in connection with formula \eqref{eq: intersection length}. 
\end{proof}

\printindex

\end{document}